\documentclass[11pt]{book}
\usepackage{graphicx}
\usepackage{amsmath, amscd, amssymb, amsthm}
\usepackage[subnum]{cases}
\usepackage{changepage}
\usepackage{marginnote}
\usepackage{hyperref}
\usepackage{cleveref}
\usepackage[all]{xy}
\usepackage{tabularx}
\usepackage{ltablex}
\usepackage{longtable}
\usepackage[T1]{fontenc}
\hypersetup{colorlinks,linkcolor={black},citecolor={black}} 
\usepackage[margin=1.1in]{geometry}
\usepackage{pdflscape}
\usepackage{mathpazo}

\newtheorem{theorem}{Theorem}[section]
\newtheorem{lemma}[theorem]{Lemma}

\newtheorem{proposition}[theorem]{Proposition}

\newtheorem{conjecture}[theorem]{Conjecture}

\theoremstyle{definition}
\newtheorem{definition}[theorem]{Definition}
\newtheorem{remark}[theorem]{Remark}
\numberwithin{equation}{section}
\newtheorem{example}[theorem]{Example}
\newtheorem{assumption}[theorem]{Assumption}
\newtheorem{setting}[theorem]{Setting}

\usepackage{fancyhdr}
\begin{document}

\normalfont

\title{Topologization and Functional Analytification II: $\infty$-Categorical Motivic Constructions for Homotopical Contexts}
\author{Xin Tong}
\date{}

\maketitle

\subsection*{Abstract}
This is our second scope of the consideration on the corresponding topologization and the corresponding functional analytification. We will focus on the corresponding functorial and motivic constructions in our current consideration. We consider topological motivic derived $I$-adic and derived $(p,I)$-adic cohomologies through derived de Rham complexes of Bhatt, Guo, Illusie, Morrow, Scholze, Frobenius sheaves over Robba rings of Kedlaya-Liu in certain derived $I$-adic and derived $(p,I)$-adic geometric context as what we defined for Bambozzi-Ben-Bassat-Kremnizer $\infty$-prestacks in our previous work in this series. The foundation we will work on will be based on the work of Bambozzi-Ben-Bassat-Kremnizer, Ben-Bassat-Mukherjee, Clausen-Scholze and Kelly-Kremnizer-Mukherjee, in order to promote the construction to even more general homotopical and $\infty$-categorical contexts. This gives us the chance to construct the functional analytic derived prismatic cohomology and derived preperfectoidizations, as well as the functional analytic derived logarithmic prismatic cohomology and derived logarithmic preperfectoidizations after Bhatt-Scholze and Koshikawa, in the framework of Bambozzi-Ben-Bassat-Kremnizer, Ben-Bassat-Mukherjee, Clausen-Scholze and Kelly-Kremnizer-Mukherjee.


\newpage

\footnote{\textit{Keywords and Phrases}: $E_\infty$-de Rham Cohomologies, Topological $(\infty,1)$-Rings, $(\infty,1)$-Bornological Rings, $(\infty,1)$-Ind-Fr\'echet Rings, $\infty$-Categorical Motivic Constructions for $\infty$-Categorical Homotopical Contexts.}

\newpage

\tableofcontents

\newpage

\chapter{Introduction}

\section{Introduction}

\indent In our previous paper \cite{12XT1} on the intrinsic morphisms of the corresponding topological, bornological and ind-Fr\'echet rings, we discussed some generalization of the corresponding contexts of \cite{12Huber1} and \cite{12Huber2}, along two directions. One is along some answer to a question of Kedlaya in \cite[Appendix 5]{12Ked1} namely the corresponding affinoid morphisms, and the other is along some stacky understanding of these such as in \cite{12Dr1}, \cite{12Dr2}, \cite{12R}.\\

\indent Under some certain foundation, \cite{12Pau1} actually studied many parallel considerations of our project. The parallel  considerations are the corresponding ind-Banach analytic spaces, the corresponding \'etale and pro-\'etale sites of ind-Banach analytic spaces, the corresponding de Rham stacks of such analytic spaces, global de Rham cohomology and comparison over ind-Banach spaces, and algebraic derived de Rham complex over the corresponding ind-Banach spaces. And the corresponding descent over the corresponding derived ind-Banach spaces. We will consider the corresponding topologization and functional analytification of the cohomology theories (topological derived de Rham after \cite{12B1}, \cite{12Bei},  \cite{12GL}, \cite{12Ill1}, \cite{12Ill2} topological logarithmic derived de Rham after Gabber such as in \cite{12B1} and \cite{12O}, certainly functional analytification construction after \cite{12KL1} and \cite{12KL2} and etc).  We would like to mention that our project is also largely inspired by the corresponding development in \cite{12CS1}, \cite{12CS2}, \cite{12FS}, \cite{12Sch2} where the corresponding $v$-stacky consideration is extensively developed. Note that $v$-stacks are very significant analytic spaces, especially in the corresponding geometrization of local Langlands correspondence in \cite{12FS}, where complicated derived categories are constructed through deep foundations in \cite{12Sch2} and \cite{12CS1}, \cite{12CS2} on condensed sets. One certainly believes that the corresponding foundations by using normed sets in \cite{12BBBK}, \cite{12BBK}, \cite{BBM}, \cite{12BK} and \cite{KKM} will also have potential applications to certainly motivic and functorial constructions in analytic geometry which are parallel to those given in \cite{M} by using the foundation from \cite{12CS1}, \cite{12CS2}. Note that the programs in \cite{12B1}, \cite{12BMS} and \cite{12BS} have already indicated that working with the corresponding correct simplicial spaces in the analytic setting is not only a pure generalization but also a very significant point around even the non simplicial analytic spaces especially in the corresponding singular situations.\\

\indent Then we follow \cite{12An1}, \cite{12An2}, \cite{12B1}, \cite{12Bei}, \cite{12BS}, \cite{12G1}, \cite{12GL}, \cite{12Ill1}, \cite{12Ill2}, \cite{12Qui}  to extend the corresponding discussion to certainly spaces, we mainly have focused on the corresponding rigid analytic spaces, pseudorigid spaces in some explicit way. And we then extend the corresponding discussion to the corresponding derived setting. The corresponding derived setting is literally following \cite{12An1}, \cite{12An2}, \cite{12B1}, \cite{12B2}, \cite{12Bei}, \cite{12BMS}, \cite{12BS}, \cite{12G1}, \cite{12GL}, \cite{12Ill1}, \cite{12Ill2}, \cite{12Qui} and the corresponding derived logarithmic setting is literally following \cite{12B1}, \cite{12Ko1} and \cite{12O}\footnote{The corresponding \cite{12Ko1} has already mentioned the corresponding derived logarithmic prismatic cohomology.}. \\

\indent We promote the construction from Bhatt, Illusie, Guo, Morrow, Scholze, Gabber \cite{12B1}, \cite{12G1}, \cite{12Ill1}, \cite{12Ill2}, \cite{12BMS} \cite{12O} in the context of topological derived de Rham complexes and topological derived logarithmic de Rham complexes, the construction from Nicolaus-Scholze \cite{12NS} in the context of derived THH, TP and TC on the level of $\mathbb{E}_1$-rings , the construction from Kedlaya-Liu in the context of the derived Robba rings and derived Frobenius sheaves \cite{12KL1} and \cite{12KL2}, the construction from Bhatt-Scholze, Koshikawa and in the context of derived prismatic cohomology and derived logarithmic prismatic cohomology \cite{12BS} and \cite{12Ko1} to the level of $\infty$-categorical functional analytic level after Lurie \cite{12Lu1}, \cite{12Lu2} in the stable $\infty$-cateogory of derived $J$-complete simplicial commutative algebras, Bambozzi-Ben-Bassat-Kremnizer \cite{12BBBK}, Ben-Bassat-Mukherjee \cite{BBM}, Bambozzi-Kremnizer \cite{12BK}, Clausen-Scholze \cite{12CS1} \cite{12CS2} and Kelly-Kremnizer-Mukherjee \cite{KKM} in the stable $\infty$-cateogory of simplicial functional analytic commutative algebras.\\

\indent Furthermore we promote the construction from Bhatt, Illusie, Guo, Morrow, Scholze, Gabber \cite{12B1}, \cite{12G1}, \cite{12Ill1}, \cite{12Ill2}, \cite{BMS2} \cite{12O} in the context of topological derived de Rham complexes and topological derived logarithmic de Rham complexes, the construction from Nicolaus-Scholze \cite{12NS} in the context of derived THH, TP and TC on the level of $\mathbb{E}_1$-rings, the construction from Kedlaya-Liu in the context of the derived Robba rings and derived Frobenius sheaves \cite{12KL1} and \cite{12KL2}, the construction from Bhatt-Scholze, Koshikawa and in the context of derived prismatic cohomology and derived logarithmic prismatic cohomology \cite{12BS} and \cite{12Ko1} not only to the level of $\infty$-categorical functional analytic level after \cite{12Lu1}, \cite{12Lu2}, but also to the corresponding $(\infty,1)$-ringed toposes level after Lurie \cite{12Lu1}, \cite{12Lu2} in the $\infty$-category of $\infty$-ringed toposes, Bambozzi-Ben-Bassat-Kremnizer \cite{12BBBK}, Ben-Bassat-Mukherjee \cite{BBM}, Bambozzi-Kremnizer \cite{12BK}, Clausen-Scholze \cite{12CS1} \cite{12CS2} and Kelly-Kremnizer-Mukherjee \cite{KKM} in the stable $\infty$-cateogory of $\infty$-functional analytic ringed toposes.\\

\indent We then work in the noncommutative setting after \cite{Kon1}, \cite{Ta}, \cite{KR1} and \cite{KR2}, with some philosophy rooted in some noncommtative motives and the corresponding nonabelian applications in noncommutative analytic geometry in the derived sense, and the noncommutative analogues of the corresponding Riemann hypothesis and the corresponding Tamagawa number conjectures, and so on. The issue is certainly that the usual Frobenius map looks somehow strange, which reminds us of the fact that actually we need to consider really large objects such as the corresponding Topological Hochschild Homologies and the corresponding nearby objects, in order to define the corresponding analogues of the corresponding prismatic cohomology through THH and nearby objects, the corresponding noncommutative 'perfectoidizations'. Here we choose to consider \cite{12NS} in order to apply the constructions to certain $\infty$-rings, which we will call them Fukaya-Kato analytifications from \cite{12FK} as a noncommutative analog of the constructions in \cite{BBM}.\\

\indent We first consider the following list of motivic constructions in this paper which we hope to establish in very general topological and Banach setting, where one can actually believe that the constructions are somehow parallel and in some sense equivalent even on the $\infty$-categorical level\footnote{We are actually talking about the spaces which might not be smooth.}:

\begin{setting}\mbox{}\\
A. Derived Topological de Rham complexes and Derived Topological Logarithmic de Rham complexes of simplicial derived $I$-complete rings and over pro-\'etale site, after \cite{12B1}, \cite{12B2}, \cite{12Bei}, \cite{12BMS}, \cite{12BS}, \cite{12G1}, \cite{12GL}, \cite{12Ill1}, \cite{12Ill2}, \cite{12O};\\
B. $\mathcal{O}\mathrm{B}_\mathrm{dR}$-sheaves and $\mathcal{O}\mathrm{B}_\mathrm{dR,log}$-sheaves over general analytic adic spaces\footnote{These are not necessarily $p$-adic Tate. But as long as $p$ is topological nilpotent we do have a basis consisting of perfectoid subdomains as in \cite[Theorem 2.9.9, Remark 2.9.10]{12Ked1}.} after \cite{12DLLZ2}, \cite{12Sch2};\\
C. $\varphi$-$\widetilde{C}_X$-sheaves, relative $B$-pairs over general analytic adic spaces after \cite{12KL1} and \cite{12KL2};\\
D. Derived Prismatic Cohomology and Derived Logarithmic Prismatic Cohomology, after \cite{12BS}, \cite{12Ko1}.\\
\end{setting}

\begin{remark}
One could also consider the corresponding derived $I$-complete version\footnote{In fact we are taking derived $I$-completion of the spectra instead of the $I$-completion of the spectra in $\infty$-category as in  \cite[Chapter 7.3]{12Lu2}.} of the corresponding left Kan extended THH and HH $\mathbb{E}_1$-ring spectra of derived $I$-topological $\mathbb{E}_1$-ring spectra as in \cite{12NS} and \cite{KKM}. In certain situation, this should be able to be compared to the constructions above.
\end{remark}

\newpage

\section{Preliminary}

\subsection{The Rings in Algebraic Topology}

We consider some very general topological rings. We will consider the corresponding Huber's adic rings, namely we allow the corresponding rings to contain some adic open subrings which have then the corresponding linear topology. And we consider more general simplicial topological rings which are more general adic in some derived sense.

\begin{setting}
We consider the corresponding category $C_{\infty,\mathrm{E}_\infty}$ of all the $\mathbb{E}_\infty$ rings\footnote{These are the corresponding ring objects in the algebraic topology, see \cite[Chapter 7]{12Lu1}. So we want to consider the corresponding interesting objects in classical algebraic topology such as in \cite{12MP} and \cite{12N}.}. We then consider the corresponding category $C_{\infty,\mathrm{E}_\infty,\mathrm{rat},I}$ of all the $\mathbb{E}_\infty$ rings which have subrings being derived $I$-adically complete\footnote{In fact we are taking derived $I$-completion of the spectra instead of the $I$-completion of the spectra in $\infty$-category as in  \cite[Chapter 7.3]{12Lu2}.}, where $I$ is finite generating set coming from elements in $\pi_0$. And we have the smaller category $C_{\infty,\mathrm{E}_\infty,\mathrm{int},I}$ of all the all the $\mathbb{E}_\infty$ rings being derived $I$-adically complete.	
\end{setting}

\begin{example}
We have many interesting ring spectra from classical algebraic topology. One can further takes the corresponding derived $p$-completion to achieve many interesting objects\footnote{In fact we are taking derived $I$-completion of the spectra instead of the $I$-completion of the spectra in $\infty$-category as in \cite[Chapter 7.3]{12Lu2}. This should be different in more general setting from the consideration in algebraic topology and more general $\infty$-category theory.}. For more discussion see \cite[Chapter 5-13]{12MP} and \cite{12N}.	
\end{example}

\begin{example}
As also discussed in \cite[Chapter 14-19]{12MP} on the level of just model categories or more general construction in \cite[Chapter 1.3]{12Lu1} one considers the stable homotopy category $D(R)$ of some ring spectra, namely the corresponding $\infty$-enhancement of the corresponding classical triangulated categories. One considers the corresponding derived complete objects in this $\infty$-category, then one will have some interesting $\mathbb{E}_\infty$-spectra. 	
\end{example}

\begin{example}
We then have the example that is the corresponding rigid analytic affinoids in \cite[Definition 4.1]{12Ta}. They have open subrings which are actually derived $p$-adically complete.	
\end{example}

\begin{example}
We then have the example that is the corresponding pseudorigid analytic affinoids as in \cite[Definition 3.1]{12Bel1},  \cite{12Bel2} and \cite[Definition 4.1]{12L}. They have open subrings which are actually derived $t$-adically complete for some specific pseudouniformizer $t$.
\end{example}

\begin{example}
We now fix a bounded morphism of simplicial adic rings $A\rightarrow B$ over $A^*$ where $A^*$ contains a corresponding ring of definition $A^*_0$ which is complete with respect to the $(p,I)$-topology and we assume that $A$ is adic and we assume that $(A^*_0,I)$ is a prism in \cite{12BS} namely we at least require that the corresponding $\delta$-structure on the corresponding ring will induce the map $\varphi(.):=.^p+p\delta(.)$ such that we have the situation where $p\in (I,\varphi(I))$. For $A$ or $B$ respectively we assume this contains a subring $A_0$ or $B_0$ (over $A_0^*$) respectively such that we have $A_0$ or $B_0$ respectively is derived complete with respect to the corresponding derived $(p,I)$-topology and we assume that $B=B_0[1/f,f\in I]$ (same for $A$). All the adic rings are assumed to be open mapping. We use the notation $d$ to denote a corresponding primitive element as in \cite[Section 2.3]{12BS} for $A^*$. 
\end{example}

\indent We also consider the corresponding simplicial prelog rings as in \cite[Chapter 6]{12B1}, and also consider the corresponding application of such algebraic derived constructions to some derived $I$-adic rings carrying the corresponding prelog structures such as in \cite[Chapter 2]{12DLLZ1}.

\begin{setting}
We consider the corresponding category $C_{\infty,\mathrm{E}_\infty,\mathrm{prelog}}$ of all the $\mathbb{E}_\infty$ rings carrying the corresponding prelog structures. We then consider the corresponding category $C_{\infty,\mathrm{E}_\infty,\mathrm{rat},I,\mathrm{prelog}}$ of all the $\mathbb{E}_\infty$ rings carrying prelog structures which have subrings being derived $I$-adically complete\footnote{In fact again we are taking derived $I$-completion of the spectra instead of the $I$-completion of the spectra in $\infty$-category as in  \cite[Chapter 7.3]{12Lu2}.}, where $I$ is finite generating set coming from elements in $\pi_0$. And we have the smaller category $C_{\infty,\mathrm{E}_\infty,\mathrm{int},I,\mathrm{prelog}}$ of all the all the $\mathbb{E}_\infty$ rings carrying prelog structures being derived $I$-adically complete.	
\end{setting}

\begin{example}
The first example is the corresponding rigid analytic affinoids in \cite[Definition 4.1]{12Ta}. They have open subrings which are actually derived $p$-adically complete. Then one adds the corresponding prelog structures from \cite[Chapter 2]{12DLLZ1}. 
\end{example}

\begin{example}
The second example is the corresponding pseudorigid analytic affinoids as in \cite[Definition 3.1]{12Bel1}, \cite{12Bel2} and \cite[Definition 4.1]{12L}. They have open subrings which are actually derived $t$-adically complete for some specific pseudouniformizer $t$. Then one adds the corresponding prelog structures from \cite[Chapter 2]{12DLLZ1}. 
\end{example}

\begin{example}
We now fix a bounded morphism of logarithmic simplicial adic rings $(A,M)\rightarrow (B,N)$ over $A^*$ where $A^*$ contains a corresponding ring of definition $A^*_0$ which is complete with respect to the $(p,I)$-topology and we assume that $A$ is adic and we assume that $(A^*_0,I)$ is a prism namely we at least require that the corresponding $\delta$-structure on the corresponding ring will induce the map $\varphi(.):=.^p+p\delta(.)$ such that we have the situation where $p\in (I,\varphi(I))$. For $(A,M)$ or $(B,N_0)$ respectively we assume this contains a subring $(A_0,M_0)$ or $(B_0,N_0)$ (over $A_0^*$) respectively such that we have $(A_0,M_0)$ or $(B_0,N_0)$ respectively is derived complete with respect to the corresponding derived $(p,I)$-topology and we assume that $B=B_0[1/f,f\in I]$ (same for $A$). All the adic rings are assumed to be open mapping. We use the notation $d$ to denote a corresponding primitive element as in \cite[Section 2.3]{12BS} for $A^*$. \\
\end{example}

\newpage

\section{Notations on $\infty$-Categories of $\infty$-Rings}

\begin{center}
\begin{tabularx}{\linewidth}{lX}
Notation & Description\\
\hline
$C_{\infty,\mathrm{E}_\infty}$ &$(\infty,1)$-category of $\mathbb{E}_\infty$ rings.\\
$C_{\infty,\mathrm{E}_\infty,\mathrm{rat},I}$ &$(\infty,1)$-category of $\mathbb{E}_\infty$ rings in the sense that the rings contain some derived rings rationally.\\
$C_{\infty,\mathrm{E}_\infty,\mathrm{int},I}$ &$(\infty,1)$-category of $\mathbb{E}_\infty$ rings which are derived adic rings.\\
$C_{\infty,\mathrm{E}_\infty,\text{prelog}}$ &$(\infty,1)$-category of $\mathbb{E}_\infty$ logarithmic rings.\\
$C_{\infty,\mathrm{E}_\infty,\mathrm{rat},I,\text{prelog}}$ &    $(\infty,1)$-category of $\mathbb{E}_\infty$ logarithmic rings in the sense that the rings contain some derived rings rationally.\\
$C_{\infty,\mathrm{E}_\infty,\mathrm{int},I,\text{prelog}}$ &    $(\infty,1)$-category of $\mathbb{E}_\infty$ logarithmic rings which are derived adic rings.\\

$\mathrm{Object}_{\mathrm{E}_\infty\mathrm{commutativealgebra},\mathrm{Simplicial}}(\mathrm{IndSNorm}_R)$& $(\infty,1)$-category of $\mathbb{E}_\infty$ commutative algebra objects in the $(\infty,1)$-category of $\mathbb{E}_\infty$ ind-seminormed modules over $R$. \\
$\mathrm{Object}_{\mathrm{E}_\infty\mathrm{commutativealgebra},\mathrm{Simplicial}}(\mathrm{Ind}^m\mathrm{SNorm}_R)$& $(\infty,1)$-category of $\mathbb{E}_\infty$ commutative algebra objects in the $(\infty,1)$-category of $\mathbb{E}_\infty$ monomorphic ind-seminormed modules over $R$. \\
$\mathrm{Object}_{\mathrm{E}_\infty\mathrm{commutativealgebra},\mathrm{Simplicial}}(\mathrm{IndNorm}_R)$& $(\infty,1)$-category of $\mathbb{E}_\infty$ commutative algebra objects in the $(\infty,1)$-category of $\mathbb{E}_\infty$ ind-normed modules over $R$. \\
$\mathrm{Object}_{\mathrm{E}_\infty\mathrm{commutativealgebra},\mathrm{Simplicial}}(\mathrm{Ind}^m\mathrm{Norm}_R)$& $(\infty,1)$-category of $\mathbb{E}_\infty$ commutative algebra objects in the $(\infty,1)$-category of $\mathbb{E}_\infty$ monomorphic ind-normed modules over $R$. \\
$\mathrm{Object}_{\mathrm{E}_\infty\mathrm{commutativealgebra},\mathrm{Simplicial}}(\mathrm{IndBan}_R)$& $(\infty,1)$-category of $\mathbb{E}_\infty$ commutative algebra objects in the $(\infty,1)$-category of $\mathbb{E}_\infty$ ind-Banach modules over $R$. \\
$\mathrm{Object}_{\mathrm{E}_\infty\mathrm{commutativealgebra},\mathrm{Simplicial}}(\mathrm{Ind}^m\mathrm{Ban}_R)$& $(\infty,1)$-category of $\mathbb{E}_\infty$ commutative algebra objects in the $(\infty,1)$-category of $\mathbb{E}_\infty$ monomorphic ind-Banach modules over $R$. \\
$\mathrm{Object}_{\mathrm{E}_\infty\mathrm{commutativealgebra},\mathrm{Simplicial}}(\mathrm{IndSNorm}_{\mathbb{F}_1})$& $(\infty,1)$-category of $\mathbb{E}_\infty$ commutative algebra objects in the $(\infty,1)$-category of $\mathbb{E}_\infty$ ind-seminormed sets over $\mathbb{F}_1$. \\
$\mathrm{Object}_{\mathrm{E}_\infty\mathrm{commutativealgebra},\mathrm{Simplicial}}(\mathrm{Ind}^m\mathrm{SNorm}_{\mathbb{F}_1})$& $(\infty,1)$-category of $\mathbb{E}_\infty$ commutative algebra objects in the $(\infty,1)$-category of $\mathbb{E}_\infty$ monomorphic ind-seminormed sets over $\mathbb{F}_1$. \\
$\mathrm{Object}_{\mathrm{E}_\infty\mathrm{commutativealgebra},\mathrm{Simplicial}}(\mathrm{IndNorm}_{\mathbb{F}_1})$& $(\infty,1)$-category of $\mathbb{E}_\infty$ commutative algebra objects in the $(\infty,1)$-category of $\mathbb{E}_\infty$ ind-normed sets over $\mathbb{F}_1$. \\
$\mathrm{Object}_{\mathrm{E}_\infty\mathrm{commutativealgebra},\mathrm{Simplicial}}(\mathrm{Ind}^m\mathrm{Norm}_{\mathbb{F}_1})$ & $(\infty,1)$-category of $\mathbb{E}_\infty$ commutative algebra objects in the $(\infty,1)$-category of $\mathbb{E}_\infty$ monomorphic ind-normed sets over $\mathbb{F}_1$. \\
$\mathrm{Object}_{\mathrm{E}_\infty\mathrm{commutativealgebra},\mathrm{Simplicial}}(\mathrm{IndBan}_{\mathbb{F}_1})$& $(\infty,1)$-category of $\mathbb{E}_\infty$ commutative algebra objects in the $(\infty,1)$-category of $\mathbb{E}_\infty$ ind-Banach sets over $\mathbb{F}_1$. \\
$\mathrm{Object}_{\mathrm{E}_\infty\mathrm{commutativealgebra},\mathrm{Simplicial}}(\mathrm{Ind}^m\mathrm{Ban}_{\mathbb{F}_1})$& $(\infty,1)$-category of $\mathbb{E}_\infty$ commutative algebra objects in the $(\infty,1)$-category of $\mathbb{E}_\infty$ monomorphic ind-Banach sets over $\mathbb{F}_1$. \\

 $\mathrm{Object}_{\mathrm{E}_\infty\mathrm{commutativealgebra},\mathrm{Simplicial}}(\mathrm{IndSNorm}_R)^\text{prelog}$& $(\infty,1)$-category of $\mathbb{E}_\infty$ commutative algebra objects in the $(\infty,1)$-category of $\mathbb{E}_\infty$ ind-seminormed modules over $R$ carrying the corresponding logarithmic structures. \\
$\mathrm{Object}_{\mathrm{E}_\infty\mathrm{commutativealgebra},\mathrm{Simplicial}}(\mathrm{Ind}^m\mathrm{SNorm}_R)^\text{prelog}$& $(\infty,1)$-category of $\mathbb{E}_\infty$ commutative algebra objects in the $(\infty,1)$-category of $\mathbb{E}_\infty$ monomorphic ind-seminormed modules over $R$ carrying the corresponding logarithmic structures. \\
$\mathrm{Object}_{\mathrm{E}_\infty\mathrm{commutativealgebra},\mathrm{Simplicial}}(\mathrm{IndNorm}_R)^\text{prelog}$& $(\infty,1)$-category of $\mathbb{E}_\infty$ commutative algebra objects in the $(\infty,1)$-category of $\mathbb{E}_\infty$ ind-normed modules over $R$ carrying the corresponding logarithmic structures. \\
$\mathrm{Object}_{\mathrm{E}_\infty\mathrm{commutativealgebra},\mathrm{Simplicial}}(\mathrm{Ind}^m\mathrm{Norm}_R)^\text{prelog}$& $(\infty,1)$-category of $\mathbb{E}_\infty$ commutative algebra objects in the $(\infty,1)$-category of $\mathbb{E}_\infty$ monomorphic ind-normed modules over $R$ carrying the corresponding logarithmic structures. \\
$\mathrm{Object}_{\mathrm{E}_\infty\mathrm{commutativealgebra},\mathrm{Simplicial}}(\mathrm{IndBan}_R)^\text{prelog}$& $(\infty,1)$-category of $\mathbb{E}_\infty$ commutative algebra objects in the $(\infty,1)$-category of $\mathbb{E}_\infty$ ind-Banach modules over $R$ carrying the corresponding logarithmic structures. \\
$\mathrm{Object}_{\mathrm{E}_\infty\mathrm{commutativealgebra},\mathrm{Simplicial}}(\mathrm{Ind}^m\mathrm{Ban}_R)^\text{prelog}$& $(\infty,1)$-category of $\mathbb{E}_\infty$ commutative algebra objects in the $(\infty,1)$-category of $\mathbb{E}_\infty$ monomorphic ind-Banach modules over $R$ carrying the corresponding logarithmic structures. \\
$\mathrm{Object}_{\mathrm{E}_\infty\mathrm{commutativealgebra},\mathrm{Simplicial}}(\mathrm{IndSNorm}_{\mathbb{F}_1})^\text{prelog}$& $(\infty,1)$-category of $\mathbb{E}_\infty$ commutative algebra objects in the $(\infty,1)$-category of $\mathbb{E}_\infty$ ind-seminormed sets over $\mathbb{F}_1$ carrying the corresponding logarithmic structures. \\
$\mathrm{Object}_{\mathrm{E}_\infty\mathrm{commutativealgebra},\mathrm{Simplicial}}(\mathrm{Ind}^m\mathrm{SNorm}_{\mathbb{F}_1})^\text{prelog}$& $(\infty,1)$-category of $\mathbb{E}_\infty$ commutative algebra objects in the $(\infty,1)$-category of $\mathbb{E}_\infty$ monomorphic ind-seminormed sets over $\mathbb{F}_1$ carrying the corresponding logarithmic structures. \\
$\mathrm{Object}_{\mathrm{E}_\infty\mathrm{commutativealgebra},\mathrm{Simplicial}}(\mathrm{IndNorm}_{\mathbb{F}_1})^\text{prelog}$& $(\infty,1)$-category of $\mathbb{E}_\infty$ commutative algebra objects in the $(\infty,1)$-category of $\mathbb{E}_\infty$ ind-normed sets over $\mathbb{F}_1$ carrying the corresponding logarithmic structures. \\
$\mathrm{Object}_{\mathrm{E}_\infty\mathrm{commutativealgebra},\mathrm{Simplicial}}(\mathrm{Ind}^m\mathrm{Norm}_{\mathbb{F}_1})^\text{prelog}$& $(\infty,1)$-category of $\mathbb{E}_\infty$ commutative algebra objects in the $(\infty,1)$-category of $\mathbb{E}_\infty$ monomorphic ind-normed sets over $\mathbb{F}_1$ carrying the corresponding logarithmic structures. \\
$\mathrm{Object}_{\mathrm{E}_\infty\mathrm{commutativealgebra},\mathrm{Simplicial}}(\mathrm{IndBan}_{\mathbb{F}_1})^\text{prelog}$& $(\infty,1)$-category of $\mathbb{E}_\infty$ commutative algebra objects in the $(\infty,1)$-category of $\mathbb{E}_\infty$ ind-Banach sets over $\mathbb{F}_1$ carrying the corresponding logarithmic structures. \\
$\mathrm{Object}_{\mathrm{E}_\infty\mathrm{commutativealgebra},\mathrm{Simplicial}}(\mathrm{Ind}^m\mathrm{Ban}_{\mathbb{F}_1})^\text{prelog}$& $(\infty,1)$-category of $\mathbb{E}_\infty$ commutative algebra objects in the $(\infty,1)$-category of $\mathbb{E}_\infty$ monomorphic ind-Banach sets over $\mathbb{F}_1$ carrying the corresponding logarithmic structures. \\

 $\mathrm{Object}_{\mathrm{E}_\infty\mathrm{commutativealgebra},\mathrm{Simplicial}}(\mathrm{IndSNorm}_R)^{\square}$ & $(\infty,1)$-category of $\mathbb{E}_\infty$ commutative algebra objects in the $(\infty,1)$-category of $\mathbb{E}_\infty$ ind-seminormed modules over $R$, generated by the corresponding formal series rings over $R$. Here $\square$ denotes \text{smoothformalseriesclosure}.  \\
$\mathrm{Object}_{\mathrm{E}_\infty\mathrm{commutativealgebra},\mathrm{Simplicial}}(\mathrm{Ind}^m\mathrm{SNorm}_R)^{\square}$& $(\infty,1)$-category of $\mathbb{E}_\infty$ commutative algebra objects in the $(\infty,1)$-category of $\mathbb{E}_\infty$ monomorphic ind-seminormed modules over $R$, generated by the corresponding formal series rings over $R$. \\
$\mathrm{Object}_{\mathrm{E}_\infty\mathrm{commutativealgebra},\mathrm{Simplicial}}(\mathrm{IndNorm}_R)^{\square}$& $(\infty,1)$-category of $\mathbb{E}_\infty$ commutative algebra objects in the $(\infty,1)$-category of $\mathbb{E}_\infty$ ind-normed modules over $R$, generated by the corresponding formal series rings over $R$, generated by the corresponding formal series rings over $R$. \\
$\mathrm{Object}_{\mathrm{E}_\infty\mathrm{commutativealgebra},\mathrm{Simplicial}}(\mathrm{Ind}^m\mathrm{Norm}_R)^{\square}$& $(\infty,1)$-category of $\mathbb{E}_\infty$ commutative algebra objects in the $(\infty,1)$-category of $\mathbb{E}_\infty$ monomorphic ind-normed modules over $R$, generated by the corresponding formal series rings over $R$. \\
$\mathrm{Object}_{\mathrm{E}_\infty\mathrm{commutativealgebra},\mathrm{Simplicial}}(\mathrm{IndBan}_R)^{\square}$& $(\infty,1)$-category of $\mathbb{E}_\infty$ commutative algebra objects in the $(\infty,1)$-category of $\mathbb{E}_\infty$ ind-Banach modules over $R$, generated by the corresponding formal series rings over $R$. \\
$\mathrm{Object}_{\mathrm{E}_\infty\mathrm{commutativealgebra},\mathrm{Simplicial}}(\mathrm{Ind}^m\mathrm{Ban}_R)^{\square}$& $(\infty,1)$-category of $\mathbb{E}_\infty$ commutative algebra objects in the $(\infty,1)$-category of $\mathbb{E}_\infty$ monomorphic ind-Banach modules over $R$, generated by the corresponding formal series rings over $R$. \\
$\mathrm{Object}_{\mathrm{E}_\infty\mathrm{commutativealgebra},\mathrm{Simplicial}}(\mathrm{IndSNorm}_{\mathbb{F}_1})^{\square}$& $(\infty,1)$-category of $\mathbb{E}_\infty$ commutative algebra objects in the $(\infty,1)$-category of $\mathbb{E}_\infty$ ind-seminormed sets over $\mathbb{F}_1$, generated by the corresponding formal series rings over $\mathbb{F}_1$. \\
$\mathrm{Object}_{\mathrm{E}_\infty\mathrm{commutativealgebra},\mathrm{Simplicial}}(\mathrm{Ind}^m\mathrm{SNorm}_{\mathbb{F}_1})^{\square}$& $(\infty,1)$-category of $\mathbb{E}_\infty$ commutative algebra objects in the $(\infty,1)$-category of $\mathbb{E}_\infty$ monomorphic ind-seminormed sets over $\mathbb{F}_1$, generated by the corresponding formal series rings over $\mathbb{F}_1$. \\
$\mathrm{Object}_{\mathrm{E}_\infty\mathrm{commutativealgebra},\mathrm{Simplicial}}(\mathrm{IndNorm}_{\mathbb{F}_1})^{\square}$& $(\infty,1)$-category of $\mathbb{E}_\infty$ commutative algebra objects in the $(\infty,1)$-category of $\mathbb{E}_\infty$ ind-normed sets over $\mathbb{F}_1$, generated by the corresponding formal series rings over $\mathbb{F}_1$. \\
$\mathrm{Object}_{\mathrm{E}_\infty\mathrm{commutativealgebra},\mathrm{Simplicial}}(\mathrm{Ind}^m\mathrm{Norm}_{\mathbb{F}_1})^{\square}$& $(\infty,1)$-category of $\mathbb{E}_\infty$ commutative algebra objects in the $(\infty,1)$-category of $\mathbb{E}_\infty$ monomorphic ind-normed sets over $\mathbb{F}_1$, generated by the corresponding formal series rings over $\mathbb{F}_1$. \\
$\mathrm{Object}_{\mathrm{E}_\infty\mathrm{commutativealgebra},\mathrm{Simplicial}}(\mathrm{IndBan}_{\mathbb{F}_1})^{\square}$& $(\infty,1)$-category of $\mathbb{E}_\infty$ commutative algebra objects in the $(\infty,1)$-category of $\mathbb{E}_\infty$ ind-Banach sets over $\mathbb{F}_1$, generated by the corresponding formal series rings over $\mathbb{F}_1$. \\
$\mathrm{Object}_{\mathrm{E}_\infty\mathrm{commutativealgebra},\mathrm{Simplicial}}(\mathrm{Ind}^m\mathrm{Ban}_{\mathbb{F}_1})^{\square}$& $(\infty,1)$-category of $\mathbb{E}_\infty$ commutative algebra objects in the $(\infty,1)$-category of $\mathbb{E}_\infty$ monomorphic ind-Banach sets over $\mathbb{F}_1$, generated by the corresponding formal series rings over $\mathbb{F}_1$. \\

$\mathrm{Object}_{\mathrm{E}_\infty\mathrm{commutativealgebra},\mathrm{Simplicial}}(\mathrm{IndSNorm}_R)^{\square,\text{prelog}}$& $(\infty,1)$-category of $\mathbb{E}_\infty$ commutative algebra objects in the $(\infty,1)$-category of $\mathbb{E}_\infty$ ind-seminormed modules over $R$ carrying the corresponding logarithmic structures, generated by the corresponding formal series rings over $R$. \\
$\mathrm{Object}_{\mathrm{E}_\infty\mathrm{commutativealgebra},\mathrm{Simplicial}}(\mathrm{Ind}^m\mathrm{SNorm}_R)^{\square,\text{prelog}}$& $(\infty,1)$-category of $\mathbb{E}_\infty$ commutative algebra objects in the $(\infty,1)$-category of $\mathbb{E}_\infty$ monomorphic ind-seminormed modules over $R$ carrying the corresponding logarithmic structures, generated by the corresponding formal series rings over $R$. \\
$\mathrm{Object}_{\mathrm{E}_\infty\mathrm{commutativealgebra},\mathrm{Simplicial}}(\mathrm{IndNorm}_R)^{\square,\text{prelog}}$& $(\infty,1)$-category of $\mathbb{E}_\infty$ commutative algebra objects in the $(\infty,1)$-category of $\mathbb{E}_\infty$ ind-normed modules over $R$ carrying the corresponding logarithmic structures, generated by the corresponding formal series rings over $R$. \\
$\mathrm{Object}_{\mathrm{E}_\infty\mathrm{commutativealgebra},\mathrm{Simplicial}}(\mathrm{Ind}^m\mathrm{Norm}_R)^{\square,\text{prelog}}$& $(\infty,1)$-category of $\mathbb{E}_\infty$ commutative algebra objects in the $(\infty,1)$-category of $\mathbb{E}_\infty$ monomorphic ind-normed modules over $R$ carrying the corresponding logarithmic structures, generated by the corresponding formal series rings over $R$. \\
$\mathrm{Object}_{\mathrm{E}_\infty\mathrm{commutativealgebra},\mathrm{Simplicial}}(\mathrm{IndBan}_R)^{\square,\text{prelog}}$& $(\infty,1)$-category of $\mathbb{E}_\infty$ commutative algebra objects in the $(\infty,1)$-category of $\mathbb{E}_\infty$ ind-Banach modules over $R$ carrying the corresponding logarithmic structures, generated by the corresponding formal series rings over $R$. \\
$\mathrm{Object}_{\mathrm{E}_\infty\mathrm{commutativealgebra},\mathrm{Simplicial}}(\mathrm{Ind}^m\mathrm{Ban}_R)^{\square,\text{prelog}}$& $(\infty,1)$-category of $\mathbb{E}_\infty$ commutative algebra objects in the $(\infty,1)$-category of $\mathbb{E}_\infty$ monomorphic ind-Banach modules over $R$ carrying the corresponding logarithmic structures, generated by the corresponding formal series rings over $R$. \\
$\mathrm{Object}_{\mathrm{E}_\infty\mathrm{commutativealgebra},\mathrm{Simplicial}}(\mathrm{IndSNorm}_{\mathbb{F}_1})^{\square,\text{prelog}}$& $(\infty,1)$-category of $\mathbb{E}_\infty$ commutative algebra objects in the $(\infty,1)$-category of $\mathbb{E}_\infty$ ind-seminormed sets over $\mathbb{F}_1$ carrying the corresponding logarithmic structures, generated by the corresponding formal series rings over ${\mathbb{F}_1}$. \\
$\mathrm{Object}_{\mathrm{E}_\infty\mathrm{commutativealgebra},\mathrm{Simplicial}}(\mathrm{Ind}^m\mathrm{SNorm}_{\mathbb{F}_1})^{\square,\text{prelog}}$& $(\infty,1)$-category of $\mathbb{E}_\infty$ commutative algebra objects in the $(\infty,1)$-category of $\mathbb{E}_\infty$ monomorphic ind-seminormed sets over $\mathbb{F}_1$ carrying the corresponding logarithmic structures, generated by the corresponding formal series rings over ${\mathbb{F}_1}$. \\
$\mathrm{Object}_{\mathrm{E}_\infty\mathrm{commutativealgebra},\mathrm{Simplicial}}(\mathrm{IndNorm}_{\mathbb{F}_1})^{\square,\text{prelog}}$& $(\infty,1)$-category of $\mathbb{E}_\infty$ commutative algebra objects in the $(\infty,1)$-category of $\mathbb{E}_\infty$ ind-normed sets over $\mathbb{F}_1$ carrying the corresponding logarithmic structures, generated by the corresponding formal series rings over ${\mathbb{F}_1}$. \\
$\mathrm{Object}_{\mathrm{E}_\infty\mathrm{commutativealgebra},\mathrm{Simplicial}}(\mathrm{Ind}^m\mathrm{Norm}_{\mathbb{F}_1})^{\square,\text{prelog}}$& $(\infty,1)$-category of $\mathbb{E}_\infty$ commutative algebra objects in the $(\infty,1)$-category of $\mathbb{E}_\infty$ monomorphic ind-normed sets over $\mathbb{F}_1$ carrying the corresponding logarithmic structures, generated by the corresponding formal series rings over ${\mathbb{F}_1}$. \\
$\mathrm{Object}_{\mathrm{E}_\infty\mathrm{commutativealgebra},\mathrm{Simplicial}}(\mathrm{IndBan}_{\mathbb{F}_1})^{\square,\text{prelog}}$& $(\infty,1)$-category of $\mathbb{E}_\infty$ commutative algebra objects in the $(\infty,1)$-category of $\mathbb{E}_\infty$ ind-Banach sets over $\mathbb{F}_1$ carrying the corresponding logarithmic structures, generated by the corresponding formal series rings over ${\mathbb{F}_1}$. \\
$\mathrm{Object}_{\mathrm{E}_\infty\mathrm{commutativealgebra},\mathrm{Simplicial}}(\mathrm{Ind}^m\mathrm{Ban}_{\mathbb{F}_1})^{\square,\text{prelog}}$& $(\infty,1)$-category of $\mathbb{E}_\infty$ commutative algebra objects in the $(\infty,1)$-category of $\mathbb{E}_\infty$ monomorphic ind-Banach sets over $\mathbb{F}_1$ carrying the corresponding logarithmic structures, generated by the corresponding formal series rings over ${\mathbb{F}_1}$. \\

\end{tabularx}
\end{center}

\begin{remark}
The generating process above is by adding all the colimits in the homotopy sense which are assumed to be sifted. $R$ will be Banach and carrying the corresponding $p$-adic topology.
\end{remark}

\newpage
\section{Notations on $\infty$-Categories of Commutative $\infty$-Ringed Toposes}

\indent We change the notations for $\infty$-ringed toposes slightly, therefore let us start from the rings.\\

\noindent Rings:\\

\noindent $\mathrm{Ind}^\text{smoothformalseriesclosure}\mathrm{Commutativealgebra}_{\mathrm{simplicial}}(\mathrm{Ind}\mathrm{Seminormed}_R)$: $(\infty,1)$-category of simplicial commutative algebra objects in the $\infty$-category of colimit completion of seminormed modules over a general Banach ring $R$, generated by the corresponding power series rings by taking colimits in the homotopical sense. \\ 
\noindent $\mathrm{Ind}^\text{smoothformalseriesclosure}\mathrm{Commutativealgebra}_{\mathrm{simplicial}}(\mathrm{Ind}^m\mathrm{Seminormed}_R)$: $(\infty,1)$-category of simplicial commutative algebra objects in the $\infty$-category of monomorphic colimit completion of seminormed modules over a general Banach ring $R$, generated by the corresponding power series rings by taking colimits in the homotopical sense. \\ 
\noindent $\mathrm{Ind}^\text{smoothformalseriesclosure}\mathrm{Commutativealgebra}_{\mathrm{simplicial}}(\mathrm{Ind}\mathrm{Normed}_R)$: $(\infty,1)$-category of simplicial commutative algebra objects in the $\infty$-category of colimit completion of normed modules over a general Banach ring $R$, generated by the corresponding power series rings by taking colimits in the homotopical sense. \\ 
\noindent $\mathrm{Ind}^\text{smoothformalseriesclosure}\mathrm{Commutativealgebra}_{\mathrm{simplicial}}(\mathrm{Ind}^m\mathrm{Normed}_R)$: $(\infty,1)$-category of simplicial commutative algebra objects in the $\infty$-category of monomorphic colimit completion of normed modules over a general Banach ring $R$, generated by the corresponding power series rings by taking colimits in the homotopical sense. \\ 
\noindent $\mathrm{Ind}^\text{smoothformalseriesclosure}\mathrm{Commutativealgebra}_{\mathrm{simplicial}}(\mathrm{Ind}\mathrm{Banach}_R)$: $(\infty,1)$-category of simplicial commutative algebra objects in the $\infty$-category of colimit completion of Banach modules over a general Banach ring $R$, generated by the corresponding power series rings by taking colimits in the homotopical sense. \\ 
\noindent $\mathrm{Ind}^\text{smoothformalseriesclosure}\mathrm{Commutativealgebra}_{\mathrm{simplicial}}(\mathrm{Ind}^m\mathrm{Banach}_R)$: $(\infty,1)$-category of simplicial commutative algebra objects in the $\infty$-category of monomorphic colimit completion of Banach modules over a general Banach ring $R$, generated by the corresponding power series rings by taking colimits in the homotopical sense. \\

\noindent $\mathrm{Ind}^\text{smoothformalseriesclosure}\mathrm{Commutativealgebra}_{\mathrm{simplicial}}(\mathrm{Ind}\mathrm{Seminormed}_{\mathbb{F}_1})$: $(\infty,1)$-category of simplicial commutative algebra objects in the $\infty$-category of colimit completion of seminormed sets over a general Banach ring $\mathbb{F}_1$, generated by the corresponding power series rings by taking colimits in the homotopical sense. \\ 
\noindent $\mathrm{Ind}^\text{smoothformalseriesclosure}\mathrm{Commutativealgebra}_{\mathrm{simplicial}}(\mathrm{Ind}^m\mathrm{Seminormed}_{\mathbb{F}_1})$: $(\infty,1)$-category of simplicial commutative algebra objects in the $\infty$-category of monomorphic colimit completion of seminormed sets over a general Banach ring $\mathbb{F}_1$, generated by the corresponding power series rings by taking colimits in the homotopical sense. \\ 
\noindent $\mathrm{Ind}^\text{smoothformalseriesclosure}\mathrm{Commutativealgebra}_{\mathrm{simplicial}}(\mathrm{Ind}\mathrm{Normed}_{\mathbb{F}_1})$: $(\infty,1)$-category of simplicial commutative algebra objects in the $\infty$-category of colimit completion of normed sets over a general Banach ring $\mathbb{F}_1$, generated by the corresponding power series rings by taking colimits in the homotopical sense. \\ 
\noindent $\mathrm{Ind}^\text{smoothformalseriesclosure}\mathrm{Commutativealgebra}_{\mathrm{simplicial}}(\mathrm{Ind}^m\mathrm{Normed}_{\mathbb{F}_1})$: $(\infty,1)$-category of simplicial commutative algebra objects in the $\infty$-category of monomorphic colimit completion of normed sets over a general Banach ring $\mathbb{F}_1$, generated by the corresponding power series rings by taking colimits in the homotopical sense. \\ 
\noindent $\mathrm{Ind}^\text{smoothformalseriesclosure}\mathrm{Commutativealgebra}_{\mathrm{simplicial}}(\mathrm{Ind}\mathrm{Banach}_{\mathbb{F}_1})$: $(\infty,1)$-category of simplicial commutative algebra objects in the $\infty$-category of colimit completion of Banach sets over a general Banach ring $\mathbb{F}_1$, generated by the corresponding power series rings by taking colimits in the homotopical sense. \\ 
\noindent $\mathrm{Ind}^\text{smoothformalseriesclosure}\mathrm{Commutativealgebra}_{\mathrm{simplicial}}(\mathrm{Ind}^m\mathrm{Banach}_{\mathbb{F}_1})$: $(\infty,1)$-category of simplicial commutative algebra objects in the $\infty$-category of monomorphic colimit completion of Banach sets over a general Banach ring $\mathbb{F}_1$, generated by the corresponding power series rings by taking colimits in the homotopical sense. \\

\noindent Prestacks:\\

\noindent $\infty-\mathrm{Prestack}_{\mathrm{Commutativealgebra}_{\mathrm{simplicial}}(\mathrm{Ind}\mathrm{Seminormed}_?)^\mathrm{opposite},\mathrm{Grotopology,homotopyepimorphism}}$: $\infty$-presheaves into $\infty$-groupoid over corresponding opposite category carrying the corresponding Grothendieck topology, and we will mainly consider the corresponding homotopy epimorphisms. $?=R,\mathbb{F}_1$.\\
\noindent $\infty-\mathrm{Prestack}_{\mathrm{Commutativealgebra}_{\mathrm{simplicial}}(\mathrm{Ind}^m\mathrm{Seminormed}_?)^\mathrm{opposite},\mathrm{Grotopology,homotopyepimorphism}}$: $\infty$-presheaves into $\infty$-groupoid over corresponding opposite category carrying the corresponding Grothendieck topology, and we will mainly consider the corresponding homotopy epimorphisms. $?=R,\mathbb{F}_1$.\\
\noindent $\infty-\mathrm{Prestack}_{\mathrm{Commutativealgebra}_{\mathrm{simplicial}}(\mathrm{Ind}\mathrm{Normed}_?)^\mathrm{opposite},\mathrm{Grotopology,homotopyepimorphism}}$: $\infty$-presheaves into $\infty$-groupoid over corresponding opposite category carrying the corresponding Grothendieck topology, and we will mainly consider the corresponding homotopy epimorphisms. $?=R,\mathbb{F}_1$.\\
\noindent $\infty-\mathrm{Prestack}_{\mathrm{Commutativealgebra}_{\mathrm{simplicial}}(\mathrm{Ind}^m\mathrm{Normed}_?)^\mathrm{opposite},\mathrm{Grotopology,homotopyepimorphism}}$: $\infty$-presheaves into $\infty$-groupoid over corresponding opposite category carrying the corresponding Grothendieck topology, and we will mainly consider the corresponding homotopy epimorphisms. $?=R,\mathbb{F}_1$.\\
\noindent $\infty-\mathrm{Prestack}_{\mathrm{Commutativealgebra}_{\mathrm{simplicial}}(\mathrm{Ind}\mathrm{Banach}_?)^\mathrm{opposite},\mathrm{Grotopology,homotopyepimorphism}}$: $\infty$-presheaves into $\infty$-groupoid over corresponding opposite category carrying the corresponding Grothendieck topology, and we will mainly consider the corresponding homotopy epimorphisms. $?=R,\mathbb{F}_1$.\\
\noindent $\infty-\mathrm{Prestack}_{\mathrm{Commutativealgebra}_{\mathrm{simplicial}}(\mathrm{Ind}^m\mathrm{Banach}_?)^\mathrm{opposite},\mathrm{Grotopology,homotopyepimorphism}}$: $\infty$-presheaves into $\infty$-groupoid over corresponding opposite category carrying the corresponding Grothendieck topology, and we will mainly consider the corresponding homotopy epimorphisms. $?=R,\mathbb{F}_1$.\\

\noindent Stacks:\\
 
 \noindent $\infty-\mathrm{Stack}_{\mathrm{Commutativealgebra}_{\mathrm{simplicial}}(\mathrm{Ind}\mathrm{Seminormed}_?)^\mathrm{opposite},\mathrm{Grotopology,homotopyepimorphism}}$: $\infty$-sheaves into $\infty$-groupoid over corresponding opposite category carrying the corresponding Grothendieck topology, and we will mainly consider the corresponding homotopy epimorphisms. $?=R,\mathbb{F}_1$.\\
\noindent $\infty-\mathrm{Stack}_{\mathrm{Commutativealgebra}_{\mathrm{simplicial}}(\mathrm{Ind}^m\mathrm{Seminormed}_?)^\mathrm{opposite},\mathrm{Grotopology,homotopyepimorphism}}$: $\infty$-sheaves into $\infty$-groupoid over corresponding opposite category carrying the corresponding Grothendieck topology, and we will mainly consider the corresponding homotopy epimorphisms. $?=R,\mathbb{F}_1$.\\
\noindent $\infty-\mathrm{Stack}_{\mathrm{Commutativealgebra}_{\mathrm{simplicial}}(\mathrm{Ind}\mathrm{Normed}_?)^\mathrm{opposite},\mathrm{Grotopology,homotopyepimorphism}}$: $\infty$-sheaves into $\infty$-groupoid over corresponding opposite category carrying the corresponding Grothendieck topology, and we will mainly consider the corresponding homotopy epimorphisms. $?=R,\mathbb{F}_1$.\\
\noindent $\infty-\mathrm{Stack}_{\mathrm{Commutativealgebra}_{\mathrm{simplicial}}(\mathrm{Ind}^m\mathrm{Normed}_?)^\mathrm{opposite},\mathrm{Grotopology,homotopyepimorphism}}$: $\infty$-sheaves into $\infty$-groupoid over corresponding opposite category carrying the corresponding Grothendieck topology, and we will mainly consider the corresponding homotopy epimorphisms. $?=R,\mathbb{F}_1$.\\
\noindent $\infty-\mathrm{Stack}_{\mathrm{Commutativealgebra}_{\mathrm{simplicial}}(\mathrm{Ind}\mathrm{Banach}_?)^\mathrm{opposite},\mathrm{Grotopology,homotopyepimorphism}}$: $\infty$-sheaves into $\infty$-groupoid over corresponding opposite category carrying the corresponding Grothendieck topology, and we will mainly consider the corresponding homotopy epimorphisms. $?=R,\mathbb{F}_1$.\\
\noindent $\infty-\mathrm{Stack}_{\mathrm{Commutativealgebra}_{\mathrm{simplicial}}(\mathrm{Ind}^m\mathrm{Banach}_?)^\mathrm{opposite},\mathrm{Grotopology,homotopyepimorphism}}$: $\infty$-sheaves into $\infty$-groupoid over corresponding opposite category carrying the corresponding Grothendieck topology, and we will mainly consider the corresponding homotopy epimorphisms. $?=R,\mathbb{F}_1$.\\

\noindent Ringed Toposes: \\
 
 \noindent $\infty-\mathrm{Toposes}^{\mathrm{ringed},\mathrm{commutativealgebra}_{\mathrm{simplicial}}(\mathrm{Ind}\mathrm{Seminormed}_?)}_{\mathrm{Commutativealgebra}_{\mathrm{simplicial}}(\mathrm{Ind}\mathrm{Seminormed}_?)^\mathrm{opposite},\mathrm{Grotopology,homotopyepimorphism}}$: $\infty$-sheaves into $\infty$-groupoid over corresponding opposite category carrying the corresponding Grothendieck topology, and we will mainly consider the corresponding homotopy epimorphisms. $?=R,\mathbb{F}_1$. And we assume the stack carries $\infty$-ringed toposes structure. \\
\noindent $\infty-\mathrm{Toposes}^{\mathrm{ringed},\mathrm{Commutativealgebra}_{\mathrm{simplicial}}(\mathrm{Ind}^m\mathrm{Seminormed}_?)}_{\mathrm{Commutativealgebra}_{\mathrm{simplicial}}(\mathrm{Ind}^m\mathrm{Seminormed}_?)^\mathrm{opposite},\mathrm{Grotopology,homotopyepimorphism}}$: $\infty$-sheaves into $\infty$-groupoid over corresponding opposite category carrying the corresponding Grothendieck topology, and we will mainly consider the corresponding homotopy epimorphisms. $?=R,\mathbb{F}_1$. And we assume the stack carries $\infty$-ringed toposes structure.\\
\noindent $\infty-\mathrm{Toposes}^{\mathrm{ringed},\mathrm{Commutativealgebra}_{\mathrm{simplicial}}(\mathrm{Ind}\mathrm{Normed}_?)}_{\mathrm{Commutativealgebra}_{\mathrm{simplicial}}(\mathrm{Ind}\mathrm{Normed}_?)^\mathrm{opposite},\mathrm{Grotopology,homotopyepimorphism}}$: $\infty$-sheaves into $\infty$-groupoid over corresponding opposite category carrying the corresponding Grothendieck topology, and we will mainly consider the corresponding homotopy epimorphisms. $?=R,\mathbb{F}_1$. And we assume the stack carries $\infty$-ringed toposes structure.\\
\noindent $\infty-\mathrm{Toposes}^{\mathrm{ringed},\mathrm{Commutativealgebra}_{\mathrm{simplicial}}(\mathrm{Ind}^m\mathrm{Normed}_?)}_{\mathrm{Commutativealgebra}_{\mathrm{simplicial}}(\mathrm{Ind}^m\mathrm{Normed}_?)^\mathrm{opposite},\mathrm{Grotopology,homotopyepimorphism}}$: $\infty$-sheaves into $\infty$-groupoid over corresponding opposite category carrying the corresponding Grothendieck topology, and we will mainly consider the corresponding homotopy epimorphisms. $?=R,\mathbb{F}_1$. And we assume the stack carries $\infty$-ringed toposes structure.\\
\noindent $\infty-\mathrm{Toposes}^{\mathrm{ringed},\mathrm{Commutativealgebra}_{\mathrm{simplicial}}(\mathrm{Ind}\mathrm{Banach}_?)}_{\mathrm{Commutativealgebra}_{\mathrm{simplicial}}(\mathrm{Ind}\mathrm{Banach}_?)^\mathrm{opposite},\mathrm{Grotopology,homotopyepimorphism}}$: $\infty$-sheaves into $\infty$-groupoid over corresponding opposite category carrying the corresponding Grothendieck topology, and we will mainly consider the corresponding homotopy epimorphisms. $?=R,\mathbb{F}_1$. And we assume the stack carries $\infty$-ringed toposes structure.\\
\noindent $\infty-\mathrm{Toposes}^{\mathrm{ringed},\mathrm{Commutativealgebra}_{\mathrm{simplicial}}(\mathrm{Ind}^m\mathrm{Banach}_?)}_{\mathrm{Commutativealgebra}_{\mathrm{simplicial}}(\mathrm{Ind}^m\mathrm{Banach}_?)^\mathrm{opposite},\mathrm{Grotopology,homotopyepimorphism}}$: $\infty$-sheaves into $\infty$-groupoid over corresponding opposite category carrying the corresponding Grothendieck topology, and we will mainly consider the corresponding homotopy epimorphisms. $?=R,\mathbb{F}_1$. And we assume the stack carries $\infty$-ringed toposes structure.\\

 \noindent $\mathrm{Proj}^\text{smoothformalseriesclosure}\infty-\mathrm{Toposes}^{\mathrm{ringed},\mathrm{commutativealgebra}_{\mathrm{simplicial}}(\mathrm{Ind}\mathrm{Seminormed}_?)}_{\mathrm{Commutativealgebra}_{\mathrm{simplicial}}(\mathrm{Ind}\mathrm{Seminormed}_?)^\mathrm{opposite},\mathrm{Grotopology,homotopyepimorphism}}$: $\infty$-sheaves into $\infty$-groupoid over corresponding opposite category carrying the corresponding Grothendieck topology, and we will mainly consider the corresponding homotopy epimorphisms. $?=R,\mathbb{F}_1$. And we assume the stack carries $\infty$-ringed toposes structure by specializing a ring object $\mathcal{O}$. \\
\noindent $\mathrm{Proj}^\text{smoothformalseriesclosure}\infty-\mathrm{Toposes}^{\mathrm{ringed},\mathrm{Commutativealgebra}_{\mathrm{simplicial}}(\mathrm{Ind}^m\mathrm{Seminormed}_?)}_{\mathrm{Commutativealgebra}_{\mathrm{simplicial}}(\mathrm{Ind}^m\mathrm{Seminormed}_?)^\mathrm{opposite},\mathrm{Grotopology,homotopyepimorphism}}$: $\infty$-sheaves into $\infty$-groupoid over corresponding opposite category carrying the corresponding Grothendieck topology, and we will mainly consider the corresponding homotopy epimorphisms. $?=R,\mathbb{F}_1$. And we assume the stack carries $\infty$-ringed toposes structure by specializing a ring object $\mathcal{O}$.\\
\noindent $\mathrm{Proj}^\text{smoothformalseriesclosure}\infty-\mathrm{Toposes}^{\mathrm{ringed},\mathrm{Commutativealgebra}_{\mathrm{simplicial}}(\mathrm{Ind}\mathrm{Normed}_?)}_{\mathrm{Commutativealgebra}_{\mathrm{simplicial}}(\mathrm{Ind}\mathrm{Normed}_?)^\mathrm{opposite},\mathrm{Grotopology,homotopyepimorphism}}$: $\infty$-sheaves into $\infty$-groupoid over corresponding opposite category carrying the corresponding Grothendieck topology, and we will mainly consider the corresponding homotopy epimorphisms. $?=R,\mathbb{F}_1$. And we assume the stack carries $\infty$-ringed toposes structure by specializing a ring object $\mathcal{O}$.\\
\noindent $\mathrm{Proj}^\text{smoothformalseriesclosure}\infty-\mathrm{Toposes}^{\mathrm{ringed},\mathrm{Commutativealgebra}_{\mathrm{simplicial}}(\mathrm{Ind}^m\mathrm{Normed}_?)}_{\mathrm{Commutativealgebra}_{\mathrm{simplicial}}(\mathrm{Ind}^m\mathrm{Normed}_?)^\mathrm{opposite},\mathrm{Grotopology,homotopyepimorphism}}$: $\infty$-sheaves into $\infty$-groupoid over corresponding opposite category carrying the corresponding Grothendieck topology, and we will mainly consider the corresponding homotopy epimorphisms. $?=R,\mathbb{F}_1$. And we assume the stack carries $\infty$-ringed toposes structure by specializing a ring object $\mathcal{O}$.\\
\noindent $\mathrm{Proj}^\text{smoothformalseriesclosure}\infty-\mathrm{Toposes}^{\mathrm{ringed},\mathrm{Commutativealgebra}_{\mathrm{simplicial}}(\mathrm{Ind}\mathrm{Banach}_?)}_{\mathrm{Commutativealgebra}_{\mathrm{simplicial}}(\mathrm{Ind}\mathrm{Banach}_?)^\mathrm{opposite},\mathrm{Grotopology,homotopyepimorphism}}$: $\infty$-sheaves into $\infty$-groupoid over corresponding opposite category carrying the corresponding Grothendieck topology, and we will mainly consider the corresponding homotopy epimorphisms. $?=R,\mathbb{F}_1$. And we assume the stack carries $\infty$-ringed toposes structure by specializing a ring object $\mathcal{O}$.\\
\noindent $\mathrm{Proj}^\text{smoothformalseriesclosure}\infty-\mathrm{Toposes}^{\mathrm{ringed},\mathrm{Commutativealgebra}_{\mathrm{simplicial}}(\mathrm{Ind}^m\mathrm{Banach}_?)}_{\mathrm{Commutativealgebra}_{\mathrm{simplicial}}(\mathrm{Ind}^m\mathrm{Banach}_?)^\mathrm{opposite},\mathrm{Grotopology,homotopyepimorphism}}$: $\infty$-sheaves into $\infty$-groupoid over corresponding opposite category carrying the corresponding Grothendieck topology, and we will mainly consider the corresponding homotopy epimorphisms. $?=R,\mathbb{F}_1$. And we assume the stack carries $\infty$-ringed toposes structure by specializing a ring object $\mathcal{O}$.\\

\begin{remark}
In $p$-adic Hodge theory, one usually will need to construct presheaves $M$ out of from the $\infty$-ring object $\mathcal{O}$.	Also one can define the corresponding pre-$\infty$-ringed Toposes, we will not continue provide the corresponding group of notations in the parallel way. 
\end{remark}

\noindent $\infty$-Quasicoherent Sheaves of Functional Analytic Modules over Ringed Toposes $\sharp=\mathrm{Seminormed},\mathrm{Normed},\mathrm{Banach}$: \\
 
 \noindent $\mathrm{Ind}\mathrm{\sharp Quasicoherent}_{\infty-\mathrm{Toposes}^{\mathrm{ringed},\mathrm{commutativealgebra}_{\mathrm{simplicial}}(\mathrm{Ind}\mathrm{Seminormed}_?)}_{\mathrm{Commutativealgebra}_{\mathrm{simplicial}}(\mathrm{Ind}\mathrm{Seminormed}_?)^\mathrm{opposite},\mathrm{Grotopology,homotopyepimorphism}}}$: Colimits completion of $\infty$-Quasicoherent Sheaves of Functional Analytic Modules over $\infty$-sheaves into $\infty$-groupoid over corresponding opposite category carrying the corresponding Grothendieck topology, and we will mainly consider the corresponding homotopy epimorphisms. $?=R,\mathbb{F}_1$. And we assume the stack carries $\infty$-ringed toposes structure. \\
\noindent $\mathrm{Ind}\mathrm{\sharp Quasicoherent}_{\infty-\mathrm{Toposes}^{\mathrm{ringed},\mathrm{Commutativealgebra}_{\mathrm{simplicial}}(\mathrm{Ind}^m\mathrm{Seminormed}_?)}_{\mathrm{Commutativealgebra}_{\mathrm{simplicial}}(\mathrm{Ind}^m\mathrm{Seminormed}_?)^\mathrm{opposite},\mathrm{Grotopology,homotopyepimorphism}}}$: Colimits completion of $\infty$-Quasicoherent Sheaves of Functional Analytic Modules over $\infty$-sheaves into $\infty$-groupoid over corresponding opposite category carrying the corresponding Grothendieck topology, and we will mainly consider the corresponding homotopy epimorphisms. $?=R,\mathbb{F}_1$. And we assume the stack carries $\infty$-ringed toposes structure.\\
\noindent $\mathrm{Ind}\mathrm{\sharp Quasicoherent}_{\infty-\mathrm{Toposes}^{\mathrm{ringed},\mathrm{Commutativealgebra}_{\mathrm{simplicial}}(\mathrm{Ind}\mathrm{Normed}_?)}_{\mathrm{Commutativealgebra}_{\mathrm{simplicial}}(\mathrm{Ind}\mathrm{Normed}_?)^\mathrm{opposite},\mathrm{Grotopology,homotopyepimorphism}}}$: Colimits completion of $\infty$-Quasicoherent Sheaves of Functional Analytic Modules over $\infty$-sheaves into $\infty$-groupoid over corresponding opposite category carrying the corresponding Grothendieck topology, and we will mainly consider the corresponding homotopy epimorphisms. $?=R,\mathbb{F}_1$. And we assume the stack carries $\infty$-ringed toposes structure.\\
\noindent $\mathrm{Ind}\mathrm{\sharp Quasicoherent}_{\infty-\mathrm{Toposes}^{\mathrm{ringed},\mathrm{Commutativealgebra}_{\mathrm{simplicial}}(\mathrm{Ind}^m\mathrm{Normed}_?)}_{\mathrm{Commutativealgebra}_{\mathrm{simplicial}}(\mathrm{Ind}^m\mathrm{Normed}_?)^\mathrm{opposite},\mathrm{Grotopology,homotopyepimorphism}}}$: Colimits completion of $\infty$-Quasicoherent Sheaves of Functional Analytic Modules over $\infty$-sheaves into $\infty$-groupoid over corresponding opposite category carrying the corresponding Grothendieck topology, and we will mainly consider the corresponding homotopy epimorphisms. $?=R,\mathbb{F}_1$. And we assume the stack carries $\infty$-ringed toposes structure.\\
\noindent $\mathrm{Ind}\mathrm{\sharp Quasicoherent}_{\infty-\mathrm{Toposes}^{\mathrm{ringed},\mathrm{Commutativealgebra}_{\mathrm{simplicial}}(\mathrm{Ind}\mathrm{Banach}_?)}_{\mathrm{Commutativealgebra}_{\mathrm{simplicial}}(\mathrm{Ind}\mathrm{Banach}_?)^\mathrm{opposite},\mathrm{Grotopology,homotopyepimorphism}}}$: Colimits completion of $\infty$-Quasicoherent Sheaves of Functional Analytic Modules over $\infty$-sheaves into $\infty$-groupoid over corresponding opposite category carrying the corresponding Grothendieck topology, and we will mainly consider the corresponding homotopy epimorphisms. $?=R,\mathbb{F}_1$. And we assume the stack carries $\infty$-ringed toposes structure.\\
\noindent $\mathrm{Ind}\mathrm{\sharp Quasicoherent}_{\infty-\mathrm{Toposes}^{\mathrm{ringed},\mathrm{Commutativealgebra}_{\mathrm{simplicial}}(\mathrm{Ind}^m\mathrm{Banach}_?)}_{\mathrm{Commutativealgebra}_{\mathrm{simplicial}}(\mathrm{Ind}^m\mathrm{Banach}_?)^\mathrm{opposite},\mathrm{Grotopology,homotopyepimorphism}}}$: Colimits completion of $\infty$-Quasicoherent Sheaves of Functional Analytic Modules over $\infty$-sheaves into $\infty$-groupoid over corresponding opposite category carrying the corresponding Grothendieck topology, and we will mainly consider the corresponding homotopy epimorphisms. $?=R,\mathbb{F}_1$. And we assume the stack carries $\infty$-ringed toposes structure.\\

 \noindent $\mathrm{Ind}\mathrm{\sharp Quasicoherent}_{\mathrm{Ind}^\text{smoothformalseriesclosure}\infty-\mathrm{Toposes}^{\mathrm{ringed},\mathrm{commutativealgebra}_{\mathrm{simplicial}}(\mathrm{Ind}\mathrm{Seminormed}_?)}_{\mathrm{Commutativealgebra}_{\mathrm{simplicial}}(\mathrm{Ind}\mathrm{Seminormed}_?)^\mathrm{opposite},\mathrm{Grotopology,homotopyepimorphism}}}$: Colimits completion of $\infty$-Quasicoherent Sheaves of Functional Analytic Modules over $\infty$-sheaves into $\infty$-groupoid over corresponding opposite category carrying the corresponding Grothendieck topology, and we will mainly consider the corresponding homotopy epimorphisms. $?=R,\mathbb{F}_1$. And we assume the stack carries $\infty$-ringed toposes structure by specializing a ring object $\mathcal{O}$. The difference is that we consider the corresponding $\mathcal{O}$ living in the colimit completion closure. \\
\noindent $\mathrm{Ind}\mathrm{\sharp Quasicoherent}_{\mathrm{Ind}^\text{smoothformalseriesclosure}\infty-\mathrm{Toposes}^{\mathrm{ringed},\mathrm{Commutativealgebra}_{\mathrm{simplicial}}(\mathrm{Ind}^m\mathrm{Seminormed}_?)}_{\mathrm{Commutativealgebra}_{\mathrm{simplicial}}(\mathrm{Ind}^m\mathrm{Seminormed}_?)^\mathrm{opposite},\mathrm{Grotopology,homotopyepimorphism}}}$: Colimits completion of $\infty$-Quasicoherent Sheaves of Functional Analytic Modules over $\infty$-sheaves into $\infty$-groupoid over corresponding opposite category carrying the corresponding Grothendieck topology, and we will mainly consider the corresponding homotopy epimorphisms. $?=R,\mathbb{F}_1$. And we assume the stack carries $\infty$-ringed toposes structure by specializing a ring object $\mathcal{O}$. The difference is that we consider the corresponding $\mathcal{O}$ living in the colimit completion closure.\\
\noindent $\mathrm{Ind}\mathrm{\sharp Quasicoherent}_{\mathrm{Ind}^\text{smoothformalseriesclosure}\infty-\mathrm{Toposes}^{\mathrm{ringed},\mathrm{Commutativealgebra}_{\mathrm{simplicial}}(\mathrm{Ind}\mathrm{Normed}_?)}_{\mathrm{Commutativealgebra}_{\mathrm{simplicial}}(\mathrm{Ind}\mathrm{Normed}_?)^\mathrm{opposite},\mathrm{Grotopology,homotopyepimorphism}}}$: Colimits completion of $\infty$-Quasicoherent Sheaves of Functional Analytic Modules over $\infty$-sheaves into $\infty$-groupoid over corresponding opposite category carrying the corresponding Grothendieck topology, and we will mainly consider the corresponding homotopy epimorphisms. $?=R,\mathbb{F}_1$. And we assume the stack carries $\infty$-ringed toposes structure by specializing a ring object $\mathcal{O}$. The difference is that we consider the corresponding $\mathcal{O}$ living in the colimit completion closure.\\
\noindent $\mathrm{Ind}\mathrm{\sharp Quasicoherent}_{\mathrm{Ind}^\text{smoothformalseriesclosure}\infty-\mathrm{Toposes}^{\mathrm{ringed},\mathrm{Commutativealgebra}_{\mathrm{simplicial}}(\mathrm{Ind}^m\mathrm{Normed}_?)}_{\mathrm{Commutativealgebra}_{\mathrm{simplicial}}(\mathrm{Ind}^m\mathrm{Normed}_?)^\mathrm{opposite},\mathrm{Grotopology,homotopyepimorphism}}}$: Colimits completion of $\infty$-Quasicoherent Sheaves of Functional Analytic Modules over $\infty$-sheaves into $\infty$-groupoid over corresponding opposite category carrying the corresponding Grothendieck topology, and we will mainly consider the corresponding homotopy epimorphisms. $?=R,\mathbb{F}_1$. And we assume the stack carries $\infty$-ringed toposes structure by specializing a ring object $\mathcal{O}$. The difference is that we consider the corresponding $\mathcal{O}$ living in the colimit completion closure.\\
\noindent $\mathrm{Ind}\mathrm{\sharp Quasicoherent}_{\mathrm{Ind}^\text{smoothformalseriesclosure}\infty-\mathrm{Toposes}^{\mathrm{ringed},\mathrm{Commutativealgebra}_{\mathrm{simplicial}}(\mathrm{Ind}\mathrm{Banach}_?)}_{\mathrm{Commutativealgebra}_{\mathrm{simplicial}}(\mathrm{Ind}\mathrm{Banach}_?)^\mathrm{opposite},\mathrm{Grotopology,homotopyepimorphism}}}$: Colimits completion of $\infty$-Quasicoherent Sheaves of Functional Analytic Modules over $\infty$-sheaves into $\infty$-groupoid over corresponding opposite category carrying the corresponding Grothendieck topology, and we will mainly consider the corresponding homotopy epimorphisms. $?=R,\mathbb{F}_1$. And we assume the stack carries $\infty$-ringed toposes structure by specializing a ring object $\mathcal{O}$. The difference is that we consider the corresponding $\mathcal{O}$ living in the colimit completion closure.\\
\noindent $\mathrm{Ind}\mathrm{\sharp Quasicoherent}_{\mathrm{Ind}^\text{smoothformalseriesclosure}\infty-\mathrm{Toposes}^{\mathrm{ringed},\mathrm{Commutativealgebra}_{\mathrm{simplicial}}(\mathrm{Ind}^m\mathrm{Banach}_?)}_{\mathrm{Commutativealgebra}_{\mathrm{simplicial}}(\mathrm{Ind}^m\mathrm{Banach}_?)^\mathrm{opposite},\mathrm{Grotopology,homotopyepimorphism}}}$: Colimits completion of $\infty$-Quasicoherent Sheaves of Functional Analytic Modules over $\infty$-sheaves into $\infty$-groupoid over corresponding opposite category carrying the corresponding Grothendieck topology, and we will mainly consider the corresponding homotopy epimorphisms. $?=R,\mathbb{F}_1$. And we assume the stack carries $\infty$-ringed toposes structure by specializing a ring object $\mathcal{O}$. The difference is that we consider the corresponding $\mathcal{O}$ living in the colimit completion closure.\\

\newpage
\section{Notations on $\infty$-Categories of Noncommutative $\infty$-Ringed Toposes}

\

\indent We change the notations for $\infty$-ringed noncommutative toposes slightly, therefore let us start from the rings. The corresponding generators we will choose in order to take the corresponding homotopy colimit completion and the corresponding homotopy limit completion are Fukaya-Kato rings in \cite{12FK}:
\begin{align}
R\left<Z_1,...,Z_n\right>,n\geq 1,\\
R\left[[Z_1,...,Z_n\right]],n\geq 1,	
\end{align}
with $Z_1,...,Z_n$ are noncommuting free variables. This would be the specific completions of the polynomials:
\begin{align}
R\left[Z_1,...,Z_n\right],n\geq 1.\\	
\end{align}

\noindent The corresponding analogs of analytification functors from Ben-Bassat-Mukherjee \cite[Section 4.2]{BBM} are given in the following. First we consider the $\infty$-category of $\mathbb{E}_1$-rings from \cite[Proposition 7.1.4.18, as well as the discussion above Proposition 7.1.4.18 on page 1225]{12Lu2} which we denote it by $\mathrm{Noncommutative}_{\mathbb{E}_1,\mathrm{Simplicial}}$, then we consider the corresponding category of all the polynomial rings with free variables over $R$, which we denote it by $\mathrm{Polynomial}^\mathrm{free}_R$, then we have the corresponding fully faithful embedding:
\begin{align}
\mathrm{Polynomial}^\mathrm{free}_R\rightarrow \mathrm{Noncommutative}_{\mathbb{E}_1,\mathrm{Simplicial}}.	
\end{align}
Then take the corresponding completion for each:
\begin{align}
R\left[Z_1,...,Z_n\right],n\geq 1,\\	
\end{align}
we have the Fukaya-Kato adic ring:
\begin{align}
R\left<Z_1,...,Z_n\right>,n\geq 1,\\
R\left[[Z_1,...,Z_n\right]],n\geq 1.	
\end{align}
As in \cite[Section 4.2]{BBM}, this will give the process what we call smooth formal series analytification by taking into account the corresponding homotopy colimit completion:\\
\begin{align}
\mathrm{Ind}^{\mathrm{smoothformalseriesclosure}}\mathrm{Polynomial}^\mathrm{free}_R\rightarrow \mathrm{Noncommutativealgebra}_{\mathrm{simplicial}}(\mathrm{Ind}\mathrm{Seminormed}_R),\\
\mathrm{Ind}^{\mathrm{smoothformalseriesclosure}}\mathrm{Polynomial}^\mathrm{free}_R\rightarrow \mathrm{Noncommutativealgebra}_{\mathrm{simplicial}}(\mathrm{Ind}\mathrm{Normed}_R),\\
\mathrm{Ind}^{\mathrm{smoothformalseriesclosure}}\mathrm{Polynomial}^\mathrm{free}_R\rightarrow \mathrm{Noncommutativealgebra}_{\mathrm{simplicial}}(\mathrm{Ind}\mathrm{Banach}_R).\\	
\end{align}
\begin{remark}
\indent The left actually spans all the $\mathbb{E}_1$-algebra in our setting by regarding the free variable polynomials as tensor algebras over $R$ as explained in \cite[Proposition 7.1.4.18, as well as the discussion above Proposition 7.1.4.18 on page 1225]{12Lu1} in analogy of \cite[Proposition 7.1.4.20, as well as the discussion above Proposition 7.1.4.20]{12Lu1} in the commutative situation.
\end{remark}

\

\newpage

\noindent Noncommutative Rings:\\

\noindent $\mathrm{Ind}^\text{smoothformalseriesclosure}\mathrm{Noncommutativealgebra}_{\mathrm{simplicial}}(\mathrm{Ind}\mathrm{Seminormed}_R)$: $(\infty,1)$-category of simplicial noncommutative algebra objects in the $\infty$-category of colimit completion of seminormed modules over a general Banach ring $R$, generated by the corresponding power series rings by taking colimits in the homotopical sense. \\ 
\noindent $\mathrm{Ind}^\text{smoothformalseriesclosure}\mathrm{Noncommutativealgebra}_{\mathrm{simplicial}}(\mathrm{Ind}^m\mathrm{Seminormed}_R)$: $(\infty,1)$-category of simplicial noncommutative algebra objects in the $\infty$-category of monomorphic colimit completion of seminormed modules over a general Banach ring $R$, generated by the corresponding power series rings by taking colimits in the homotopical sense. \\ 
\noindent $\mathrm{Ind}^\text{smoothformalseriesclosure}\mathrm{Noncommutativealgebra}_{\mathrm{simplicial}}(\mathrm{Ind}\mathrm{Normed}_R)$: $(\infty,1)$-category of simplicial noncommutative algebra objects in the $\infty$-category of colimit completion of normed modules over a general Banach ring $R$, generated by the corresponding power series rings by taking colimits in the homotopical sense. \\ 
\noindent $\mathrm{Ind}^\text{smoothformalseriesclosure}\mathrm{Noncommutativealgebra}_{\mathrm{simplicial}}(\mathrm{Ind}^m\mathrm{Normed}_R)$: $(\infty,1)$-category of simplicial noncommutative algebra objects in the $\infty$-category of monomorphic colimit completion of normed modules over a general Banach ring $R$, generated by the corresponding power series rings by taking colimits in the homotopical sense. \\ 
\noindent $\mathrm{Ind}^\text{smoothformalseriesclosure}\mathrm{Noncommutativealgebra}_{\mathrm{simplicial}}(\mathrm{Ind}\mathrm{Banach}_R)$: $(\infty,1)$-category of simplicial noncommutative algebra objects in the $\infty$-category of colimit completion of Banach modules over a general Banach ring $R$, generated by the corresponding power series rings by taking colimits in the homotopical sense. \\ 
\noindent $\mathrm{Ind}^\text{smoothformalseriesclosure}\mathrm{Noncommutativealgebra}_{\mathrm{simplicial}}(\mathrm{Ind}^m\mathrm{Banach}_R)$: $(\infty,1)$-category of simplicial noncommutative algebra objects in the $\infty$-category of monomorphic colimit completion of Banach modules over a general Banach ring $R$, generated by the corresponding power series rings by taking colimits in the homotopical sense. \\

\noindent $\mathrm{Ind}^\text{smoothformalseriesclosure}\mathrm{Noncommutativealgebra}_{\mathrm{simplicial}}(\mathrm{Ind}\mathrm{Seminormed}_{\mathbb{F}_1})$: $(\infty,1)$-category of simplicial noncommutative algebra objects in the $\infty$-category of colimit completion of seminormed sets over a general Banach ring $\mathbb{F}_1$, generated by the corresponding power series rings by taking colimits in the homotopical sense. \\ 
\noindent $\mathrm{Ind}^\text{smoothformalseriesclosure}\mathrm{Noncommutativealgebra}_{\mathrm{simplicial}}(\mathrm{Ind}^m\mathrm{Seminormed}_{\mathbb{F}_1})$: $(\infty,1)$-category of simplicial noncommutative algebra objects in the $\infty$-category of monomorphic colimit completion of seminormed sets over a general Banach ring $\mathbb{F}_1$, generated by the corresponding power series rings by taking colimits in the homotopical sense. \\ 
\noindent $\mathrm{Ind}^\text{smoothformalseriesclosure}\mathrm{Noncommutativealgebra}_{\mathrm{simplicial}}(\mathrm{Ind}\mathrm{Normed}_{\mathbb{F}_1})$: $(\infty,1)$-category of simplicial noncommutative algebra objects in the $\infty$-category of colimit completion of normed sets over a general Banach ring $\mathbb{F}_1$, generated by the corresponding power series rings by taking colimits in the homotopical sense. \\ 
\noindent $\mathrm{Ind}^\text{smoothformalseriesclosure}\mathrm{Noncommutativealgebra}_{\mathrm{simplicial}}(\mathrm{Ind}^m\mathrm{Normed}_{\mathbb{F}_1})$: $(\infty,1)$-category of simplicial noncommutative algebra objects in the $\infty$-category of monomorphic colimit completion of normed sets over a general Banach ring $\mathbb{F}_1$, generated by the corresponding power series rings by taking colimits in the homotopical sense. \\ 
\noindent $\mathrm{Ind}^\text{smoothformalseriesclosure}\mathrm{Noncommutativealgebra}_{\mathrm{simplicial}}(\mathrm{Ind}\mathrm{Banach}_{\mathbb{F}_1})$: $(\infty,1)$-category of simplicial noncommutative algebra objects in the $\infty$-category of colimit completion of Banach sets over a general Banach ring $\mathbb{F}_1$, generated by the corresponding power series rings by taking colimits in the homotopical sense. \\ 
\noindent $\mathrm{Ind}^\text{smoothformalseriesclosure}\mathrm{Noncommutativealgebra}_{\mathrm{simplicial}}(\mathrm{Ind}^m\mathrm{Banach}_{\mathbb{F}_1})$: $(\infty,1)$-category of simplicial noncommutative algebra objects in the $\infty$-category of monomorphic colimit completion of Banach sets over a general Banach ring $\mathbb{F}_1$, generated by the corresponding power series rings by taking colimits in the homotopical sense. \\

\noindent Prestacks:\\

\noindent $\infty-\mathrm{Prestack}_{\mathrm{Noncommutativealgebra}_{\mathrm{simplicial}}(\mathrm{Ind}\mathrm{Seminormed}_?)^\mathrm{opposite},\mathrm{Grotopology,homotopyepimorphism}}$: $\infty$-presheaves into $\infty$-groupoid over corresponding opposite category carrying the corresponding Grothendieck topology, and we will mainly consider the corresponding homotopy epimorphisms. $?=R,\mathbb{F}_1$.\\
\noindent $\infty-\mathrm{Prestack}_{\mathrm{Noncommutativealgebra}_{\mathrm{simplicial}}(\mathrm{Ind}^m\mathrm{Seminormed}_?)^\mathrm{opposite},\mathrm{Grotopology,homotopyepimorphism}}$: $\infty$-presheaves into $\infty$-groupoid over corresponding opposite category carrying the corresponding Grothendieck topology, and we will mainly consider the corresponding homotopy epimorphisms. $?=R,\mathbb{F}_1$.\\
\noindent $\infty-\mathrm{Prestack}_{\mathrm{Noncommutativealgebra}_{\mathrm{simplicial}}(\mathrm{Ind}\mathrm{Normed}_?)^\mathrm{opposite},\mathrm{Grotopology,homotopyepimorphism}}$: $\infty$-presheaves into $\infty$-groupoid over corresponding opposite category carrying the corresponding Grothendieck topology, and we will mainly consider the corresponding homotopy epimorphisms. $?=R,\mathbb{F}_1$.\\
\noindent $\infty-\mathrm{Prestack}_{\mathrm{Noncommutativealgebra}_{\mathrm{simplicial}}(\mathrm{Ind}^m\mathrm{Normed}_?)^\mathrm{opposite},\mathrm{Grotopology,homotopyepimorphism}}$: $\infty$-presheaves into $\infty$-groupoid over corresponding opposite category carrying the corresponding Grothendieck topology, and we will mainly consider the corresponding homotopy epimorphisms. $?=R,\mathbb{F}_1$.\\
\noindent $\infty-\mathrm{Prestack}_{\mathrm{Noncommutativealgebra}_{\mathrm{simplicial}}(\mathrm{Ind}\mathrm{Banach}_?)^\mathrm{opposite},\mathrm{Grotopology,homotopyepimorphism}}$: $\infty$-presheaves into $\infty$-groupoid over corresponding opposite category carrying the corresponding Grothendieck topology, and we will mainly consider the corresponding homotopy epimorphisms. $?=R,\mathbb{F}_1$.\\
\noindent $\infty-\mathrm{Prestack}_{\mathrm{Noncommutativealgebra}_{\mathrm{simplicial}}(\mathrm{Ind}^m\mathrm{Banach}_?)^\mathrm{opposite},\mathrm{Grotopology,homotopyepimorphism}}$: $\infty$-presheaves into $\infty$-groupoid over corresponding opposite category carrying the corresponding Grothendieck topology, and we will mainly consider the corresponding homotopy epimorphisms. $?=R,\mathbb{F}_1$.\\

\noindent Stacks:\\
 
 \noindent $\infty-\mathrm{Stack}_{\mathrm{Noncommutativealgebra}_{\mathrm{simplicial}}(\mathrm{Ind}\mathrm{Seminormed}_?)^\mathrm{opposite},\mathrm{Grotopology,homotopyepimorphism}}$: $\infty$-sheaves into $\infty$-groupoid over corresponding opposite category carrying the corresponding Grothendieck topology, and we will mainly consider the corresponding homotopy epimorphisms. $?=R,\mathbb{F}_1$.\\
\noindent $\infty-\mathrm{Stack}_{\mathrm{Noncommutativealgebra}_{\mathrm{simplicial}}(\mathrm{Ind}^m\mathrm{Seminormed}_?)^\mathrm{opposite},\mathrm{Grotopology,homotopyepimorphism}}$: $\infty$-sheaves into $\infty$-groupoid over corresponding opposite category carrying the corresponding Grothendieck topology, and we will mainly consider the corresponding homotopy epimorphisms. $?=R,\mathbb{F}_1$.\\
\noindent $\infty-\mathrm{Stack}_{\mathrm{Noncommutativealgebra}_{\mathrm{simplicial}}(\mathrm{Ind}\mathrm{Normed}_?)^\mathrm{opposite},\mathrm{Grotopology,homotopyepimorphism}}$: $\infty$-sheaves into $\infty$-groupoid over corresponding opposite category carrying the corresponding Grothendieck topology, and we will mainly consider the corresponding homotopy epimorphisms. $?=R,\mathbb{F}_1$.\\
\noindent $\infty-\mathrm{Stack}_{\mathrm{Noncommutativealgebra}_{\mathrm{simplicial}}(\mathrm{Ind}^m\mathrm{Normed}_?)^\mathrm{opposite},\mathrm{Grotopology,homotopyepimorphism}}$: $\infty$-sheaves into $\infty$-groupoid over corresponding opposite category carrying the corresponding Grothendieck topology, and we will mainly consider the corresponding homotopy epimorphisms. $?=R,\mathbb{F}_1$.\\
\noindent $\infty-\mathrm{Stack}_{\mathrm{Noncommutativealgebra}_{\mathrm{simplicial}}(\mathrm{Ind}\mathrm{Banach}_?)^\mathrm{opposite},\mathrm{Grotopology,homotopyepimorphism}}$: $\infty$-sheaves into $\infty$-groupoid over corresponding opposite category carrying the corresponding Grothendieck topology, and we will mainly consider the corresponding homotopy epimorphisms. $?=R,\mathbb{F}_1$.\\
\noindent $\infty-\mathrm{Stack}_{\mathrm{Noncommutativealgebra}_{\mathrm{simplicial}}(\mathrm{Ind}^m\mathrm{Banach}_?)^\mathrm{opposite},\mathrm{Grotopology,homotopyepimorphism}}$: $\infty$-sheaves into $\infty$-groupoid over corresponding opposite category carrying the corresponding Grothendieck topology, and we will mainly consider the corresponding homotopy epimorphisms. $?=R,\mathbb{F}_1$.\\

\indent The following is a noncommutative analog of \cite[Definition 5.6]{12BBBK}:

\begin{definition}
Here a homotopy epimorphism is defined to be such a morphism $A\rightarrow B$ such that $B\overline{\otimes}_A B^\mathrm{opp}\rightarrow B$ reflects isomorphism in the homotopy category.
\end{definition}

\noindent Ringed Toposes: \\
 
 \noindent $\infty-\mathrm{Toposes}^{\mathrm{ringed},\mathrm{Noncommutativealgebra}_{\mathrm{simplicial}}(\mathrm{Ind}\mathrm{Seminormed}_?)}_{\mathrm{Noncommutativealgebra}_{\mathrm{simplicial}}(\mathrm{Ind}\mathrm{Seminormed}_?)^\mathrm{opposite},\mathrm{Grotopology,homotopyepimorphism}}$: $\infty$-sheaves into $\infty$-groupoid over corresponding opposite category carrying the corresponding Grothendieck topology, and we will mainly consider the corresponding homotopy epimorphisms. $?=R,\mathbb{F}_1$. And we assume the stack carries $\infty$-ringed toposes structure. \\
\noindent $\infty-\mathrm{Toposes}^{\mathrm{ringed},\mathrm{Noncommutativealgebra}_{\mathrm{simplicial}}(\mathrm{Ind}^m\mathrm{Seminormed}_?)}_{\mathrm{Noncommutativealgebra}_{\mathrm{simplicial}}(\mathrm{Ind}^m\mathrm{Seminormed}_?)^\mathrm{opposite},\mathrm{Grotopology,homotopyepimorphism}}$: $\infty$-sheaves into $\infty$-groupoid over corresponding opposite category carrying the corresponding Grothendieck topology, and we will mainly consider the corresponding homotopy epimorphisms. $?=R,\mathbb{F}_1$. And we assume the stack carries $\infty$-ringed toposes structure.\\
\noindent $\infty-\mathrm{Toposes}^{\mathrm{ringed},\mathrm{Noncommutativealgebra}_{\mathrm{simplicial}}(\mathrm{Ind}\mathrm{Normed}_?)}_{\mathrm{Noncommutativealgebra}_{\mathrm{simplicial}}(\mathrm{Ind}\mathrm{Normed}_?)^\mathrm{opposite},\mathrm{Grotopology,homotopyepimorphism}}$: $\infty$-sheaves into $\infty$-groupoid over corresponding opposite category carrying the corresponding Grothendieck topology, and we will mainly consider the corresponding homotopy epimorphisms. $?=R,\mathbb{F}_1$. And we assume the stack carries $\infty$-ringed toposes structure.\\
\noindent $\infty-\mathrm{Toposes}^{\mathrm{ringed},\mathrm{Noncommutativealgebra}_{\mathrm{simplicial}}(\mathrm{Ind}^m\mathrm{Normed}_?)}_{\mathrm{Noncommutativealgebra}_{\mathrm{simplicial}}(\mathrm{Ind}^m\mathrm{Normed}_?)^\mathrm{opposite},\mathrm{Grotopology,homotopyepimorphism}}$: $\infty$-sheaves into $\infty$-groupoid over corresponding opposite category carrying the corresponding Grothendieck topology, and we will mainly consider the corresponding homotopy epimorphisms. $?=R,\mathbb{F}_1$. And we assume the stack carries $\infty$-ringed toposes structure.\\
\noindent $\infty-\mathrm{Toposes}^{\mathrm{ringed},\mathrm{Noncommutativealgebra}_{\mathrm{simplicial}}(\mathrm{Ind}\mathrm{Banach}_?)}_{\mathrm{Noncommutativealgebra}_{\mathrm{simplicial}}(\mathrm{Ind}\mathrm{Banach}_?)^\mathrm{opposite},\mathrm{Grotopology,homotopyepimorphism}}$: $\infty$-sheaves into $\infty$-groupoid over corresponding opposite category carrying the corresponding Grothendieck topology, and we will mainly consider the corresponding homotopy epimorphisms. $?=R,\mathbb{F}_1$. And we assume the stack carries $\infty$-ringed toposes structure.\\
\noindent $\infty-\mathrm{Toposes}^{\mathrm{ringed},\mathrm{Noncommutativealgebra}_{\mathrm{simplicial}}(\mathrm{Ind}^m\mathrm{Banach}_?)}_{\mathrm{Noncommutativealgebra}_{\mathrm{simplicial}}(\mathrm{Ind}^m\mathrm{Banach}_?)^\mathrm{opposite},\mathrm{Grotopology,homotopyepimorphism}}$: $\infty$-sheaves into $\infty$-groupoid over corresponding opposite category carrying the corresponding Grothendieck topology, and we will mainly consider the corresponding homotopy epimorphisms. $?=R,\mathbb{F}_1$. And we assume the stack carries $\infty$-ringed toposes structure.\\

 \noindent $\mathrm{Proj}^\text{smoothformalseriesclosure}\infty-\mathrm{Toposes}^{\mathrm{ringed},\mathrm{Noncommutativealgebra}_{\mathrm{simplicial}}(\mathrm{Ind}\mathrm{Seminormed}_?)}_{\mathrm{Noncommutativealgebra}_{\mathrm{simplicial}}(\mathrm{Ind}\mathrm{Seminormed}_?)^\mathrm{opposite},\mathrm{Grotopology,homotopyepimorphism}}$: $\infty$-sheaves into $\infty$-groupoid over corresponding opposite category carrying the corresponding Grothendieck topology, and we will mainly consider the corresponding homotopy epimorphisms. $?=R,\mathbb{F}_1$. And we assume the stack carries $\infty$-ringed toposes structure by specializing a ring object $\mathcal{O}$. \\
\noindent $\mathrm{Proj}^\text{smoothformalseriesclosure}\infty-\mathrm{Toposes}^{\mathrm{ringed},\mathrm{Noncommutativealgebra}_{\mathrm{simplicial}}(\mathrm{Ind}^m\mathrm{Seminormed}_?)}_{\mathrm{Noncommutativealgebra}_{\mathrm{simplicial}}(\mathrm{Ind}^m\mathrm{Seminormed}_?)^\mathrm{opposite},\mathrm{Grotopology,homotopyepimorphism}}$: $\infty$-sheaves into $\infty$-groupoid over corresponding opposite category carrying the corresponding Grothendieck topology, and we will mainly consider the corresponding homotopy epimorphisms. $?=R,\mathbb{F}_1$. And we assume the stack carries $\infty$-ringed toposes structure by specializing a ring object $\mathcal{O}$.\\
\noindent $\mathrm{Proj}^\text{smoothformalseriesclosure}\infty-\mathrm{Toposes}^{\mathrm{ringed},\mathrm{Noncommutativealgebra}_{\mathrm{simplicial}}(\mathrm{Ind}\mathrm{Normed}_?)}_{\mathrm{Noncommutativealgebra}_{\mathrm{simplicial}}(\mathrm{Ind}\mathrm{Normed}_?)^\mathrm{opposite},\mathrm{Grotopology,homotopyepimorphism}}$: $\infty$-sheaves into $\infty$-groupoid over corresponding opposite category carrying the corresponding Grothendieck topology, and we will mainly consider the corresponding homotopy epimorphisms. $?=R,\mathbb{F}_1$. And we assume the stack carries $\infty$-ringed toposes structure by specializing a ring object $\mathcal{O}$.\\
\noindent $\mathrm{Proj}^\text{smoothformalseriesclosure}\infty-\mathrm{Toposes}^{\mathrm{ringed},\mathrm{Noncommutativealgebra}_{\mathrm{simplicial}}(\mathrm{Ind}^m\mathrm{Normed}_?)}_{\mathrm{Noncommutativealgebra}_{\mathrm{simplicial}}(\mathrm{Ind}^m\mathrm{Normed}_?)^\mathrm{opposite},\mathrm{Grotopology,homotopyepimorphism}}$: $\infty$-sheaves into $\infty$-groupoid over corresponding opposite category carrying the corresponding Grothendieck topology, and we will mainly consider the corresponding homotopy epimorphisms. $?=R,\mathbb{F}_1$. And we assume the stack carries $\infty$-ringed toposes structure by specializing a ring object $\mathcal{O}$.\\
\noindent $\mathrm{Proj}^\text{smoothformalseriesclosure}\infty-\mathrm{Toposes}^{\mathrm{ringed},\mathrm{Noncommutativealgebra}_{\mathrm{simplicial}}(\mathrm{Ind}\mathrm{Banach}_?)}_{\mathrm{Noncommutativealgebra}_{\mathrm{simplicial}}(\mathrm{Ind}\mathrm{Banach}_?)^\mathrm{opposite},\mathrm{Grotopology,homotopyepimorphism}}$: $\infty$-sheaves into $\infty$-groupoid over corresponding opposite category carrying the corresponding Grothendieck topology, and we will mainly consider the corresponding homotopy epimorphisms. $?=R,\mathbb{F}_1$. And we assume the stack carries $\infty$-ringed toposes structure by specializing a ring object $\mathcal{O}$.\\
\noindent $\mathrm{Proj}^\text{smoothformalseriesclosure}\infty-\mathrm{Toposes}^{\mathrm{ringed},\mathrm{Noncommutativealgebra}_{\mathrm{simplicial}}(\mathrm{Ind}^m\mathrm{Banach}_?)}_{\mathrm{Noncommutativealgebra}_{\mathrm{simplicial}}(\mathrm{Ind}^m\mathrm{Banach}_?)^\mathrm{opposite},\mathrm{Grotopology,homotopyepimorphism}}$: $\infty$-sheaves into $\infty$-groupoid over corresponding opposite category carrying the corresponding Grothendieck topology, and we will mainly consider the corresponding homotopy epimorphisms. $?=R,\mathbb{F}_1$. And we assume the stack carries $\infty$-ringed toposes structure by specializing a ring object $\mathcal{O}$.\\

\noindent $\infty$-Quasicoherent Sheaves of Functional Analytic Modules over Ringed Toposes $\sharp=\mathrm{Seminormed},\mathrm{Normed},\mathrm{Banach}$: \\
 
 \noindent $\mathrm{Ind}\mathrm{\sharp Quasicoherent}_{\infty-\mathrm{Toposes}^{\mathrm{ringed},\mathrm{Noncommutativealgebra}_{\mathrm{simplicial}}(\mathrm{Ind}\mathrm{Seminormed}_?)}_{\mathrm{Noncommutativealgebra}_{\mathrm{simplicial}}(\mathrm{Ind}\mathrm{Seminormed}_?)^\mathrm{opposite},\mathrm{Grotopology,homotopyepimorphism}}}$: Colimits completion of $\infty$-Quasicoherent Sheaves of Functional Analytic Modules over $\infty$-sheaves into $\infty$-groupoid over corresponding opposite category carrying the corresponding Grothendieck topology, and we will mainly consider the corresponding homotopy epimorphisms. $?=R,\mathbb{F}_1$. And we assume the stack carries $\infty$-ringed toposes structure. \\
\noindent $\mathrm{Ind}\mathrm{\sharp Quasicoherent}_{\infty-\mathrm{Toposes}^{\mathrm{ringed},\mathrm{Noncommutativealgebra}_{\mathrm{simplicial}}(\mathrm{Ind}^m\mathrm{Seminormed}_?)}_{\mathrm{Noncommutativealgebra}_{\mathrm{simplicial}}(\mathrm{Ind}^m\mathrm{Seminormed}_?)^\mathrm{opposite},\mathrm{Grotopology,homotopyepimorphism}}}$: Colimits completion of $\infty$-Quasicoherent Sheaves of Functional Analytic Modules over $\infty$-sheaves into $\infty$-groupoid over corresponding opposite category carrying the corresponding Grothendieck topology, and we will mainly consider the corresponding homotopy epimorphisms. $?=R,\mathbb{F}_1$. And we assume the stack carries $\infty$-ringed toposes structure.\\
\noindent $\mathrm{Ind}\mathrm{\sharp Quasicoherent}_{\infty-\mathrm{Toposes}^{\mathrm{ringed},\mathrm{Noncommutativealgebra}_{\mathrm{simplicial}}(\mathrm{Ind}\mathrm{Normed}_?)}_{\mathrm{Noncommutativealgebra}_{\mathrm{simplicial}}(\mathrm{Ind}\mathrm{Normed}_?)^\mathrm{opposite},\mathrm{Grotopology,homotopyepimorphism}}}$: Colimits completion of $\infty$-Quasicoherent Sheaves of Functional Analytic Modules over $\infty$-sheaves into $\infty$-groupoid over corresponding opposite category carrying the corresponding Grothendieck topology, and we will mainly consider the corresponding homotopy epimorphisms. $?=R,\mathbb{F}_1$. And we assume the stack carries $\infty$-ringed toposes structure.\\
\noindent $\mathrm{Ind}\mathrm{\sharp Quasicoherent}_{\infty-\mathrm{Toposes}^{\mathrm{ringed},\mathrm{Noncommutativealgebra}_{\mathrm{simplicial}}(\mathrm{Ind}^m\mathrm{Normed}_?)}_{\mathrm{Noncommutativealgebra}_{\mathrm{simplicial}}(\mathrm{Ind}^m\mathrm{Normed}_?)^\mathrm{opposite},\mathrm{Grotopology,homotopyepimorphism}}}$: Colimits completion of $\infty$-Quasicoherent Sheaves of Functional Analytic Modules over $\infty$-sheaves into $\infty$-groupoid over corresponding opposite category carrying the corresponding Grothendieck topology, and we will mainly consider the corresponding homotopy epimorphisms. $?=R,\mathbb{F}_1$. And we assume the stack carries $\infty$-ringed toposes structure.\\
\noindent $\mathrm{Ind}\mathrm{\sharp Quasicoherent}_{\infty-\mathrm{Toposes}^{\mathrm{ringed},\mathrm{Noncommutativealgebra}_{\mathrm{simplicial}}(\mathrm{Ind}\mathrm{Banach}_?)}_{\mathrm{Noncommutativealgebra}_{\mathrm{simplicial}}(\mathrm{Ind}\mathrm{Banach}_?)^\mathrm{opposite},\mathrm{Grotopology,homotopyepimorphism}}}$: Colimits completion of $\infty$-Quasicoherent Sheaves of Functional Analytic Modules over $\infty$-sheaves into $\infty$-groupoid over corresponding opposite category carrying the corresponding Grothendieck topology, and we will mainly consider the corresponding homotopy epimorphisms. $?=R,\mathbb{F}_1$. And we assume the stack carries $\infty$-ringed toposes structure.\\
\noindent $\mathrm{Ind}\mathrm{\sharp Quasicoherent}_{\infty-\mathrm{Toposes}^{\mathrm{ringed},\mathrm{Noncommutativealgebra}_{\mathrm{simplicial}}(\mathrm{Ind}^m\mathrm{Banach}_?)}_{\mathrm{Noncommutativealgebra}_{\mathrm{simplicial}}(\mathrm{Ind}^m\mathrm{Banach}_?)^\mathrm{opposite},\mathrm{Grotopology,homotopyepimorphism}}}$: Colimits completion of $\infty$-Quasicoherent Sheaves of Functional Analytic Modules over $\infty$-sheaves into $\infty$-groupoid over corresponding opposite category carrying the corresponding Grothendieck topology, and we will mainly consider the corresponding homotopy epimorphisms. $?=R,\mathbb{F}_1$. And we assume the stack carries $\infty$-ringed toposes structure.\\

 \noindent $\mathrm{Ind}\mathrm{\sharp Quasicoherent}_{\mathrm{Ind}^\text{smoothformalseriesclosure}\infty-\mathrm{Toposes}^{\mathrm{ringed},\mathrm{Noncommutativealgebra}_{\mathrm{simplicial}}(\mathrm{Ind}\mathrm{Seminormed}_?)}_{\mathrm{Noncommutativealgebra}_{\mathrm{simplicial}}(\mathrm{Ind}\mathrm{Seminormed}_?)^\mathrm{opposite},\mathrm{Grotopology,homotopyepimorphism}}}$: Colimits completion of $\infty$-Quasicoherent Sheaves of Functional Analytic Modules over $\infty$-sheaves into $\infty$-groupoid over corresponding opposite category carrying the corresponding Grothendieck topology, and we will mainly consider the corresponding homotopy epimorphisms. $?=R,\mathbb{F}_1$. And we assume the stack carries $\infty$-ringed toposes structure by specializing a ring object $\mathcal{O}$. The difference is that we consider the corresponding $\mathcal{O}$ living in the colimit completion closure. \\
\noindent $\mathrm{Ind}\mathrm{\sharp Quasicoherent}_{\mathrm{Ind}^\text{smoothformalseriesclosure}\infty-\mathrm{Toposes}^{\mathrm{ringed},\mathrm{Noncommutativealgebra}_{\mathrm{simplicial}}(\mathrm{Ind}^m\mathrm{Seminormed}_?)}_{\mathrm{Noncommutativealgebra}_{\mathrm{simplicial}}(\mathrm{Ind}^m\mathrm{Seminormed}_?)^\mathrm{opposite},\mathrm{Grotopology,homotopyepimorphism}}}$: Colimits completion of $\infty$-Quasicoherent Sheaves of Functional Analytic Modules over $\infty$-sheaves into $\infty$-groupoid over corresponding opposite category carrying the corresponding Grothendieck topology, and we will mainly consider the corresponding homotopy epimorphisms. $?=R,\mathbb{F}_1$. And we assume the stack carries $\infty$-ringed toposes structure by specializing a ring object $\mathcal{O}$. The difference is that we consider the corresponding $\mathcal{O}$ living in the colimit completion closure.\\
\noindent $\mathrm{Ind}\mathrm{\sharp Quasicoherent}_{\mathrm{Ind}^\text{smoothformalseriesclosure}\infty-\mathrm{Toposes}^{\mathrm{ringed},\mathrm{Noncommutativealgebra}_{\mathrm{simplicial}}(\mathrm{Ind}\mathrm{Normed}_?)}_{\mathrm{Noncommutativealgebra}_{\mathrm{simplicial}}(\mathrm{Ind}\mathrm{Normed}_?)^\mathrm{opposite},\mathrm{Grotopology,homotopyepimorphism}}}$: Colimits completion of $\infty$-Quasicoherent Sheaves of Functional Analytic Modules over $\infty$-sheaves into $\infty$-groupoid over corresponding opposite category carrying the corresponding Grothendieck topology, and we will mainly consider the corresponding homotopy epimorphisms. $?=R,\mathbb{F}_1$. And we assume the stack carries $\infty$-ringed toposes structure by specializing a ring object $\mathcal{O}$. The difference is that we consider the corresponding $\mathcal{O}$ living in the colimit completion closure.\\
\noindent $\mathrm{Ind}\mathrm{\sharp Quasicoherent}_{\mathrm{Ind}^\text{smoothformalseriesclosure}\infty-\mathrm{Toposes}^{\mathrm{ringed},\mathrm{Noncommutativealgebra}_{\mathrm{simplicial}}(\mathrm{Ind}^m\mathrm{Normed}_?)}_{\mathrm{Noncommutativealgebra}_{\mathrm{simplicial}}(\mathrm{Ind}^m\mathrm{Normed}_?)^\mathrm{opposite},\mathrm{Grotopology,homotopyepimorphism}}}$: Colimits completion of $\infty$-Quasicoherent Sheaves of Functional Analytic Modules over $\infty$-sheaves into $\infty$-groupoid over corresponding opposite category carrying the corresponding Grothendieck topology, and we will mainly consider the corresponding homotopy epimorphisms. $?=R,\mathbb{F}_1$. And we assume the stack carries $\infty$-ringed toposes structure by specializing a ring object $\mathcal{O}$. The difference is that we consider the corresponding $\mathcal{O}$ living in the colimit completion closure.\\
\noindent $\mathrm{Ind}\mathrm{\sharp Quasicoherent}_{\mathrm{Ind}^\text{smoothformalseriesclosure}\infty-\mathrm{Toposes}^{\mathrm{ringed},\mathrm{Noncommutativealgebra}_{\mathrm{simplicial}}(\mathrm{Ind}\mathrm{Banach}_?)}_{\mathrm{Noncommutativealgebra}_{\mathrm{simplicial}}(\mathrm{Ind}\mathrm{Banach}_?)^\mathrm{opposite},\mathrm{Grotopology,homotopyepimorphism}}}$: Colimits completion of $\infty$-Quasicoherent Sheaves of Functional Analytic Modules over $\infty$-sheaves into $\infty$-groupoid over corresponding opposite category carrying the corresponding Grothendieck topology, and we will mainly consider the corresponding homotopy epimorphisms. $?=R,\mathbb{F}_1$. And we assume the stack carries $\infty$-ringed toposes structure by specializing a ring object $\mathcal{O}$. The difference is that we consider the corresponding $\mathcal{O}$ living in the colimit completion closure.\\
\noindent $\mathrm{Ind}\mathrm{\sharp Quasicoherent}_{\mathrm{Ind}^\text{smoothformalseriesclosure}\infty-\mathrm{Toposes}^{\mathrm{ringed},\mathrm{Noncommutativealgebra}_{\mathrm{simplicial}}(\mathrm{Ind}^m\mathrm{Banach}_?)}_{\mathrm{Noncommutativealgebra}_{\mathrm{simplicial}}(\mathrm{Ind}^m\mathrm{Banach}_?)^\mathrm{opposite},\mathrm{Grotopology,homotopyepimorphism}}}$: Colimits completion of $\infty$-Quasicoherent Sheaves of Functional Analytic Modules over $\infty$-sheaves into $\infty$-groupoid over corresponding opposite category carrying the corresponding Grothendieck topology, and we will mainly consider the corresponding homotopy epimorphisms. $?=R,\mathbb{F}_1$. And we assume the stack carries $\infty$-ringed toposes structure by specializing a ring object $\mathcal{O}$. The difference is that we consider the corresponding $\mathcal{O}$ living in the colimit completion closure.\\

Let us explain slightly what is happening here, the base $\infty$-rings are noncommutative in certain sense, which is not the same as in the foundation of \cite{12BBK}, \cite{BBM}, \cite{KKM}, \cite{12BK}. Certainly for instance one considers the corresponding $\infty$-category $\mathrm{Simpicial}(\mathrm{Ind}\mathrm{Banach}_{\mathbb{F}_1})$, then takes the corresponding fibrations over the corresponding noncommutative rings to achieve so.

\newpage

\chapter{Topological Theory}

\section{Topological Andr\'e-Quillen Homology and Topological Derived de Rham Complexes}

\subsection{Derived $p$-Complete Derived de Rham Complex}

\indent We now first discuss the corresponding Banach version of  Andr\'e-Quillen Homology and the corresponding Banach version of  Derived de Rham complex parallel to \cite[Chapitre 3]{12An1}, \cite{12An2}, \cite[Chapter 2, Chapter 8]{12B1}, \cite[Chapter 1]{12Bei}, \cite[Chapter 5]{12G1}, \cite[Chapter 3, Chapter 4]{12GL}, \cite[Chapitre II, Chapitre III]{12Ill1}, \cite[Chapitre VIII]{12Ill2}, \cite[Section 4]{12Qui}. We would like to start from the corresponding context of \cite[Chapter 3, Chapter 4]{12GL}, and represent the construction for the convenience of the readers. We start from the corresponding construction of the algebraic $p$-adic derived de Rham complex for a map $A\rightarrow B$ of $p$-complete rings. This is the corresponding derived differential complex attached to the polynomial resolution of $B$:
\begin{align}
A[A[B]]...,	
\end{align}
which is now denoted by $\mathrm{Kan}_\mathrm{Left}\mathrm{deRham}^\text{degreenumber}_{B/A,,\mathrm{alg}}:=\mathrm{Kan}_\mathrm{Left}\mathrm{deRham}^\text{degreenumber}_{-/A,,\mathrm{alg}}(B)$ after taking the corresponding left Kan extension which will be the same for all the following constructions \footnote{We have already considered the corresponding left Kan extension to all the rings which are not concentrated at degree zero after \cite[Example 5.11, Example 5.12]{12BMS}, which is also discussed in \cite[Lecture 7]{12B2}.}. The corresponding cotangent complex associated is defined to be just:
\begin{align}
\mathbb{L}_{B/A,\mathrm{alg}}:=	\mathrm{deRham}^1_{A[B]^\text{degreenumber}/A,,\mathrm{alg}}\otimes_{A[B]^\text{degreenumber}} B.
\end{align}
The corresponding algebraic Andr\'e-Quillen homologies are defined to be:
\begin{align}
H_{\text{degreenumber},{\mathrm{AQ}},\mathrm{alg}}:=\pi_\text{degreenumber} (\mathbb{L}_{B/A,\mathrm{alg}}). 	
\end{align}

\indent The corresponding topological Andr\'e-Quillen complex is actually the completed version of the corresponding algebraic ones above by considering the corresponding certain $p$-completion over the simplicial module structure.

\indent Then we consider the corresponding derived algebraic de Rham complex which is just defined to be:
\begin{align}
\mathrm{Kan}_\mathrm{Left}\mathrm{deRham}^\text{degreenumber}_{B/A,\mathrm{alg}},\mathrm{Kan}_\mathrm{Left}\mathrm{Fil}^*_{\mathrm{deRham}^\text{degreenumber}_{B/A,\mathrm{alg}}}.	
\end{align}
We then take the corresponding Banach completion and we denote that by:
\begin{align}
\mathrm{Kan}_\mathrm{Left}\mathrm{deRham}^\text{degreenumber}_{B/A,\mathrm{topo}},\mathrm{Kan}_\mathrm{Left}\mathrm{Fil}^*_{\mathrm{deRham}^\text{degreenumber}_{B/A,\mathrm{topo}}}.	
\end{align}

\indent Then we need to take the corresponding Hodge-Filtered completion by using the corresponding filtration associated as above:
\begin{align}
\mathrm{Kan}_\mathrm{Left}\widehat{\mathrm{deRham}}^\text{degreenumber}_{B/A,\mathrm{topo}},\mathrm{Kan}_\mathrm{Left}\mathrm{Fil}^*_{\widehat{\mathrm{deRham}}^\text{degreenumber}_{B/A,\mathrm{topo}}}.	
\end{align}

This is basically the corresponding analytic and complete version the corresponding algebraic de Rham complex. Furthermore we allow large coefficients with rigid affinoid algebra $Z$ over $\mathbb{Q}_p$. Therefore we take the corresponding completed tensor product in the following. 

\begin{definition}
We define the following $Z$ deformed version of the corresponding complete version of the corresponding Andr\'e-Quillen homology and the corresponding complete version of the corresponding derived de Rham complex. We start from the corresponding construction of the algebraic $p$-adic derived de Rham complex for a map $A\rightarrow B$. Fix a pair of ring of definitions $A_0,B_0$ in $A,B$ respectively. Then this is the corresponding derived differential complex attached to the polynomial resolution of $B_0$:
\begin{align}
A_0[A_0[B_0]]...,	
\end{align}
which is now denoted by $\mathrm{Kan}_\mathrm{Left}\mathrm{deRham}^\text{degreenumber}_{B_0/A_0,\mathrm{alg}}$. The corresponding cotangent complex associated is defined to be just:
\begin{align}
\mathbb{L}_{B_0/A_0,\mathrm{alg}}:=	\mathrm{deRham}^1_{A_0[B_0]^\text{degreenumber}/A_0,\mathrm{alg}}\otimes_{A_0[B_0]^\text{degreenumber}} B_0.
\end{align}
The corresponding algebraic Andr\'e-Quillen homologies are defined to be:
\begin{align}
H_{\text{degreenumber},{\mathrm{AQ}},\mathrm{alg}}:=\pi_\text{degreenumber} (\mathbb{L}_{B_0/A_0,\mathrm{alg}}). 	
\end{align}
The corresponding topological Andr\'e-Quillen complex is actually the complete version of the corresponding algebraic ones above by considering the corresponding derived $p$-completion over the simplicial module structure:
\begin{align}
\mathbb{L}_{B_0/A_0,\mathrm{topo}}:=R\varprojlim_k	\mathrm{Kos}_{p^k}\left((\mathrm{deRham}^1_{A_0[B_0]^\text{degreenumber}/A_0,\mathrm{alg}}\otimes_{A_0[B_0]^\text{degreenumber}} B_0)\right).
\end{align}
Taking the product with $\mathcal{O}_Z$ we have the corresponding integral version of the topological Andr\'e-Quillen complex:
\begin{align}
\mathbb{L}_{B_0/A_0,\mathrm{topo},Z}:=\mathbb{L}_{B_0/A_0,\mathrm{topo}}\widehat{\otimes}_{\mathbb{Z}_p}\mathcal{O}_Z
\end{align}
Then we consider the corresponding derived algebraic de Rham complex which is just defined to be:
\begin{align}
\mathrm{Kan}_\mathrm{Left}\mathrm{deRham}^\text{degreenumber}_{B_0/A_0,\mathrm{alg}},\mathrm{Kan}_\mathrm{Left}\mathrm{Fil}^*_{\mathrm{deRham}^1_{B_0/A_0,\mathrm{alg}}}.	
\end{align}
We then take the corresponding derived $p$-completion and we denote that by:
\begin{align}
\mathrm{Kan}_\mathrm{Left}\mathrm{deRham}^\text{degreenumber}_{B_0/A_0,\mathrm{topo}}:=R\varprojlim_k\mathrm{Kos}_{p^k}\left(\mathrm{Kan}_\mathrm{Left}\mathrm{deRham}^\text{degreenumber}_{B_0/A_0,\mathrm{alg}},\mathrm{Kan}_\mathrm{Left}\mathrm{Fil}^*_{\mathrm{deRham}^\text{degreenumber}_{B_0/A_0,\mathrm{alg}}}\right),\\
\mathrm{Kan}_\mathrm{Left}\mathrm{Fil}^*_{\mathrm{deRham}^\text{degreenumber}_{B_0/A_0,\mathrm{topo}}}:=R\varprojlim_k\mathrm{Kos}_{p^k}\left(\mathrm{Kan}_\mathrm{Left}\mathrm{Fil}^*_{\mathrm{deRham}^\text{degreenumber}_{B_0/A_0,\mathrm{alg}}}\right).	
\end{align}
Before considering the corresponding integral version we just consider the corresponding product of these $\mathbb{E}_\infty$-rings with $\mathcal{O}_Z$ to get:
\begin{align}
\mathrm{Kan}_\mathrm{Left}\mathrm{deRham}^\text{degreenumber}_{B_0/A_0,\mathrm{topo},Z}:=\mathrm{Kan}_\mathrm{Left}\mathrm{deRham}^\text{degreenumber}_{B_0/A_0,\mathrm{topo}}\widehat{\otimes}^\mathbb{L}_{\mathbb{Z}_p}\mathcal{O}_Z,\\
\mathrm{Kan}_\mathrm{Left}\mathrm{Fil}^*_{\mathrm{deRham}^\text{degreenumber}_{B_0/A_0,\mathrm{topo}},Z}:=\mathrm{Kan}_\mathrm{Left}\mathrm{Fil}^*_{\mathrm{deRham}^\text{degreenumber}_{B_0/A_0,\mathrm{topo}}}\widehat{\otimes}^\mathbb{L}_{\mathbb{Z}_p}\mathcal{O}_Z.	
\end{align}
Then we consider the following construction for the map $A\rightarrow B$ by putting:
\begin{align}
\mathbb{L}_{B_0/A_0,\mathrm{topo},Z}:= \mathrm{Colim}_{A_0\rightarrow B_0}\mathbb{L}_{B_0/A_0,\mathrm{topo},Z}[1/p],\\
H_{\text{degreenumber},{\mathrm{AQ}},\mathrm{topo},Z}:=\pi_\text{degreenumber} (\mathbb{L}_{B/A,\mathrm{topo},Z}),	\\
\mathrm{Kan}_\mathrm{Left}\mathrm{deRham}^\text{degreenumber}_{B/A,\mathrm{topo},Z}:=\mathrm{Colim}_{A_0\rightarrow B_0}\mathrm{Kan}_\mathrm{Left}\mathrm{deRham}^\text{degreenumber}_{B_0/A_0,\mathrm{topo},Z}[1/p],\\
\mathrm{Kan}_\mathrm{Left}\mathrm{Fil}^*_{\mathrm{deRham}^\text{degreenumber}_{B/A,\mathrm{topo}},Z}:=\mathrm{Colim}_{A_0\rightarrow B_0}\mathrm{Kan}_\mathrm{Left}\mathrm{Fil}^*_{\mathrm{deRham}^1_{B_0/A_0,\mathrm{topo}},Z}[1/p].
\end{align}
Then we need to take the corresponding Hodge-Filtered completion by using the corresponding filtration associated as above to achieve the corresponding Hodge-complete objects in the corresponding filtered $\infty$-categories:
\begin{align}
\mathrm{Kan}_\mathrm{Left}{\mathrm{deRham}}^\text{degreenumber}_{B/A,\mathrm{topo,Hodge},Z},\mathrm{Kan}_\mathrm{Left}\mathrm{Fil}^*_{{\mathrm{deRham}}^\text{degreenumber}_{B/A,\mathrm{topo,Hodge}},Z}.	
\end{align} 	
\end{definition}

\begin{definition}
We define the corresponding finite projective filtered crystals to be the corresponding finite projective module spectra over the topological filtered $E_\infty$-ring $\mathrm{Kan}_\mathrm{Left}\mathrm{deRham}^\text{degreenumber}_{B/A,\mathrm{topo},Z}$ with the corresponding induced filtrations.	
\end{definition}

\begin{definition}
We define the corresponding almost perfect \footnote{This is the corresponding derived version of pseudocoherence from \cite{12Lu1}, \cite{12Lu2}.} filtered crystals to be the corresponding almost perfect module spectra over the topological filtered $E_\infty$-ring $\mathrm{Kan}_\mathrm{Left}\mathrm{deRham}^\text{degreenumber}_{B/A,\mathrm{topo},Z}$ with the corresponding induced filtrations.	
\end{definition}

\indent The following is derived from the main Poincar\'e Lemma from \cite[Theorem 1.2]{12GL} in the non-deformed situation. Consider a corresponding smooth rigid analytic space $X$ over $k/\mathbb{Q}_p$ (where $k$ is a corresponding unramified analytic field which is discretely-valued and the corresponding residue field is finite). Then we have the following:

\newpage

\begin{landscape}
\begin{proposition}
Consider the corresponding projective map $g:X_{\text{pro-\'etale}}\rightarrow X_{\text{\'et}}$ and the the projective map $f:X_{\text{pro-\'etale}}\rightarrow X$. Then we have the following two strictly exact long exact sequences 
\footnote{We should have the corresponding naturality taking into the following form:

\[\tiny
\xymatrix@C+0.4pc@R+0pc{
0\ar[r]\ar[r]\ar[r] \ar[r] &\mathrm{Kan}_\mathrm{Left}\mathrm{deRham}^\text{degreenumber}_{k[\widehat{\mathcal{O}}]^\text{degreenumber}/k,\mathrm{topo},Z}\ar[d]\ar[d]\ar[d] \ar[d] \ar[r]^\partial\ar[r]\ar[r] \ar[r] &\mathrm{Kan}_\mathrm{Left}\mathrm{deRham}^\text{degreenumber}_{X[\widehat{\mathcal{O}}]^\text{degreenumber}/X,\mathrm{topo},Z} \ar[d]\ar[d]\ar[d] \ar[d]\ar[r]^\partial\ar[r]\ar[r] \ar[r] &\mathrm{Kan}_\mathrm{Left}\mathrm{deRham}^\text{degreenumber}_{X[\widehat{\mathcal{O}}]^\text{degreenumber}/X,\mathrm{topo},Z}{\otimes}f^{-1} \mathrm{deRham}^1_{X,\mathrm{topo}} \ar[d]\ar[d]\ar[d] \ar[d]\ar[r]^\partial\ar[r]\ar[r] \ar[r]&\\
0\ar[r]\ar[r]\ar[r] \ar[r] &\mathrm{Kan}_\mathrm{Left}\mathrm{deRham}^\text{degreenumber}_{k[\widehat{\mathcal{O}}]^\text{degreenumber}/k,\mathrm{topo},Z} \ar[r]^\partial\ar[r]\ar[r] \ar[r] &\mathrm{Kan}_\mathrm{Left}\mathrm{deRham}^\text{degreenumber}_{X[\widehat{\mathcal{O}}]^\text{degreenumber}/X_{\text{\'et}},\mathrm{topo},Z} \ar[r]^\partial\ar[r]\ar[r] \ar[r] &\mathrm{Kan}_\mathrm{Left}\mathrm{deRham}^\text{degreenumber}_{X[\widehat{\mathcal{O}}]^\text{degreenumber}/X,\mathrm{topo},Z}{\otimes}f^{-1} \mathrm{deRham}^1_{X_{\text{\'et}},\mathrm{topo}} \ar[r]^\partial\ar[r]\ar[r] \ar[r]&.
}
\]
}:
\[\tiny
\xymatrix@C+0.4pc@R+0pc{
0\ar[r]\ar[r]\ar[r] \ar[r] &\mathrm{Kan}_\mathrm{Left}\mathrm{deRham}^\text{degreenumber}_{k[\widehat{\mathcal{O}}]^\text{degreenumber}/k,\mathrm{topo},Z} \ar[r]^\partial\ar[r]\ar[r] \ar[r] &\mathrm{Kan}_\mathrm{Left}\mathrm{deRham}^\text{degreenumber}_{X[\widehat{\mathcal{O}}]^\text{degreenumber}/X,\mathrm{topo},Z} \ar[r]^\partial\ar[r]\ar[r] \ar[r] &\mathrm{Kan}_\mathrm{Left}\mathrm{deRham}^\text{degreenumber}_{X[\widehat{\mathcal{O}}]^\text{degreenumber}/X,\mathrm{topo},Z}{\otimes}f^{-1} \mathrm{deRham}^1_{X,\mathrm{topo}} \ar[r]^\partial\ar[r]\ar[r] \ar[r]&...,
}
\]
and
\[\tiny
\xymatrix@C+0.4pc@R+0pc{
0\ar[r]\ar[r]\ar[r] \ar[r] &\mathrm{Kan}_\mathrm{Left}\mathrm{deRham}^\text{degreenumber}_{k[\widehat{\mathcal{O}}]^\text{degreenumber}/k,\mathrm{topo},Z} \ar[r]^\partial\ar[r]\ar[r] \ar[r] &\mathrm{Kan}_\mathrm{Left}\mathrm{deRham}^\text{degreenumber}_{X[\widehat{\mathcal{O}}]^\text{degreenumber}/X_{\text{\'et}},\mathrm{topo},Z} \ar[r]^\partial\ar[r]\ar[r] \ar[r] &\mathrm{Kan}_\mathrm{Left}\mathrm{deRham}^\text{degreenumber}_{X[\widehat{\mathcal{O}}]^\text{degreenumber}/X,\mathrm{topo},Z}{\otimes}f^{-1} \mathrm{deRham}^1_{X_{\text{\'et}},\mathrm{topo}} \ar[r]^\partial\ar[r]\ar[r] \ar[r] &....
}
\]
\end{proposition}

\begin{proof}
This is actually a direct consequence \cite[Theorem 1.2]{12GL}.	
\end{proof}

\begin{proposition}
Consider the corresponding projective map $g:X_{\text{pro-\'etale}}\rightarrow X_{\text{\'et}}$ and the the projective map $f:X_{\text{pro-\'etale}}\rightarrow X$. Let $M$ be a corresponding $Z$-projective differential crystal spectrum. Then we have the following two strictly exact long exact sequences \footnote{Here we drop the notation $\mathrm{Kan}_\mathrm{Left}$.}:
\[\tiny
\xymatrix@C+0.4pc@R+0pc{
M{\otimes}^\mathbb{L}(0\ar[r]\ar[r]\ar[r] \ar[r] &\mathrm{deRham}^\text{degreenumber}_{k[\widehat{\mathcal{O}}]^\text{degreenumber}/k,\mathrm{topo},Z} \ar[r]^\partial\ar[r]\ar[r] \ar[r] &\mathrm{deRham}^\text{degreenumber}_{X[\widehat{\mathcal{O}}]^\text{degreenumber}/X,\mathrm{topo},Z} \ar[r]^\partial\ar[r]\ar[r] \ar[r] &\mathrm{deRham}^\text{degreenumber}_{X[\widehat{\mathcal{O}}]^\text{degreenumber}/X,\mathrm{topo},Z}{\otimes}f^{-1} \mathrm{deRham}^1_{X,\mathrm{topo}} \ar[r]^\partial\ar[r]\ar[r] \ar[r] &...),
}
\]
and
\[\tiny
\xymatrix@C+0.4pc@R+0pc{
M{\otimes}^\mathbb{L}(0\ar[r]\ar[r]\ar[r] \ar[r] &\mathrm{deRham}^\text{degreenumber}_{k[\widehat{\mathcal{O}}]^\text{degreenumber}/k,\mathrm{topo},Z} \ar[r]^\partial\ar[r]\ar[r] \ar[r] &\mathrm{deRham}^\text{degreenumber}_{X[\widehat{\mathcal{O}}]^\text{degreenumber}/X_{\text{\'et}},\mathrm{topo},Z} \ar[r]^\partial\ar[r]\ar[r] \ar[r] &\mathrm{deRham}^\text{degreenumber}_{X[\widehat{\mathcal{O}}]^\text{degreenumber}/X,\mathrm{topo},Z}{\otimes}f^{-1} \mathrm{deRham}^1_{X_{\text{\'et}},\mathrm{topo}} \ar[r]^\partial\ar[r]\ar[r] \ar[r] &...).
}
\]
\end{proposition}

\begin{proof}
This is actually a direct consequence of the previous proposition due to the corresponding flatness.	
\end{proof}

\end{landscape}

\newpage

\subsection{Analytic Andr\'e-Quillen Homology and Analytic Derived de Rham complex of Pseudorigid Spaces}



\indent We now consider the analytic Andr\'e-Quillen Homology and analytic Derived de Rham complex of pseudorigid space, which are very crucial in some development in \cite{12Bel1} in the arithmetic family. Therefore we just investigate the corresponding picture in the corresponding geometric family.

\begin{setting}
We consider now a corresponding morphism taking the corresponding form of $A\rightarrow B$ where $A$ is going to be a pseudorigid affinoid algebra over $\mathbb{Z}_p$, and $B$ is going to be a perfectoid chart of $A$ in the corresponding pro-\'etale site of the pseudorigid affinoid space attached to $A$. As in \cite[Definition 3.1, and below Definition 3.1]{12Bel1} in our situation $A$ is of topologically finite type over $\mathcal{O}_K[[t]]\left<\pi^a/t^b\right>[1/t]$, where $K$ is a discrete valued field containing $\mathbb{Q}_p$ and $(a,b)=1$ \footnote{Certainly one can also consider the characteristic $p$ situation.}.	
\end{setting}

\indent From \cite[Definition 3.1, and below Definition 3.1]{12Bel1} we have the following:

\begin{lemma} We have the following statements:\\
A. $A$ as above is Tate, complete over $\mathcal{O}_K$;\\
B. The ring $A$ has a ring of definition $A_0$ which is of $\mathcal{O}_K$-formally finite type;\\
C. The ring $A_0$ is of topologically finite type over $\mathcal{O}_K[[t]]\left<\pi^a/t^b\right>$.
\end{lemma}

\begin{proof}
It is very easy to see that the corresponding results hold in our current situation over $\mathcal{O}_K$.	
\end{proof}

As in the corresponding construction in the rigid situation in the previous section following \cite[Chapitre 3]{12An1}, \cite{12An2}, \cite[Chapter 2, Chapter 8]{12B1}, \cite[Chapter 1]{12Bei}, \cite[Chapter 5]{12G1}, \cite[Chapter 3, Chapter 4]{12GL}, \cite[Chapitre II, Chapitre III]{12Ill1}, \cite[Chapitre VIII]{12Ill2}, \cite[Section 4]{12Qui} we give the following definitions. First we consider the following new setting:

\begin{setting}
Now as in the above notion we will consider a general map of rings $A\rightarrow B$ over $\mathcal{O}_K[[t]]\left<\pi^a/t^b\right>[1/t]$ where we have the corresponding map of the associated ring of definitions $A_0\rightarrow B_0$ over $\mathcal{O}_K[[t]]\left<\pi^a/t^b\right>$ such that $A_0,B_0$ are basically $I$-adic (where $\mathcal{O}_K[[t]]\left<\pi^a/t^b\right>$ is $I$-adic).	
\end{setting}


\begin{definition}
We start from the corresponding construction of the algebraic $p$-adic derived de Rham complex for a map $A\rightarrow B$. Fix a pair of ring of definitions $A_0,B_0$ in $A,B$ respectively. Then this is the corresponding derived differential complex attached to the polynomial resolution of $B_0$:
\begin{align}
A_0[A_0[B_0]]...,	
\end{align}
which is now denoted by $\mathrm{Kan}_\mathrm{Left}\mathrm{deRham}^\text{degreenumber}_{B_0/A_0,\mathrm{alg}}$ after taking the corresponding left Kan extension. The corresponding cotangent complex associated is defined to be just:
\begin{align}
\mathbb{L}_{B_0/A_0,\mathrm{alg}}:=	\mathrm{deRham}^1_{A_0[B_0]^\text{degreenumber}/A_0,\mathrm{alg}}\otimes_{A_0[B_0]^\text{degreenumber}} B_0.
\end{align}
The corresponding algebraic Andr\'e-Quillen homologies are defined to be:
\begin{align}
H_{\text{degreenumber},{\mathrm{AQ}},\mathrm{alg}}:=\pi_\text{degreenumber} (\mathbb{L}_{B_0/A_0,\mathrm{alg}}). 	
\end{align}
The corresponding topological Andr\'e-Quillen complex is actually the complete version of the corresponding algebraic ones above by considering the corresponding derived $I$-completion over the simplicial module structure:
\begin{align}
\mathbb{L}_{B_0/A_0,\mathrm{topo}}:=R\varprojlim_I \mathrm{Kos}_{I}	\left((\mathrm{deRham}^1_{A_0[B_0]^\text{degreenumber}/A_0,\mathrm{alg}}\otimes_{A_0[B_0]^\text{degreenumber}} B_0)\right).
\end{align}
Then we consider the corresponding derived algebraic de Rham complex which is just defined to be:
\begin{align}
\mathrm{Kan}_\mathrm{Left}\mathrm{deRham}^\text{degreenumber}_{B_0/A_0,\mathrm{alg}},\mathrm{Kan}_\mathrm{Left}\mathrm{Fil}^*_{\mathrm{Kan}_\mathrm{Left}\mathrm{deRham}^1_{B_0/A_0,\mathrm{alg}}}.	
\end{align}
We then take the corresponding derived $I$-completion and we denote that by:
\begin{align}
\mathrm{Kan}_\mathrm{Left}\mathrm{deRham}^\text{degreenumber}_{B_0/A_0,\mathrm{topo}}:=R\varprojlim_I\mathrm{Kos}_I\left(\mathrm{Kan}_\mathrm{Left}\mathrm{deRham}^\text{degreenumber}_{B_0/A_0,\mathrm{alg}},\mathrm{Kan}_\mathrm{Left}\mathrm{Fil}^*_{\mathrm{Kan}_\mathrm{Left}\mathrm{deRham}^\text{degreenumber}_{B_0/A_0,\mathrm{alg}}}\right),\\
\mathrm{Fil}^*_{\mathrm{deRham}^\text{degreenumber}_{B_0/A_0,\mathrm{topo}}}:=R\varprojlim_I\mathrm{Kos}_I\left(\mathrm{Kan}_\mathrm{Left}\mathrm{Fil}^*_{\mathrm{Kan}_\mathrm{Left}\mathrm{deRham}^\text{degreenumber}_{B_0/A_0,\mathrm{alg}}}\right).	
\end{align}
Then we consider the following construction for the map $A\rightarrow B$ by putting:
\begin{align}
\mathbb{L}_{B_0/A_0,\mathrm{topo}}:= \mathrm{Colim}_{A_0\rightarrow B_0}\mathbb{L}_{B_0/A_0,\mathrm{topo}}[1/t],\\
H_{\text{degreenumber},{\mathrm{AQ}},\mathrm{topo}}:=\pi_\text{degreenumber} (\mathbb{L}_{B/A,\mathrm{topo}}),	\\
\mathrm{Kan}_\mathrm{Left}\mathrm{deRham}^\text{degreenumber}_{B/A,\mathrm{topo}}:=\mathrm{Colim}_{A_0\rightarrow B_0}\mathrm{Kan}_\mathrm{Left}\mathrm{deRham}^\text{degreenumber}_{B_0/A_0,\mathrm{topo}}[1/t],\\
\mathrm{Kan}_\mathrm{Left}\mathrm{Fil}^*_{\mathrm{deRham}^\text{degreenumber}_{B/A,\mathrm{topo}}}:=\mathrm{Colim}_{A_0\rightarrow B_0}\mathrm{Fil}^*_{\mathrm{Kan}_\mathrm{Left}\mathrm{deRham}^\text{degreenumber}_{B_0/A_0,\mathrm{topo}}}[1/t].
\end{align}
Then we need to take the corresponding Hodge-Filtered completion by using the corresponding filtration associated as above to achieve the corresponding Hodge-complete objects in the corresponding filtered $\infty$-categories:
\begin{align}
\mathrm{Kan}_\mathrm{Left}{\mathrm{deRham}}^\text{degreenumber}_{B/A,\mathrm{topo,Hodge}},\mathrm{Fil}^*_{\mathrm{Kan}_\mathrm{Left}{\mathrm{deRham}}^\text{degreenumber}_{B/A,\mathrm{topo,Hodge}}}.	
\end{align}
\end{definition}

\begin{definition}
We define the corresponding finite projective filtered crystals to be the corresponding finite projective module spectra over the topological filtered $E_\infty$-ring $\mathrm{Kan}_\mathrm{Left}\mathrm{deRham}^\text{degreenumber}_{B/A,\mathrm{topo}}$ with the corresponding induced filtrations.	
\end{definition}

\begin{definition}
We define the corresponding almost perfect \footnote{This is the corresponding derived version of pseudocoherence from \cite{12Lu1}, \cite{12Lu2}.} filtered crystals to be the corresponding almost perfect module spectra over the topological filtered $E_\infty$-ring $\mathrm{Kan}_\mathrm{Left}\mathrm{deRham}^\text{degreenumber}_{B /A,\mathrm{topo}}$ with the corresponding induced filtrations.	
\end{definition}

\indent We now consider the corresponding large coefficient local systems over pseudorigid spaces where $Z$ now is a topological algebra over $\mathbb{Z}_p$\footnote{It is better to assume that it is basically completely flat.}.


\begin{definition}
We define the following $Z$ deformed version of the corresponding complete version of the corresponding Andr\'e-Quillen homology and the corresponding complete version of the corresponding derived de Rham complex. We start from the corresponding construction of the algebraic $p$-adic derived de Rham complex for a map $A\rightarrow B$. Fix a pair of ring of definitions $A_0,B_0$ in $A,B$ respectively. Then this is the corresponding derived differential complex attached to the polynomial resolution of $B_0$:
\begin{align}
A_0[A_0[B_0]]...,	
\end{align}
which is now denoted by $\mathrm{deRham}^\text{degreenumber}_{A_0[B_0]^\text{degreenumber}/A_0,\mathrm{alg}}$. The corresponding cotangent complex associated is defined to be just:
\begin{align}
\mathbb{L}_{B_0/A_0,\mathrm{alg}}:=	\mathrm{deRham}^1_{A_0[B_0]^\text{degreenumber}/A_0,\mathrm{alg}}\otimes_{A_0[B_0]^\text{degreenumber}} B_0.
\end{align}
The corresponding algebraic Andr\'e-Quillen homologies are defined to be:
\begin{align}
H_{\text{degreenumber},{\mathrm{AQ}},\mathrm{alg}}:=\pi_\text{degreenumber} (\mathbb{L}_{B_0/A_0,\mathrm{alg}}). 	
\end{align}
The corresponding topological Andr\'e-Quillen complex is actually the complete version of the corresponding algebraic ones above by considering the corresponding derived $I$-completion over the simplicial module structure:
\begin{align}
\mathbb{L}_{B_0/A_0,\mathrm{topo}}:=R\varprojlim_k	\mathrm{Kos}_I\left((\mathrm{deRham}^1_{A_0[B_0]^\text{degreenumber}/A_0,\mathrm{alg}}\otimes_{A_0[B_0]^\text{degreenumber}} B_0)\right).
\end{align}
Taking the product with $Z$ we have the corresponding integral version of the topological Andr\'e-Quillen complex\footnote{Here we did not take the corresponding derived completion, but in some situation this is achievable.}:
\begin{align}
\mathbb{L}_{B_0/A_0,\mathrm{topo},Z}:=\mathbb{L}_{B_0/A_0,\mathrm{topo}}{\otimes}^\mathbb{L}_{\mathbb{Z}_p}Z
\end{align}
Then we consider the corresponding derived algebraic de Rham complex which is just defined to be:
\begin{align}
\mathrm{Kan}_\mathrm{Left}\mathrm{deRham}^\text{degreenumber}_{B_0/A_0,\mathrm{alg}},\mathrm{Fil}^*_{\mathrm{Kan}_\mathrm{Left}\mathrm{deRham}^\text{degreenumber}_{B_0/A_0,\mathrm{alg}}}.	
\end{align}
We then take the corresponding derived $I$-completion and we denote that by:
\begin{align}
\mathrm{Kan}_\mathrm{Left}\mathrm{deRham}^\text{degreenumber}_{B_0/A_0,\mathrm{topo}}:=R\varprojlim_k\mathrm{Kos}_I\left(\mathrm{Kan}_\mathrm{Left}\mathrm{deRham}^\text{degreenumber}_{B_0/A_0,\mathrm{alg}},\mathrm{Fil}^*_{\mathrm{Kan}_\mathrm{Left}\mathrm{deRham}^\text{degreenumber}_{B_0/A_0,\mathrm{alg}}}\right),\\
\mathrm{Kan}_\mathrm{Left}\mathrm{Fil}^*_{\mathrm{deRham}^\text{degreenumber}_{B_0/A_0,\mathrm{topo}}}:=R\varprojlim_k\mathrm{Kos}_I\left(\mathrm{Kan}_\mathrm{Left}\mathrm{Fil}^*_{\mathrm{Kan}_\mathrm{Left}\mathrm{deRham}^\text{degreenumber}_{B_0/A_0,\mathrm{alg}}}\right).	
\end{align}
Before considering the corresponding integral version we just consider the corresponding product of these $\mathbb{E}_\infty$-rings with $Z$ to get\footnote{Here again we did not take the corresponding derived completion, but in some situation this is achievable, for instance when the ring $Z$ is endowed with derived $J$-complete topology.}:
\begin{align}
\mathrm{Kan}_\mathrm{Left}\mathrm{deRham}^\text{degreenumber}_{B_0/A_0,\mathrm{topo},Z}:=\mathrm{Kan}_\mathrm{Left}\mathrm{deRham}^\text{degreenumber}_{B_0/A_0,\mathrm{topo}}{\otimes}^\mathbb{L}_{\mathbb{Z}_p}Z,\\
\mathrm{Kan}_\mathrm{Left}\mathrm{Fil}^*_{\mathrm{deRham}^\text{degreenumber}_{B_0/A_0,\mathrm{topo}},Z}:=\mathrm{Fil}^*_{\mathrm{Kan}_\mathrm{Left}\mathrm{deRham}^\text{degreenumber}_{B_0/A_0,\mathrm{topo}}}{\otimes}^\mathbb{L}_{\mathbb{Z}_p}Z.	
\end{align}
Then we consider the following construction for the map $A\rightarrow B$ by putting:
\begin{align}
\mathbb{L}_{B_0/A_0,\mathrm{topo},Z}:= \mathrm{Colim}_{A_0\rightarrow B_0}\mathbb{L}_{B_0/A_0,\mathrm{topo},Z}[1/t],\\
H_{\text{degreenumber},{\mathrm{AQ}},\mathrm{topo},Z}:=\pi_\text{degreenumber} (\mathbb{L}_{B/A,\mathrm{topo},Z}),	\\
\mathrm{Kan}_\mathrm{Left}\mathrm{deRham}^\text{degreenumber}_{B/A,\mathrm{topo},Z}:=\mathrm{Colim}_{A_0\rightarrow B_0}\mathrm{Kan}_\mathrm{Left}\mathrm{deRham}^\text{degreenumber}_{B_0/A_0,\mathrm{topo},Z}[1/t],\\
\mathrm{Kan}_\mathrm{Left}\mathrm{Fil}^*_{\mathrm{Kan}_\mathrm{Left}\mathrm{deRham}^\text{degreenumber}_{B/A,\mathrm{topo}},Z}:=\mathrm{Colim}_{A_0\rightarrow B_0}\mathrm{Kan}_\mathrm{Left}\mathrm{Fil}^*_{\mathrm{Kan}_\mathrm{Left}\mathrm{deRham}^1_{B_0/A_0,\mathrm{topo}},Z}[1/t].
\end{align}
Then we need to take the corresponding Hodge-Filtered completion by using the corresponding filtration associated as above to achieve the corresponding Hodge-complete objects in the corresponding filtered $\infty$-categories:
\begin{align}
\mathrm{Kan}_\mathrm{Left}{\mathrm{deRham}}^\text{degreenumber}_{B/A,\mathrm{topo,Hodge},Z},\mathrm{Kan}_\mathrm{Left}\mathrm{Fil}^*_{\mathrm{Kan}_\mathrm{Left}{\mathrm{deRham}}^1_{B/A,\mathrm{topo,Hodge}},Z}.	
\end{align} 	
\end{definition}

\begin{example}
Now we construct the corresponding pseudorigid analog of the corresponding example in \cite[Example 4.7]{12GL}. Now we consider the corresponding the map:
\begin{align}
\mathbb{Z}_p[[u]]\left<\frac{p^a}{u^b}\right>[1/u]\left<X^\pm_1,X^\pm_2,...,X^\pm_d\right>\longrightarrow \mathbb{Z}_p[[u]]\left<\frac{p^a}{u^b}\right>[1/u]\left<X^{\pm/p^\infty}_1,X^{\pm/p^\infty}_2,...,X^{\pm/p^\infty}_d\right>	
\end{align}
which could be written as:	
\begin{align}
&\mathbb{Z}_p[[u]]\left<\frac{p^a}{u^b}\right>[1/u]\left<X^\pm_1,X^\pm_2,...,X^\pm_d\right>\longrightarrow \\
&\mathbb{Z}_p[[u]]\left<\frac{p^a}{u^b}\right>[1/u]\left<X^\pm_1,X^\pm_2,...,X^\pm_d\right>\left<Y^{\pm/p^\infty}_1,Y^{\pm/p^\infty}_2,...,Y^{\pm/p^\infty}_d\right>/(X_i-Y_i,i=1,...,d).	
\end{align}
So in our situation the corresponding Hodge complete topological derived de Rham complex will be just:
\begin{align}
\mathbb{Z}_p[[u]]\left<\frac{p^a}{u^b}\right>[1/u]\left<Y^{\pm/p^\infty}_1,Y^{\pm/p^\infty}_2,...,Y^{\pm/p^\infty}_d\right>[[Z_1,...,Z_d,Z_i=X_i-Y_i,i=1,...,d]].	
\end{align}

\end{example}


\indent We now follow the corresponding previous works \cite[Chapter 2, Chapter 8]{12B1}, \cite[Chapter 1]{12Bei}, \cite[Chapter 5]{12G1}, \cite[Chapter 3, Chapter 4]{12GL}, \cite{12Ill1}, \cite[Chapitre VIII]{12Ill2} to study the corresponding functoriality for some triple $A\rightarrow B \rightarrow C$ where we have first in the algebraic setting the following result. Here we first consider more general setting where we will consider those rings over $\mathcal{O}_K[[t]]\left<\frac{\pi^a}{t^b}\right>$ which is $t$-adically complete.


\begin{proposition}
The corresponding functoriality for the corresponding algebraic derived de Rham complex holds in our situation for the corresponding integral adic rings $A,B,C$ in the context of this section, namely we have the following corresponding commutative diagram:
\[
\xymatrix@C+0pc@R+3pc{
\mathrm{Kan}_\mathrm{Left}\mathrm{deRham}^\text{degreenumber}_{B/A,\mathrm{alg}} \ar[r] \ar[r] \ar[r]\ar[d] \ar[d] \ar[d] &\mathrm{Kan}_\mathrm{Left}\mathrm{deRham}^\text{degreenumber}_{B/A,\mathrm{alg}} \ar[d] \ar[d] \ar[d]\\
B \ar[r] \ar[r] \ar[r] &\mathrm{Kan}_\mathrm{Left}\mathrm{deRham}^\text{degreenumber}_{B/A,\mathrm{alg}}.
}
\]	
\end{proposition}

\begin{proof}
See \cite[Lemma 3.3]{12GL}.	
\end{proof}

\begin{proposition}
The corresponding functoriality for the corresponding algebraic derived de Rham complex holds in our situation for the corresponding integral adic rings $A,B,C$ in the context of this section, namely we have the following corresponding commutative diagram:
\[
\xymatrix@C+0pc@R+3pc{
\mathrm{Kan}_\mathrm{Left}\mathrm{deRham}^\text{degreenumber}_{B/A,\mathrm{topo,Hodge}} \ar[r] \ar[r] \ar[r]\ar[d] \ar[d] \ar[d] &\mathrm{Kan}_\mathrm{Left}\mathrm{deRham}^\text{degreenumber}_{B/A,\mathrm{topo,Hodge}} \ar[d] \ar[d] \ar[d]\\
B \ar[r] \ar[r] \ar[r] &\mathrm{Kan}_\mathrm{Left}\mathrm{deRham}^\text{degreenumber}_{B/A,\mathrm{topo,Hodge}}.
}
\]	
\end{proposition}

\begin{proof}
Just take the corresponding derived $I$-completion of the diagram in the previous proposition.	
\end{proof}


\indent The following conjectures are literally inspired by the rigid analytic situation in \cite[Theorem 1.2]{12GL}. We use the notation $*$ to denote the ring $\mathcal{O}_K[[t]]\left<\pi^a/t^b\right>[1/t]$. We assume the spaces in the following two conjectures are smooth pseudo-rigid analytic spaces over $\mathcal{O}_K[[t]]\left<\pi^a/t^b\right>[1/t]$.

\newpage

\begin{landscape}
\begin{conjecture}
The construction in this section could be carried over the pro-\'etale sites \footnote{For more on the foundations here for the pro-\'etale sites see \cite[Theorem 2.9.9, Remark 2.9.10]{12Ked1}}. Consider the corresponding projective map $g:X_{\text{pro-\'etale}}\rightarrow X_{\text{\'et}}$ and the the projective map $f:X_{\text{pro-\'etale}}\rightarrow X$. Then we have the following two strictly exact long exact sequences 
\footnote{We should have the corresponding naturality taking into the following form:
\[\tiny
\xymatrix@C+0.4pc@R+0pc{
0\ar[r]\ar[r]\ar[r] \ar[r] &\mathrm{deRham}^\text{degreenumber}_{*[\widehat{\mathcal{O}}]^\text{degreenumber}/*,\mathrm{topo},Z}\ar[d]\ar[d]\ar[d] \ar[d] \ar[r]^\partial\ar[r]\ar[r] \ar[r] &\mathrm{deRham}^\text{degreenumber}_{X[\widehat{\mathcal{O}}]^\text{degreenumber}/X,\mathrm{topo},Z} \ar[d]\ar[d]\ar[d] \ar[d]\ar[r]^\partial\ar[r]\ar[r] \ar[r] &\mathrm{deRham}^\text{degreenumber}_{X[\widehat{\mathcal{O}}]^\text{degreenumber}/X,\mathrm{topo},Z}{\otimes}f^{-1} \mathrm{deRham}^1_{X,\mathrm{topo}} \ar[d]\ar[d]\ar[d] \ar[d]\ar[r]^\partial\ar[r]\ar[r] \ar[r]&...\\
0\ar[r]\ar[r]\ar[r] \ar[r] &\mathrm{deRham}^\text{degreenumber}_{*[\widehat{\mathcal{O}}]^\text{degreenumber}/*,\mathrm{topo},Z} \ar[r]^\partial\ar[r]\ar[r] \ar[r] &\mathrm{deRham}^\text{degreenumber}_{X[\widehat{\mathcal{O}}]^\text{degreenumber}/X_{\text{\'et}},\mathrm{topo},Z} \ar[r]^\partial\ar[r]\ar[r] \ar[r] &\mathrm{deRham}^\text{degreenumber}_{X[\widehat{\mathcal{O}}]^\text{degreenumber}/X,\mathrm{topo},Z}{\otimes}f^{-1} \mathrm{deRham}^1_{X_{\text{\'et}},\mathrm{topo}} \ar[r]^\partial\ar[r]\ar[r] \ar[r] &....
}
\]
}:
\[\tiny
\xymatrix@C+0.4pc@R+0pc{
0\ar[r]\ar[r]\ar[r] \ar[r] &\mathrm{deRham}^\text{degreenumber}_{*[\widehat{\mathcal{O}}]^\text{degreenumber}/*,\mathrm{topo},Z} \ar[r]^\partial\ar[r]\ar[r] \ar[r] &\mathrm{deRham}^\text{degreenumber}_{X[\widehat{\mathcal{O}}]^\text{degreenumber}/X,\mathrm{topo},Z} \ar[r]^\partial\ar[r]\ar[r] \ar[r] &\mathrm{deRham}^\text{degreenumber}_{X[\widehat{\mathcal{O}}]^\text{degreenumber}/X,\mathrm{topo},Z}{\otimes}f^{-1} \mathrm{deRham}^1_{X,\mathrm{topo}} \ar[r]^\partial\ar[r]\ar[r] \ar[r]&...,
}
\]
and
\[\tiny
\xymatrix@C+0.4pc@R+0pc{
0\ar[r]\ar[r]\ar[r] \ar[r] &\mathrm{deRham}^\text{degreenumber}_{*[\widehat{\mathcal{O}}]^\text{degreenumber}/*,\mathrm{topo},Z} \ar[r]^\partial\ar[r]\ar[r] \ar[r] &\mathrm{deRham}^\text{degreenumber}_{X[\widehat{\mathcal{O}}]^\text{degreenumber}/X_{\text{\'et}},\mathrm{topo},Z} \ar[r]^\partial\ar[r]\ar[r] \ar[r] &\mathrm{deRham}^\text{degreenumber}_{X[\widehat{\mathcal{O}}]^\text{degreenumber}/X,\mathrm{topo},Z}{\otimes}f^{-1} \mathrm{deRham}^1_{X_{\text{\'et}},\mathrm{topo}} \ar[r]^\partial\ar[r]\ar[r] \ar[r]&....
}
\]
\end{conjecture}

\begin{conjecture}
Consider the corresponding projective map $g:X_{\text{pro-\'etale}}\rightarrow X_{\text{\'et}}$ and the the projective map $f:X_{\text{pro-\'etale}}\rightarrow X$. Let $M$ be a corresponding $Z$-projective differential crystal spectrum. Then we have the following two strictly exact long exact sequences:
\[\tiny
\xymatrix@C+0.4pc@R+0pc{
M{\otimes}^\mathbb{L}(0\ar[r]\ar[r]\ar[r] \ar[r] &\mathrm{deRham}^\text{degreenumber}_{\mathcal{O}_k[\widehat{\mathcal{O}}]^\text{degreenumber}/\mathcal{O}_k,\mathrm{topo},Z} \ar[r]^\partial\ar[r]\ar[r] \ar[r] &\mathrm{deRham}^\text{degreenumber}_{X[\widehat{\mathcal{O}}]^\text{degreenumber}/X,\mathrm{topo},Z} \ar[r]^\partial\ar[r]\ar[r] \ar[r] &\mathrm{deRham}^\text{degreenumber}_{X[\widehat{\mathcal{O}}]^\text{degreenumber}/X,\mathrm{topo},Z}{\otimes}f^{-1} \mathrm{deRham}^1_{X,\mathrm{topo}} \ar[r]^\partial\ar[r]\ar[r] \ar[r] &...),
}
\]
and
\[\tiny
\xymatrix@C+0.4pc@R+0pc{
M{\otimes}^\mathbb{L}(0\ar[r]\ar[r]\ar[r] \ar[r] &\mathrm{deRham}^\text{degreenumber}_{\mathcal{O}_k[\widehat{\mathcal{O}}]^\text{degreenumber}/\mathcal{O}_k,\mathrm{topo},Z} \ar[r]^\partial\ar[r]\ar[r] \ar[r] &\mathrm{deRham}^\text{degreenumber}_{X[\widehat{\mathcal{O}}]^\text{degreenumber}/X_{\text{\'et}},\mathrm{topo},Z} \ar[r]^\partial\ar[r]\ar[r] \ar[r] &\mathrm{deRham}^\text{degreenumber}_{X[\widehat{\mathcal{O}}]^\text{degreenumber}/X,\mathrm{topo},Z}{\otimes}f^{-1} \mathrm{deRham}^1_{X_{\text{\'et}},\mathrm{topo}} \ar[r]^\partial\ar[r]\ar[r] \ar[r]&...).
}
\]
\end{conjecture}

\end{landscape}


\newpage

\subsection{Derived $(p,I)$-Complete Derived de Rham complex of Derived Adic Rings}


\indent The corresponding construction of the derived de Rham complex could be defined for general derived spaces. Along our discussion in the situations of rigid analytic spaces and the corresponding pseudorigid spaces we now focus on the corresponding derived adic rings. 

\begin{setting}
We now fix a bounded morphism of simplicial adic rings $A\rightarrow B$ over $A^*$ where $A^*$ contains a corresponding ring of definition $A^*_0$ which is derived complete with respect to the $(p,I)$-topology and we assume that $A$ is adic and we assume that $(A^*_0,I)$ is a prism in \cite{12BS} namely we at least require that the corresponding $\delta$-structure on the corresponding ring will induce the map $\varphi(.):=.^p+p\delta(.)$ such that we have the situation where $p\in (I,\varphi(I))$. For $A$ or $B$ respectively we assume this contains a subring $A_0$ or $B_0$ (over $A_0^*$) respectively such that we have $A_0$ or $B_0$ respectively is derived complete with respect to the corresponding derived $(p,I)$-topology and we assume that $B=B_0[1/f,f\in I]$ (same for $A$). All the adic rings are assumed to be open mapping. We use the notation $d$ to denote a corresponding primitive element as in \cite[Section 2.3]{12BS} for $A^*$. We are going to assume that $p$ is a topologically nilpotent element.
\end{setting}

\indent As in the corresponding construction in the rigid and pseudocoherent situation in the previous section following \cite[Chapitre 3]{12An1}, \cite{12An2}, \cite[Chapter 2, Chapter 8]{12B1}, \cite[Chapter 1]{12Bei}, \cite[Chapter 5]{12G1}, \cite[Chapter 3, Chapter 4]{12GL}, \cite[Chapitre II, Chapitre III]{12Ill1}, \cite[Chapitre VIII]{12Ill2}, \cite[Section 4]{12Qui}, we have the following: 

\begin{definition}
We start from the corresponding construction of the algebraic $p$-adic derived de Rham complex for a map $A\rightarrow B$. Fix a pair of ring of definitions $A_0,B_0$ in $A,B$ respectively. Then this is the corresponding left Kan extended derived differential complex attached to $B_0/A_0$ which is now denoted by $\mathrm{Kan}_\mathrm{Left}\mathrm{deRham}^\text{degreenumber}_{{B_0}/A_0,\mathrm{alg}}$\footnote{Namely over $A_0$ we take the suitable corresponding left Kan extension, then apply to $B_0$.} from the corresponding $(p,I)$-complete commutative rings. The corresponding cotangent complex associated is defined to be just:
\begin{align}
\mathbb{L}_{B_0/A_0,\mathrm{alg}}:=	\mathrm{deRham}^1_{P^\text{degreenumber}_{B_0}/A_0,\mathrm{alg}}\otimes_{P^\text{degreenumber}_{B_0}} B_0.
\end{align}
The corresponding algebraic Andr\'e-Quillen homologies are defined to be:
\begin{align}
H_{\text{degreenumber},{\mathrm{AQ}},\mathrm{alg}}:=\pi_\text{degreenumber} (\mathbb{L}_{B_0/A_0,\mathrm{alg}}). 	
\end{align}
The corresponding topological Andr\'e-Quillen complex is actually the complete version of the corresponding algebraic ones above by considering the corresponding derived $(p,I)$-completion over the simplicial module structure:
\begin{align}
\mathbb{L}_{B_0/A_0,\mathrm{topo}}:=R\varprojlim \mathrm{Kos}_{(p,I)}	\left((\mathrm{deRham}^1_{P^\text{degreenumber}_{B_0}/A_0,\mathrm{alg}}\otimes_{P^\text{degreenumber}_{B_0}} B_0)\right).
\end{align}
Then we consider the corresponding derived algebraic de Rham complex which is just defined to be:
\begin{align}
\mathrm{Kan}_\mathrm{Left}\mathrm{deRham}^\text{degreenumber}_{-/A_0,\mathrm{alg}}(B_0),\mathrm{Fil}^*_{\mathrm{Kan}_\mathrm{Left}\mathrm{deRham}^\text{degreenumber}_{B_0/A_0,\mathrm{alg}}}.	
\end{align}
We then take the corresponding derived $(p,I)$-completion and we denote that by:
\begin{align}
\mathrm{Kan}_\mathrm{Left}\mathrm{deRham}^\text{degreenumber}_{B_0/A_0,\mathrm{topo}}:=R\varprojlim \mathrm{Kos}_{(p,I)}\left(\mathrm{Kan}_\mathrm{Left}\mathrm{deRham}^\text{degreenumber}_{B_0/A_0,\mathrm{alg}},\mathrm{Fil}^*_{\mathrm{Kan}_\mathrm{Left}\mathrm{deRham}^\text{degreenumber}_{B_0/A_0,\mathrm{alg}}}\right),\\
\mathrm{Kan}_\mathrm{Left}\mathrm{Fil}^*_{\mathrm{Kan}_\mathrm{Left}\mathrm{deRham}^\text{degreenumber}_{B_0/A_0,\mathrm{topo}}}:=R\varprojlim \mathrm{Kos}_{(p,I)}\left(\mathrm{Fil}^*_{\mathrm{Kan}_\mathrm{Left}\mathrm{deRham}^\text{degreenumber}_{B_0/A_0,\mathrm{alg}}}\right).	
\end{align}
Then in the situation that all the rings are classical adic rings we consider the following construction for the map $A\rightarrow B$ by putting:
\begin{align}
\mathbb{L}_{B/A,\mathrm{topo}}:= \mathrm{Colim}_{A_0\rightarrow B_0}\mathbb{L}_{B_0/A_0,\mathrm{topo}}[1/(d)],\\
H_{\text{degreenumber},{\mathrm{AQ}},\mathrm{topo}}:=\pi_\text{degreenumber} (\mathbb{L}_{B/A,\mathrm{topo}}),	\\
\mathrm{Kan}_\mathrm{Left}\mathrm{deRham}^\text{degreenumber}_{B/A,\mathrm{topo}}:=\mathrm{Colim}_{A_0\rightarrow B_0}\mathrm{Kan}_\mathrm{Left}\mathrm{deRham}^\text{degreenumber}_{B_0/A_0,\mathrm{topo}}[1/(d)],\\
\mathrm{Kan}_\mathrm{Left}\mathrm{Fil}^*_{\mathrm{Kan}_\mathrm{Left}\mathrm{deRham}^\text{degreenumber}_{B/A,\mathrm{topo}}}:=\mathrm{Colim}_{A_0\rightarrow B_0}\mathrm{Fil}^*_{\mathrm{Kan}_\mathrm{Left}\mathrm{deRham}^\text{degreenumber}_{B_0/A_0,\mathrm{topo}}}[1/(d)].
\end{align}
Then we need to take the corresponding Hodge-Filtered completion by using the corresponding filtration associated as above to achieve the corresponding Hodge-complete objects in the corresponding filtered $\infty$-categories:
\begin{align}
\mathrm{Kan}_\mathrm{Left}{\mathrm{deRham}}^\text{degreenumber}_{B/A,\mathrm{topo,Hodge}},\mathrm{Fil}^*_{\mathrm{Kan}_\mathrm{Left}{\mathrm{deRham}}^\text{degreenumber}_{B/A,\mathrm{topo,Hodge}}}.	
\end{align}
\end{definition}

\begin{definition}
We define the corresponding finite projective filtered crystals to be the corresponding finite projective module spectra over the topological filtered $E_\infty$-ring $\mathrm{Kan}_\mathrm{Left}\mathrm{deRham}^\text{degreenumber}_{B/A,\mathrm{topo}}$ with the corresponding induced filtrations.	
\end{definition}

\begin{definition}
We define the corresponding almost perfect \footnote{This is the corresponding derived version of pseudocoherence from \cite{12Lu1}, \cite{12Lu2}.} filtered crystals to be the corresponding almost perfect module spectra over the topological filtered $E_\infty$-ring $\mathrm{Kan}_\mathrm{Left}\mathrm{deRham}^\text{degreenumber}_{B/A,\mathrm{topo}}$ with the corresponding induced filtrations.	
\end{definition}

\indent We now consider the corresponding large coefficient local systems over pseudorigid spaces where $Z$ now is a simplicial topological algebra over $\mathbb{Z}_p$\footnote{It is better to assume that it is basically derived completely flat.}.


\begin{definition}
We define the following $Z$ deformed version of the corresponding complete version of the corresponding Andr\'e-Quillen homology and the corresponding complete version of the corresponding derived de Rham complex. We start from the corresponding construction of the algebraic $p$-adic derived de Rham complex for a map $A\rightarrow B$. Fix a pair of ring of definitions $A_0,B_0$ in $A,B$ respectively. Then this is the corresponding left Kan extended derived differential complex attached to $B_0$ from the corresponding derived $(p,I)$-complete commutative rings which is now denoted by $\mathrm{Kan}_\mathrm{Left}\mathrm{deRham}^\text{degreenumber}_{B_0/A_0,\mathrm{alg}}$. The corresponding cotangent complex associated is defined to be just:
\begin{align}
\mathbb{L}_{B_0/A_0,\mathrm{alg}}:=	\mathrm{deRham}^1_{P^\text{degreenumber}_{B_0}/A_0,\mathrm{alg}}\otimes_{P^\text{degreenumber}_{B_0}} B_0.
\end{align}
The corresponding algebraic Andr\'e-Quillen homologies are defined to be:
\begin{align}
H_{\text{degreenumber},{\mathrm{AQ}},\mathrm{alg}}:=\pi_\text{degreenumber} (\mathbb{L}_{B_0/A_0,\mathrm{alg}}). 	
\end{align}
The corresponding topological Andr\'e-Quillen complex is actually the complete version of the corresponding algebraic ones above by considering the corresponding derived $(p,I)$-completion over the simplicial module structure:
\begin{align}
\mathbb{L}_{B_0/A_0,\mathrm{topo}}:=R\varprojlim \mathrm{Kos}_{(p,I)}	\left((\mathrm{deRham}^1_{P^\text{degreenumber}_{B_0}/A_0,\mathrm{alg}}\otimes_{P^\text{degreenumber}_{B_0}} B_0)\right).
\end{align}
Taking the product with $Z$ we have the corresponding integral version of the topological Andr\'e-Quillen complex:
\begin{align}
\mathbb{L}_{B_0/A_0,\mathrm{topo},Z}:=\mathbb{L}_{B_0/A_0,\mathrm{topo}}{\otimes}^\mathbb{L}_{\mathbb{Z}_p}Z
\end{align}
Then we consider the corresponding derived algebraic de Rham complex which is just defined to be:
\begin{align}
\mathrm{Kan}_\mathrm{Left}\mathrm{deRham}^\text{degreenumber}_{B_0/A_0,\mathrm{alg}},\mathrm{Fil}^*_{\mathrm{Kan}_\mathrm{Left}\mathrm{deRham}^\text{degreenumber}_{B_0/A_0,\mathrm{alg}}}.	
\end{align}
We then take the corresponding derived $(p,I)$-completion and we denote that by:
\begin{align}
\mathrm{Kan}_\mathrm{Left}\mathrm{deRham}^\text{degreenumber}_{B_0/A_0,\mathrm{topo}}:=R\varprojlim_k \mathrm{Kos}_{(p,I)}\left(\mathrm{Kan}_\mathrm{Left}\mathrm{deRham}^\text{degreenumber}_{B_0/A_0,\mathrm{alg}},\mathrm{Fil}^*_{\mathrm{deRham}^\text{degreenumber}_{B_0/A_0,\mathrm{alg}}}\right),\\
\mathrm{Fil}^*_{\mathrm{Kan}_\mathrm{Left}\mathrm{deRham}^\text{degreenumber}_{B_0/A_0,\mathrm{topo}}}:=R\varprojlim_k\left(\mathrm{Fil}^*_{\mathrm{Kan}_\mathrm{Left}\mathrm{deRham}^\text{degreenumber}_{B_0/A_0,\mathrm{alg}}}\right).	
\end{align}
Before considering the corresponding integral version we just consider the corresponding product of these $\mathbb{E}_\infty$-rings with $Z$ to get:
\begin{align}
\mathrm{Kan}_\mathrm{Left}\mathrm{deRham}^\text{degreenumber}_{B_0/A_0,\mathrm{topo},Z}:=\mathrm{Kan}_\mathrm{Left}\mathrm{deRham}^\text{degreenumber}_{B_0/A_0,\mathrm{topo}}{\otimes}^\mathbb{L}_{\mathbb{Z}_p}Z,\\
\mathrm{Fil}^*_{\mathrm{Kan}_\mathrm{Left}\mathrm{deRham}^\text{degreenumber}_{B_0/A_0,\mathrm{topo}},Z}:=\mathrm{Fil}^*_{\mathrm{Kan}_\mathrm{Left}\mathrm{deRham}^\text{degreenumber}_{B_0/A_0,\mathrm{topo}}}{\otimes}^\mathbb{L}_{\mathbb{Z}_p}Z.	
\end{align}
When we have that the corresponding ring $Z$ is also derived $(p,I)$-topologized and commutative, then we can further take the correspding derived $(p,I)$-completion to achieve the corresponding derived completed version:
\begin{align}
\mathrm{Kan}_\mathrm{Left}\mathrm{deRham}^\text{degreenumber}_{B_0/A_0,\mathrm{topo},Z}:=\mathrm{Kan}_\mathrm{Left}\mathrm{deRham}^\text{degreenumber}_{B_0/A_0,\mathrm{topo}}\widehat{\otimes}^\mathbb{L}_{\mathbb{Z}_p}Z,\\
\mathrm{Fil}^*_{\mathrm{Kan}_\mathrm{Left}\mathrm{deRham}^\text{degreenumber}_{B_0/A_0,\mathrm{topo}},Z}:=\mathrm{Fil}^*_{\mathrm{Kan}_\mathrm{Left}\mathrm{deRham}^\text{degreenumber}_{B_0/A_0,\mathrm{topo}}}\widehat{\otimes}^\mathbb{L}_{\mathbb{Z}_p}Z.	
\end{align}
Then in the situation that all the rings are classical adic rings we consider the following construction for the map $A\rightarrow B$ by putting:
\begin{align}
\mathbb{L}_{B/A,\mathrm{topo},Z}:= \mathrm{Colim}_{A_0\rightarrow B_0}\mathbb{L}_{B_0/A_0,\mathrm{topo},Z}[1/(d)],\\
H_{\text{degreenumber},{\mathrm{AQ}},\mathrm{topo},Z}:=\pi_\text{degreenumber} (\mathbb{L}_{B/A,\mathrm{topo},Z}),	\\
\mathrm{Kan}_\mathrm{Left}\mathrm{deRham}^\text{degreenumber}_{B/A,\mathrm{topo},Z}:=\mathrm{Colim}_{A_0\rightarrow B_0}\mathrm{Kan}_\mathrm{Left}\mathrm{deRham}^\text{degreenumber}_{B_0/A_0,\mathrm{topo},Z}[1/(d)],\\
\mathrm{Fil}^*_{\mathrm{Kan}_\mathrm{Left}\mathrm{deRham}^\text{degreenumber}_{B/A,\mathrm{topo}},Z}:=\mathrm{Colim}_{A_0\rightarrow B_0}\mathrm{Fil}^*_{\mathrm{Kan}_\mathrm{Left}\mathrm{deRham}^\text{degreenumber}_{B_0/A_0,\mathrm{topo}},Z}[1/(d)].
\end{align}
Then we need to take the corresponding Hodge-Filtered completion by using the corresponding filtration associated as above to achieve the corresponding Hodge-complete objects in the corresponding filtered $\infty$-categories:
\begin{align}
\mathrm{Kan}_\mathrm{Left}{\mathrm{deRham}}^\text{degreenumber}_{B/A,\mathrm{topo,Hodge},Z},\mathrm{Fil}^*_{\mathrm{Kan}_\mathrm{Left}{\mathrm{deRham}}^\text{degreenumber}_{B/A,\mathrm{topo,Hodge}},Z}.	
\end{align} 	
\end{definition}


\newpage

\section{Topological Logarithmic Derived De Rham Complexes}

\subsection{Logarithmic Setting for Rigid Analytic Spaces}


\indent In this section we are now going to follow Gabber's construction as in \cite[Chapter 8]{12O} and the extension by \cite[Chapter 5, Chapter 6, Chapter 7]{12B1} to consider the corresponding construction of the topological complete logarithmic cotangent complexes and the corresponding topological logarithmic derived de Rham complex following the corresponding construction for rigid spaces. Certainly the corresponding construction will also follow closely the corresponding \cite[Chapitre 3]{12An1}, \cite{12An2}, \cite[Chapter 2, Chapter 8]{12B1}, \cite[Chapter 1]{12Bei}, \cite[Chapter 5]{12G1}, \cite[Chapter 3, Chapter 4]{12GL}, \cite[Chapitre II, Chapitre III]{12Ill1}, \cite[Chapitre VIII]{12Ill2}, \cite[Section 4]{12Qui} as if we do not have the corresponding log structures, in particular the geometry context underlying is just as in \cite[Chapter 3, Chapter 4]{12GL}. For the geometric foundation of log adic spaces, see \cite{12DLLZ1}. We start from the corresponding construction of the algebraic $p$-adic logarithmic derived de Rham complex for a map $(A,M)\rightarrow (B,N)$ of $p$-complete rings carrying the corresponding log structures, here we use the corresponding notation $(*,?),*=A,B$ to denote the corresponding admissible log rings where $?$ represents the corresponding monoids in the consideration. This is the corresponding derived differential complex attached to the canonical resolution (which we will denote it by the same notation as in the non logarithmic setting) of $(B,N)$:
\begin{align}
(A,M)[(B,N)]^\text{degreenumber},	
\end{align}
which is now denoted by $\mathrm{Kan}_\mathrm{Left}\mathrm{deRham}^\text{degreenumber}_{(B,N)/(A,M),\mathrm{alg}}:=\mathrm{Kan}_\mathrm{Left}\mathrm{deRham}^\text{degreenumber}_{-/(A,M),\mathrm{alg}}((B,N))$ after taking the left Kan extension as in \cite[Chapter 6]{12B1} and applying to the ring $(B,N)$. The corresponding cotangent complex associated is defined to be just:
\begin{align}
\mathbb{L}_{(B,N)/(A,M),\mathrm{alg}}:=	\mathrm{Kan}_\mathrm{Left}\mathrm{deRham}^1_{(A,M)[(B,N)]^\text{degreenumber}/(A,M),\mathrm{alg}}\otimes_{(A,M)[(B,N)]^\text{degreenumber}} (B,N).
\end{align}
The corresponding algebraic Andr\'e-Quillen homologies are defined to be:
\begin{align}
H_{\text{degreenumber},{\mathrm{AQ}},\mathrm{alg}}:=\pi_\text{degreenumber} (\mathbb{L}_{(B,N)/(A,M),\mathrm{alg}}). 	
\end{align}

\indent The corresponding topological Andr\'e-Quillen complex is actually the complete version of the corresponding algebraic ones above by considering the corresponding certain $p$-completion over the simplicial module structure.

\indent Then we consider the corresponding derived algebraic de Rham complex which is just defined to be:
\begin{align}
\mathrm{Kan}_\mathrm{Left}\mathrm{deRham}^\text{degreenumber}_{(B,N)/(A,M),\mathrm{alg}},\mathrm{Fil}^*_{\mathrm{Kan}_\mathrm{Left}\mathrm{deRham}^1_{(B,N)/(A,M),\mathrm{alg}}}.	
\end{align}
We then take the corresponding Banach completion and we denote that by:
\begin{align}
\mathrm{Kan}_\mathrm{Left}\mathrm{deRham}^\text{degreenumber}_{(B,N)/(A,M),\mathrm{topo}},\mathrm{Fil}^*_{\mathrm{Kan}_\mathrm{Left}\mathrm{deRham}^1_{(B,N)/(A,M),\mathrm{topo}}}.	
\end{align}

\indent Then we need to take the corresponding Hodge-Filtered completion by using the corresponding filtration associated as above:
\begin{align}
\mathrm{Kan}_\mathrm{Left}\widehat{\mathrm{deRham}}^\text{degreenumber}_{(B,N)/(A,M),\mathrm{topo}},\mathrm{Fil}^*_{\mathrm{Kan}_\mathrm{Left}\widehat{\mathrm{deRham}}^1_{(B,N)/(A,M),\mathrm{topo}}}.	
\end{align}

This is basically the corresponding analytic and complete version the corresponding algebraic log de Rham complex. Furthermore we allow large coefficients with rigid affinoid algebra $Z$ over $\mathbb{Q}_p$. Therefore we take the corresponding completed tensor product in the following.

\begin{definition}
We define the following $Z$ deformed version of the corresponding complete version of the corresponding logarithmic Andr\'e-Quillen homology and the corresponding complete version of the corresponding logarithmic derived de Rham complex. We start from the corresponding construction of the algebraic $p$-adic logarithmic derived de Rham complex for a map $(A,M)\rightarrow (B,N)$. Fix a pair of ring of definitions $(A_0,M_0),(B_0,N_0)$ in $(A,M),(B,N)$ respectively. Then this is the corresponding derived differential complex attached to the canonical resolution of $(B_0,N_0)$:
\begin{align}
(A_0,M_0)[(B_0,N_0)]^\text{degreenumber},	
\end{align}
which is now denoted by $\mathrm{Kan}_\mathrm{Left}\mathrm{deRham}^\text{degreenumber}_{(A_0,M_0)[(B_0,N_0)]^\text{degreenumber}/(A_0,M_0),\mathrm{alg}}$ after taking the corresponding left Kan extension as in \cite[Chapter 6]{12B1}. The corresponding cotangent complex associated is defined to be just:
\begin{align}
\mathbb{L}_{B_0/(A_0,M_0),\mathrm{alg}}:=	\mathrm{Kan}_\mathrm{Left}\mathrm{deRham}^1_{(A_0,M_0)[B_0,N_0]^\text{degreenumber}/(A_0,M_0),\mathrm{alg}}\otimes_{(A_0,M_0)[B_0,N_0]^\text{degreenumber}} (B_0,N_0).
\end{align}
The corresponding algebraic logarithmic Andr\'e-Quillen homologies are defined to be:
\begin{align}
H_{\text{degreenumber},{\mathrm{AQ}},\mathrm{alg,log}}:=\pi_\text{degreenumber} (\mathbb{L}_{(B_0,N_0)/(A_0,M_0),\mathrm{alg}}). 	
\end{align}
The corresponding topological logarithmic Andr\'e-Quillen complex is actually the complete version of the corresponding algebraic ones above by considering the corresponding derived $p$-completion over the simplicial module structure:
\begin{align}
&\mathbb{L}_{(B_0,N_0)/(A_0,M_0),\mathrm{topo}}\\
&:=R\varprojlim_k\mathrm{Kos}_{p^k}	\left((\mathrm{Kan}_\mathrm{Left}\mathrm{deRham}^1_{(A_0,M_0)[(B_0,N_0)]^\text{degreenumber}/(A_0,M_0),\mathrm{alg}}\otimes_{(A_0,N_0)[(B_0,N_0)]^\text{degreenumber}} (B_0,N_0))\right).
\end{align}
Taking the product with $\mathcal{O}_Z$ we have the corresponding integral version of the topological logarithmic Andr\'e-Quillen complex:
\begin{align}
\mathbb{L}_{(B_0,N_0)/(A_0,M_0),\mathrm{topo},Z}:=\mathbb{L}_{(B_0,N_0)/(A_0,M_0),\mathrm{topo}}\widehat{\otimes}_{\mathbb{Z}_p}\mathcal{O}_Z
\end{align}
Then we consider the corresponding derived algebraic de Rham complex which is just defined to be:
\begin{align}
\mathrm{Kan}_\mathrm{Left}\mathrm{deRham}^\text{degreenumber}_{(B_0,N_0)/(A_0,M_0),\mathrm{alg}},\mathrm{Fil}^*_{\mathrm{Kan}_\mathrm{Left}\mathrm{deRham}^1_{(B_0,N_0)/(A_0,M_0),\mathrm{alg}}}.	
\end{align}
We then take the corresponding derived $p$-completion and we denote that by:
\begin{align}
&\mathrm{Kan}_\mathrm{Left}\mathrm{deRham}^\text{degreenumber}_{(B_0,N_0)/(A_0,M_0),\mathrm{topo}}\\
&:=R\varprojlim_k\mathrm{Kos}_{p^k}\left(\mathrm{Kan}_\mathrm{Left}\mathrm{deRham}^\text{degreenumber}_{(B_0,N_0)/(A_0,M_0),\mathrm{alg}},\mathrm{Fil}^*_{\mathrm{Kan}_\mathrm{Left}\mathrm{deRham}^\text{degreenumber}_{(B_0,N_0)/(A_0,M_0),\mathrm{alg}}}\right),\\
&\mathrm{Fil}^*_{\mathrm{Kan}_\mathrm{Left}\mathrm{deRham}^\text{degreenumber}_{(B_0,N_0)/(A_0,M_0),\mathrm{topo}}}:=R\varprojlim_k\mathrm{Kos}_{p^k}\left(\mathrm{Fil}^*_{\mathrm{Kan}_\mathrm{Left}\mathrm{deRham}^\text{degreenumber}_{(B_0,N_0)/(A_0,M_0),\mathrm{alg}}}\right).	
\end{align}
Before considering the corresponding integral version we just consider the corresponding product of these $\mathbb{E}_\infty$-rings with $\mathcal{O}_Z$ to get:
\begin{align}
\mathrm{Kan}_\mathrm{Left}\mathrm{deRham}^\text{degreenumber}_{(B_0,N_0)/(A_0,M_0),\mathrm{topo},Z}:=\mathrm{Kan}_\mathrm{Left}\mathrm{deRham}^\text{degreenumber}_{(B_0,N_0)/(A_0,M_0),\mathrm{topo}}\widehat{\otimes}^\mathbb{L}_{\mathbb{Z}_p}\mathcal{O}_Z,\\
\mathrm{Fil}^*_{\mathrm{Kan}_\mathrm{Left}\mathrm{deRham}^\text{degreenumber}_{(B_0,N_0)/(A_0,M_0),\mathrm{topo}},Z}:=\mathrm{Fil}^*_{\mathrm{Kan}_\mathrm{Left}\mathrm{deRham}^\text{degreenumber}_{(B_0,N_0)/(A_0,M_0),\mathrm{topo}}}\widehat{\otimes}^\mathbb{L}_{\mathbb{Z}_p}\mathcal{O}_Z.	
\end{align}
Then we consider the following construction for the map $(A,M)\rightarrow (B,N)$ by putting:

\begin{align}
&\mathbb{L}_{(B_0,N_0)/(A_0,M_0),\mathrm{topo},Z}:= \mathrm{Colim}_{(A_0,M_0)\rightarrow (B_0,N_0)}\mathbb{L}_{(B_0,N_0)/(A_0,M_0),\mathrm{topo},Z}[1/p],\\
&H_{\text{degreenumber},{\mathrm{AQ}},\mathrm{topo},Z}:=\pi_\text{degreenumber} (\mathbb{L}_{(B,N)/(A,M),\mathrm{topo},Z}),	\\
&\mathrm{Kan}_\mathrm{Left}\mathrm{deRham}^\text{degreenumber}_{(B,N)/(A,M),\mathrm{topo},Z}:=\mathrm{Colim}_{(A_0,M_0)\rightarrow (B_0,N_0)}\mathrm{Kan}_\mathrm{Left}\mathrm{deRham}^\text{degreenumber}_{(B_0,N_0)/(A_0,M_0),\mathrm{topo},Z}[1/p],\\
&\mathrm{Fil}^*_{\mathrm{Kan}_\mathrm{Left}\mathrm{deRham}^\text{degreenumber}_{(B,N)/(A,M),\mathrm{topo}},Z}:=\mathrm{Colim}\mathrm{Fil}^*_{\mathrm{Kan}_\mathrm{Left}\mathrm{deRham}^\text{degreenumber}_{(B_0,N_0)/(A_0,M_0),\mathrm{topo}},Z}[1/p].
\end{align}
Then we need to take the corresponding Hodge-Filtered completion by using the corresponding filtration associated as above to achieve the corresponding Hodge-complete objects in the corresponding filtered $\infty$-categories:
\begin{align}
{\mathrm{Kan}_\mathrm{Left}\mathrm{deRham}}^\text{degreenumber}_{(B,N)/(A,M),\mathrm{topo,Hodge},Z},\mathrm{Fil}^*_{{\mathrm{Kan}_\mathrm{Left}\mathrm{deRham}}^1_{(B,N)/(A,M),\mathrm{topo,Hodge}},Z}.	
\end{align} 	
\end{definition}

\begin{definition}
We define the corresponding finite projective filtered crystals to be the corresponding finite projective module spectra over the topological filtered $E_\infty$-ring $\mathrm{Kan}_\mathrm{Left}\mathrm{deRham}^\text{degreenumber}_{(B,N)/(A,M),\mathrm{topo},Z}$ with the corresponding induced filtrations.	
\end{definition}

\begin{definition}
We define the corresponding almost perfect \footnote{This is the corresponding derived version of pseudocoherence from \cite{12Lu1}, \cite{12Lu2}.} filtered crystals to be the corresponding almost perfect module spectra over the topological filtered $E_\infty$-ring $\mathrm{Kan}_\mathrm{Left}\mathrm{deRham}^\text{degreenumber}_{(B,N)/(A,N),\mathrm{topo},Z}$ with the corresponding induced filtrations.	
\end{definition}

\indent The following is inspired by the main Poincar\'e Lemma from \cite[Theorem 1.2]{12GL} in the non-deformed situation. Consider a corresponding smooth log rigid analytic space $X$ over $k/\mathbb{Q}_p$ (where $k$ is a corresponding unramified analytic field which is discretely-valued and the corresponding residue field is finite). Then we have the following:

\begin{conjecture}
Consider the corresponding projective map $g:X_{\text{Kummer-pro-\'etale}}\rightarrow X_{\text{Kummer-\'et}}$ and the the projective map $f:X_{\text{Kummer-pro-\'etale}}\rightarrow X$. Then we have the logarithmic versions of the strictly exact long exact sequences as in the Poincar\'e lemma in the rigid analytic situation. 
\end{conjecture}


\newpage

\subsection{Logarithmic Setting for Pseudorigid Spaces}


\indent We now consider the analytic logarithmic Andr\'e-Quillen Homology and analytic logarithmic Derived de Rham complex of pseudorigid space, which are very crucial in some development in \cite{12Bel1} in the arithmetic family. Therefore we just investigate the corresponding picture in the corresponding geometric family.

\begin{setting}
We consider now a corresponding morphism taking the corresponding form of $(A,M)\rightarrow (B,N)$ where $(A,M)$ is going to be a log pseudorigid affinoid algebra over $\mathbb{Z}_p$ (where we consider the corresponding foundation in \cite[Chapter 2]{12DLLZ1} for the corresponding log adic rings), and $(B,N)$ is going to be a Kummer perfectoid chart of $A$ in the corresponding Kummer pro-\'etale site of the log pseudorigid affinoid space attached to $A$ (where we consider the corresponding foundation in \cite[Chapter 5]{12DLLZ1} for the corresponding Kummer pro-\'etale sites). As in \cite[Definition 3.1, and below Definition 3.1]{12Bel1} in our situation $A$ is of topologically finite type over $\mathcal{O}_K[[t]]\left<\pi^a/t^b\right>[1/t]$, where $K$ is a discrete valued field containing $\mathbb{Q}_p$ and $(a,b)=1$ \footnote{Certainly one can also consider the characteristic $p$ situation.}.	
\end{setting}

\indent From \cite[Definition 3.1, and below Definition 3.1]{12Bel1} we have the following:

\begin{lemma} We have the following statements:\\
A. $A$ as above is Tate, complete over $\mathcal{O}_K$;\\
B. The ring $A$ has a ring of definition $A_0$ which is of $\mathcal{O}_K$-formally finite type;\\
C. The ring $A_0$ is of topologically finite type over $\mathcal{O}_K[[t]]\left<\pi^a/t^b\right>$.
\end{lemma}

\begin{proof}
Since we did not change the corresponding assumption on the corresponding ring theoretic consideration.
\end{proof}

As in the corresponding construction in the rigid situation in the previous section following \cite[Chapitre 3]{12An1}, \cite{12An2}, \cite[Chapter 5, Chapter 6, Chapter 7]{12B1}, \cite[Chapter 1]{12Bei}, \cite[Chapter 5]{12G1}, \cite[Chapter 3, Chapter 4]{12GL}, \cite[Chapitre II, Chapitre III]{12Ill1}, \cite[Chapitre VIII]{12Ill2}, \cite[Chapter 8]{12O}, \cite[Section 4]{12Qui}. We keep now the following setting:

\begin{setting}
Now as in the above notion we will consider a general map of logarithmic rings $(A,M)\rightarrow (B,N)$ over $\mathcal{O}_K[[t]]\left<\pi^a/t^b\right>[1/t]$ where we have the corresponding map of the associated ring of definitions $(A_0,M_0)\rightarrow (B_0,N_0)$ over $\mathcal{O}_K[[t]]\left<\pi^a/t^b\right>$ such that $A_0,B_0$ are basically $I$-adic (where $\mathcal{O}_K[[t]]\left<\pi^a/t^b\right>$ is $I$-adic).	
\end{setting}

\begin{definition}
We start from the corresponding construction of the algebraic $p$-adic log derived de Rham complex for a map $(A,M)\rightarrow (B,N)$. Fix a pair of ring of definitions $(A_0,M_0),(B_0,N_0)$ in $(A,M),(B,N)$ respectively. Then this is the corresponding derived differential complex attached to the canonical resolution of $(B_0,N_0)$:
\begin{align}
(A_0,M_0)[(B_0,N_0)]^\text{degreenumber},	
\end{align}
which is now denoted by $\mathrm{Kan}_\mathrm{Left}\mathrm{deRham}^\text{degreenumber}_{(B_0,N_0)/(A_0,M_0),\mathrm{alg}}$. This is after taking the left Kan extension as in \cite[Chapter 6]{12B1}. The corresponding logarithmic cotangent complex associated is defined to be just:
\begin{align}
\mathbb{L}_{(B_0,N_0)/(A_0,M_0),\mathrm{alg}}:=	\mathrm{Kan}_\mathrm{Left}\mathrm{deRham}^1_{(A_0,M_0)[(B_0,N_0)]^\text{degreenumber}/(A_0,M_0),\mathrm{alg}}\otimes_{(A_0,M_0)[(B_0,N_0)]^\text{degreenumber}} (B_0,N_0).
\end{align}
The corresponding algebraic logarithmic Andr\'e-Quillen homologies are defined to be:
\begin{align}
H_{\text{degreenumber},{\mathrm{AQ}},\mathrm{alg}}:=\pi_\text{degreenumber} (\mathbb{L}_{(B_0,N_0)/(A_0,M_0),\mathrm{alg}}). 	
\end{align}
The corresponding topological logarithmic Andr\'e-Quillen complex is actually the complete version of the corresponding algebraic ones above by considering the corresponding derived $I$-completion over the simplicial module structure:

\begin{align}
&\mathbb{L}_{(B_0,N_0)/(A_0,M_0),\mathrm{topo}}\\
&:=R\varprojlim_I\mathrm{Kos}_I	\left((\mathrm{Kan}_\mathrm{Left}\mathrm{deRham}^1_{(A_0,M_0)[(B_0,N_0)]^\text{degreenumber}/(A_0,M_0),\mathrm{alg}}\otimes_{(A_0,M_0)[(B_0,N_0)]^\text{degreenumber}} (B_0,N_0))\right).
\end{align}
Then we consider the corresponding logarithmic derived algebraic de Rham complex which is just defined to be:
\begin{align}
\mathrm{Kan}_\mathrm{Left}\mathrm{deRham}^\text{degreenumber}_{(B_0,N_0)/(A_0,M_0),\mathrm{alg}},\mathrm{Fil}^*_{\mathrm{Kan}_\mathrm{Left}\mathrm{deRham}^1_{(B_0,N_0)/(A_0,M_0),\mathrm{alg}}}.	
\end{align}
We then take the corresponding derived $I$-completion and we denote that by:
\begin{align}
&\mathrm{Kan}_\mathrm{Left}\mathrm{deRham}^\text{degreenumber}_{(B_0,N_0)/(A_0,M_0),\mathrm{topo}}:=\\
&R\varprojlim_I\mathrm{Kos}_I\left(\mathrm{Kan}_\mathrm{Left}\mathrm{deRham}^\text{degreenumber}_{(A_0,M_0)[(B_0,N_0)]^\text{degreenumber}/(A_0,M_0),\mathrm{alg}},\mathrm{Fil}^*_{\mathrm{Kan}_\mathrm{Left}\mathrm{deRham}^\text{degreenumber}_{(B_0,N_0)/(A_0,M_0),\mathrm{alg}}}\right),\\
&\mathrm{Fil}^*_{\mathrm{Kan}_\mathrm{Left}\mathrm{deRham}^\text{degreenumber}_{(B_0,N_0)/(A_0,M_0),\mathrm{topo}}}:=R\varprojlim_I\mathrm{Kos}_I\left(\mathrm{Fil}^*_{\mathrm{Kan}_\mathrm{Left}\mathrm{deRham}^\text{degreenumber}_{(B_0,N_0)/(A_0,M_0),\mathrm{alg}}}\right).	
\end{align}
Then we consider the following construction for the map $(A,M)\rightarrow (B,N)$ by putting:
\begin{align}
&\mathbb{L}_{(B_0,N_0)/(A_0,M_0),\mathrm{topo}}:= \mathrm{Colim}_{(A_0,M_0)\rightarrow (B_0,N_0)}\mathbb{L}_{(B_0,N_0)/(A_0,M_0),\mathrm{topo}}[1/t],\\
&H_{\text{degreenumber},{\mathrm{AQ}},\mathrm{topo}}:=\pi_\text{degreenumber} (\mathbb{L}_{(B,M)/(A,M),\mathrm{topo}}),	\\
&\mathrm{Kan}_\mathrm{Left}\mathrm{deRham}^\text{degreenumber}_{(B,N)/(A,M),\mathrm{topo}}:=\mathrm{Colim}_{(A_0,M_0)\rightarrow (B_0,N_0)}\mathrm{Kan}_\mathrm{Left}\mathrm{deRham}^\text{degreenumber}_{(B_0,N_0)/(A_0,M_0),\mathrm{topo}}[1/t],\\
&\mathrm{Fil}^*_{\mathrm{Kan}_\mathrm{Left}\mathrm{deRham}^\text{degreenumber}_{(B,N)/(A,M),\mathrm{topo}}}:=\mathrm{Colim}_{(A_0,M_0)\rightarrow (B_0,N_0)}\mathrm{Fil}^*_{\mathrm{Kan}_\mathrm{Left}\mathrm{deRham}^\text{degreenumber}_{(B_0,N_0)/(A_0,M_0),\mathrm{topo}}}[1/t].
\end{align}
Then we need to take the corresponding Hodge-Filtered completion by using the corresponding filtration associated as above to achieve the corresponding Hodge-complete objects in the corresponding filtered $\infty$-categories:
\begin{align}
{\mathrm{Kan}_\mathrm{Left}\mathrm{deRham}}^\text{degreenumber}_{(B,N)/(A,M),\mathrm{topo,Hodge}},\mathrm{Fil}^*_{{\mathrm{Kan}_\mathrm{Left}\mathrm{deRham}}^\text{degreenumber}_{(B,N)/(A,M),\mathrm{topo,Hodge}}}.	
\end{align}
\end{definition}

\begin{definition}
We define the corresponding finite projective filtered crystals to be the corresponding finite projective module spectra over the topological filtered $E_\infty$-ring $\mathrm{Kan}_\mathrm{Left}\mathrm{deRham}^\text{degreenumber}_{(B,N)/(A,M),\mathrm{topo}}$ with the corresponding induced filtrations.	
\end{definition}

\begin{definition}
We define the corresponding almost perfect \footnote{This is the corresponding derived version of pseudocoherence from \cite{12Lu1}, \cite{12Lu2}.} filtered crystals to be the corresponding almost perfect module spectra over the topological filtered $E_\infty$-ring $\mathrm{Kan}_\mathrm{Left}\mathrm{deRham}^\text{degreenumber}_{(B,N)/(A,M),\mathrm{topo}}$ with the corresponding induced filtrations.	
\end{definition}

\indent We now consider the corresponding large coefficient local systems over logarithmic pseudorigid spaces where $Z$ now is a topological algebra over $\mathbb{Z}_p$\footnote{It is better to assume that it is basically completely flat.}.

\begin{definition}
We define the following $Z$ deformed version of the corresponding complete version of the corresponding logarithmic Andr\'e-Quillen homology and the corresponding complete version of the corresponding logarithmic derived de Rham complex. We start from the corresponding construction of the logarithmic algebraic $p$-adic derived de Rham complex for a map $(A,M)\rightarrow (B,N)$. Fix a pair of ring of definitions $(A_0,M_0),(B_0,N_0)$ in $(A,M),(B,N)$ respectively. Then this is the corresponding derived differential complex attached to the canonical resolution of $(B_0,N_0)$:
\begin{align}
(A_0,M_0)[(B_0,N_0)]^\text{degreenumber},	
\end{align}
which is now denoted by $\mathrm{Kan}_\mathrm{Left}\mathrm{deRham}^\text{degreenumber}_{(B_0,N_0)/(A_0,M_0),\mathrm{alg}}$. The corresponding logarithmic cotangent complex associated is defined to be just:
\begin{align}
\mathbb{L}_{(B_0,N_0)/(A_0,M_0),\mathrm{alg}}:=	\mathrm{Kan}_\mathrm{Left}\mathrm{deRham}^1_{(A_0,M_0)[(B_0,N_0)]^\text{degreenumber}/(A_0,M_0),\mathrm{alg}}\otimes_{(A_0,M_0)[(B_0,N_0)]^\text{degreenumber}} (B_0,N_0).
\end{align}
The corresponding algebraic logarithmic Andr\'e-Quillen homologies are defined to be:
\begin{align}
H_{\text{degreenumber},{\mathrm{AQ}},\mathrm{alg}}:=\pi_\text{degreenumber} (\mathbb{L}_{(B_0,N_0)/(A_0,M_0),\mathrm{alg}}). 	
\end{align}
The corresponding topological logarithmic Andr\'e-Quillen complex is actually the complete version of the corresponding algebraic ones above by considering the corresponding derived $I$-completion over the simplicial module structure:
\begin{align}
&\mathbb{L}_{(B_0,N_0)/(A_0,M_0),\mathrm{topo}}\\
&:=R\varprojlim_k\mathrm{Kos}_I	\left((\mathrm{Kan}_\mathrm{Left}\mathrm{deRham}^1_{(A_0,M_0)[(B_0,N_0)]^\text{degreenumber}/(A_0,M_0),\mathrm{alg}}\otimes_{(A_0,M_0)[(B_0,N_0)]^\text{degreenumber}} (B_0,N_0))\right).
\end{align}
Taking the product with $Z$ we have the corresponding integral version of the topological logarithmic version of the corresponding Andr\'e-Quillen complex:
\begin{align}
\mathbb{L}_{(B_0,N_0)/(A_0,M_0),\mathrm{topo},Z}:=\mathbb{L}_{(B_0,N_0)/(A_0,M_0),\mathrm{topo}}{\otimes}_{\mathbb{Z}_p}\mathcal{O}_Z
\end{align}
Then we consider the corresponding logarithmic derived algebraic de Rham complex which is just defined to be:
\begin{align}
\mathrm{Kan}_\mathrm{Left}\mathrm{deRham}^\text{degreenumber}_{(B_0,N_0)/(A_0,M_0),\mathrm{alg}},\mathrm{Fil}^*_{\mathrm{Kan}_\mathrm{Left}\mathrm{deRham}^1_{(B_0,N_0)/(A_0,M_0),\mathrm{alg}}}.	
\end{align}
We then take the corresponding derived $I$-completion and we denote that by:
\begin{align}
&\mathrm{Kan}_\mathrm{Left}\mathrm{deRham}^\text{degreenumber}_{(B_0,N_0)/(A_0,M_0),\mathrm{topo}}:=\\
&R\varprojlim_k\mathrm{Kos}_I\left(\mathrm{Kan}_\mathrm{Left}\mathrm{deRham}^\text{degreenumber}_{(B_0,N_0)/(A_0,M_0),\mathrm{alg}},\mathrm{Fil}^*_{\mathrm{Kan}_\mathrm{Left}\mathrm{deRham}^\text{degreenumber}_{(B_0,N_0)/(A_0,M_0),\mathrm{alg}}}\right),\\
&\mathrm{Fil}^*_{\mathrm{Kan}_\mathrm{Left}\mathrm{deRham}^\text{degreenumber}_{(B_0,N_0)/(A_0,M_0),\mathrm{topo}}}:=R\varprojlim_k\mathrm{Kos}_I\left(\mathrm{Fil}^*_{\mathrm{Kan}_\mathrm{Left}\mathrm{deRham}^\text{degreenumber}_{(B_0,N_0)/(A_0,M_0),\mathrm{alg}}}\right).	
\end{align}
Before considering the corresponding integral version we just consider the corresponding product of these $\mathbb{E}_\infty$-rings with $Z$ to get:
\begin{align}
&\mathrm{Kan}_\mathrm{Left}\mathrm{deRham}^\text{degreenumber}_{(B_0,N_0)/(A_0,M_0),\mathrm{topo},Z}:=\mathrm{Kan}_\mathrm{Left}\mathrm{deRham}^\text{degreenumber}_{(B_0,N_0)/(A_0,M_0),\mathrm{topo}}{\otimes}^\mathbb{L}_{\mathbb{Z}_p}Z,\\
&\mathrm{Fil}^*_{\mathrm{Kan}_\mathrm{Left}\mathrm{deRham}^\text{degreenumber}_{(B_0,N_0)/(A_0,M_0),\mathrm{topo}},Z}:=\mathrm{Fil}^*_{\mathrm{Kan}_\mathrm{Left}\mathrm{deRham}^\text{degreenumber}_{(B_0,N_0)/(A_0,M_0),\mathrm{topo}}}{\otimes}^\mathbb{L}_{\mathbb{Z}_p}Z.	
\end{align}
Then we consider the following construction for the map $(A,M)\rightarrow (B,N)$ by putting:
\begin{align}
&\mathbb{L}_{(B_0,N_0)/(A_0,M_0),\mathrm{topo},Z}:= \mathrm{Colim}_{(A_0,M_0)\rightarrow (B_0,N_0)}\mathbb{L}_{(B_0,N_0)/(A_0,M_0),\mathrm{topo},Z}[1/t],\\
&H_{\text{degreenumber},{\mathrm{AQ}},\mathrm{topo},Z}:=\pi_\text{degreenumber} (\mathbb{L}_{(B,N)/(A,M),\mathrm{topo},Z}),	\\
&\mathrm{Kan}_\mathrm{Left}\mathrm{deRham}^\text{degreenumber}_{(B,N)/(A,M),\mathrm{topo},Z}:=\mathrm{Colim}_{(A_0,M_0)\rightarrow (B_0,N_0)}\mathrm{Kan}_\mathrm{Left}\mathrm{deRham}^\text{degreenumber}_{(B_0,N_0)/(A_0,M_0),\mathrm{topo},Z}[1/t],\\
&\mathrm{Fil}^*_{\mathrm{Kan}_\mathrm{Left}\mathrm{deRham}^\text{degreenumber}_{(B,N)/(A,M),\mathrm{topo}},Z}:=\mathrm{Colim}_{(A_0,M_0)\rightarrow (B_0,N_0)}\mathrm{Fil}^*_{\mathrm{Kan}_\mathrm{Left}\mathrm{deRham}^\text{degreenumber}_{(B_0,N_0)/(A_0,M_0),\mathrm{topo}},Z}[1/t].
\end{align}
Then we need to take the corresponding Hodge-Filtered completion by using the corresponding filtration associated as above to achieve the corresponding Hodge-complete objects in the corresponding filtered $\infty$-categories:
\begin{align}
{\mathrm{Kan}_\mathrm{Left}\mathrm{deRham}}^\text{degreenumber}_{(B,N)/(A,M),\mathrm{topo,Hodge},Z},\mathrm{Fil}^*_{{\mathrm{Kan}_\mathrm{Left}\mathrm{deRham}}^1_{(B,N)/(A,M),\mathrm{topo,Hodge}},Z}.	
\end{align} 	
\end{definition}

\begin{example}
Now we construct the corresponding pseudorigid analog of the corresponding example in \cite[Example 4.7]{12GL}. Now we consider the corresponding the map:
\begin{align}
\mathbb{Z}_p[[u]]\left<\frac{p^a}{u^b}\right>[1/u]\left<X^\pm_1,X^\pm_2,...,X^\pm_d,X_{d+1},...,X_e\right>\\
\longrightarrow \mathbb{Z}_p[[u]]\left<\frac{p^a}{u^b}\right>[1/u]\left<X^{\pm/p^\infty}_1,X^{\pm/p^\infty}_2,...,X^{\pm/p^\infty}_d,X^{1/p^\infty}_{d+1},...,X^{1/p^\infty}_e\right>	
\end{align}
which could be written as:	
\begin{align}
&\mathbb{Z}_p[[u]]\left<\frac{p^a}{u^b}\right>[1/u]\left<X^\pm_1,X^\pm_2,...,X^\pm_d,X_{d+1},...,X_e\right>\longrightarrow \\
&\mathbb{Z}_p[[u]]\left<\frac{p^a}{u^b}\right>[1/u]\left<X^\pm_1,X^\pm_2,...,X^\pm_d,X_{d+1},...,X_e\right>\left<Y^{\pm/p^\infty}_1,...,Y^{\pm/p^\infty}_d,Y^{1/p^\infty}_{d+1},...,Y^{1/p^\infty}_e\right>\\
&/(X_i-Y_i,i=1,...,e).		
\end{align}
So in our situation the corresponding Hodge complete topological derived de Rham complex will be just:
\begin{align}
\mathbb{Z}_p[[u]]\left<\frac{p^a}{u^b}\right>[1/u]&\left<Y^{\pm/p^\infty}_1,Y^{\pm/p^\infty}_2,...,Y^{\pm/p^\infty}_d,Y^{1/p^\infty}_{d+1},...,Y^{1/p^\infty}_e\right>\\
&[[Z_1,...,Z_e,Z_i=X_i-Y_i,i=1,...,e]].	
\end{align}

\end{example}

\

\indent The following conjectures are literally inspired by the rigid analytic situation in \cite[Theorem 1.2]{12GL}. We assume the spaces are smooth.

\begin{conjecture}
After the whole foundation of \cite{12DLLZ1} and \cite[Theorem 2.9.9, Remark 2.9.10]{12Ked1}, the construction in this section could be carried over the corresponding certain Kummer-pro-\'etale sites where logarithmic perfectoid subdomains form a corresponding basis of neighbourhood of the topology.	
\end{conjecture}

\indent Based on this conjecture one can conjecture the following:

\begin{conjecture}
Consider the corresponding projective map $g:X_{\text{Kummer-pro-\'etale}}\rightarrow X_{\text{Kummer-\'et}}$ and the the projective map $f:X_{\text{Kummer-pro-\'etale}}\rightarrow X$. Then we have the log versions of the strictly exact Poincar\'e long exact sequences as in the pseudorigid analytic situation.
\end{conjecture}

and

\begin{conjecture}
Consider the corresponding projective map $g:X_{\text{Kummer-pro-\'etale}}\rightarrow X_{\text{Kummer-\'et}}$ and the the projective map $f:X_{\text{Kummer-pro-\'etale}}\rightarrow X$. Let $M$ be a corresponding $Z$-projective differential crystal spectrum. Then we have the logarithmic versions of the strictly exact Poincar\'e long exact sequences for $M$ as in the pseudorigid analytic siatution.
\end{conjecture}

\newpage

\subsection{Derived $(p,I)$-Complete Logarithmic Derived de Rham complex}

\indent The corresponding construction of the derived de Rham complex could be defined for general derived spaces. Along our discussion in the situations of rigid analytic spaces and the corresponding pseudorigid spaces we now focus on the corresponding derived  rings. And we consider the corresponding logarithmic setting. 

\begin{setting}
We now fix a bounded morphism of log simplicial topological rings $(A,M)\rightarrow (B,N)$ over $A^*$ where $A^*$ contains a corresponding ring of definition $A^*_0$ which is derived complete with respect to the $(p,I)$-topology and we assume that $(A,M)$ is adic in the sense to be defined just below and we assume that $(A^*_0,I)$ is a prism namely we at least require that the corresponding $\delta$-structure on the corresponding ring will induce the map $\varphi(.):=.^p+p\delta(.)$ such that we have the situation where $p\in (I,\varphi(I))$. For $(A,M)$ or $(B,N)$ respectively we assume this contains a subring $(A_0,M)$ or $(B_0,N_0)$ (over $A^*_0$) respectively such that we have $A_0$ or $B_0$ respectively is derived complete with respect to the corresponding derived $(p,I)$-topology and we assume that $B=B_0[1/f,f\in I]$ (same for $A$) \footnote{One can also invert $p$ to consider the rationality with respect to $p$ as in the rigid analytic situation.}. All the adic rings are assumed to be open mapping. We use the notation $d$ to denote a corresponding primitive element as in \cite[Section 2.3]{12BS} for $A^*$. We are going to assume that $p$ is a topologically nilpotent element.
\end{setting}

\indent As in the corresponding non-logarithmic situation following \cite[Chapter 6]{12B1}\footnote{Here the corresponding left Kan extension happens along the embedding of the free prelog rings to the $\infty$-category of simplicial prelog rings. See \cite[Chapter 6]{12B1} for the construction.} we first consider the category $\mathrm{Alg}_{\mathrm{prelog},\mathrm{smooth},A_0^*}$ of all the smooth prelog algebras over $A_0^*$ and take the corresponding left Kan extension to the corresponding $\infty$-category $\mathrm{Alg}_{\infty,\mathrm{prelog},A_0^*}$ of all the prelog simplicial $A_0^*$-algebras. Then we may take the corresponding derived $(p,I)$-completion to get completed objects. Finally we could basically apply this to the relatively more specific rings in the previous setting. These are basically parallel to the situations where we work over $\mathbb{Z}_p$ and the pseudorigid disc\footnote{Certainly we could work in more general setting over any derived $J$-complete ring, such as in the situation where we do not require that the ring $A^*_0$ is a prism.}.


\begin{remark}
The corresponding idea behind this is certainly inspired by \cite{12LL} for instance where the authors compared the corresponding prismatic cohomology of some suitable ring $R$ over $A^*_0/I$ for instance and the corresponding derived de Rham cohomology canonically attached to the corresponding morphism, and the corresponding Hodge filtration is compared to the Nygaard filtration. 	
\end{remark}

\indent As in the corresponding construction in the rigid and pseudocoherent situation in the previous section following \cite[Chapitre 3]{12An1}, \cite{12An2}, \cite[Chapter 5, Chapter 6, Chapter 7]{12B1}, \cite[Chapter 1]{12Bei}, \cite[Chapter 5]{12G1}, \cite[Chapter 3, Chapter 4]{12GL}, \cite[Chapitre II, Chapitre III]{12Ill1}, \cite[Chapitre VIII]{12Ill2}, \cite[Chapter 8]{12O}, \cite[Section 4]{12Qui} we give the following definition. 

\begin{definition}
We start from the corresponding construction of the algebraic $p$-adic logarithmic derived de Rham complex for a map $(A,M)\rightarrow (B,N)$. Fix a pair of ring of definitions $(A_0,M_0),(B_0,N_0)$ in $(A,M),(B,N)$ respectively. Then this is the corresponding derived differential complex attached to the cofibrant replacement of $(B_0,N_0)$:
\begin{align}
P^\text{degreenumber}_{(B_0,N_0)},	
\end{align}
which is now denoted by $\mathrm{Kan}_\mathrm{Left}\mathrm{deRham}^\text{degreenumber}_{(B_0,N_0)/(A_0,M_0),\mathrm{alg}}$ after taking the corresponding left Kan extension as in \cite[Chapter 6]{12B1}. The corresponding logarithmic cotangent complex associated is defined to be just:
\begin{align}
\mathbb{L}_{(B_0,N_0)/(A_0,M_0),\mathrm{alg}}:=	\mathrm{Kan}_\mathrm{Left}\mathrm{deRham}^1_{P^\text{degreenumber}_{(B_0,N_0)}/(A_0,M_0),\mathrm{alg}}\otimes_{P^\text{degreenumber}_{(B_0,N_0)}} (B_0,N_0).
\end{align}
The corresponding logarithmic algebraic Andr\'e-Quillen homologies are defined to be:
\begin{align}
H_{\text{degreenumber},{\mathrm{AQ}},\mathrm{alg}}:=\pi_\text{degreenumber} (\mathbb{L}_{(B_0,N_0)/(A_0,M_0),\mathrm{alg}}). 	
\end{align}
The corresponding topological Andr\'e-Quillen complex is actually the complete version of the corresponding algebraic ones above by considering the corresponding derived $(p,I)$-completion over the simplicial module structure:
\begin{align}
\mathbb{L}_{(B_0,N_0)/(A_0,M_0),\mathrm{topo}}:=R\varprojlim \mathrm{Kos}_{(p,I)}	\left((\mathrm{Kan}_\mathrm{Left}\mathrm{deRham}^1_{P^\text{degreenumber}_{(B_0,N_0)}/(A_0,M_0),\mathrm{alg}}\otimes_{P^\text{degreenumber}_{(B_0,N_0)}} (B_0,N_0))\right).
\end{align}
Then we consider the corresponding derived algebraic de Rham complex which is just defined to be:
\begin{align}
\mathrm{Kan}_\mathrm{Left}\mathrm{deRham}^\text{degreenumber}_{(B_0,N_0)/(A_0,M_0),\mathrm{alg}},\mathrm{Fil}^*_{\mathrm{Kan}_\mathrm{Left}\mathrm{deRham}^1_{(B_0,N_0)/(A_0,M_0),\mathrm{alg}}}.	
\end{align}
We then take the corresponding derived $(p,I)$-completion and we denote that by:
\begin{align}
\mathrm{Kan}_\mathrm{Left}&\mathrm{deRham}^\text{degreenumber}_{(B_0,N_0)/(A_0,M_0),\mathrm{topo}}\\
&:=R\varprojlim  \mathrm{Kos}_{(p,I)}\left(\mathrm{Kan}_\mathrm{Left}\mathrm{deRham}^\text{degreenumber}_{(B_0,N_0)/(A_0,M_0),\mathrm{alg}},\mathrm{Fil}^*_{\mathrm{Kan}_\mathrm{Left}\mathrm{deRham}^\text{degreenumber}_{(B_0,N_0)/(A_0,M_0),\mathrm{alg}}}\right),\\
&\mathrm{Fil}^*_{\mathrm{Kan}_\mathrm{Left}\mathrm{deRham}^\text{degreenumber}_{(B_0,N_0)/(A_0,M_0),\mathrm{topo}}}:=R\varprojlim \mathrm{Kos}_{(p,I)} \left(\mathrm{Fil}^*_{\mathrm{Kan}_\mathrm{Left}\mathrm{deRham}^\text{degreenumber}_{(B_0,N_0)/(A_0,M_0),\mathrm{alg}}}\right).	
\end{align}
Then in the situation that all the rings are classical logarithmic adic rings we consider the following construction for the map $A\rightarrow B$ by putting:
\begin{align}
\mathbb{L}_{(B,N)/(A,M),\mathrm{topo}}:= \mathrm{Colim}_{(A_0,M_0)\rightarrow (B_0,N_0)}\mathbb{L}_{(B_0,N_0)/(A_0,M_0),\mathrm{topo}}[1/(d)],\\
H_{\text{degreenumber},{\mathrm{AQ}},\mathrm{topo}}:=\pi_\text{degreenumber} (\mathbb{L}_{(B,N)/(A,M),\mathrm{topo}}),	\\
\mathrm{Kan}_\mathrm{Left}\mathrm{deRham}^\text{degreenumber}_{(B,N)/(A,M),\mathrm{topo}}:=\mathrm{Colim}_{(A_0,M_0)\rightarrow (B_0,N_0)}\mathrm{Kan}_\mathrm{Left}\mathrm{deRham}^\text{degreenumber}_{(B_0,N_0)/(A_0,M_0),\mathrm{topo}}[1/(d)],\\
\mathrm{Fil}^*_{\mathrm{Kan}_\mathrm{Left}\mathrm{deRham}^\text{degreenumber}_{(B,N)/(A,M),\mathrm{topo}}}:=\mathrm{Colim}_{(A_0,M_0)\rightarrow (B_0,N_0)}\mathrm{Fil}^*_{\mathrm{Kan}_\mathrm{Left}\mathrm{deRham}^\text{degreenumber}_{(B_0,N_0)/(A_0,M_0),\mathrm{topo}}}[1/(d)].
\end{align}
Then we need to take the corresponding Hodge-Filtered completion by using the corresponding filtration associated as above to achieve the corresponding Hodge-complete objects in the corresponding filtered $\infty$-categories:
\begin{align}
{\mathrm{Kan}_\mathrm{Left}\mathrm{deRham}}^\text{degreenumber}_{(B,N)/(A,M),\mathrm{topo,Hodge}},\mathrm{Fil}^*_{{\mathrm{Kan}_\mathrm{Left}\mathrm{deRham}}^\text{degreenumber}_{(B,N)/(A,M),\mathrm{topo,Hodge}}}.	
\end{align}
\end{definition}

\begin{definition}
We define the corresponding finite projective filtered crystals to be the corresponding finite projective module spectra over the topological filtered $E_\infty$-ring $\mathrm{Kan}_\mathrm{Left}\mathrm{deRham}^\text{degreenumber}_{(B,N)/{(A,M)},\mathrm{topo}}$ with the corresponding induced filtrations.	
\end{definition}

\begin{definition}
We define the corresponding almost perfect \footnote{This is the corresponding derived version of pseudocoherence from \cite{12Lu1}, \cite{12Lu2}.} filtered crystals to be the corresponding almost perfect module spectra over the topological filtered $E_\infty$-ring $\mathrm{Kan}_\mathrm{Left}\mathrm{deRham}^\text{degreenumber}_{(B,N)/(A,M),\mathrm{topo}}$ with the corresponding induced filtrations.	
\end{definition}

\indent We now consider the corresponding large coefficient local systems over pseudorigid spaces where $Z$ now is a simplicial topological algebra over $\mathbb{Z}_p$\footnote{It is better to assume that it is basically derived completely flat.}.


\begin{definition}
We define the following $Z$ deformed version of the corresponding complete version of the corresponding logarithmic Andr\'e-Quillen homology and the corresponding complete version of the corresponding logarithmic derived de Rham complex. We start from the corresponding construction of the logarithmic algebraic derived de Rham complex for a map $(A,M)\rightarrow (B,N)$. Fix a pair of ring of definitions $(A_0,M_0),(B_0,N_0)$ in $(A,M),(B,N)$ respectively. Then this is the corresponding derived differential complex attached to the canonical resolution of $(B_0,N_0)$:
\begin{align}
(A_0,M_0)[(B_0,N_0)]^\text{degreenumber},	
\end{align}
which is now denoted by $\mathrm{Kan}_\mathrm{Left}\mathrm{deRham}^\text{degreenumber}_{(B_0,N_0)/(A_0,M_0),\mathrm{alg}}$. The corresponding logarithmic cotangent complex associated is defined to be just:
\begin{align}
\mathbb{L}_{(B_0,N_0)/(A_0,M_0),\mathrm{alg}}\\
&:=	\mathrm{Kan}_\mathrm{Left}\mathrm{deRham}^1_{(A_0,M_0)[(B_0,N_0)]^\text{degreenumber}/(A_0,M_0),\mathrm{alg}}\otimes_{(A_0,M_0)[(B_0,N_0)]^\text{degreenumber}} (B_0,N_0).
\end{align}
The corresponding algebraic logarithmic Andr\'e-Quillen homologies are defined to be:
\begin{align}
H_{\text{degreenumber},{\mathrm{AQ}},\mathrm{alg}}:=\pi_\text{degreenumber} (\mathbb{L}_{(B_0,N_0)/(A_0,M_0),\mathrm{alg}}). 	
\end{align}
The corresponding topological logarithmic Andr\'e-Quillen complex is actually the complete version of the corresponding algebraic ones above by considering the corresponding derived $(p,I)$-completion over the simplicial module structure:
\begin{align}
&\mathbb{L}_{(B_0,N_0)/(A_0,M_0),\mathrm{topo}}\\
&:=R\varprojlim_k\mathrm{Kos}_{(p,I)}	\left((\mathrm{Kan}_\mathrm{Left}\mathrm{deRham}^1_{(A_0,M_0)[(B_0,N_0)]^\text{degreenumber}/(A_0,M_0),\mathrm{alg}}\otimes_{(A_0,M_0)[(B_0,N_0)]^\text{degreenumber}} (B_0,N_0))\right).
\end{align}
Taking the product with $Z$ we have the corresponding integral version of the topological logarithmic version of the corresponding Andr\'e-Quillen complex:
\begin{align}
\mathbb{L}_{(B_0,N_0)/(A_0,M_0),\mathrm{topo},Z}:=\mathbb{L}_{(B_0,N_0)/(A_0,M_0),\mathrm{topo}}{\otimes}_{\mathbb{Z}_p}Z
\end{align}
Then we consider the corresponding logarithmic derived algebraic de Rham complex which is just defined to be:
\begin{align}
\mathrm{Kan}_\mathrm{Left}\mathrm{deRham}^\text{degreenumber}_{(B_0,N_0)/(A_0,M_0),\mathrm{alg}},\mathrm{Fil}^*_{\mathrm{Kan}_\mathrm{Left}\mathrm{deRham}^1_{(B_0,N_0)/(A_0,M_0),\mathrm{alg}}}.	
\end{align}
We then take the corresponding derived $(p,I)$-completion and we denote that by:
\begin{align}
\mathrm{Kan}_\mathrm{Left}&\mathrm{deRham}^\text{degreenumber}_{(B_0,N_0)/(A_0,M_0),\mathrm{topo}}:=\\
&R\varprojlim_k\mathrm{Kos}_{(p,I)}\left(\mathrm{Kan}_\mathrm{Left}\mathrm{deRham}^\text{degreenumber}_{(B_0,N_0)/(A_0,M_0),\mathrm{alg}},\mathrm{Fil}^*_{\mathrm{Kan}_\mathrm{Left}\mathrm{deRham}^\text{degreenumber}_{(B_0,N_0)/(A_0,M_0),\mathrm{alg}}}\right),\\
&\mathrm{Fil}^*_{\mathrm{Kan}_\mathrm{Left}\mathrm{deRham}^\text{degreenumber}_{(B_0,N_0)/(A_0,M_0),\mathrm{topo}}}:=R\varprojlim_k\mathrm{Kos}_{(p,I)}\left(\mathrm{Fil}^*_{\mathrm{Kan}_\mathrm{Left}\mathrm{deRham}^\text{degreenumber}_{(B_0,N_0)/(A_0,M_0),\mathrm{alg}}}\right).	
\end{align}
Before considering the corresponding integral version we just consider the corresponding product of these $\mathbb{E}_\infty$-rings with $Z$ to get:
\begin{align}
\mathrm{Kan}_\mathrm{Left}\mathrm{deRham}^\text{degreenumber}_{(B_0,N_0)/(A_0,M_0),\mathrm{topo},Z}:=\mathrm{Kan}_\mathrm{Left}\mathrm{deRham}^\text{degreenumber}_{(B_0,N_0)/(A_0,M_0),\mathrm{topo}}{\otimes}^\mathbb{L}_{\mathbb{Z}_p}Z,\\
\mathrm{Fil}^*_{\mathrm{Kan}_\mathrm{Left}\mathrm{deRham}^\text{degreenumber}_{(B_0,N_0)/(A_0,M_0),\mathrm{topo}},Z}:=\mathrm{Fil}^*_{\mathrm{Kan}_\mathrm{Left}\mathrm{deRham}^\text{degreenumber}_{(B_0,N_0)/(A_0,M_0),\mathrm{topo}}}{\otimes}^\mathbb{L}_{\mathbb{Z}_p}Z.	
\end{align}
When we have that the corresponding ring $Z$ is also derived $(p,I)$-topologized and commutative, then we can further take the correspding derived $(p,I)$-completion to achieve the corresponding derived completed version:
\begin{align}
\mathrm{Kan}_\mathrm{Left}\mathrm{deRham}^\text{degreenumber}_{(B_0,N_0)/(A_0,M_0),\mathrm{topo},Z}:=\mathrm{Kan}_\mathrm{Left}\mathrm{deRham}^\text{degreenumber}_{(B_0,N_0)/(A_0,M_0),\mathrm{topo}}\widehat{\otimes}^\mathbb{L}_{\mathbb{Z}_p}Z,\\
\mathrm{Fil}^*_{\mathrm{Kan}_\mathrm{Left}\mathrm{deRham}^\text{degreenumber}_{(B_0,N_0)/(A_0,M_0),\mathrm{topo}},Z}:=\mathrm{Fil}^*_{\mathrm{Kan}_\mathrm{Left}\mathrm{deRham}^\text{degreenumber}_{(B_0,N_0)/(A_0,M_0),\mathrm{topo}}}\widehat{\otimes}^\mathbb{L}_{\mathbb{Z}_p}Z.	
\end{align}
Then in the situation that all the rings are classical logarithmic adic rings we consider the following construction for the map $(A,M)\rightarrow (B,N)$ by putting:
\begin{align}
&\mathbb{L}_{(B_0,N_0)/(A_0,M_0),\mathrm{topo},Z}:= \mathrm{Colim}_{(A_0,M_0)\rightarrow (B_0,N_0)}\mathbb{L}_{(B_0,N_0)/(A_0,M_0),\mathrm{topo},Z}[1/(p,I)],\\
&H_{\text{degreenumber},{\mathrm{AQ}},\mathrm{topo},Z}:=\pi_\text{degreenumber} (\mathbb{L}_{(B,N)/(A,M),\mathrm{topo},Z}),	\\
&\mathrm{Kan}_\mathrm{Left}\mathrm{deRham}^\text{degreenumber}_{(B,N)/(A,M),\mathrm{topo},Z}:=\mathrm{Colim}_{(A_0,M_0)\rightarrow (B_0,N_0)}\mathrm{Kan}_\mathrm{Left}\mathrm{deRham}^\text{degreenumber}_{(B,N)/(A_0,M_0),\mathrm{topo},Z}[1/(p,I)],\\
&\mathrm{Fil}^*_{\mathrm{Kan}_\mathrm{Left}\mathrm{deRham}^\text{degreenumber}_{(B,N)/(A,M),\mathrm{topo}},Z}:=\mathrm{Colim}_{(A_0,M_0)\rightarrow (B_0,N_0)}\mathrm{Fil}^*_{\mathrm{Kan}_\mathrm{Left}\mathrm{deRham}^\text{degreenumber}_{(B_0,N_0)/(A_0,M_0),\mathrm{topo}},Z}[1/(p,I)].
\end{align}
Then we need to take the corresponding Hodge-Filtered completion by using the corresponding filtration associated as above to achieve the corresponding Hodge-complete objects in the corresponding filtered $\infty$-categories:
\begin{align}
{\mathrm{Kan}_\mathrm{Left}\mathrm{deRham}}^\text{degreenumber}_{(B,N)/(A,M),\mathrm{topo,Hodge},Z},\mathrm{Fil}^*_{{\mathrm{Kan}_\mathrm{Left}\mathrm{deRham}}^\text{degreenumber}_{(B,N)/(A,M),\mathrm{topo,Hodge}},Z}.	
\end{align} 	
\end{definition}

\newpage

\section{Robba Sheaves and Frobenius Sheaves}

\subsection{Pseudorigid Relative Toric Tower}

\indent In this section we discuss the corresponding Robba sheaves of \cite{12KL1} and \cite{12KL2} over general spaces. Here we consider the corresponding discussion for the pseudorigid situation which is certainly following the corresponding treatment in \cite[Chapter 8]{12KL2}. But we do not really understand if the corresponding generality could be achieved at this moment. So one definitely has to be very careful in the corresponding analysis \footnote{Pseudorigid spaces are actually along our generalization in our mind, which are the first kinds of spaces we would like to study in our project before considering more general topological or functional analytic spaces, namely those general spaces whose structure sheaves of simplicial rings carry topologies or norms such as in \cite{12BBBK} and \cite{12CS2}.}. These kinds of spaces are actually different from the rigid analytic space, although the theories are really related to each other, such as in \cite{12Bel1}, \cite{12Bel2} and \cite{12L}. Following the ideas in the rigid situation in \cite{12KL2}, we consider the towers in the smooth situation in the following:

\begin{setting}
We consider the following towers as in \cite[Chapter 5]{12KL2}. First we let $A_0$ be the ring $\mathcal{O}_K[[u]]\left<\frac{\pi^a}{u^b}\right>[1/u]\left<T_1^\pm,...,T_d^\pm\right>$ and we put $A_0^+$ to be the ring\\ $\mathcal{O}_K[[u]]\left<\frac{\pi^a}{u^b}\right>.\left<T_1^\pm,...,T_d^\pm\right>$. And for the higher level we have the following rings:
\begin{align}
&(A_n,A_n^+)\\
&:=(\mathcal{O}_K(\pi^{1/p^n})[[u]]\left<\frac{\pi^a}{u^b}\right>[1/u](u^{\pm 1/p^n})\left<T_1^{\pm 1/p^n},...,T_d^{\pm 1/p^n}\right>,\\
&\mathcal{O}_K(\pi^{1/p^n})[[u]]\left<\frac{\pi^a}{u^b}\right>[u^{1/p^n}]\left<T_1^{\pm 1/p^n},...,T_d^{\pm 1/p^n}\right>),
\end{align}
which implies that we have:
\[
\xymatrix@C+0pc@R+0pc{
A_0\ar[r] \ar[r] \ar[r] &A_1 \ar[r] \ar[r] \ar[r] &... \ar[r] \ar[r] \ar[r] &A_n \ar[r] \ar[r] \ar[r] &..., \forall n\geq 0. 
}
\]
\end{setting}

\begin{proposition}
The tower above is finite \'etale.	
\end{proposition}

\begin{proof}
Note that the corresponding pseudorigid space is actually formally of finite type over $\mathcal{O}_K$ namely we have that the ring at the zeroth level is just
\begin{center}
 $\mathcal{O}_K[[u,S_1,...,S_d]][1/u]\left<V_1^{\pm},...,V_e^{\pm}\right>$. 
\end{center} 
Therefore the corresponding tower under this sort of presentation will give:
\begin{align}
&(A_n,A_n^+)\\
&:=(\mathcal{O}_K[\pi^{1/p^n}][[u^{1/p^n},S_1^{1/p^n},...,S_d^{1/p^n}]][1/u^{1/p^n}]\left<V_1^{\pm 1/p^n},...,V_e^{\pm 1/p^n}\right>,\\
&\mathcal{O}_K[\pi^{1/p^n}][[u^{1/p^n},S_1^{1/p^n},...,S_d^{1/p^n}]]\left<V_1^{\pm 1/p^n},...,V_e^{\pm 1/p^n}\right>).
\end{align}	
This is certainly a finite \'etale tower.
\end{proof}

\begin{proposition}
The tower above is perfectoid.	
\end{proposition}

\begin{proof}
Note that the corresponding pseudorigid space is actually formally of finite type over $\mathcal{O}_K$ namely we have that the ring at the zeroth level is just
\begin{center}
 $\mathcal{O}_K[[u,S_1,...,S_d]][1/u]\left<V_1^{\pm},...,V_e^{\pm}\right>$.	
\end{center}
Therefore the corresponding tower under this sort of presentation will give:
\begin{align}
&(A_n,A_n^+)\\
&:=(\mathcal{O}_K[\pi^{1/p^n}][[u^{1/p^n},S_1^{1/p^n},...,S_d^{1/p^n}]][1/u^{1/p^n}]\left<V_1^{\pm 1/p^n},...,V_e^{\pm 1/p^n}\right>,\\
&\mathcal{O}_K[\pi^{1/p^n}][[u^{1/p^n},S_1^{1/p^n},...,S_d^{1/p^n}]]\left<V_1^{\pm 1/p^n},...,V_e^{\pm 1/p^n}\right>).
\end{align}	
The corresponding $\infty$-level is just:
\begin{align}
&A_\infty\\
&:=\mathcal{O}_K[\pi^{1/p^\infty}][[u^{1/p^\infty},S_1^{1/p^\infty},...,S_d^{1/p^\infty}]][1/u^{1/p^\infty}]\left<V_1^{\pm 1/p^\infty},...,V_e^{\pm 1/p^\infty}\right>^\wedge,\\
&A_\infty^+\\
&:=\mathcal{O}_K[\pi^{1/p^\infty}][[u^{1/p^\infty},S_1^{1/p^\infty},...,S_d^{1/p^\infty}]]\left<V_1^{\pm 1/p^\infty},...,V_e^{\pm 1/p^\infty}\right>^\wedge.
\end{align}
These are certainly the corresponding perfectoid rings in the sense of \cite[Definition 3.3.1]{12KL2}, also see \cite[Lemma 3.3.28]{12KL2}\footnote{As in \cite[Lemma 3.3.28]{12KL2} one takes the original pseudouniformizer, and goes to a suitably high level to achieve some root $u'$ of the function $X^{p^n}-Xu-u$, and evetually the corresponding element $u'^p$ will divide $p$.}.
\end{proof}

\begin{remark}
Here the corresponding rings are regarded as the corresponding topological rings instead of Banach rings, but we can certainly consider the corresponding Banach ring structure induced from the linear topology namely really the $\infty$-level of this tower is Fontaine perfectoid adic Banach ring which is Tate in the sense of \cite[Definition 3.1.1]{12KL2}.	
\end{remark}

\begin{proposition}
The tower is weakly decompleting.	
\end{proposition}

\begin{proof}
We give the proof where $K=\mathbb{Q}_p$, which is certainly carried over to more general situation. The corresponding tower is actually weakly decompleting once one considers the corresponding presentation as above for the corresponding ring of definition:
\begin{align}
&A_\infty\\
&:=\mathcal{O}_K[\pi^{1/p^\infty}][[u^{1/p^\infty},S_1^{1/p^\infty},...,S_d^{1/p^\infty}]][1/u^{1/p^\infty}]\left<V_1^{\pm 1/p^\infty},...,V_e^{\pm 1/p^\infty}\right>^\wedge,\\
&A_\infty^+\\
&:=\mathcal{O}_K[\pi^{1/p^\infty}][[u^{1/p^\infty},S_1^{1/p^\infty},...,S_d^{1/p^\infty}]]\left<V_1^{\pm 1/p^\infty},...,V_e^{\pm 1/p^\infty}\right>^\wedge.
\end{align}
This transfers to the positive characteristic situation under tilting:
\begin{align}
&R_\infty\\
&:={k}_K[\overline{\pi}^{1/p^\infty}][[\overline{u}^{1/p^\infty},\overline{S}_1^{1/p^\infty},...,\overline{S}_d^{1/p^\infty}]][1/\overline{u}^{1/p^\infty}]\left<\overline{V}_1^{\pm 1/p^\infty},...,\overline{V}_e^{\pm 1/p^\infty}\right>^\wedge,\\
&R_\infty^+\\
&:={k}_K[\overline{\pi}^{1/p^\infty}][[\overline{u}^{1/p^\infty},\overline{S}_1^{1/p^\infty},...,\overline{S}_d^{1/p^\infty}]]\left<\overline{V}_1^{\pm 1/p^\infty},...,\overline{V}_e^{\pm 1/p^\infty}\right>^\wedge.
\end{align}	
Then one can finish as in \cite[Lemma 7.1.2]{12KL2} by comparing this with $R_{(H_\text{degreenumber},H_\text{degreenumber}^+)}$.
\end{proof}

\newpage
\subsection{Pseudorigid Logarithmic Relative Toric Tower}

\begin{setting}
We consider the following towers as in \cite[Chapter 5]{12KL2}. First we let $A_0$ be the ring $\mathcal{O}_K[[u]]\left<\frac{\pi^a}{u^b}\right>[1/u]\left<T_1^\pm,...,T_d^\pm,T^-_{d+1},...,T^-_f\right>$ and we put $A_0^+$ to be the ring 
\begin{align}
\mathcal{O}_K[[u]]\left<\frac{\pi^a}{u^b}\right>.\left<T_1^\pm,...,T_d^\pm,T^-_{d+1},...,T^-_f\right>.	
\end{align}
And for the higher level we have the following rings:
\begin{align}
&(A_n,A_n^+)\\
&:=(\mathcal{O}_K(\pi^{1/p^n})[[u]]\left<\frac{\pi^a}{u^b}\right>[1/u](u^{\pm 1/p^n})\left<T_1^{\pm 1/p^n},...,T_d^{\pm 1/p^n},T^{-1/p^n}_{d+1},...,T^{-1/p^n}_f\right>,\\
&\mathcal{O}_K(\pi^{1/p^n})[[u]]\left<\frac{\pi^a}{u^b}\right>[u^{1/p^n}]\left<T_1^{\pm 1/p^n},...,T_d^{\pm 1/p^n},T^{-1/p^n}_{d+1},...,T^{-1/p^n}_f\right>),
\end{align}
which implies that we have:
\[
\xymatrix@C+0pc@R+0pc{
A_0\ar[r] \ar[r] \ar[r] &A_1 \ar[r] \ar[r] \ar[r] &... \ar[r] \ar[r] \ar[r] &A_n \ar[r] \ar[r] \ar[r] &..., \forall n\geq 0. 
}
\]
\end{setting}

\begin{proposition}
The tower above is finite \'etale.	
\end{proposition}

\begin{proof}
Note that the corresponding pseudorigid space is actually formally of finite type over $\mathcal{O}_K$ namely we have that the ring at the zeroth level is just
\begin{center}
 $\mathcal{O}_K[[u,S_1,...,S_d]][1/u]\left<V_1^{\pm},...,V_e^{\pm}\right>$. 
\end{center}
Therefore the corresponding tower under this sort of presentation will give:
\begin{align}
&(A_n,A_n^+)\\
&:=(\mathcal{O}_K[\pi^{1/p^n}][[u^{1/p^n},S_1^{1/p^n},...,S_d^{1/p^n}]][1/u^{1/p^n}]\left<V_1^{\pm 1/p^n},...,V_e^{\pm 1/p^n},T^{-1/p^n}_{e+1},...,T^{-1/p^n}_f\right>,\\
&\mathcal{O}_K[\pi^{1/p^n}][[u^{1/p^n},S_1^{1/p^n},...,S_d^{1/p^n}]]\left<V_1^{\pm 1/p^n},...,V_e^{\pm 1/p^n},T^{-1/p^n}_{d+1},...,T^{-1/p^n}_f\right>).
\end{align}	
This is certainly a finite \'etale tower.
\end{proof}

\begin{proposition}
The tower above is perfectoid.	
\end{proposition}

\begin{proof}
Note that the corresponding pseudorigid space is actually formally of finite type over $\mathcal{O}_K$ namely we have that the ring at the zeroth level is just
\begin{center}
 $\mathcal{O}_K[[u,S_1,...,S_d]][1/u]\left<V_1^{\pm},...,V_e^{\pm}\right>$. 
\end{center} 
Therefore the corresponding tower under this sort of presentation will give:
\begin{align}
&(A_n,A_n^+)\\
&:=(\mathcal{O}_K[\pi^{1/p^n}][[u^{1/p^n},S_1^{1/p^n},...,S_d^{1/p^n}]][1/u^{1/p^n}]\left<V_1^{\pm 1/p^n},...,V_e^{\pm 1/p^n},T^{-1/p^n}_{e+1},...,T^{-1/p^n}_f\right>,\\
&\mathcal{O}_K[\pi^{1/p^n}][[u^{1/p^n},S_1^{1/p^n},...,S_d^{1/p^n}]]\left<V_1^{\pm 1/p^n},...,V_e^{\pm 1/p^n},T^{-1/p^n}_{e+1},...,T^{-1/p^n}_f\right>).
\end{align}	
The corresponding $\infty$-level is just:
\begin{align}
&A_\infty:=\\
&\mathcal{O}_K[\pi^{1/p^\infty}][[u^{1/p^\infty},S_1^{1/p^\infty},...,S_d^{1/p^\infty}]][1/u^{1/p^\infty}]\left<V_1^{\pm 1/p^\infty},...,V_e^{\pm 1/p^\infty},T^{-1/p^\infty}_{e+1},...,T^{-1/p^\infty}_f\right>^\wedge,\\
&A_\infty^+\\
&:=\mathcal{O}_K[\pi^{1/p^\infty}][[u^{1/p^\infty},S_1^{1/p^\infty},...,S_d^{1/p^\infty}]]\left<V_1^{\pm 1/p^\infty},...,V_e^{\pm 1/p^\infty},T^{-1/p^\infty}_{e+1},...,T^{-1/p^\infty}_f\right>^\wedge.
\end{align}
These are certainly the corresponding perfectoid rings in the sense of \cite[Definition 3.3.1]{12KL2}, also see \cite[Lemma 3.3.28]{12KL2}\footnote{As in \cite[Lemma 3.3.28]{12KL2} one takes the original pseudouniformizer, and goes to a suitably high level to achieve some root $u'$ of the function $X^{p^n}-Xu-u$, and evetually the corresponding element $u'^p$ will divide $p$.}.  
\end{proof}

\begin{remark}
Here the corresponding rings are regarded as the corresponding topological rings instead of Banach rings, but we can certainly consider the corresponding Banach ring structure induced from the linear topology namely really the $\infty$-level of this tower is Fontaine perfectoid adic Banach ring which is Tate in the sense of \cite[Definition 3.1.1]{12KL2}.	
\end{remark}

\begin{proposition}
The tower is weakly decompleting.	
\end{proposition}

\begin{proof}
We give the proof where $K=\mathbb{Q}_p$, which is certainly carried over to more general situation. The corresponding tower is actually weakly decompleting once one considers the corresponding presentation as above for the corresponding ring of definition:
\begin{align}
&A_\infty:=\\
&\mathcal{O}_K[\pi^{1/p^\infty}][[u^{1/p^\infty},S_1^{1/p^\infty},...,S_d^{1/p^\infty}]][1/u^{1/p^\infty}]\left<V_1^{\pm 1/p^\infty},...,V_e^{\pm 1/p^\infty},T^{-1/p^\infty}_{e+1},...,T^{-1/p^\infty}_f\right>^\wedge,\\
&A_\infty^+\\
&:=\mathcal{O}_K[\pi^{1/p^\infty}][[u^{1/p^\infty},S_1^{1/p^\infty},...,S_d^{1/p^\infty}]]\left<V_1^{\pm 1/p^\infty},...,V_e^{\pm 1/p^\infty},T^{-1/p^\infty}_{e+1},...,T^{-1/p^\infty}_f\right>^\wedge.
\end{align}
This transfers to the positive characteristic situation under tilting:
\begin{align}
&R_\infty:=\\
&{k}_K[\overline{\pi}^{1/p^\infty}][[\overline{u}^{1/p^\infty},\overline{S}_1^{1/p^\infty},...,\overline{S}_d^{1/p^\infty}]][1/\overline{u}^{1/p^\infty}]\left<\overline{V}_1^{\pm 1/p^\infty},...,\overline{V}_e^{\pm 1/p^\infty},T^{-1/p^\infty}_{e+1},...,T^{-1/p^\infty}_f\right>^\wedge,\\
&R_\infty^+\\
&:={k}_K[\overline{\pi}^{1/p^\infty}][[\overline{u}^{1/p^\infty},\overline{S}_1^{1/p^\infty},...,\overline{S}_d^{1/p^\infty}]]\left<\overline{V}_1^{\pm 1/p^\infty},...,\overline{V}_e^{\pm 1/p^\infty},T^{-1/p^\infty}_{e+1},...,T^{-1/p^\infty}_f\right>^\wedge.
\end{align}	
Then one can finish as in \cite[Lemma 7.1.2]{12KL2} by comparing this with $R_{(H_\text{degreenumber},H_\text{degreenumber}^+)}$.
\end{proof}

\newpage
\subsection{Robba Sheaves over Pseudorigid Spaces}

\indent Now we consider Kedlaya-Liu's Robba sheaves, which are also some motivic and functorial construction. Our space is actually regarded over $\mathbb{Z}_p$, but note that the corresponding pseudorigid affinoid could be defined over arbitrary integral ring $\mathcal{O}_K$ of some analytic field $K$. Let $X$ be a pseudorigid space over $\mathcal{O}_K$, but we regard this as a corresponding Tate adic space over $\mathbb{Z}_p$. Now we consider the corresponding pro-\'etale site of $X$ which we denote it as $X_\text{pro\'et}$.


\begin{definition}
We now apply the corresponding definitions of the Robba rings in \cite[Definition 4.1.1]{12KL2} with $E$ therein being just $\mathbb{Q}_p$ to any perfectoid subdomain $(P_\infty,P^+_\infty)$ of $X_\text{pro\'et}$. For each such perfectoid, we promote it to be a Banach ring $(P_\infty,P^+_\infty,\|.\|_{\infty})$, then we have the corresponding construction of the Robba rings in \cite[Definition 4.1.1]{12KL2} (we use the notation $\Pi$ to represent the notation $\mathcal{R}$):
\begin{align}
\widetilde{\Pi}^{[s,r]}_{(P_\infty,P^+_\infty,\|.\|_{\infty})},\\
\varprojlim_s \widetilde{\Pi}^{[s,r]}_{(P_\infty,P^+_\infty,\|.\|_{\infty})},\\
\varinjlim_r \varprojlim_s \widetilde{\Pi}^{[s,r]}_{(P_\infty,P^+_\infty,\|.\|_{\infty})}.	
\end{align}
Then one just organizes these to be certain presheaves over the site $X_\text{pro\'et}$:
\begin{align}
\widetilde{\Pi}^{[s,r]}_{X,\text{pro\'et}},\\
\varprojlim_s \widetilde{\Pi}^{[s,r]}_{X,\text{pro\'et}},\\
\varinjlim_r \varprojlim_s \widetilde{\Pi}^{[s,r]}_{X,\text{pro\'et}}.	
\end{align}	
However the corresponding construction is not canonical since the promotion to Banach rings locally is not functorial. But we do have the situation that they are actually sheaves due to the fact that we can regard them as sheaves over some preperfectoid spaces under the corresponding tilting of the total spaces, where we use the same notation to denote the sheaves.
\begin{align}
\widetilde{\Pi}^{[s,r]}_{X,\text{pro\'et}},\\
\varprojlim_s \widetilde{\Pi}^{[s,r]}_{X,\text{pro\'et}},\\
\varinjlim_r \varprojlim_s \widetilde{\Pi}^{[s,r]}_{X,\text{pro\'et}}.	
\end{align}	
The corresponding total spaces are defined as:
\begin{align}
\mathrm{Spectrumadic}_{\mathrm{total}}(\widetilde{\Pi}^{[s,r]}_{X,\text{pro\'et}},\widetilde{\Pi}^{[s,r],+}_{X,\text{pro\'et}}),\\
\mathrm{Spectrumadic}_{\mathrm{total}}(\varprojlim_s \widetilde{\Pi}^{[s,r]}_{X,\text{pro\'et}},\varprojlim_s \widetilde{\Pi}^{[s,r],+}_{X,\text{pro\'et}}),\\
\mathrm{Spectrumadic}_{\mathrm{total}}(\varinjlim_r \varprojlim_s \widetilde{\Pi}^{[s,r]}_{X,\text{pro\'et}},\varinjlim_r \varprojlim_s \widetilde{\Pi}^{[s,r],+}_{X,\text{pro\'et}}).	
\end{align}	
\end{definition}

\begin{definition}
As in \cite[Chapter 4.3 and Chapter 8]{12KL2}, consider the corresponding total space:
\begin{align}
\mathrm{Spectrumadic}_{\mathrm{total}}(\widetilde{\Pi}^{[s,r]}_{X,\text{pro\'et}},\widetilde{\Pi}^{[s,r],+}_{X,\text{pro\'et}}).	
\end{align}
We define the corresponding $\varphi$-sheaves (where $\varphi$ is Frobenius lifting from $p$-th power Frobenius coming from the characteristic $p$ rings encoded in the construction recalled above from \cite{12KL2}) to be sheaves locally attached to \'etale-stably pseudocoherent sheaves carrying the corresponding semilinear Frobenius morphisms realizing the isomorphisms under pullback. Then over 

\begin{align}
\mathrm{Spectrumadic}_{\mathrm{total}}(\varprojlim_s \widetilde{\Pi}^{[s,r]}_{X,\text{pro\'et}},\varprojlim_s \widetilde{\Pi}^{[s,r],+}_{X,\text{pro\'et}})	
\end{align}
we define the similar pseudocoherent sheaves (complete with respect to the natural topology) locally base change to some sheaves in the previous kind. Finally we define the similar pseudocoherent sheaves (complete with respect to the natural topology) over 
\begin{align}
\mathrm{Spectrumadic}_{\mathrm{total}}(\varinjlim_r \varprojlim_s \widetilde{\Pi}^{[s,r]}_{X,\text{pro\'et}},\varinjlim_r \varprojlim_s \widetilde{\Pi}^{[s,r],+}_{X,\text{pro\'et}}),
\end{align}
locally base change to some sheaves in the previous kind\footnote{Here it is certainly not expect to be the case where we could have uniform radius $r>0$.}.

\end{definition}

\newpage
\subsection{Robba Sheaves over $(p,I)$-Adic Spaces}

\indent Now we consider Kedlaya-Liu's Robba sheaves over more general adic spaces, which is also some motivic and functorial construction. Let $X$ be a Tate adic space over $\mathbb{Z}_p$ where $p$ is assumed to be topologically nilpotent. Now we consider the corresponding pro-\'etale site of $X$ which we denote it as $X_\text{pro\'et}$.

\begin{definition}
We now apply the corresponding definitions of the Robba rings in \cite[Definition 4.1.1]{12KL2} with $E$ therein being just $\mathbb{Q}_p$ to any perfectoid subdomain $(P_\infty,P^+_\infty)$ of $X_\text{pro\'et}$. For each such perfectoid, we promote it to be a Banach ring $(P_\infty,P^+_\infty,\|.\|_{\infty})$, then we have the corresponding construction of the Robba rings in \cite[Definition 4.1.1]{12KL2} (we use the notation $\Pi$ to represent the notation $\mathcal{R}$):
\begin{align}
\widetilde{\Pi}^{[s,r]}_{(P_\infty,P^+_\infty,\|.\|_{\infty})},\\
\varprojlim_s \widetilde{\Pi}^{[s,r]}_{(P_\infty,P^+_\infty,\|.\|_{\infty})},\\
\varinjlim_r \varprojlim_s \widetilde{\Pi}^{[s,r]}_{(P_\infty,P^+_\infty,\|.\|_{\infty})}.	
\end{align}
Then one just organizes these to be certain presheaves over the site $X_\text{pro\'et}$:
\begin{align}
\widetilde{\Pi}^{[s,r]}_{X,\text{pro\'et}},\\
\varprojlim_s \widetilde{\Pi}^{[s,r]}_{X,\text{pro\'et}},\\
\varinjlim_r \varprojlim_s \widetilde{\Pi}^{[s,r]}_{X,\text{pro\'et}}.	
\end{align}	
However the corresponding construction is not canonical since the promotion to Banach rings locally is not functorial. But we do have the situation that they are actually sheaves due to the fact that we can regard them as sheaves over some preperfectoid spaces under the corresponding tilting of the total spaces, where we use the same notation to denote the sheaves.
\begin{align}
\widetilde{\Pi}^{[s,r]}_{X,\text{pro\'et}},\\
\varprojlim_s \widetilde{\Pi}^{[s,r]}_{X,\text{pro\'et}},\\
\varinjlim_r \varprojlim_s \widetilde{\Pi}^{[s,r]}_{X,\text{pro\'et}}.	
\end{align}	
The corresponding total spaces are defined as:
\begin{align}
\mathrm{Spectrumadic}_{\mathrm{total}}(\widetilde{\Pi}^{[s,r]}_{X,\text{pro\'et}},\widetilde{\Pi}^{[s,r],+}_{X,\text{pro\'et}}),\\
\mathrm{Spectrumadic}_{\mathrm{total}}(\varprojlim_s \widetilde{\Pi}^{[s,r]}_{X,\text{pro\'et}},\varprojlim_s \widetilde{\Pi}^{[s,r],+}_{X,\text{pro\'et}}),\\
\mathrm{Spectrumadic}_{\mathrm{total}}(\varinjlim_r \varprojlim_s \widetilde{\Pi}^{[s,r]}_{X,\text{pro\'et}},\varinjlim_r \varprojlim_s \widetilde{\Pi}^{[s,r],+}_{X,\text{pro\'et}}).	
\end{align}	
\end{definition}

\begin{definition}
As in \cite[Chapter 4.3 and Chapter 8]{12KL2}, consider the corresponding total space:
\begin{align}
\mathrm{Spectrumadic}_{\mathrm{total}}(\widetilde{\Pi}^{[s,r]}_{X,\text{pro\'et}},\widetilde{\Pi}^{[s,r],+}_{X,\text{pro\'et}}).	
\end{align}
We define the corresponding $\varphi$-sheaves (where $\varphi$ is Frobenius lifting from $p$-th power Frobenius coming from the characteristic $p$ rings encoded in the construction recalled above from \cite{12KL2}) to be sheaves locally attached to \'etale-stably pseudocoherent sheaves carrying the corresponding semilinear Frobenius morphism realizing the isomorphism under pullback. Then over 

\begin{align}
\mathrm{Spectrumadic}_{\mathrm{total}}(\varprojlim_s \widetilde{\Pi}^{[s,r]}_{X,\text{pro\'et}},\varprojlim_s \widetilde{\Pi}^{[s,r],+}_{X,\text{pro\'et}})	
\end{align}
we define the similar pseudocoherent sheaves (complete with respect to the natural topology) locally base change to some sheaves in the previous kind. Finally we define the similar pseudocoherent sheaves (complete with respect to the natural topology) over 
\begin{align}
\mathrm{Spectrumadic}_{\mathrm{total}}(\varinjlim_r \varprojlim_s \widetilde{\Pi}^{[s,r]}_{X,\text{pro\'et}},\varinjlim_r \varprojlim_s \widetilde{\Pi}^{[s,r],+}_{X,\text{pro\'et}}),
\end{align}
locally base change to some sheaves in the previous kind\footnote{Here it is certainly not expected to be the case where we could have uniform radius $r>0$.}.

\end{definition}

\begin{remark}
The corresponding sheaves of the full Robba rings are well-defined since one could regard them as structure sheaves of some total spaces.	
\end{remark}

\newpage

\section{Derived $I$-Complete THH and HH of Derived $I$-Complete Rings}

\subsection{Derived $I$-Complete Objects}

\indent The corresponding THH and HH of derived $I$-complete rings are topological constructions which are very closely related to the corresponding relative $p$-adic motives. And certain derived $I$-complete versions are also relevant in some highly nontrivial way. We make some discussion closely after \cite[Section 2.2, Section 2.3]{12BMS}, \cite{12BS} and \cite[Chapter 3]{12NS}. We now consider the following rings:

\begin{setting}
We consider the derived $I$-complete $\mathbb{E}_1$-rings. For instance one can consider some derived $I$-complete $\mathbb{E}_1$-rings relative to the corresponding integral pseudorigid disc. In the following $I$ will be some two sided ideal in the $\pi_0$ of the base spectrum $R$. We regard all the ring spectra as being in the $\infty$-category of all ring spectra over $R$ and in the derived $\infty$-category $\mathbb{D}(R)$ of all the module spectra over $R$. We regard all the derived $I$-complete ring spectra as being in the $\infty$-category of all derived $I$-complete ring spectra over $R$ and in the derived $\infty$-category $\mathbb{D}_I(R)$ of all the derived $I$-complete module spectra over $R$. We assume that $I$ is central in the noncommutative setting. 	
\end{setting}

\begin{definition}
\indent For any such ring $R$ which is assumed to be $\mathbb{E}_\infty$, we will use the corresponding notations $L\mathrm{THH}(R)$ and $\mathrm{LHH}(R)$ to denote the corresponding left Kan extended THH or the corresponding HH functor from \cite[Chapter 3]{12NS}. And we will use the corresponding notations $L\mathrm{THH}(R)_I$ and $\mathrm{LHH}(R)_I$ to denote the corresponding left Kan extended THH or the corresponding HH functor after taking the corresponding derived $I$-completion as in \cite[Chapter 3]{12NS}.	
\end{definition}

\begin{definition}
\indent For any such ring $R$ which is assumed to be $\mathbb{E}_1$, we will use the corresponding notations $\mathrm{THH}(R)$ and $\mathrm{HH}(R)$ to denote the corresponding THH or the corresponding HH functor from \cite[Chapter 3]{12NS}. And we will use the corresponding notations $\mathrm{THH}(R)_I$ and $\mathrm{HH}(R)_I$ to denote the corresponding THH or the corresponding HH functor after taking the corresponding derived $I$-completion as in \cite[Chapter 3]{12NS} and \cite[Chapter 1 Notation]{12BS}.	
\end{definition}

\begin{example}
For instance one considers the adic rings in \cite[Section 1.4]{12FK}, one takes such a ring $R$ which is complete with respect to a corresponding two sided finitely generated ideal $I$, then the corresponding construction above namely $L\mathrm{THH}(R)_I$ and $L\mathrm{HH}(R)_I$ will also be basically topological versions of the spectra in \cite[Definition 1.3, Definition 1.4]{12ELS} in the corresponding derived nonocommutative deformation theory.
\end{example}

\indent As in \cite[Chapter 5]{12ELS} and more generally, one has the corresponding noncommutative version of the corresponding relative K\"ahler differential complexes $\mathrm{deRham}^\text{degreenumber}_{\mathrm{noncommutative}, B/A}$. For instance if we have the corresponding map $A\rightarrow B$ of adic rings in \cite[Section 1.4]{12FK} we could then take the corresponding derived $I$-completion from $\mathrm{deRham}^\text{degreenumber}_{\mathrm{noncommutative}, B/A}$ to achieve the corresponding topological version $\mathrm{deRham}^\text{degreenumber}_{\mathrm{noncommutative}, B/A,\mathrm{topo}}$.

\begin{remark}
As mentioned in the introduction, the corresponding topological derived $I$-adic version of the corresponding $L$THH and $L$HH spectra should be interesting to study, as over a quasisyntomic site of a quasisyntomic ring $Q$, \cite[Proposition 5.15]{12BMS} shows that the corresponding sheaf $\pi_0\mathrm{HC}^-(-/Q)_p$ will be closely related to the corresponding $p$-adic complete and Hodge-complete derived de Rham sheaf over the same site after taking the corresponding unfolding through the quasiregular semiperfectoids. We do not know if this relation could hold in some sense if we consider derived $I$-adic rings, but one might want to believe that in the situations we considered before the story should be in some sense easier to be established.	
\end{remark}

\newpage

\chapter{Functional Analytic Theory}

\section{Functional Analytic Andr\'e-Quillen Homology and Topological Derived de Rham Complexes}

\subsection{The Construction from Kelly-Kremnizer-Mukherjee}

\indent We now work in the corresponding foundations from \cite{12BBBK}, \cite{12BBK}, \cite{BBM}, \cite{12BK} and \cite{KKM}, namely what we are going to consider will be literally the following $(\infty,1)$-categories and the associated constructions with some fixed Banach ring $R$ or $\mathbb{F}_1$\footnote{At this moment we do not work over more general rings such as $\mathbb{Z}$, but this is crucial in the corresponding globalization. That being said, working over $\mathbb{F}_1$ is in some sense more general even than these situations.}:
\begin{align}
\mathrm{Object}_{\mathrm{E}_\infty\mathrm{commutativealgebra},\mathrm{Simplicial}}(\mathrm{IndSNorm}_R),\\
\mathrm{Object}_{\mathrm{E}_\infty\mathrm{commutativealgebra},\mathrm{Simplicial}}(\mathrm{Ind}^m\mathrm{SNorm}_R),\\
\mathrm{Object}_{\mathrm{E}_\infty\mathrm{commutativealgebra},\mathrm{Simplicial}}(\mathrm{IndNorm}_R),\\
\mathrm{Object}_{\mathrm{E}_\infty\mathrm{commutativealgebra},\mathrm{Simplicial}}(\mathrm{Ind}^m\mathrm{Norm}_R),\\
\mathrm{Object}_{\mathrm{E}_\infty\mathrm{commutativealgebra},\mathrm{Simplicial}}(\mathrm{IndBan}_R),\\
\mathrm{Object}_{\mathrm{E}_\infty\mathrm{commutativealgebra},\mathrm{Simplicial}}(\mathrm{Ind}^m\mathrm{Ban}_R),
\end{align}
with
\begin{align}
\mathrm{Object}_{\mathrm{E}_\infty\mathrm{commutativealgebra},\mathrm{Simplicial}}(\mathrm{IndSNorm}_{\mathbb{F}_1}),\\
\mathrm{Object}_{\mathrm{E}_\infty\mathrm{commutativealgebra},\mathrm{Simplicial}}(\mathrm{Ind}^m\mathrm{SNorm}_{\mathbb{F}_1}),\\
\mathrm{Object}_{\mathrm{E}_\infty\mathrm{commutativealgebra},\mathrm{Simplicial}}(\mathrm{IndNorm}_{\mathbb{F}_1}),\\
\mathrm{Object}_{\mathrm{E}_\infty\mathrm{commutativealgebra},\mathrm{Simplicial}}(\mathrm{Ind}^m\mathrm{Norm}_{\mathbb{F}_1}),\\
\mathrm{Object}_{\mathrm{E}_\infty\mathrm{commutativealgebra},\mathrm{Simplicial}}(\mathrm{IndBan}_{\mathbb{F}_1}),\\
\mathrm{Object}_{\mathrm{E}_\infty\mathrm{commutativealgebra},\mathrm{Simplicial}}(\mathrm{Ind}^m\mathrm{Ban}_{\mathbb{F}_1}).
\end{align}
We then use the notation $\mathcal{X}$ to denote any of these. We now first discuss the corresponding topological version of  Andr\'e-Quillen Homology and the corresponding topological version of  Derived de Rham complex parallel to \cite[Chapitre 3]{12An1}, \cite{12An2}, \cite[Chapter 2, Chapter 8]{12B1}, \cite[Chapter 1]{12Bei}, \cite[Chapter 5]{12G1}, \cite[Chapter 3, Chapter 4]{12GL}, \cite[Chapitre II, Chapitre III]{12Ill1}, \cite[Chapitre VIII]{12Ill2}, \cite[Section 4]{12Qui}. We would like to start from the corresponding context of \cite[Section 5.2.1, Definition 5.5, Proposition 5.6, Definition 5.8, Definition 5.9]{KKM}\footnote{Also see \cite[Definition 4.4.7, Construction 4.4.10, Theorem 5.3.6, Definition 5.2.4, Corollary 5.3.9]{Ra}.}, and represent the construction for the convenience of the readers \footnote{Here we just present the homotopical contexts in Kelly-Kremnizer-Mukherjee for the general $(\infty,1)$-categories as above.}. Namely as in \cite[Section 5.2.1, Definition 5.5, Proposition 5.6]{KKM} we consider the corresponding cotangent complex associated to any pair object
\begin{center}
 $(A,B)\in \mathcal{X}\times \mathcal{X}_A$ 
\end{center} 
is defined to be (as in \cite[Section 5.2.1, Definition 5.5, Proposition 5.6]{KKM}) just:
\begin{align}
\mathbb{L}_{B/A,\textrm{KKM}}:=	\mathrm{deRham}^1_{A[B]^\text{degreenumber}/A,\textrm{KKM}}\otimes_{A[B]^\text{degreenumber}} B.
\end{align}
The corresponding topological Andr\'e-Quillen homologies (as in \cite[Section 5.2.1, Definition 5.5, Proposition 5.6]{KKM}) are defined to be:
\begin{align}
H_{\text{degreenumber},{\mathrm{AQ}},\textrm{KKM}}:=\pi_\text{degreenumber} (\mathbb{L}_{B/A,\textrm{KKM}}). 	
\end{align}
We then as in the definition \cite[Definition 5.8, Definition 5.9, Section 5.2.1]{KKM} have the corresponding different kinds of $(\infty,1)$-rings of de Rham complexes and we denote that by:
\begin{align}
\mathrm{Kan}_\mathrm{Left}\mathrm{deRham}^\text{degreenumber}_{B/A,\mathrm{functionalanalytic},\textrm{KKM}},\mathrm{Kan}_\mathrm{Left}\mathrm{Fil}^*_{\mathrm{deRham}^\text{degreenumber}_{B/A,\mathrm{functionalanalytic},\textrm{KKM}}}.	
\end{align}
Then we need to take the corresponding Hodge-Filtered completion by using the corresponding filtration associated as above:
\begin{align}
\mathrm{Kan}_\mathrm{Left}\widehat{\mathrm{deRham}}^\text{degreenumber}_{B/A,\mathrm{functionalanalytic},\mathrm{KKM}},\mathrm{Kan}_\mathrm{Left}\mathrm{Fil}^*_{\widehat{\mathrm{deRham}}^\text{degreenumber}_{B/A,\mathrm{functionalanalytic},\mathrm{KKM}}}.	
\end{align}

\newpage

\section{Functional Analytic Derived Prismatic Cohomology and Functional Analytic Derived Perfectoidizations}

\subsection{General Constructions}

\indent Lurie's book \cite[Section 5.5.8]{Lu3} illustrates in a very general way the corresponding constructions on how we could extend in a certain $(\infty,n)$-category the corresponding nonabelian derived functor from a smaller class of generators in some discrete sense to a large $(\infty,n)$-categorical closure by considering enough colimits being sifted. Therefore in our situation by using the foundation in \cite{12BBBK}, \cite{12BBK}, \cite{BBM}, \cite{12BK} and \cite{KKM} one can define and extend from Tate series rings, Stein series rings, formal series rings and dagger series rings (as those in \cite[Section 4.2]{BBM}) into very large $(\infty,1)$-categories in the following sense:
\begin{align}
\mathrm{Object}_{\mathrm{E}_\infty\mathrm{commutativealgebra},\mathrm{Simplicial}}(\mathrm{IndSNorm}_R),\\
\mathrm{Object}_{\mathrm{E}_\infty\mathrm{commutativealgebra},\mathrm{Simplicial}}(\mathrm{Ind}^m\mathrm{SNorm}_R),\\
\mathrm{Object}_{\mathrm{E}_\infty\mathrm{commutativealgebra},\mathrm{Simplicial}}(\mathrm{IndNorm}_R),\\
\mathrm{Object}_{\mathrm{E}_\infty\mathrm{commutativealgebra},\mathrm{Simplicial}}(\mathrm{Ind}^m\mathrm{Norm}_R),\\
\mathrm{Object}_{\mathrm{E}_\infty\mathrm{commutativealgebra},\mathrm{Simplicial}}(\mathrm{IndBan}_R),\\
\mathrm{Object}_{\mathrm{E}_\infty\mathrm{commutativealgebra},\mathrm{Simplicial}}(\mathrm{Ind}^m\mathrm{Ban}_R),
\end{align}
with
\begin{align}
\mathrm{Object}_{\mathrm{E}_\infty\mathrm{commutativealgebra},\mathrm{Simplicial}}(\mathrm{IndSNorm}_{\mathbb{F}_1}),\\
\mathrm{Object}_{\mathrm{E}_\infty\mathrm{commutativealgebra},\mathrm{Simplicial}}(\mathrm{Ind}^m\mathrm{SNorm}_{\mathbb{F}_1}),\\
\mathrm{Object}_{\mathrm{E}_\infty\mathrm{commutativealgebra},\mathrm{Simplicial}}(\mathrm{IndNorm}_{\mathbb{F}_1}),\\
\mathrm{Object}_{\mathrm{E}_\infty\mathrm{commutativealgebra},\mathrm{Simplicial}}(\mathrm{Ind}^m\mathrm{Norm}_{\mathbb{F}_1}),\\
\mathrm{Object}_{\mathrm{E}_\infty\mathrm{commutativealgebra},\mathrm{Simplicial}}(\mathrm{IndBan}_{\mathbb{F}_1}),\\
\mathrm{Object}_{\mathrm{E}_\infty\mathrm{commutativealgebra},\mathrm{Simplicial}}(\mathrm{Ind}^m\mathrm{Ban}_{\mathbb{F}_1}).
\end{align}

\begin{definition}
Since we are going to consider Bhatt-Scholze's derived prismatic functors \cite[Construction 7.6]{12BS}, so we take the formal series over a general prism $(A,I)$ (assume this to be bounded) such that $A/I$ is also Banach. Then we consider the corresponding formal series ring over $A/I$, then take the correponsding projection resolution compactly generated closure of them to get the sub $(\infty,1)$-categories of above $(\infty,1)$-categories, which we are going to denote them as:
\begin{align}
\mathrm{Object}_{\mathrm{E}_\infty\mathrm{commutativealgebra},\mathrm{Simplicial}}(\mathrm{IndSNorm}_{A/I})^{\text{smoothformalseriesclosure}},\\
\mathrm{Object}_{\mathrm{E}_\infty\mathrm{commutativealgebra},\mathrm{Simplicial}}(\mathrm{Ind}^m\mathrm{SNorm}_{A/I})^{\text{smoothformalseriesclosure}},\\
\mathrm{Object}_{\mathrm{E}_\infty\mathrm{commutativealgebra},\mathrm{Simplicial}}(\mathrm{IndNorm}_{A/I})^{\text{smoothformalseriesclosure}},\\
\mathrm{Object}_{\mathrm{E}_\infty\mathrm{commutativealgebra},\mathrm{Simplicial}}(\mathrm{Ind}^m\mathrm{Norm}_{A/I})^{\text{smoothformalseriesclosure}},\\
\mathrm{Object}_{\mathrm{E}_\infty\mathrm{commutativealgebra},\mathrm{Simplicial}}(\mathrm{IndBan}_{A/I})^{\text{smoothformalseriesclosure}},\\
\mathrm{Object}_{\mathrm{E}_\infty\mathrm{commutativealgebra},\mathrm{Simplicial}}(\mathrm{Ind}^m\mathrm{Ban}_{A/I})^{\text{smoothformalseriesclosure}},
\end{align}	
as in \cite[Section 4.2]{BBM} on the corresponding analytification defined from extension from formal series rings.\\
\end{definition}

\indent Consequently one applies this idea directly to Bhatt-Scholze's prismatic construction \cite{12BS}, one directly gets the corresponding functional analytic $(\infty,1)$-categorical prismatic cohomologies.

\newpage

\subsection{$\infty$-Categorical Functional Analytic Prismatic Cohomologies and $\infty$-Categorical Functional Analytic Preperfectoidizations}

\begin{definition}
\indent One can actually define the derived prismatic cohomologies through derived topological Hochschild cohomologies, derived topological period cohomologies and derived topological cyclic cohomologies as in \cite[Section 2.2, Section 2.3]{12BMS}, \cite[Theorem 1.13]{12BS}:
\begin{align}
	\mathrm{Kan}_{\mathrm{Left}}\mathrm{THH},\mathrm{Kan}_{\mathrm{Left}}\mathrm{TP},\mathrm{Kan}_{\mathrm{Left}}\mathrm{TC},
\end{align}
on the following $(\infty,1)$-compactly generated closures of the corresponding polynomials\footnote{Definitely, we need to put certain norms over in some relatively canonical way, as in \cite[Section 4.2]{BBM} one can basically consider rigid ones and dagger ones, and so on. We restrict to the \textit{formal} one. This means that we are going to consider the corresponding $p$-adic topology only such as the norm over $A/I\left<T_1,...,T_n\right>$, while rigid analytic situation is usually over a field such as in \cite{G2} over $B_\mathrm{dR}\left<T_1,...,T_n\right>$ for instance. One can then in the same way define the corresponding:
\begin{align}
	\mathrm{Kan}_{\mathrm{Left}}\mathrm{THH},\mathrm{Kan}_{\mathrm{Left}}\mathrm{TP},\mathrm{Kan}_{\mathrm{Left}}\mathrm{TC},
\end{align}
over 
\begin{align}
\mathrm{Object}_{\mathrm{E}_\infty\mathrm{commutativealgebra},\mathrm{Simplicial}}(\mathrm{IndSNorm}_{B_\mathrm{dR}})^{\mathrm{smoothformalseriesclosure}},\\
\mathrm{Object}_{\mathrm{E}_\infty\mathrm{commutativealgebra},\mathrm{Simplicial}}(\mathrm{Ind}^m\mathrm{SNorm}_{B_\mathrm{dR}})^{\mathrm{smoothformalseriesclosure}},\\
\mathrm{Object}_{\mathrm{E}_\infty\mathrm{commutativealgebra},\mathrm{Simplicial}}(\mathrm{IndNorm}_{B_\mathrm{dR}})^{\mathrm{smoothformalseriesclosure}},\\
\mathrm{Object}_{\mathrm{E}_\infty\mathrm{commutativealgebra},\mathrm{Simplicial}}(\mathrm{Ind}^m\mathrm{Norm}_{B_\mathrm{dR}})^{\mathrm{smoothformalseriesclosure}},\\
\mathrm{Object}_{\mathrm{E}_\infty\mathrm{commutativealgebra},\mathrm{Simplicial}}(\mathrm{IndBan}_{B_\mathrm{dR}})^{\mathrm{smoothformalseriesclosure}},\\
\mathrm{Object}_{\mathrm{E}_\infty\mathrm{commutativealgebra},\mathrm{Simplicial}}(\mathrm{Ind}^m\mathrm{Ban}_{B_\mathrm{dR}})^{\mathrm{smoothformalseriesclosure}},
\end{align}
by
\begin{align}
	&\mathrm{Kan}_{\mathrm{Left}}\mathrm{THH}_{\text{functionalanalytic,KKM},\text{BBM,formalanalytification}}(\mathcal{O}):=\\
	&(\underset{i}{\text{homotopycolimit}}_{\text{sifted},\text{derivedcategory}_{\infty}(B_\mathrm{dR}-\text{Module})}\mathrm{Kan}_{\mathrm{Left}}\mathrm{THH}_{\text{functionalanalytic,KKM}}(\mathcal{O}_i))_\text{BBM,formalanalytification},\\
	&\mathrm{Kan}_{\mathrm{Left}}\mathrm{TP}_{\text{functionalanalytic,KKM},\text{BBM,formalanalytification}}(\mathcal{O}):=\\
	&(\underset{i}{\text{homotopycolimit}}_{\text{sifted},\text{derivedcategory}_{\infty}(B_\mathrm{dR}-\text{Module})}\mathrm{Kan}_{\mathrm{Left}}\mathrm{TP}_{\text{functionalanalytic,KKM}}(\mathcal{O}_i))_\text{BBM,formalanalytification},\\
	&\mathrm{Kan}_{\mathrm{Left}}\mathrm{TC}_{\text{functionalanalytic,KKM},\text{BBM,formalanalytification}}(\mathcal{O}):=\\
	&(\underset{i}{\text{homotopycolimit}}_{\text{sifted},\text{derivedcategory}_{\infty}(B_\mathrm{dR}-\text{Module})}\mathrm{Kan}_{\mathrm{Left}}\mathrm{TC}_{\text{functionalanalytic,KKM}}(\mathcal{O}_i))_\text{BBM,formalanalytification},
\end{align}
where each $\mathcal{O}_i$ is given as some $B_\mathrm{dR}\left<T_1,...,T_n\right>$. See \cite{G2}, \cite{12G1}.} given over $A/I$ with a chosen prism $(A,I)$\footnote{In all the following, we assume this prism to be bounded and satisfy that $A/I$ is Banach.}:

\begin{align}
\mathrm{Object}_{\mathrm{E}_\infty\mathrm{commutativealgebra},\mathrm{Simplicial}}(\mathrm{IndSNorm}_{A/I})^{\mathrm{smoothformalseriesclosure}},\\
\mathrm{Object}_{\mathrm{E}_\infty\mathrm{commutativealgebra},\mathrm{Simplicial}}(\mathrm{Ind}^m\mathrm{SNorm}_{A/I})^{\mathrm{smoothformalseriesclosure}},\\
\mathrm{Object}_{\mathrm{E}_\infty\mathrm{commutativealgebra},\mathrm{Simplicial}}(\mathrm{IndNorm}_{A/I})^{\mathrm{smoothformalseriesclosure}},\\
\mathrm{Object}_{\mathrm{E}_\infty\mathrm{commutativealgebra},\mathrm{Simplicial}}(\mathrm{Ind}^m\mathrm{Norm}_{A/I})^{\mathrm{smoothformalseriesclosure}},\\
\mathrm{Object}_{\mathrm{E}_\infty\mathrm{commutativealgebra},\mathrm{Simplicial}}(\mathrm{IndBan}_{A/I})^{\mathrm{smoothformalseriesclosure}},\\
\mathrm{Object}_{\mathrm{E}_\infty\mathrm{commutativealgebra},\mathrm{Simplicial}}(\mathrm{Ind}^m\mathrm{Ban}_{A/I})^{\mathrm{smoothformalseriesclosure}}.
\end{align}
We call the corresponding functors are derived functional analytic Hochschild cohomologies, derived functional analytic period cohomologies and derived functional analytic cyclic cohomologies, which we are going to denote them as in the following:
\begin{align}
	&\mathrm{Kan}_{\mathrm{Left}}\mathrm{THH}_{\text{functionalanalytic,KKM},\text{BBM,formalanalytification}}(\mathcal{O}):=\\
	&(\underset{i}{\text{homotopycolimit}}_{\text{sifted},\text{derivedcategory}_{\infty}(A/I-\text{Module})}\mathrm{Kan}_{\mathrm{Left}}\mathrm{THH}_{\text{functionalanalytic,KKM}}(\mathcal{O}_i)\\
	&)_\text{BBM,formalanalytification},\\
	&\mathrm{Kan}_{\mathrm{Left}}\mathrm{TP}_{\text{functionalanalytic,KKM},\text{BBM,formalanalytification}}(\mathcal{O}):=\\
	&(\underset{i}{\text{homotopycolimit}}_{\text{sifted},\text{derivedcategory}_{\infty}(A/I-\text{Module})}\mathrm{Kan}_{\mathrm{Left}}\mathrm{TP}_{\text{functionalanalytic,KKM}}(\mathcal{O}_i)\\
	&)_\text{BBM,formalanalytification},\\
	&\mathrm{Kan}_{\mathrm{Left}}\mathrm{TC}_{\text{functionalanalytic,KKM},\text{BBM,formalanalytification}}(\mathcal{O}):=\\
	&(\underset{i}{\text{homotopycolimit}}_{\text{sifted},\text{derivedcategory}_{\infty}(A/I-\text{Module})}\mathrm{Kan}_{\mathrm{Left}}\mathrm{TC}_{\text{functionalanalytic,KKM}}(\mathcal{O}_i)\\
	&)_\text{BBM,formalanalytification},
\end{align}
by writing any object $\mathcal{O}$ as the corresponding colimit 
\begin{center}
$\underset{i}{\text{homotopycolimit}}_\text{sifted}\mathcal{O}_i$.
\end{center}
These are quite large $(\infty,1)$-commutative ring objects in the corresponding $(\infty,1)$-categories for $R=A/I$:

\begin{align}
\mathrm{Object}_{\mathrm{E}_\infty\mathrm{commutativealgebra},\mathrm{Simplicial}}(\mathrm{IndSNorm}_R),\\
\mathrm{Object}_{\mathrm{E}_\infty\mathrm{commutativealgebra},\mathrm{Simplicial}}(\mathrm{Ind}^m\mathrm{SNorm}_R),\\
\mathrm{Object}_{\mathrm{E}_\infty\mathrm{commutativealgebra},\mathrm{Simplicial}}(\mathrm{IndNorm}_R),\\
\mathrm{Object}_{\mathrm{E}_\infty\mathrm{commutativealgebra},\mathrm{Simplicial}}(\mathrm{Ind}^m\mathrm{Norm}_R),\\
\mathrm{Object}_{\mathrm{E}_\infty\mathrm{commutativealgebra},\mathrm{Simplicial}}(\mathrm{IndBan}_R),\\
\mathrm{Object}_{\mathrm{E}_\infty\mathrm{commutativealgebra},\mathrm{Simplicial}}(\mathrm{Ind}^m\mathrm{Ban}_R),
\end{align}
after taking the formal series ring left Kan extension analytification from \cite[Section 4.2]{BBM}, which is defined by taking the left Kan extension to all the $(\infty,1)$-ring objects in the $\infty$-derived category of all $A$-modules from formal series rings over $A$, into:
\begin{align}
\mathrm{Object}_{\mathrm{E}_\infty\mathrm{commutativealgebra},\mathrm{Simplicial}}(\mathrm{IndSNorm}_{\mathbb{F}_1})_A,\\
\mathrm{Object}_{\mathrm{E}_\infty\mathrm{commutativealgebra},\mathrm{Simplicial}}(\mathrm{Ind}^m\mathrm{SNorm}_{\mathbb{F}_1})_A,\\
\mathrm{Object}_{\mathrm{E}_\infty\mathrm{commutativealgebra},\mathrm{Simplicial}}(\mathrm{IndNorm}_{\mathbb{F}_1})_A,\\
\mathrm{Object}_{\mathrm{E}_\infty\mathrm{commutativealgebra},\mathrm{Simplicial}}(\mathrm{Ind}^m\mathrm{Norm}_{\mathbb{F}_1})_A,\\
\mathrm{Object}_{\mathrm{E}_\infty\mathrm{commutativealgebra},\mathrm{Simplicial}}(\mathrm{IndBan}_{\mathbb{F}_1})_A,\\
\mathrm{Object}_{\mathrm{E}_\infty\mathrm{commutativealgebra},\mathrm{Simplicial}}(\mathrm{Ind}^m\mathrm{Ban}_{\mathbb{F}_1})_A.
\end{align}
\end{definition}

\

\begin{remark}
One should actually do this in a more coherent way as in \cite{Ra}, \cite{12NS}, \cite{KKM}, by applying directly the corresponding construction to objects in the corresponding $(\infty,1)$-objects above, even the corresponding $A_\infty$-objects. The constructions here are the corresponding functional analytic counterpart of the corresponding condensed constructions in \cite{M} after \cite{12CS1} and \cite{12CS2}.
\end{remark}

\

\begin{definition}
Then we can in the same fashion consider the corresponding derived prismatic complexes \cite[Construction 7.6]{12BS}\footnote{One just applies \cite[Construction 7.6]{12BS} and then takes the left Kan extensions.} for the commutative algebras as in the above (for a given prism $(A,I)$):
\begin{align}
\mathrm{Kan}_{\mathrm{Left}}\Delta_{?/A},	
\end{align}
by the regular corresponding left Kan extension techniques on the following $(\infty,1)$-compactly generated closures of the corresponding polynomials given over $A/I$ with a chosen prism $(A,I)$:

\begin{align}
\mathrm{Object}_{\mathrm{E}_\infty\mathrm{commutativealgebra},\mathrm{Simplicial}}(\mathrm{IndSNorm}_{A/I})^{\mathrm{smoothformalseriesclosure}},\\
\mathrm{Object}_{\mathrm{E}_\infty\mathrm{commutativealgebra},\mathrm{Simplicial}}(\mathrm{Ind}^m\mathrm{SNorm}_{A/I})^{\mathrm{smoothformalseriesclosure}},\\
\mathrm{Object}_{\mathrm{E}_\infty\mathrm{commutativealgebra},\mathrm{Simplicial}}(\mathrm{IndNorm}_{A/I})^{\mathrm{smoothformalseriesclosure}},\\
\mathrm{Object}_{\mathrm{E}_\infty\mathrm{commutativealgebra},\mathrm{Simplicial}}(\mathrm{Ind}^m\mathrm{Norm}_{A/I})^{\mathrm{smoothformalseriesclosure}},\\
\mathrm{Object}_{\mathrm{E}_\infty\mathrm{commutativealgebra},\mathrm{Simplicial}}(\mathrm{IndBan}_{A/I})^{\mathrm{smoothformalseriesclosure}},\\
\mathrm{Object}_{\mathrm{E}_\infty\mathrm{commutativealgebra},\mathrm{Simplicial}}(\mathrm{Ind}^m\mathrm{Ban}_{A/I})^{\mathrm{smoothformalseriesclosure}}.
\end{align}
We call the corresponding functors functional analytic derived prismatic complexes which we are going to denote that as in the following:
\begin{align}
\mathrm{Kan}_{\mathrm{Left}}\Delta_{?/A,\text{functionalanalytic,KKM},\text{BBM,formalanalytification}}.	
\end{align}
This would mean the following definition:
\begin{align}
&\mathrm{Kan}_{\mathrm{Left}}\Delta_{?/A,\text{functionalanalytic,KKM},\text{BBM,formalanalytification}}(\mathcal{O})\\
&:=	((\underset{i}{\text{homotopycolimit}}_{\text{sifted},\text{derivedcategory}_{\infty}(A/I-\text{Module})}\mathrm{Kan}_{\mathrm{Left}}\Delta_{?/A,\text{functionalanalytic,KKM}}(\mathcal{O}_i))^\wedge\\
&)_\text{BBM,formalanalytification}
\end{align}
\footnote{Before the Ben-Bassat-Mukherjee $p$-adic formal analytification we take the corresponding derived $(p,I)$-completion.}by writing any object $\mathcal{O}$ as the corresponding colimit 
\begin{center}
$\underset{i}{\text{homotopycolimit}}_\text{sifted}\mathcal{O}_i$.
\end{center}
These are quite large $(\infty,1)$-commutative ring objects in the corresponding $(\infty,1)$-categories for $R=A/I$:
\begin{align}
\mathrm{Object}_{\mathrm{E}_\infty\mathrm{commutativealgebra},\mathrm{Simplicial}}(\mathrm{IndSNorm}_R),\\
\mathrm{Object}_{\mathrm{E}_\infty\mathrm{commutativealgebra},\mathrm{Simplicial}}(\mathrm{Ind}^m\mathrm{SNorm}_R),\\
\mathrm{Object}_{\mathrm{E}_\infty\mathrm{commutativealgebra},\mathrm{Simplicial}}(\mathrm{IndNorm}_R),\\
\mathrm{Object}_{\mathrm{E}_\infty\mathrm{commutativealgebra},\mathrm{Simplicial}}(\mathrm{Ind}^m\mathrm{Norm}_R),\\
\mathrm{Object}_{\mathrm{E}_\infty\mathrm{commutativealgebra},\mathrm{Simplicial}}(\mathrm{IndBan}_R),\\
\mathrm{Object}_{\mathrm{E}_\infty\mathrm{commutativealgebra},\mathrm{Simplicial}}(\mathrm{Ind}^m\mathrm{Ban}_R),\\
\end{align}
after taking the formal series ring left Kan extension analytification from \cite[Section 4.2]{BBM}, which is defined by taking the left Kan extension to all the $(\infty,1)$-ring objects in the $\infty$-derived category of all $A$-modules from formal series rings over $A$, into:
\begin{align}
\mathrm{Object}_{\mathrm{E}_\infty\mathrm{commutativealgebra},\mathrm{Simplicial}}(\mathrm{IndSNorm}_{\mathbb{F}_1})_A,\\
\mathrm{Object}_{\mathrm{E}_\infty\mathrm{commutativealgebra},\mathrm{Simplicial}}(\mathrm{Ind}^m\mathrm{SNorm}_{\mathbb{F}_1})_A,\\
\mathrm{Object}_{\mathrm{E}_\infty\mathrm{commutativealgebra},\mathrm{Simplicial}}(\mathrm{IndNorm}_{\mathbb{F}_1})_A,\\
\mathrm{Object}_{\mathrm{E}_\infty\mathrm{commutativealgebra},\mathrm{Simplicial}}(\mathrm{Ind}^m\mathrm{Norm}_{\mathbb{F}_1})_A,\\
\mathrm{Object}_{\mathrm{E}_\infty\mathrm{commutativealgebra},\mathrm{Simplicial}}(\mathrm{IndBan}_{\mathbb{F}_1})_A,\\
\mathrm{Object}_{\mathrm{E}_\infty\mathrm{commutativealgebra},\mathrm{Simplicial}}(\mathrm{Ind}^m\mathrm{Ban}_{\mathbb{F}_1})_A.
\end{align}
\end{definition}

\

\indent Then as in \cite[Definition 8.2]{12BS} we consider the corresponding perfectoidization in this analytic setting.

\begin{definition}
Let $(A,I)$ be a perfectoid prism, and we consider any $\mathrm{E}_\infty$-ring $\mathcal{O}$ in the following
\begin{align}
\mathrm{Object}_{\mathrm{E}_\infty\mathrm{commutativealgebra},\mathrm{Simplicial}}(\mathrm{IndSNorm}_{A/I})^{\mathrm{smoothformalseriesclosure}},\\
\mathrm{Object}_{\mathrm{E}_\infty\mathrm{commutativealgebra},\mathrm{Simplicial}}(\mathrm{Ind}^m\mathrm{SNorm}_{A/I})^{\mathrm{smoothformalseriesclosure}},\\
\mathrm{Object}_{\mathrm{E}_\infty\mathrm{commutativealgebra},\mathrm{Simplicial}}(\mathrm{IndNorm}_{A/I})^{\mathrm{smoothformalseriesclosure}},\\
\mathrm{Object}_{\mathrm{E}_\infty\mathrm{commutativealgebra},\mathrm{Simplicial}}(\mathrm{Ind}^m\mathrm{Norm}_{A/I})^{\mathrm{smoothformalseriesclosure}},\\
\mathrm{Object}_{\mathrm{E}_\infty\mathrm{commutativealgebra},\mathrm{Simplicial}}(\mathrm{IndBan}_{A/I})^{\mathrm{smoothformalseriesclosure}},\\
\mathrm{Object}_{\mathrm{E}_\infty\mathrm{commutativealgebra},\mathrm{Simplicial}}(\mathrm{Ind}^m\mathrm{Ban}_{A/I})^{\mathrm{smoothformalseriesclosure}}.
\end{align}
Then consider the derived prismatic object:
\begin{align}
\mathrm{Kan}_{\mathrm{Left}}\Delta_{?/A,\text{functionalanalytic,KKM},\text{BBM,formalanalytification}}(\mathcal{O}).
\end{align}	
Then as in \cite[Definition 8.2]{12BS} we have the following preperfectoidization:
\begin{align}
&(\mathcal{O})^{\text{preperfectoidization}}\\
&:=\mathrm{Colimit}(\mathrm{Kan}_{\mathrm{Left}}\Delta_{?/A,\text{functionalanalytic,KKM},\text{BBM,formalanalytification}}(\mathcal{O})\rightarrow \\
&\mathrm{Fro}_*\mathrm{Kan}_{\mathrm{Left}}\Delta_{?/A,\text{functionalanalytic,KKM},\text{BBM,formalanalytification}}(\mathcal{O})\\
&\rightarrow \mathrm{Fro}_* \mathrm{Fro}_*\mathrm{Kan}_{\mathrm{Left}}\Delta_{?/A,\text{functionalanalytic,KKM},\text{BBM,formalanalytification}}(\mathcal{O})\rightarrow...)^{\text{BBM,formalanalytification}},	
\end{align}
after taking the formal series ring left Kan extension analytification from \cite[Section 4.2]{BBM}, which is defined by taking the left Kan extension to all the $(\infty,1)$-ring object in the $\infty$-derived category of all $A$-modules from formal series rings over $A$. Then we define the corresponding perfectoidization:
\begin{align}
&(\mathcal{O})^{\text{perfectoidization}}\\
&:=\mathrm{Colimit}(\mathrm{Kan}_{\mathrm{Left}}\Delta_{?/A,\text{functionalanalytic,KKM},\text{BBM,formalanalytification}}(\mathcal{O})\longrightarrow \\
&\mathrm{Fro}_*\mathrm{Kan}_{\mathrm{Left}}\Delta_{?/A,\text{functionalanalytic,KKM},\text{BBM,formalanalytification}}(\mathcal{O})\\
&\longrightarrow \mathrm{Fro}_* \mathrm{Fro}_*\mathrm{Kan}_{\mathrm{Left}}\Delta_{?/A,\text{functionalanalytic,KKM},\text{BBM,formalanalytification}}(\mathcal{O})\longrightarrow...)^{\text{BBM,formalanalytification}}\times A/I.	
\end{align}
Furthermore one can take derived $(p,I)$-completion to achieve the derived $(p,I)$-completed versions:
\begin{align}
\mathcal{O}^\text{preperfectoidization,derivedcomplete}:=(\mathcal{O}^\text{preperfectoidization})^{\wedge},\\
\mathcal{O}^\text{perfectoidization,derivedcomplete}:=\mathcal{O}^\text{preperfectoidization,derivedcomplete}\times A/I.\\
\end{align}
These are large $(\infty,1)$-commutative algebra objects in the corresponding categories as in the above, attached to also large $(\infty,1)$-commutative algebra objects. When we apply this to the corresponding sub-$(\infty,1)$-categories of Banach perfectoid objects in \cite{BMS2}, \cite{GR}, \cite{12KL1}, \cite{12KL2}, \cite{12Ked1}, \cite{12Sch3},  we will recover the corresponding distinguished elemental deformation processes defined in \cite{BMS2}, \cite{GR}, \cite{12KL1}, \cite{12KL2}, \cite{12Ked1}, \cite{12Sch3}.
\end{definition}

\

\begin{remark}
One can then define such ring $\mathcal{O}$ to be \textit{preperfectoid} if we have the equivalence:
\begin{align}
\mathcal{O}^{\text{preperfectoidization}} \overset{\sim}{\longrightarrow}	\mathcal{O}.
\end{align}
One can then define such ring $\mathcal{O}$ to be \textit{perfectoid} if we have the equivalence:
\begin{align}
\mathcal{O}^{\text{preperfectoidization}}\times A/I \overset{\sim}{\longrightarrow}	\mathcal{O}.
\end{align}
	
\end{remark}

\newpage

\section{Functional Analytic Derived Logarithmic Prismatic Cohomology and Functional Analytic Derived Logarithmic Perfectoidizations}

\subsection{Functional Analytic Derived Logarithmic Prismatic Cohomology}

\indent We now extend the previous functional analytic discussion to logarithmic context, which is related to the logarithmic context for instance in \cite[Chapter 8]{12O}, \cite[Chapter 5, Chapter 6, Chapter 7]{12B1} and \cite{12Ko1}.

\begin{definition}
We define the logarithmic version of the $(\infty,1)$-categories we considered in the previous section (let $(B,*)$ be any Banach logarithmic ring or $\mathbb{F}_1$):
\begin{align}
\mathrm{Object}_{\mathrm{E}_\infty\mathrm{commutativealgebra},\mathrm{Simplicial}}(\mathrm{IndSNorm}_B)^{\text{prelog}},\\
\mathrm{Object}_{\mathrm{E}_\infty\mathrm{commutativealgebra},\mathrm{Simplicial}}(\mathrm{Ind}^m\mathrm{SNorm}_B)^{\text{prelog}},\\
\mathrm{Object}_{\mathrm{E}_\infty\mathrm{commutativealgebra},\mathrm{Simplicial}}(\mathrm{IndNorm}_B)^{\text{prelog}},\\
\mathrm{Object}_{\mathrm{E}_\infty\mathrm{commutativealgebra},\mathrm{Simplicial}}(\mathrm{Ind}^m\mathrm{Norm}_B)^{\text{prelog}},\\
\mathrm{Object}_{\mathrm{E}_\infty\mathrm{commutativealgebra},\mathrm{Simplicial}}(\mathrm{IndBan}_B)^{\text{prelog}},\\
\mathrm{Object}_{\mathrm{E}_\infty\mathrm{commutativealgebra},\mathrm{Simplicial}}(\mathrm{Ind}^m\mathrm{Ban}_B)^{\text{prelog}},
\end{align}
where the prelog object means a morphism $M\rightarrow ?$ where	$M$ is
in the following categories of $(\infty,1)$-monoid objects:
\begin{align}
\mathrm{Object}_{\mathrm{E}_\infty\mathrm{monoid},\mathrm{Simplicial}}(\mathrm{IndSNorm}_B),\\
\mathrm{Object}_{\mathrm{E}_\infty\mathrm{monoid},\mathrm{Simplicial}}(\mathrm{Ind}^m\mathrm{SNorm}_B),\\
\mathrm{Object}_{\mathrm{E}_\infty\mathrm{monoid},\mathrm{Simplicial}}(\mathrm{IndNorm}_B),\\
\mathrm{Object}_{\mathrm{E}_\infty\mathrm{monoid},\mathrm{Simplicial}}(\mathrm{Ind}^m\mathrm{Norm}_B),\\
\mathrm{Object}_{\mathrm{E}_\infty\mathrm{monoid},\mathrm{Simplicial}}(\mathrm{IndBan}_B),\\
\mathrm{Object}_{\mathrm{E}_\infty\mathrm{monoid},\mathrm{Simplicial}}(\mathrm{Ind}^m\mathrm{Ban}_B),\\
\end{align}
with $?$ being commutative algebra object in the previous section.
\end{definition}

\begin{assumption}
We are going to consider the possible colimits in the $(\infty,1)$-categories:
\begin{align}
\mathrm{Object}_{\mathrm{E}_\infty\mathrm{commutativealgebra},\mathrm{Simplicial}}(\mathrm{IndSNorm}_B)^{\text{prelog}},\\
\mathrm{Object}_{\mathrm{E}_\infty\mathrm{commutativealgebra},\mathrm{Simplicial}}(\mathrm{Ind}^m\mathrm{SNorm}_B)^{\text{prelog}},\\
\mathrm{Object}_{\mathrm{E}_\infty\mathrm{commutativealgebra},\mathrm{Simplicial}}(\mathrm{IndNorm}_B)^{\text{prelog}},\\
\mathrm{Object}_{\mathrm{E}_\infty\mathrm{commutativealgebra},\mathrm{Simplicial}}(\mathrm{Ind}^m\mathrm{Norm}_B)^{\text{prelog}},\\
\mathrm{Object}_{\mathrm{E}_\infty\mathrm{commutativealgebra},\mathrm{Simplicial}}(\mathrm{IndBan}_B)^{\text{prelog}},\\
\mathrm{Object}_{\mathrm{E}_\infty\mathrm{commutativealgebra},\mathrm{Simplicial}}(\mathrm{Ind}^m\mathrm{Ban}_B)^{\text{prelog}},
\end{align}
which means essentially that all the construction will only apply to the object $\mathcal{M}\rightarrow\mathcal{O}$ which could be written as the colimit of logarithmic formal series rings over $(B,*)$:
\begin{center}
$\underset{i}{\text{homotopycolimit}}_\text{sifted}(\mathcal{M}_i\rightarrow\mathcal{O}_i)$
\end{center}
in certain large enough $\infty$-category:
\begin{align}
\mathrm{Object}_{\mathrm{Simplicial}}(\mathrm{IndSets})\times \mathrm{Object}_{\mathrm{Simplicial}}(\mathrm{IndSets}).\\
\end{align}
We conjecture the corresponding homotopy colimits exist throughout in order to drop this technical assumption. The resulting $\infty$-categories are denoted to be:
\begin{align}
\mathrm{Object}_{\mathrm{E}_\infty\mathrm{commutativealgebra},\mathrm{Simplicial}}(\mathrm{IndSNorm}_{B})^{\mathrm{smoothformalseriesclosure},\text{prelog}},\\
\mathrm{Object}_{\mathrm{E}_\infty\mathrm{commutativealgebra},\mathrm{Simplicial}}(\mathrm{Ind}^m\mathrm{SNorm}_{B})^{\mathrm{smoothformalseriesclosure},\text{prelog}},\\
\mathrm{Object}_{\mathrm{E}_\infty\mathrm{commutativealgebra},\mathrm{Simplicial}}(\mathrm{IndNorm}_{B})^{\mathrm{smoothformalseriesclosure},\text{prelog}},\\
\mathrm{Object}_{\mathrm{E}_\infty\mathrm{commutativealgebra},\mathrm{Simplicial}}(\mathrm{Ind}^m\mathrm{Norm}_{B})^{\mathrm{smoothformalseriesclosure},\text{prelog}},\\
\mathrm{Object}_{\mathrm{E}_\infty\mathrm{commutativealgebra},\mathrm{Simplicial}}(\mathrm{IndBan}_{B})^{\mathrm{smoothformalseriesclosure},\text{prelog}},\\
\mathrm{Object}_{\mathrm{E}_\infty\mathrm{commutativealgebra},\mathrm{Simplicial}}(\mathrm{Ind}^m\mathrm{Ban}_{B})^{\mathrm{smoothformalseriesclosure},\text{prelog}}.
\end{align}	
Namely one considers all the logarithmic formal series rings, and then takes in 
\begin{align}
\mathrm{Object}_{\mathrm{Simplicial}}(\mathrm{IndSets})\times \mathrm{Object}_{\mathrm{Simplicial}}(\mathrm{IndSets})
\end{align}
the closure by colimits.\\
\end{assumption}

\

\begin{definition}
Then we can in the same fashion consider the corresponding derived prismatic complexes \cite[Construction 7.6]{12BS}, \cite[Definition 4.1]{12Ko1}\footnote{One just applies \cite[Construction 7.6]{12BS}, \cite[Definition 4.1]{12Ko1} and then takes the left Kan extensions. We note that the derived logarithmic prismatic cohomologies are considered in \cite[Just above Notation in Section 1]{12Ko1}.} for the commutative algebras as in the above (for a given log prism $(A,I,M)$):
\begin{align}
\mathrm{Kan}_{\mathrm{Left}}\Delta_{?/A},	
\end{align}
by the regular corresponding left Kan extension techniques on the following $(\infty,1)$-compactly generated closures of the corresponding polynomials given over $A/I$ with a chosen log prism $(A,I,M)$:

\begin{align}
\mathrm{Object}_{\mathrm{E}_\infty\mathrm{commutativealgebra},\mathrm{Simplicial}}(\mathrm{IndSNorm}_{A/I})^{\mathrm{smoothformalseriesclosure},\text{prelog}},\\
\mathrm{Object}_{\mathrm{E}_\infty\mathrm{commutativealgebra},\mathrm{Simplicial}}(\mathrm{Ind}^m\mathrm{SNorm}_{A/I})^{\mathrm{smoothformalseriesclosure},\text{prelog}},\\
\mathrm{Object}_{\mathrm{E}_\infty\mathrm{commutativealgebra},\mathrm{Simplicial}}(\mathrm{IndNorm}_{A/I})^{\mathrm{smoothformalseriesclosure},\text{prelog}},\\
\mathrm{Object}_{\mathrm{E}_\infty\mathrm{commutativealgebra},\mathrm{Simplicial}}(\mathrm{Ind}^m\mathrm{Norm}_{A/I})^{\mathrm{smoothformalseriesclosure},\text{prelog}},\\
\mathrm{Object}_{\mathrm{E}_\infty\mathrm{commutativealgebra},\mathrm{Simplicial}}(\mathrm{IndBan}_{A/I})^{\mathrm{smoothformalseriesclosure},\text{prelog}},\\
\mathrm{Object}_{\mathrm{E}_\infty\mathrm{commutativealgebra},\mathrm{Simplicial}}(\mathrm{Ind}^m\mathrm{Ban}_{A/I})^{\mathrm{smoothformalseriesclosure},\text{prelog}}.
\end{align}
We call the corresponding functors functional analytic derived logarithmic prismatic complexes which we are going to denote that as in the following:
\begin{align}
\mathrm{Kan}_{\mathrm{Left}}\Delta_{?/A,\text{functionalanalytic,logarithmic,KKM},\text{BBM,formalanalytification}}.	
\end{align}
This would mean the following definition{\footnote{Before the Ben-Bassat-Mukherjee $p$-adic formal analytification we take the corresponding derived $(p,I)$-completion.}}:
\begin{align}
\mathrm{Kan}_{\mathrm{Left}}&\Delta_{?/A,\text{functionalanalytic,logarithmic,KKM},\text{BBM,formalanalytification}}(\mathcal{O})\\
&:=	((\underset{i}{\text{homotopycolimit}}_{\text{sifted},\text{derivedcategory}_{\infty}(A/I-\text{Module})}\mathrm{Kan}_{\mathrm{Left}}\Delta_{?/A,\text{functionalanalytic,logarithmic,KKM}}\\
&(\mathcal{O}_i))^\wedge\\
&)_\text{BBM,formalanalytification}
\end{align}
by writing any object $\mathcal{O}$ as the corresponding colimit 
\begin{center}
$\underset{i}{\text{homotopycolimit}}_\text{sifted}\mathcal{O}_i$.
\end{center}
These are quite large $(\infty,1)$-commutative ring objects in the corresponding $(\infty,1)$-categories for $(R=A/I,*)$:

\begin{align}
\mathrm{Object}_{\mathrm{E}_\infty\mathrm{commutativealgebra},\mathrm{Simplicial}}(\mathrm{IndSNorm}_R)^{\text{prelog}},\\
\mathrm{Object}_{\mathrm{E}_\infty\mathrm{commutativealgebra},\mathrm{Simplicial}}(\mathrm{Ind}^m\mathrm{SNorm}_R)^{\text{prelog}},\\
\mathrm{Object}_{\mathrm{E}_\infty\mathrm{commutativealgebra},\mathrm{Simplicial}}(\mathrm{IndNorm}_R)^{\text{prelog}},\\
\mathrm{Object}_{\mathrm{E}_\infty\mathrm{commutativealgebra},\mathrm{Simplicial}}(\mathrm{Ind}^m\mathrm{Norm}_R)^{\text{prelog}},\\
\mathrm{Object}_{\mathrm{E}_\infty\mathrm{commutativealgebra},\mathrm{Simplicial}}(\mathrm{IndBan}_R)^{\text{prelog}},\\
\mathrm{Object}_{\mathrm{E}_\infty\mathrm{commutativealgebra},\mathrm{Simplicial}}(\mathrm{Ind}^m\mathrm{Ban}_R)^{\text{prelog}},\\
\end{align}
after taking the formal series ring left Kan extension analytification from \cite[Section 4.2]{BBM}, which is defined by taking the left Kan extension to all the $(\infty,1)$-ring objects in the $\infty$-derived category of all $A$-modules from formal series rings over $A$.
\end{definition}

\newpage

\subsection{Functional Analytic Derived Logarithmic Perfectoidizations}

\indent Then as in \cite[Definition 8.2]{12BS} we consider the corresponding perfectoidization in this analytic setting.

\begin{definition}
Let $(A,I,M)$ be a perfectoid logarithmic prism as in \cite[Definition 3.3]{12Ko1}, and we consider any $\mathrm{E}_\infty$-ring $\mathcal{O}$ in the following
\begin{align}
\mathrm{Object}_{\mathrm{E}_\infty\mathrm{commutativealgebra},\mathrm{Simplicial}}(\mathrm{IndSNorm}_{A/I})^{\mathrm{smoothformalseriesclosure},\text{prelog}},\\
\mathrm{Object}_{\mathrm{E}_\infty\mathrm{commutativealgebra},\mathrm{Simplicial}}(\mathrm{Ind}^m\mathrm{SNorm}_{A/I})^{\mathrm{smoothformalseriesclosure},\text{prelog}},\\
\mathrm{Object}_{\mathrm{E}_\infty\mathrm{commutativealgebra},\mathrm{Simplicial}}(\mathrm{IndNorm}_{A/I})^{\mathrm{smoothformalseriesclosure},\text{prelog}},\\
\mathrm{Object}_{\mathrm{E}_\infty\mathrm{commutativealgebra},\mathrm{Simplicial}}(\mathrm{Ind}^m\mathrm{Norm}_{A/I})^{\mathrm{smoothformalseriesclosure},\text{prelog}},\\
\mathrm{Object}_{\mathrm{E}_\infty\mathrm{commutativealgebra},\mathrm{Simplicial}}(\mathrm{IndBan}_{A/I})^{\mathrm{smoothformalseriesclosure},\text{prelog}},\\
\mathrm{Object}_{\mathrm{E}_\infty\mathrm{commutativealgebra},\mathrm{Simplicial}}(\mathrm{Ind}^m\mathrm{Ban}_{A/I})^{\mathrm{smoothformalseriesclosure},\text{prelog}}.
\end{align}
Then consider the derived prismatic object:
\begin{align}
\mathrm{Kan}_{\mathrm{Left}}&\Delta_{?/A,\text{functionalanalytic,logarithmic,KKM},\text{BBM,formalanalytification}}(\mathcal{O})\\
&:=	((\underset{i}{\text{homotopycolimit}}_{\text{sifted},\text{derivedcategory}_{\infty}(A/I-\text{Module})}\mathrm{Kan}_{\mathrm{Left}}\Delta_{?/A,\text{functionalanalytic,logarithmic,KKM}}\\
&(\mathcal{O}_i))^\wedge\\
&)_\text{BBM,formalanalytification}.
\end{align}
Then as in \cite[Definition 8.2]{12BS} we have the following preperfectoidization:
\begin{align}
&(\mathcal{O})^{\text{preperfectoidization}}\\
&:=\mathrm{Colimit}(\mathrm{Kan}_{\mathrm{Left}}\Delta_{?/A,\text{functionalanalytic,logarithmic,KKM},\text{BBM,formalanalytification}}(\mathcal{O})\longrightarrow \\
&\mathrm{Fro}_*\mathrm{Kan}_{\mathrm{Left}}\Delta_{?/A,\text{functionalanalytic,logarithmic,KKM},\text{BBM,formalanalytification}}(\mathcal{O})\\
&\longrightarrow \mathrm{Fro}_* \mathrm{Fro}_*\mathrm{Kan}_{\mathrm{Left}}\Delta_{?/A,\text{functionalanalytic,logarithmic,KKM},\text{BBM,formalanalytification}}(\mathcal{O})\longrightarrow...\\
&)^{\text{BBM,formalanalytification}}.	
\end{align}
\footnote{Again after taking the formal series ring left Kan extension analytification from \cite[Section 4.2]{BBM}, which is defined by taking the left Kan extension to all the $(\infty,1)$-ring objects in the $\infty$-derived category of all $A$-modules from formal series rings over $A$, into:
\begin{align}
\mathrm{Object}_{\mathrm{E}_\infty\mathrm{commutativealgebra},\mathrm{Simplicial}}(\mathrm{IndSNorm}_R)^{\text{prelog}},\\
\mathrm{Object}_{\mathrm{E}_\infty\mathrm{commutativealgebra},\mathrm{Simplicial}}(\mathrm{Ind}^m\mathrm{SNorm}_R)^{\text{prelog}},\\
\mathrm{Object}_{\mathrm{E}_\infty\mathrm{commutativealgebra},\mathrm{Simplicial}}(\mathrm{IndNorm}_R)^{\text{prelog}},\\
\mathrm{Object}_{\mathrm{E}_\infty\mathrm{commutativealgebra},\mathrm{Simplicial}}(\mathrm{Ind}^m\mathrm{Norm}_R)^{\text{prelog}},\\
\mathrm{Object}_{\mathrm{E}_\infty\mathrm{commutativealgebra},\mathrm{Simplicial}}(\mathrm{IndBan}_R)^{\text{prelog}},\\
\mathrm{Object}_{\mathrm{E}_\infty\mathrm{commutativealgebra},\mathrm{Simplicial}}(\mathrm{Ind}^m\mathrm{Ban}_R)^{\text{prelog}}.\\
\end{align}
}Then we define the corresponding perfectoidization:
\begin{align}
&(\mathcal{O})^{\text{perfectoidization}}\\
&:=\mathrm{Colimit}(\mathrm{Kan}_{\mathrm{Left}}\Delta_{?/A,\text{functionalanalytic,logarithmic,KKM},\text{BBM,formalanalytification}}(\mathcal{O})\longrightarrow \\
&\mathrm{Fro}_*\mathrm{Kan}_{\mathrm{Left}}\Delta_{?/A,\text{functionalanalytic,logarithmic,KKM},\text{BBM,formalanalytification}}(\mathcal{O})\\
&\longrightarrow \mathrm{Fro}_* \mathrm{Fro}_*\mathrm{Kan}_{\mathrm{Left}}\Delta_{?/A,\text{functionalanalytic,logarithmic,KKM},\text{BBM,formalanalytification}}(\mathcal{O})\longrightarrow\\
&...)^{\text{BBM,formalanalytification}}\times A/I.	
\end{align}
Furthermore one can take derived $(p,I)$-completion to achieve the derived $(p,I)$-completed versions:
\begin{align}
\mathcal{O}^\text{preperfectoidization,derivedcomplete}:=(\mathcal{O}^\text{preperfectoidization})^{\wedge},\\
\mathcal{O}^\text{perfectoidization,derivedcomplete}:=\mathcal{O}^\text{preperfectoidization,derivedcomplete}\times A/I.\\
\end{align}

\noindent These are large $(\infty,1)$-commutative algebra objects in the corresponding categories as in the above, attached to also large $(\infty,1)$-commutative algebra objects. When we apply this to the corresponding sub-$(\infty,1)$-categories of Banach perfectoid objects in \cite{BMS2}, \cite{12DLLZ1}, \cite{12DLLZ2}, \cite{GR}, \cite{12KL1}, \cite{12KL2}, \cite{12Ked1}, \cite{12Sch3} we will recover the corresponding distinguished elemental deformation processes defined in \cite{BMS2}, \cite{12DLLZ1}, \cite{12DLLZ2}, \cite{GR}, \cite{12KL1}, \cite{12KL2}, \cite{12Ked1}, \cite{12Sch3}.
\end{definition}

\

\begin{remark}
One can then define such ring $\mathcal{O}$ to be \textit{logarithmicpreperfectoid} if we have the equivalence:
\begin{align}
\mathcal{O}^{\text{preperfectoidization}} \overset{\sim}{\longrightarrow}	\mathcal{O}.
\end{align}
One can then define such ring $\mathcal{O}$ to be \textit{logarithmicperfectoid} if we have the equivalence:
\begin{align}
\mathcal{O}^{\text{preperfectoidization}}\times A/I \overset{\sim}{\longrightarrow}	\mathcal{O}.
\end{align}
	
\end{remark}

\newpage

\section{Functional Analytic Derived Prismatic Complexes for $(\infty,1)$-Analytic Stacks and the Preperfectoidizations}

\subsection{Functional Analytic Derived Prismatic Complexes for $(\infty,1)$-Analytic Stacks}

\indent We now promote the construction in the previous sections to the corresponding $(\infty,1)$-ringed toposes level after Lurie \cite{12Lu1}, \cite{12Lu2} and \cite{Lu3} in the $\infty$-category of $\infty$-ringed toposes, Bambozzi-Ben-Bassat-Kremnizer \cite{12BBBK}, Ben-Bassat-Mukherjee \cite{BBM}, Bambozzi-Kremnizer \cite{12BK}, Clausen-Scholze \cite{12CS1} \cite{12CS2} and Kelly-Kremnizer-Mukherjee \cite{KKM} in the $\infty$-cateogory of $\infty$-functional analytic ringed toposes.\\

\indent Now we consider the following $\infty$-categories of the corresponding $\infty$-analytic ringed toposes from Bambozzi-Ben-Bassat-Kremnizer \cite{12BBBK}:\\

\begin{align}
&\infty-\mathrm{Toposes}^{\mathrm{ringed},\mathrm{commutativealgebra}_{\mathrm{simplicial}}(\mathrm{Ind}\mathrm{Seminormed}_?)}_{\mathrm{Commutativealgebra}_{\mathrm{simplicial}}(\mathrm{Ind}\mathrm{Seminormed}_?)^\mathrm{opposite},\mathrm{Grothendiecktopology,homotopyepimorphism}}.\\
&\infty-\mathrm{Toposes}^{\mathrm{ringed},\mathrm{Commutativealgebra}_{\mathrm{simplicial}}(\mathrm{Ind}^m\mathrm{Seminormed}_?)}_{\mathrm{Commutativealgebra}_{\mathrm{simplicial}}(\mathrm{Ind}^m\mathrm{Seminormed}_?)^\mathrm{opposite},\mathrm{Grothendiecktopology,homotopyepimorphism}}.\\
&\infty-\mathrm{Toposes}^{\mathrm{ringed},\mathrm{Commutativealgebra}_{\mathrm{simplicial}}(\mathrm{Ind}\mathrm{Normed}_?)}_{\mathrm{Commutativealgebra}_{\mathrm{simplicial}}(\mathrm{Ind}\mathrm{Normed}_?)^\mathrm{opposite},\mathrm{Grothendiecktopology,homotopyepimorphism}}.\\
&\infty-\mathrm{Toposes}^{\mathrm{ringed},\mathrm{Commutativealgebra}_{\mathrm{simplicial}}(\mathrm{Ind}^m\mathrm{Normed}_?)}_{\mathrm{Commutativealgebra}_{\mathrm{simplicial}}(\mathrm{Ind}^m\mathrm{Normed}_?)^\mathrm{opposite},\mathrm{Grothendiecktopology,homotopyepimorphism}}.\\
&\infty-\mathrm{Toposes}^{\mathrm{ringed},\mathrm{Commutativealgebra}_{\mathrm{simplicial}}(\mathrm{Ind}\mathrm{Banach}_?)}_{\mathrm{Commutativealgebra}_{\mathrm{simplicial}}(\mathrm{Ind}\mathrm{Banach}_?)^\mathrm{opposite},\mathrm{Grothendiecktopology,homotopyepimorphism}}.\\
&\infty-\mathrm{Toposes}^{\mathrm{ringed},\mathrm{Commutativealgebra}_{\mathrm{simplicial}}(\mathrm{Ind}^m\mathrm{Banach}_?)}_{\mathrm{Commutativealgebra}_{\mathrm{simplicial}}(\mathrm{Ind}^m\mathrm{Banach}_?)^\mathrm{opposite},\mathrm{Grothendiecktopology,homotopyepimorphism}}.\\ 
&\mathrm{Proj}^\text{smoothformalseriesclosure}\infty-\mathrm{Toposes}^{\mathrm{ringed},\mathrm{commutativealgebra}_{\mathrm{simplicial}}(\mathrm{Ind}\mathrm{Seminormed}_?)}_{\mathrm{Commutativealgebra}_{\mathrm{simplicial}}(\mathrm{Ind}\mathrm{Seminormed}_?)^\mathrm{opposite},\mathrm{Grothendiecktopology,homotopyepimorphism}}. \\
&\mathrm{Proj}^\text{smoothformalseriesclosure}\infty-\mathrm{Toposes}^{\mathrm{ringed},\mathrm{Commutativealgebra}_{\mathrm{simplicial}}(\mathrm{Ind}^m\mathrm{Seminormed}_?)}_{\mathrm{Commutativealgebra}_{\mathrm{simplicial}}(\mathrm{Ind}^m\mathrm{Seminormed}_?)^\mathrm{opposite},\mathrm{Grothendiecktopology,homotopyepimorphism}}.\\
&\mathrm{Proj}^\text{smoothformalseriesclosure}\infty-\mathrm{Toposes}^{\mathrm{ringed},\mathrm{Commutativealgebra}_{\mathrm{simplicial}}(\mathrm{Ind}\mathrm{Normed}_?)}_{\mathrm{Commutativealgebra}_{\mathrm{simplicial}}(\mathrm{Ind}\mathrm{Normed}_?)^\mathrm{opposite},\mathrm{Grothendiecktopology,homotopyepimorphism}}.\\
&\mathrm{Proj}^\text{smoothformalseriesclosure}\infty-\mathrm{Toposes}^{\mathrm{ringed},\mathrm{Commutativealgebra}_{\mathrm{simplicial}}(\mathrm{Ind}^m\mathrm{Normed}_?)}_{\mathrm{Commutativealgebra}_{\mathrm{simplicial}}(\mathrm{Ind}^m\mathrm{Normed}_?)^\mathrm{opposite},\mathrm{Grothendiecktopology,homotopyepimorphism}}.\\
&\mathrm{Proj}^\text{smoothformalseriesclosure}\infty-\mathrm{Toposes}^{\mathrm{ringed},\mathrm{Commutativealgebra}_{\mathrm{simplicial}}(\mathrm{Ind}\mathrm{Banach}_?)}_{\mathrm{Commutativealgebra}_{\mathrm{simplicial}}(\mathrm{Ind}\mathrm{Banach}_?)^\mathrm{opposite},\mathrm{Grothendiecktopology,homotopyepimorphism}}.\\
&\mathrm{Proj}^\text{smoothformalseriesclosure}\infty-\mathrm{Toposes}^{\mathrm{ringed},\mathrm{Commutativealgebra}_{\mathrm{simplicial}}(\mathrm{Ind}^m\mathrm{Banach}_?)}_{\mathrm{Commutativealgebra}_{\mathrm{simplicial}}(\mathrm{Ind}^m\mathrm{Banach}_?)^\mathrm{opposite},\mathrm{Grothendiecktopology,homotopyepimorphism}}.\\ 
\end{align}

\begin{example}
\mbox{(Preadic nonsheafy spaces and their colimits)} The main example we would like consider comes from \cite{12BK} where one constructs the corresponding derived adic space $(\mathrm{Spectrumadic}_{\mathrm{BK}}(R),\mathcal{O}_{\mathrm{Spectrumadic}_{\mathrm{BK}}(R)})$ by using derived rational localization to reach some $\infty$-toposes carrying essentially $\infty$-sheaf of rings, from any Banach ring $R$\footnote{With some key essential assumptions but that is not so serious once one considers some foundation from Kedlaya in \cite{Ked2} on reified adic spaces as mentioned in the last section of \cite{12BK}.}. Then we apply this to the corresponding situation below. Let $(A,I)$ be a corresponding prism from Bhatt-Scholze, with the assumption that $A/I$ is Banach. Then what we do is consider all the formal series rings over $A/I$ then take the colimit completion in the homotopy sense, by embedding them through $\mathrm{Spectrumadic}_{\mathrm{BK}}$ into the $\infty$-toposes ringed. For instance for any such formal ring $F$, one could regard this as a inductive system of Banach rings:
\begin{align}
F=\underset{i}{\mathrm{homotopycolimit}} F_i,	
\end{align}
where we set:
\begin{align}
(\mathrm{Spectrumadic}_{\mathrm{BK}}(F),\mathcal{O}):=\underset{i}{\mathrm{homotopylimit}} (\mathrm{Spectrumadic}_{\mathrm{BK}}(F_i),\mathcal{O}_{\mathrm{Spectrumadic}_{\mathrm{BK}}(F_i)}).	
\end{align}
The resulting $\infty$-stacks generated are interesting to study. 	
\end{example}

\begin{definition}
\indent One can actually define the derived prismatic cohomology presheaves through derived topological Hochschild cohomology presheaves, derived topological period cohomology presheaves and derived topological cyclic cohomology presheaves as in \cite[Section 2.2, Section 2.3]{12BMS}, \cite[Theorem 1.13]{12BS}:
\begin{align}
	\mathrm{Kan}_{\mathrm{Left}}\mathrm{THH},\mathrm{Kan}_{\mathrm{Left}}\mathrm{TP},\mathrm{Kan}_{\mathrm{Left}}\mathrm{TC},
\end{align}
on the following $(\infty,1)$-compactly generated closures of the corresponding polynomials\footnote{Definitely, we need to put certain norms over in some relatively canonical way, as in \cite[Section 4.2]{BBM} one can basically consider rigid ones and dagger ones, and so on. We restrict to the \textit{formal} one.} given over $A/I$ with a chosen prism $(A,I)$\footnote{In all the following, we assume this prism to be bounded and satisfy that $A/I$ is Banach.}:
\begin{align}
\mathrm{Proj}^\text{smoothformalseriesclosure}\infty-\mathrm{Toposes}^{\mathrm{ringed},\mathrm{commutativealgebra}_{\mathrm{simplicial}}(\mathrm{Ind}\mathrm{Seminormed}_{A/I})}_{\mathrm{Commutativealgebra}_{\mathrm{simplicial}}(\mathrm{Ind}\mathrm{Seminormed}_{A/I})^\mathrm{opposite},\mathrm{Grothendiecktopology,homotopyepimorphism}}. \\
\mathrm{Proj}^\text{smoothformalseriesclosure}\infty-\mathrm{Toposes}^{\mathrm{ringed},\mathrm{Commutativealgebra}_{\mathrm{simplicial}}(\mathrm{Ind}^m\mathrm{Seminormed}_{A/I})}_{\mathrm{Commutativealgebra}_{\mathrm{simplicial}}(\mathrm{Ind}^m\mathrm{Seminormed}_{A/I})^\mathrm{opposite},\mathrm{Grothendiecktopology,homotopyepimorphism}}.\\
\mathrm{Proj}^\text{smoothformalseriesclosure}\infty-\mathrm{Toposes}^{\mathrm{ringed},\mathrm{Commutativealgebra}_{\mathrm{simplicial}}(\mathrm{Ind}\mathrm{Normed}_{A/I})}_{\mathrm{Commutativealgebra}_{\mathrm{simplicial}}(\mathrm{Ind}\mathrm{Normed}_{A/I})^\mathrm{opposite},\mathrm{Grothendiecktopology,homotopyepimorphism}}.\\
\mathrm{Proj}^\text{smoothformalseriesclosure}\infty-\mathrm{Toposes}^{\mathrm{ringed},\mathrm{Commutativealgebra}_{\mathrm{simplicial}}(\mathrm{Ind}^m\mathrm{Normed}_{A/I})}_{\mathrm{Commutativealgebra}_{\mathrm{simplicial}}(\mathrm{Ind}^m\mathrm{Normed}_{A/I})^\mathrm{opposite},\mathrm{Grothendiecktopology,homotopyepimorphism}}.\\
\mathrm{Proj}^\text{smoothformalseriesclosure}\infty-\mathrm{Toposes}^{\mathrm{ringed},\mathrm{Commutativealgebra}_{\mathrm{simplicial}}(\mathrm{Ind}\mathrm{Banach}_{A/I})}_{\mathrm{Commutativealgebra}_{\mathrm{simplicial}}(\mathrm{Ind}\mathrm{Banach}_{A/I})^\mathrm{opposite},\mathrm{Grothendiecktopology,homotopyepimorphism}}.\\
\mathrm{Proj}^\text{smoothformalseriesclosure}\infty-\mathrm{Toposes}^{\mathrm{ringed},\mathrm{Commutativealgebra}_{\mathrm{simplicial}}(\mathrm{Ind}^m\mathrm{Banach}_{A/I})}_{\mathrm{Commutativealgebra}_{\mathrm{simplicial}}(\mathrm{Ind}^m\mathrm{Banach}_{A/I})^\mathrm{opposite},\mathrm{Grothendiecktopology,homotopyepimorphism}}. 
\end{align}
We call the corresponding functors are derived functional analytic Hochschild cohomology presheaves, derived functional analytic period cohomology presheaves and derived functional analytic cyclic cohomology presheaves, which we are going to denote these presheaves as in the following for any $\infty$-ringed topos $(\mathbb{X},\mathcal{O})=\underset{i}{\text{homotopycolimit}}(\mathbb{X}_i,\mathcal{O}_i)$:
\begin{align}
	&\mathrm{Kan}_{\mathrm{Left}}\mathrm{THH}_{\text{functionalanalytic,KKM},\text{BBM,formalanalytification}}(\mathcal{O}):=\\
	&(\underset{i}{\text{homotopycolimit}}_{\text{sifted},\text{derivedcategory}_{\infty}(A/I-\text{Module})}\mathrm{Kan}_{\mathrm{Left}}\mathrm{THH}_{\text{functionalanalytic,KKM}}(\mathcal{O}_i)\\
	&)_\text{BBM,formalanalytification},\\
	&\mathrm{Kan}_{\mathrm{Left}}\mathrm{TP}_{\text{functionalanalytic,KKM},\text{BBM,formalanalytification}}(\mathcal{O}):=\\
	&(\underset{i}{\text{homotopycolimit}}_{\text{sifted},\text{derivedcategory}_{\infty}(A/I-\text{Module})}\mathrm{Kan}_{\mathrm{Left}}\mathrm{TP}_{\text{functionalanalytic,KKM}}(\mathcal{O}_i)\\
	&)_\text{BBM,formalanalytification},\\
	&\mathrm{Kan}_{\mathrm{Left}}\mathrm{TC}_{\text{functionalanalytic,KKM},\text{BBM,formalanalytification}}(\mathcal{O}):=\\
	&(\underset{i}{\text{homotopycolimit}}_{\text{sifted},\text{derivedcategory}_{\infty}(A/I-\text{Module})}\mathrm{Kan}_{\mathrm{Left}}\mathrm{TC}_{\text{functionalanalytic,KKM}}(\mathcal{O}_i)\\
	&)_\text{BBM,formalanalytification},
\end{align}
by writing any object $\mathcal{O}$ as the corresponding colimit 
\begin{center}
$\underset{i}{\text{homotopycolimit}}_\text{sifted}\mathcal{O}_i:=\underset{i}{\text{homotopycolimit}}_\text{sifted}\mathcal{O}_{\mathrm{Spectrumadic}_\mathrm{BK}(\mathrm{Formalseriesring}_i)}$.
\end{center}
These are quite large $(\infty,1)$-commutative ring objects in the corresponding $(\infty,1)$-categories for $?=A/I$ over $(\mathbb{X},\mathcal{O})$:

 \begin{align}
\mathrm{Ind}\mathrm{\sharp Quasicoherent}^{\text{presheaf}}_{\mathrm{Ind}^\text{smoothformalseriesclosure}\infty-\mathrm{Toposes}^{\mathrm{ringed},\mathrm{commutativealgebra}_{\mathrm{simplicial}}(\mathrm{Ind}\mathrm{Seminormed}_?)}_{\mathrm{Commutativealgebra}_{\mathrm{simplicial}}(\mathrm{Ind}\mathrm{Seminormed}_?)^\mathrm{opposite},\mathrm{Grotopology,homotopyepimorphism}}}. \\
\mathrm{Ind}\mathrm{\sharp Quasicoherent}^{\text{presheaf}}_{\mathrm{Ind}^\text{smoothformalseriesclosure}\infty-\mathrm{Toposes}^{\mathrm{ringed},\mathrm{Commutativealgebra}_{\mathrm{simplicial}}(\mathrm{Ind}^m\mathrm{Seminormed}_?)}_{\mathrm{Commutativealgebra}_{\mathrm{simplicial}}(\mathrm{Ind}^m\mathrm{Seminormed}_?)^\mathrm{opposite},\mathrm{Grotopology,homotopyepimorphism}}}.\\
\mathrm{Ind}\mathrm{\sharp Quasicoherent}^{\text{presheaf}}_{\mathrm{Ind}^\text{smoothformalseriesclosure}\infty-\mathrm{Toposes}^{\mathrm{ringed},\mathrm{Commutativealgebra}_{\mathrm{simplicial}}(\mathrm{Ind}\mathrm{Normed}_?)}_{\mathrm{Commutativealgebra}_{\mathrm{simplicial}}(\mathrm{Ind}\mathrm{Normed}_?)^\mathrm{opposite},\mathrm{Grotopology,homotopyepimorphism}}}.\\
\mathrm{Ind}\mathrm{\sharp Quasicoherent}^{\text{presheaf}}_{\mathrm{Ind}^\text{smoothformalseriesclosure}\infty-\mathrm{Toposes}^{\mathrm{ringed},\mathrm{Commutativealgebra}_{\mathrm{simplicial}}(\mathrm{Ind}^m\mathrm{Normed}_?)}_{\mathrm{Commutativealgebra}_{\mathrm{simplicial}}(\mathrm{Ind}^m\mathrm{Normed}_?)^\mathrm{opposite},\mathrm{Grotopology,homotopyepimorphism}}}.\\
\mathrm{Ind}\mathrm{\sharp Quasicoherent}^{\text{presheaf}}_{\mathrm{Ind}^\text{smoothformalseriesclosure}\infty-\mathrm{Toposes}^{\mathrm{ringed},\mathrm{Commutativealgebra}_{\mathrm{simplicial}}(\mathrm{Ind}\mathrm{Banach}_?)}_{\mathrm{Commutativealgebra}_{\mathrm{simplicial}}(\mathrm{Ind}\mathrm{Banach}_?)^\mathrm{opposite},\mathrm{Grotopology,homotopyepimorphism}}}.\\
\mathrm{Ind}\mathrm{\sharp Quasicoherent}^{\text{presheaf}}_{\mathrm{Ind}^\text{smoothformalseriesclosure}\infty-\mathrm{Toposes}^{\mathrm{ringed},\mathrm{Commutativealgebra}_{\mathrm{simplicial}}(\mathrm{Ind}^m\mathrm{Banach}_?)}_{\mathrm{Commutativealgebra}_{\mathrm{simplicial}}(\mathrm{Ind}^m\mathrm{Banach}_?)^\mathrm{opposite},\mathrm{Grotopology,homotopyepimorphism}}},\\ 
\end{align}
after taking the formal series ring left Kan extension analytification from \cite[Section 4.2]{BBM}, which is defined by taking the left Kan extension to all the $(\infty,1)$-ring objects in the $\infty$-derived category of all $A$-modules from formal series rings over $A$.
\end{definition}

\

\begin{definition}
Then we can in the same fashion consider the corresponding derived prismatic complex presheaves \cite[Construction 7.6]{12BS}\footnote{One just applies \cite[Construction 7.6]{12BS} and then takes the left Kan extensions.} for the commutative algebras as in the above (for a given prism $(A,I)$):
\begin{align}
\mathrm{Kan}_{\mathrm{Left}}\Delta_{?/A},	
\end{align}
by the regular corresponding left Kan extension techniques on the following $(\infty,1)$-compactly generated closures of the corresponding polynomials given over $A/I$ with a chosen prism $(A,I)$:\\

\begin{align}
\mathrm{Proj}^\text{smoothformalseriesclosure}\infty-\mathrm{Toposes}^{\mathrm{ringed},\mathrm{commutativealgebra}_{\mathrm{simplicial}}(\mathrm{Ind}\mathrm{Seminormed}_{A/I})}_{\mathrm{Commutativealgebra}_{\mathrm{simplicial}}(\mathrm{Ind}\mathrm{Seminormed}_{A/I})^\mathrm{opposite},\mathrm{Grothendiecktopology,homotopyepimorphism}}. \\
\mathrm{Proj}^\text{smoothformalseriesclosure}\infty-\mathrm{Toposes}^{\mathrm{ringed},\mathrm{Commutativealgebra}_{\mathrm{simplicial}}(\mathrm{Ind}^m\mathrm{Seminormed}_{A/I})}_{\mathrm{Commutativealgebra}_{\mathrm{simplicial}}(\mathrm{Ind}^m\mathrm{Seminormed}_{A/I})^\mathrm{opposite},\mathrm{Grothendiecktopology,homotopyepimorphism}}.\\
\mathrm{Proj}^\text{smoothformalseriesclosure}\infty-\mathrm{Toposes}^{\mathrm{ringed},\mathrm{Commutativealgebra}_{\mathrm{simplicial}}(\mathrm{Ind}\mathrm{Normed}_{A/I})}_{\mathrm{Commutativealgebra}_{\mathrm{simplicial}}(\mathrm{Ind}\mathrm{Normed}_{A/I})^\mathrm{opposite},\mathrm{Grothendiecktopology,homotopyepimorphism}}.\\
\mathrm{Proj}^\text{smoothformalseriesclosure}\infty-\mathrm{Toposes}^{\mathrm{ringed},\mathrm{Commutativealgebra}_{\mathrm{simplicial}}(\mathrm{Ind}^m\mathrm{Normed}_{A/I})}_{\mathrm{Commutativealgebra}_{\mathrm{simplicial}}(\mathrm{Ind}^m\mathrm{Normed}_{A/I})^\mathrm{opposite},\mathrm{Grothendiecktopology,homotopyepimorphism}}.\\
\mathrm{Proj}^\text{smoothformalseriesclosure}\infty-\mathrm{Toposes}^{\mathrm{ringed},\mathrm{Commutativealgebra}_{\mathrm{simplicial}}(\mathrm{Ind}\mathrm{Banach}_{A/I})}_{\mathrm{Commutativealgebra}_{\mathrm{simplicial}}(\mathrm{Ind}\mathrm{Banach}_{A/I})^\mathrm{opposite},\mathrm{Grothendiecktopology,homotopyepimorphism}}.\\
\mathrm{Proj}^\text{smoothformalseriesclosure}\infty-\mathrm{Toposes}^{\mathrm{ringed},\mathrm{Commutativealgebra}_{\mathrm{simplicial}}(\mathrm{Ind}^m\mathrm{Banach}_{A/I})}_{\mathrm{Commutativealgebra}_{\mathrm{simplicial}}(\mathrm{Ind}^m\mathrm{Banach}_{A/I})^\mathrm{opposite},\mathrm{Grothendiecktopology,homotopyepimorphism}}. 
\end{align}
We call the corresponding functors functional analytic derived prismatic complex presheaves which we are going to denote that as in the following:
\begin{align}
\mathrm{Kan}_{\mathrm{Left}}\Delta_{?/A,\text{functionalanalytic,KKM},\text{BBM,formalanalytification}}.	
\end{align}
This would mean the following definition{\footnote{Before the Ben-Bassat-Mukherjee $p$-adic formal analytification we take the corresponding derived $(p,I)$-completion.}}:
\begin{align}
&\mathrm{Kan}_{\mathrm{Left}}\Delta_{?/A,\text{functionalanalytic,KKM},\text{BBM,formalanalytification}}(\mathcal{O})\\
&:=	((\underset{i}{\text{homotopycolimit}}_{\text{sifted},\text{derivedcategory}_{\infty}(A/I-\text{Module})}\mathrm{Kan}_{\mathrm{Left}}\Delta_{?/A,\text{functionalanalytic,KKM}}(\mathcal{O}_i))^\wedge\\
&)_\text{BBM,formalanalytification}
\end{align}
by writing any object $\mathcal{O}$ as the corresponding colimit\footnote{Here we remind the readers of the corresponding foundation here, namely the presheaf $\mathcal{O}_i$ in fact takes the value in Koszul complex taking the following form:
\begin{align}
&\mathrm{Koszulcomplex}_{A/I\left<X_1,...,X_l\right>\left<T_1,...,T_m\right>}(a_1-T_1b_1,...,a_m-T_mb_m)\\
&= A/I\left<X_1,...,X_l\right>\left<T_1,...,T_m\right>/^\mathbb{L}(a_1-T_1b_1,...,a_m-T_mb_m).	
\end{align}
This is actually derived $p$-complete since each homotopy group is derived $p$-complete (from the corresponding Banach structure from $A/I\left<X_1,...,X_l\right>$ induced from the $p$-adic topology). Namely the definition of the presheaf:
\begin{align}
\mathrm{Kan}_{\mathrm{Left}}\Delta_{?/A,\text{functionalanalytic,KKM}}(\mathcal{O}_i)	
\end{align}
is directly the application of the derived prismatic functor from \cite[Construction 7.6]{12BS}.} 
\begin{center}
$\underset{i}{\text{homotopycolimit}}_\text{sifted}\mathcal{O}_i$.
\end{center}
These are quite large $(\infty,1)$-commutative ring objects in the corresponding $(\infty,1)$-categories for $R=A/I$:
\begin{align}
\mathrm{Ind}\mathrm{\sharp Quasicoherent}^{\text{presheaf}}_{\mathrm{Ind}^\text{smoothformalseriesclosure}\infty-\mathrm{Toposes}^{\mathrm{ringed},\mathrm{commutativealgebra}_{\mathrm{simplicial}}(\mathrm{Ind}\mathrm{Seminormed}_R)}_{\mathrm{Commutativealgebra}_{\mathrm{simplicial}}(\mathrm{Ind}\mathrm{Seminormed}_R)^\mathrm{opposite},\mathrm{Grotopology,homotopyepimorphism}}}. \\
\mathrm{Ind}\mathrm{\sharp Quasicoherent}^{\text{presheaf}}_{\mathrm{Ind}^\text{smoothformalseriesclosure}\infty-\mathrm{Toposes}^{\mathrm{ringed},\mathrm{Commutativealgebra}_{\mathrm{simplicial}}(\mathrm{Ind}^m\mathrm{Seminormed}_R)}_{\mathrm{Commutativealgebra}_{\mathrm{simplicial}}(\mathrm{Ind}^m\mathrm{Seminormed}_R)^\mathrm{opposite},\mathrm{Grotopology,homotopyepimorphism}}}.\\
\mathrm{Ind}\mathrm{\sharp Quasicoherent}^{\text{presheaf}}_{\mathrm{Ind}^\text{smoothformalseriesclosure}\infty-\mathrm{Toposes}^{\mathrm{ringed},\mathrm{Commutativealgebra}_{\mathrm{simplicial}}(\mathrm{Ind}\mathrm{Normed}_R)}_{\mathrm{Commutativealgebra}_{\mathrm{simplicial}}(\mathrm{Ind}\mathrm{Normed}_R)^\mathrm{opposite},\mathrm{Grotopology,homotopyepimorphism}}}.\\
\mathrm{Ind}\mathrm{\sharp Quasicoherent}^{\text{presheaf}}_{\mathrm{Ind}^\text{smoothformalseriesclosure}\infty-\mathrm{Toposes}^{\mathrm{ringed},\mathrm{Commutativealgebra}_{\mathrm{simplicial}}(\mathrm{Ind}^m\mathrm{Normed}_R)}_{\mathrm{Commutativealgebra}_{\mathrm{simplicial}}(\mathrm{Ind}^m\mathrm{Normed}_R)^\mathrm{opposite},\mathrm{Grotopology,homotopyepimorphism}}}.\\
\mathrm{Ind}\mathrm{\sharp Quasicoherent}^{\text{presheaf}}_{\mathrm{Ind}^\text{smoothformalseriesclosure}\infty-\mathrm{Toposes}^{\mathrm{ringed},\mathrm{Commutativealgebra}_{\mathrm{simplicial}}(\mathrm{Ind}\mathrm{Banach}_R)}_{\mathrm{Commutativealgebra}_{\mathrm{simplicial}}(\mathrm{Ind}\mathrm{Banach}_R)^\mathrm{opposite},\mathrm{Grotopology,homotopyepimorphism}}}.\\
\mathrm{Ind}\mathrm{\sharp Quasicoherent}^{\text{presheaf}}_{\mathrm{Ind}^\text{smoothformalseriesclosure}\infty-\mathrm{Toposes}^{\mathrm{ringed},\mathrm{Commutativealgebra}_{\mathrm{simplicial}}(\mathrm{Ind}^m\mathrm{Banach}_R)}_{\mathrm{Commutativealgebra}_{\mathrm{simplicial}}(\mathrm{Ind}^m\mathrm{Banach}_R)^\mathrm{opposite},\mathrm{Grotopology,homotopyepimorphism}}},\\ 
\end{align}
after taking the formal series ring left Kan extension analytification from \cite[Section 4.2]{BBM}, which is defined by taking the left Kan extension to all the $(\infty,1)$-ring objects in the $\infty$-derived category of all $A$-modules from formal series rings over $A$.\\
\end{definition}

\newpage

\subsection{Functional Analytic Derived Preperfectoidizations for $(\infty,1)$-Analytic Stacks}

\

\noindent Then as in \cite[Definition 8.2]{12BS} we consider the corresponding perfectoidization in this analytic setting.

\begin{definition}
Let $(A,I)$ be a perfectoid prism, and we consider any $\mathrm{E}_\infty$-ring $\mathcal{O}$ in the following
\begin{align}
\mathrm{Proj}^\text{smoothformalseriesclosure}\infty-\mathrm{Toposes}^{\mathrm{ringed},\mathrm{commutativealgebra}_{\mathrm{simplicial}}(\mathrm{Ind}\mathrm{Seminormed}_{A/I})}_{\mathrm{Commutativealgebra}_{\mathrm{simplicial}}(\mathrm{Ind}\mathrm{Seminormed}_{A/I})^\mathrm{opposite},\mathrm{Grothendiecktopology,homotopyepimorphism}}. \\
\mathrm{Proj}^\text{smoothformalseriesclosure}\infty-\mathrm{Toposes}^{\mathrm{ringed},\mathrm{Commutativealgebra}_{\mathrm{simplicial}}(\mathrm{Ind}^m\mathrm{Seminormed}_{A/I})}_{\mathrm{Commutativealgebra}_{\mathrm{simplicial}}(\mathrm{Ind}^m\mathrm{Seminormed}_{A/I})^\mathrm{opposite},\mathrm{Grothendiecktopology,homotopyepimorphism}}.\\
\mathrm{Proj}^\text{smoothformalseriesclosure}\infty-\mathrm{Toposes}^{\mathrm{ringed},\mathrm{Commutativealgebra}_{\mathrm{simplicial}}(\mathrm{Ind}\mathrm{Normed}_{A/I})}_{\mathrm{Commutativealgebra}_{\mathrm{simplicial}}(\mathrm{Ind}\mathrm{Normed}_{A/I})^\mathrm{opposite},\mathrm{Grothendiecktopology,homotopyepimorphism}}.\\
\mathrm{Proj}^\text{smoothformalseriesclosure}\infty-\mathrm{Toposes}^{\mathrm{ringed},\mathrm{Commutativealgebra}_{\mathrm{simplicial}}(\mathrm{Ind}^m\mathrm{Normed}_{A/I})}_{\mathrm{Commutativealgebra}_{\mathrm{simplicial}}(\mathrm{Ind}^m\mathrm{Normed}_{A/I})^\mathrm{opposite},\mathrm{Grothendiecktopology,homotopyepimorphism}}.\\
\mathrm{Proj}^\text{smoothformalseriesclosure}\infty-\mathrm{Toposes}^{\mathrm{ringed},\mathrm{Commutativealgebra}_{\mathrm{simplicial}}(\mathrm{Ind}\mathrm{Banach}_{A/I})}_{\mathrm{Commutativealgebra}_{\mathrm{simplicial}}(\mathrm{Ind}\mathrm{Banach}_{A/I})^\mathrm{opposite},\mathrm{Grothendiecktopology,homotopyepimorphism}}.\\
\mathrm{Proj}^\text{smoothformalseriesclosure}\infty-\mathrm{Toposes}^{\mathrm{ringed},\mathrm{Commutativealgebra}_{\mathrm{simplicial}}(\mathrm{Ind}^m\mathrm{Banach}_{A/I})}_{\mathrm{Commutativealgebra}_{\mathrm{simplicial}}(\mathrm{Ind}^m\mathrm{Banach}_{A/I})^\mathrm{opposite},\mathrm{Grothendiecktopology,homotopyepimorphism}}. 
\end{align}
Then consider the derived prismatic object:
\begin{align}
\mathrm{Kan}_{\mathrm{Left}}\Delta_{?/A,\text{functionalanalytic,KKM},\text{BBM,formalanalytification}}(\mathcal{O}).
\end{align}	
Then as in \cite[Definition 8.2]{12BS} we have the following preperfectoidization:
\begin{align}
&(\mathcal{O})^{\text{preperfectoidization}}\\
&:=\mathrm{Colimit}(\mathrm{Kan}_{\mathrm{Left}}\Delta_{?/A,\text{functionalanalytic,KKM},\text{BBM,formalanalytification}}(\mathcal{O})\rightarrow \\
&\mathrm{Fro}_*\mathrm{Kan}_{\mathrm{Left}}\Delta_{?/A,\text{functionalanalytic,KKM},\text{BBM,formalanalytification}}(\mathcal{O})\\
&\rightarrow \mathrm{Fro}_* \mathrm{Fro}_*\mathrm{Kan}_{\mathrm{Left}}\Delta_{?/A,\text{functionalanalytic,KKM},\text{BBM,formalanalytification}}(\mathcal{O})\rightarrow...)^{\text{BBM,formalanalytification}},	
\end{align}
after taking the formal series ring left Kan extension analytification from \cite[Section 4.2]{BBM}, which is defined by taking the left Kan extension to all the $(\infty,1)$-ring object in the $\infty$-derived category of all $A$-modules from formal series rings over $A$. Then we define the corresponding perfectoidization:
\begin{align}
&(\mathcal{O})^{\text{perfectoidization}}\\
&:=\mathrm{Colimit}(\mathrm{Kan}_{\mathrm{Left}}\Delta_{?/A,\text{functionalanalytic,KKM},\text{BBM,formalanalytification}}(\mathcal{O})\longrightarrow \\
&\mathrm{Fro}_*\mathrm{Kan}_{\mathrm{Left}}\Delta_{?/A,\text{functionalanalytic,KKM},\text{BBM,formalanalytification}}(\mathcal{O})\\
&\longrightarrow \mathrm{Fro}_* \mathrm{Fro}_*\mathrm{Kan}_{\mathrm{Left}}\Delta_{?/A,\text{functionalanalytic,KKM},\text{BBM,formalanalytification}}(\mathcal{O})\longrightarrow...)^{\text{BBM,formalanalytification}}\times A/I.	
\end{align}
Furthermore one can take derived $(p,I)$-completion to achieve the derived $(p,I)$-completed versions:
\begin{align}
\mathcal{O}^\text{preperfectoidization,derivedcomplete}:=(\mathcal{O}^\text{preperfectoidization})^{\wedge},\\
\mathcal{O}^\text{perfectoidization,derivedcomplete}:=\mathcal{O}^\text{preperfectoidization,derivedcomplete}\times A/I.\\
\end{align}
These are large $(\infty,1)$-commutative algebra objects in the corresponding categories as in the above, attached to also large $(\infty,1)$-commutative algebra objects. When we apply this to the corresponding sub-$(\infty,1)$-categories of Banach perfectoid objects in \cite{BMS2}, \cite{GR}, \cite{12KL1}, \cite{12KL2}, \cite{12Ked1}, \cite{12Sch3},  we will recover the corresponding distinguished elemental deformation processes defined in \cite{BMS2}, \cite{GR}, \cite{12KL1}, \cite{12KL2}, \cite{12Ked1}, \cite{12Sch3}. 
\end{definition}

\

\indent The $\infty$-presheaves in this section in the $\infty$-category:	
\begin{align}
\mathrm{Ind}\mathrm{\sharp Quasicoherent}^{\text{presheaf}}_{\mathrm{Ind}^\text{smoothformalseriesclosure}\infty-\mathrm{Toposes}^{\mathrm{ringed},\mathrm{commutativealgebra}_{\mathrm{simplicial}}(\mathrm{Ind}\mathrm{Seminormed}_R)}_{\mathrm{Commutativealgebra}_{\mathrm{simplicial}}(\mathrm{Ind}\mathrm{Seminormed}_R)^\mathrm{opposite},\mathrm{Grotopology,homotopyepimorphism}}}. \\
\mathrm{Ind}\mathrm{\sharp Quasicoherent}^{\text{presheaf}}_{\mathrm{Ind}^\text{smoothformalseriesclosure}\infty-\mathrm{Toposes}^{\mathrm{ringed},\mathrm{Commutativealgebra}_{\mathrm{simplicial}}(\mathrm{Ind}^m\mathrm{Seminormed}_R)}_{\mathrm{Commutativealgebra}_{\mathrm{simplicial}}(\mathrm{Ind}^m\mathrm{Seminormed}_R)^\mathrm{opposite},\mathrm{Grotopology,homotopyepimorphism}}}.\\
\mathrm{Ind}\mathrm{\sharp Quasicoherent}^{\text{presheaf}}_{\mathrm{Ind}^\text{smoothformalseriesclosure}\infty-\mathrm{Toposes}^{\mathrm{ringed},\mathrm{Commutativealgebra}_{\mathrm{simplicial}}(\mathrm{Ind}\mathrm{Normed}_R)}_{\mathrm{Commutativealgebra}_{\mathrm{simplicial}}(\mathrm{Ind}\mathrm{Normed}_R)^\mathrm{opposite},\mathrm{Grotopology,homotopyepimorphism}}}.\\
\mathrm{Ind}\mathrm{\sharp Quasicoherent}^{\text{presheaf}}_{\mathrm{Ind}^\text{smoothformalseriesclosure}\infty-\mathrm{Toposes}^{\mathrm{ringed},\mathrm{Commutativealgebra}_{\mathrm{simplicial}}(\mathrm{Ind}^m\mathrm{Normed}_R)}_{\mathrm{Commutativealgebra}_{\mathrm{simplicial}}(\mathrm{Ind}^m\mathrm{Normed}_R)^\mathrm{opposite},\mathrm{Grotopology,homotopyepimorphism}}}.\\
\mathrm{Ind}\mathrm{\sharp Quasicoherent}^{\text{presheaf}}_{\mathrm{Ind}^\text{smoothformalseriesclosure}\infty-\mathrm{Toposes}^{\mathrm{ringed},\mathrm{Commutativealgebra}_{\mathrm{simplicial}}(\mathrm{Ind}\mathrm{Banach}_R)}_{\mathrm{Commutativealgebra}_{\mathrm{simplicial}}(\mathrm{Ind}\mathrm{Banach}_R)^\mathrm{opposite},\mathrm{Grotopology,homotopyepimorphism}}}.\\
\mathrm{Ind}\mathrm{\sharp Quasicoherent}^{\text{presheaf}}_{\mathrm{Ind}^\text{smoothformalseriesclosure}\infty-\mathrm{Toposes}^{\mathrm{ringed},\mathrm{Commutativealgebra}_{\mathrm{simplicial}}(\mathrm{Ind}^m\mathrm{Banach}_R)}_{\mathrm{Commutativealgebra}_{\mathrm{simplicial}}(\mathrm{Ind}^m\mathrm{Banach}_R)^\mathrm{opposite},\mathrm{Grotopology,homotopyepimorphism}}},\\
\end{align}
are expected to be $\infty$-sheaves as long as one considers in the admissible situations the corresponding \v{C}ech $\infty$-descent for general seminormed modules as in \cite[Section 9.3]{KKM} and \cite{12BBK} such as Bambozzi-Kremnizer spaces in \cite{12BK}. Therefore we have:\\

\begin{proposition}
The motivic complex $\infty$-presheaf 
\begin{align}
\mathrm{Kan}_{\mathrm{Left}}\Delta_{?/A,\mathrm{functionalanalytic,KKM},\mathrm{BBM,formalanalytification}}(\mathcal{O}),\\
\end{align}
as well as the corresponding Hodge-Tate $\infty$-presheaf as in \cite{12BS}{\footnote{Before the Ben-Bassat-Mukherjee $p$-adic formal analytification we take the corresponding derived $(p,I)$-completion.}}:
\begin{align}
\mathrm{Kan}_{\mathrm{Left}}&\Delta_{?/A,\mathrm{functionalanalytic,KKM},\mathrm{BBM,formalanalytification}}(\mathcal{O})^{\mathrm{HodgeTate}}\\
&:=\mathrm{Kan}_{\mathrm{Left}}\Delta_{?/A,\mathrm{functionalanalytic,KKM},\mathrm{BBM,formalanalytification}}(\mathcal{O})\times A/I\\
&=((\underset{i}{\mathrm{homotopycolimit}}_{\mathrm{sifted},\mathrm{derivedcategory}_{\infty}(A/I-\mathrm{Module})}\mathrm{Kan}_{\mathrm{Left}}\overline{\Delta}_{?/A,\mathrm{functionalanalytic,KKM}}(\mathcal{O}_i))^\wedge\\
&)_\mathrm{BBM,formalanalytification}	
\end{align}
and the preperfectoidization $\infty$-presheaves:
\begin{align}
&(\mathcal{O})^{\mathrm{preperfectoidization}},\\
&(\mathcal{O})^{\mathrm{perfectoidization}},\\
&(\mathcal{O})^{\mathrm{preperfectoidization,derivedcompleted}},\\
&(\mathcal{O})^{\mathrm{perfectoidization,derivedcompleted}}
\end{align}
are $\infty$-sheaves over Bambozzi-Kremnizer topos
\begin{align}
(\mathrm{Spectrumadic}_{\mathrm{BK}}(F),\mathcal{O}),	
\end{align}
attached to any colimit of formal series rings $F$ over $A/I$ in \cite{12BK}.\\	
\end{proposition}

\begin{remark}
One can then define such ring $\mathcal{O}$ to be \textit{preperfectoid} if we have the equivalence:
\begin{align}
\mathcal{O}^{\text{preperfectoidization}} \overset{\sim}{\longrightarrow}	\mathcal{O}.
\end{align}
One can then define such ring $\mathcal{O}$ to be \textit{perfectoid} if we have the equivalence:
\begin{align}
\mathcal{O}^{\text{preperfectoidization}}\times A/I \overset{\sim}{\longrightarrow}	\mathcal{O}.
\end{align}
	
\end{remark}

\newpage

\subsection{Functional Analytic Derived de Rham Complexes for $(\infty,1)$-Analytic Stacks and de Rham Preperfectoidizations}

\indent As in \cite{12LL} we have the comparison between the derived prismatic cohomology and the corresponding derived de Rham cohomology in some very well-defined way which respects the corresponding filtrations, we can then in our situation take the corresponding definition of some derived de Rham complex as the one side of the comparison from \cite{12LL}\footnote{Our goal here is actually study the corresponding derived de Rham period rings and the corresponding applications in $p$-adic Hodge theory extending work of \cite{12DLLZ1}, \cite{12DLLZ2}, \cite{12Sch2} in the motivation from \cite{12GL}, when applyting the construction to derived $p$-adic formal stacks and derived logarithmic $p$-adic formal stacks.}. To be more precise after \cite[Chapitre 3]{12An1}, \cite{12An2}, \cite[Chapter 2, Chapter 8]{12B1}, \cite[Chapter 1]{12Bei}, \cite[Chapter 5]{12G1}, \cite[Chapter 3, Chapter 4]{12GL}, \cite[Chapitre II, Chapitre III]{12Ill1}, \cite[Chapitre VIII]{12Ill2}, \cite[Section 4]{12Qui}, and \cite[Example 5.11, Example 5.12]{BMS2} we define the corresponding:

\begin{definition}
Then we can in the same fashion consider the corresponding derived de Rham complex presheaves for the commutative algebras as in the above (for a given prism $(A,I)$):
\begin{align}
\mathrm{Kan}_{\mathrm{Left}}\mathrm{deRham}_{?/A},	
\end{align}
\footnote{After taking derived $p$-completion.}by the regular corresponding left Kan extension techniques on the following $(\infty,1)$-compactly generated closures of the corresponding polynomials given over $A/I$ with a chosen prism $(A,I)$:\\

\begin{align}
\mathrm{Proj}^\text{smoothformalseriesclosure}\infty-\mathrm{Toposes}^{\mathrm{ringed},\mathrm{commutativealgebra}_{\mathrm{simplicial}}(\mathrm{Ind}\mathrm{Seminormed}_{A/I})}_{\mathrm{Commutativealgebra}_{\mathrm{simplicial}}(\mathrm{Ind}\mathrm{Seminormed}_{A/I})^\mathrm{opposite},\mathrm{Grothendiecktopology,homotopyepimorphism}}. \\
\mathrm{Proj}^\text{smoothformalseriesclosure}\infty-\mathrm{Toposes}^{\mathrm{ringed},\mathrm{Commutativealgebra}_{\mathrm{simplicial}}(\mathrm{Ind}^m\mathrm{Seminormed}_{A/I})}_{\mathrm{Commutativealgebra}_{\mathrm{simplicial}}(\mathrm{Ind}^m\mathrm{Seminormed}_{A/I})^\mathrm{opposite},\mathrm{Grothendiecktopology,homotopyepimorphism}}.\\
\mathrm{Proj}^\text{smoothformalseriesclosure}\infty-\mathrm{Toposes}^{\mathrm{ringed},\mathrm{Commutativealgebra}_{\mathrm{simplicial}}(\mathrm{Ind}\mathrm{Normed}_{A/I})}_{\mathrm{Commutativealgebra}_{\mathrm{simplicial}}(\mathrm{Ind}\mathrm{Normed}_{A/I})^\mathrm{opposite},\mathrm{Grothendiecktopology,homotopyepimorphism}}.\\
\mathrm{Proj}^\text{smoothformalseriesclosure}\infty-\mathrm{Toposes}^{\mathrm{ringed},\mathrm{Commutativealgebra}_{\mathrm{simplicial}}(\mathrm{Ind}^m\mathrm{Normed}_{A/I})}_{\mathrm{Commutativealgebra}_{\mathrm{simplicial}}(\mathrm{Ind}^m\mathrm{Normed}_{A/I})^\mathrm{opposite},\mathrm{Grothendiecktopology,homotopyepimorphism}}.\\
\mathrm{Proj}^\text{smoothformalseriesclosure}\infty-\mathrm{Toposes}^{\mathrm{ringed},\mathrm{Commutativealgebra}_{\mathrm{simplicial}}(\mathrm{Ind}\mathrm{Banach}_{A/I})}_{\mathrm{Commutativealgebra}_{\mathrm{simplicial}}(\mathrm{Ind}\mathrm{Banach}_{A/I})^\mathrm{opposite},\mathrm{Grothendiecktopology,homotopyepimorphism}}.\\
\mathrm{Proj}^\text{smoothformalseriesclosure}\infty-\mathrm{Toposes}^{\mathrm{ringed},\mathrm{Commutativealgebra}_{\mathrm{simplicial}}(\mathrm{Ind}^m\mathrm{Banach}_{A/I})}_{\mathrm{Commutativealgebra}_{\mathrm{simplicial}}(\mathrm{Ind}^m\mathrm{Banach}_{A/I})^\mathrm{opposite},\mathrm{Grothendiecktopology,homotopyepimorphism}}. 
\end{align}
We call the corresponding functors functional analytic derived de Rham complex presheaves which we are going to denote that as in the following:
\begin{align}
\mathrm{Kan}_{\mathrm{Left}}\mathrm{deRham}_{?/A,\text{functionalanalytic,KKM},\text{BBM,formalanalytification}}.	
\end{align}
This would mean the following definition{\footnote{Before the Ben-Bassat-Mukherjee $p$-adic formal analytification we take the corresponding derived $(p,I)$-completion.}}:
\begin{align}
\mathrm{Kan}_{\mathrm{Left}}&\mathrm{deRham}_{?/A,\text{functionalanalytic,KKM},\text{BBM,formalanalytification}}(\mathcal{O})\\
&:=	((\underset{i}{\text{homotopycolimit}}_{\text{sifted},\text{derivedcategory}_{\infty}(A/I-\text{Module})}\mathrm{Kan}_{\mathrm{Left}}\mathrm{deRham}_{?/A,\text{functionalanalytic,KKM}}\\
&(\mathcal{O}_i))^\wedge\\
&)_\text{BBM,formalanalytification}
\end{align}
by writing any object $\mathcal{O}$ as the corresponding colimit 
\begin{center}
$\underset{i}{\text{homotopycolimit}}_\text{sifted}\mathcal{O}_i$.
\end{center}
These are quite large $(\infty,1)$-commutative ring objects in the corresponding $(\infty,1)$-categories for $R=A/I$:
\begin{align}
\mathrm{Ind}\mathrm{\sharp Quasicoherent}^{\text{presheaf}}_{\mathrm{Ind}^\text{smoothformalseriesclosure}\infty-\mathrm{Toposes}^{\mathrm{ringed},\mathrm{commutativealgebra}_{\mathrm{simplicial}}(\mathrm{Ind}\mathrm{Seminormed}_R)}_{\mathrm{Commutativealgebra}_{\mathrm{simplicial}}(\mathrm{Ind}\mathrm{Seminormed}_R)^\mathrm{opposite},\mathrm{Grotopology,homotopyepimorphism}}}. \\
\mathrm{Ind}\mathrm{\sharp Quasicoherent}^{\text{presheaf}}_{\mathrm{Ind}^\text{smoothformalseriesclosure}\infty-\mathrm{Toposes}^{\mathrm{ringed},\mathrm{Commutativealgebra}_{\mathrm{simplicial}}(\mathrm{Ind}^m\mathrm{Seminormed}_R)}_{\mathrm{Commutativealgebra}_{\mathrm{simplicial}}(\mathrm{Ind}^m\mathrm{Seminormed}_R)^\mathrm{opposite},\mathrm{Grotopology,homotopyepimorphism}}}.\\
\mathrm{Ind}\mathrm{\sharp Quasicoherent}^{\text{presheaf}}_{\mathrm{Ind}^\text{smoothformalseriesclosure}\infty-\mathrm{Toposes}^{\mathrm{ringed},\mathrm{Commutativealgebra}_{\mathrm{simplicial}}(\mathrm{Ind}\mathrm{Normed}_R)}_{\mathrm{Commutativealgebra}_{\mathrm{simplicial}}(\mathrm{Ind}\mathrm{Normed}_R)^\mathrm{opposite},\mathrm{Grotopology,homotopyepimorphism}}}.\\
\mathrm{Ind}\mathrm{\sharp Quasicoherent}^{\text{presheaf}}_{\mathrm{Ind}^\text{smoothformalseriesclosure}\infty-\mathrm{Toposes}^{\mathrm{ringed},\mathrm{Commutativealgebra}_{\mathrm{simplicial}}(\mathrm{Ind}^m\mathrm{Normed}_R)}_{\mathrm{Commutativealgebra}_{\mathrm{simplicial}}(\mathrm{Ind}^m\mathrm{Normed}_R)^\mathrm{opposite},\mathrm{Grotopology,homotopyepimorphism}}}.\\
\mathrm{Ind}\mathrm{\sharp Quasicoherent}^{\text{presheaf}}_{\mathrm{Ind}^\text{smoothformalseriesclosure}\infty-\mathrm{Toposes}^{\mathrm{ringed},\mathrm{Commutativealgebra}_{\mathrm{simplicial}}(\mathrm{Ind}\mathrm{Banach}_R)}_{\mathrm{Commutativealgebra}_{\mathrm{simplicial}}(\mathrm{Ind}\mathrm{Banach}_R)^\mathrm{opposite},\mathrm{Grotopology,homotopyepimorphism}}}.\\
\mathrm{Ind}\mathrm{\sharp Quasicoherent}^{\text{presheaf}}_{\mathrm{Ind}^\text{smoothformalseriesclosure}\infty-\mathrm{Toposes}^{\mathrm{ringed},\mathrm{Commutativealgebra}_{\mathrm{simplicial}}(\mathrm{Ind}^m\mathrm{Banach}_R)}_{\mathrm{Commutativealgebra}_{\mathrm{simplicial}}(\mathrm{Ind}^m\mathrm{Banach}_R)^\mathrm{opposite},\mathrm{Grotopology,homotopyepimorphism}}},\\ 
\end{align}
after taking the formal series ring left Kan extension analytification from \cite[Section 4.2]{BBM}, which is defined by taking the left Kan extension to all the $(\infty,1)$-ring objects in the $\infty$-derived category of all $A$-modules from formal series rings over $A$.\\
\end{definition}

\noindent Now we apply this construction to any $A/I$ formal scheme $(X,\mathcal{O}_X)$. But in order to use prismatic technology to reach the corresponding period derived de Rham sheaves, we need to consider more, this would be some definition for 'perfectoidizations':

\begin{definition}
Let $(A,I)$ be a perfectoid prism, and we consider any $\mathrm{E}_\infty$-ring $\mathcal{O}$ in the following
\begin{align}
\mathrm{Proj}^\text{smoothformalseriesclosure}\infty-\mathrm{Toposes}^{\mathrm{ringed},\mathrm{commutativealgebra}_{\mathrm{simplicial}}(\mathrm{Ind}\mathrm{Seminormed}_{A/I})}_{\mathrm{Commutativealgebra}_{\mathrm{simplicial}}(\mathrm{Ind}\mathrm{Seminormed}_{A/I})^\mathrm{opposite},\mathrm{Grothendiecktopology,homotopyepimorphism}}. \\
\mathrm{Proj}^\text{smoothformalseriesclosure}\infty-\mathrm{Toposes}^{\mathrm{ringed},\mathrm{Commutativealgebra}_{\mathrm{simplicial}}(\mathrm{Ind}^m\mathrm{Seminormed}_{A/I})}_{\mathrm{Commutativealgebra}_{\mathrm{simplicial}}(\mathrm{Ind}^m\mathrm{Seminormed}_{A/I})^\mathrm{opposite},\mathrm{Grothendiecktopology,homotopyepimorphism}}.\\
\mathrm{Proj}^\text{smoothformalseriesclosure}\infty-\mathrm{Toposes}^{\mathrm{ringed},\mathrm{Commutativealgebra}_{\mathrm{simplicial}}(\mathrm{Ind}\mathrm{Normed}_{A/I})}_{\mathrm{Commutativealgebra}_{\mathrm{simplicial}}(\mathrm{Ind}\mathrm{Normed}_{A/I})^\mathrm{opposite},\mathrm{Grothendiecktopology,homotopyepimorphism}}.\\
\mathrm{Proj}^\text{smoothformalseriesclosure}\infty-\mathrm{Toposes}^{\mathrm{ringed},\mathrm{Commutativealgebra}_{\mathrm{simplicial}}(\mathrm{Ind}^m\mathrm{Normed}_{A/I})}_{\mathrm{Commutativealgebra}_{\mathrm{simplicial}}(\mathrm{Ind}^m\mathrm{Normed}_{A/I})^\mathrm{opposite},\mathrm{Grothendiecktopology,homotopyepimorphism}}.\\
\mathrm{Proj}^\text{smoothformalseriesclosure}\infty-\mathrm{Toposes}^{\mathrm{ringed},\mathrm{Commutativealgebra}_{\mathrm{simplicial}}(\mathrm{Ind}\mathrm{Banach}_{A/I})}_{\mathrm{Commutativealgebra}_{\mathrm{simplicial}}(\mathrm{Ind}\mathrm{Banach}_{A/I})^\mathrm{opposite},\mathrm{Grothendiecktopology,homotopyepimorphism}}.\\
\mathrm{Proj}^\text{smoothformalseriesclosure}\infty-\mathrm{Toposes}^{\mathrm{ringed},\mathrm{Commutativealgebra}_{\mathrm{simplicial}}(\mathrm{Ind}^m\mathrm{Banach}_{A/I})}_{\mathrm{Commutativealgebra}_{\mathrm{simplicial}}(\mathrm{Ind}^m\mathrm{Banach}_{A/I})^\mathrm{opposite},\mathrm{Grothendiecktopology,homotopyepimorphism}}. 
\end{align}
Then consider the derived prismatic object:
\begin{align}
\mathrm{Kan}_{\mathrm{Left}}\mathrm{deRham}_{?/A,\text{functionalanalytic,KKM},\text{BBM,formalanalytification}}(\mathcal{O}).
\end{align}	
Then as in \cite[Definition 8.2]{12BS} we have the following de Rham  preperfectoidization:
\begin{align}
&(\mathcal{O})^{\text{deRham,preperfectoidization}}\\
&:=\mathrm{Colimit}(\mathrm{Kan}_{\mathrm{Left}}\mathrm{deRham}_{?/A,\text{functionalanalytic,KKM},\text{BBM,formalanalytification}}(\mathcal{O})\rightarrow \\
&\mathrm{Fro}_*\mathrm{Kan}_{\mathrm{Left}}\mathrm{deRham}_{?/A,\text{functionalanalytic,KKM},\text{BBM,formalanalytification}}(\mathcal{O})\\
&\rightarrow \mathrm{Fro}_* \mathrm{Fro}_*\mathrm{Kan}_{\mathrm{Left}}\mathrm{deRham}_{?/A,\text{functionalanalytic,KKM},\text{BBM,formalanalytification}}(\mathcal{O})\rightarrow...\\
&)^{\text{BBM,formalanalytification}},	
\end{align}
after taking the formal series ring left Kan extension analytification from \cite[Section 4.2]{BBM}, which is defined by taking the left Kan extension to all the $(\infty,1)$-ring object in the $\infty$-derived category of all $A/I$-modules from formal series rings over $A/I$. 
Furthermore one can take derived $(p,I)$-completion to achieve the derived $(p,I)$-completed versions.

\end{definition}

\subsection{Functional Analytic Derived Prismatic Complexes for Inductive $(\infty,1)$-Analytic Stacks}

\indent We now promote the construction in the previous sections to the corresponding inductive systems of $(\infty,1)$-ringed toposes level after Lurie \cite{12Lu1}, \cite{12Lu2} and \cite{Lu3} in the $\infty$-category of $\infty$-ringed toposes, Bambozzi-Ben-Bassat-Kremnizer \cite{12BBBK}, Ben-Bassat-Mukherjee \cite{BBM}, Bambozzi-Kremnizer \cite{12BK}, Clausen-Scholze \cite{12CS1} \cite{12CS2} and Kelly-Kremnizer-Mukherjee \cite{KKM} in the $\infty$-cateogory of $\infty$-functional analytic ringed toposes.\\

\indent Now we consider the following $\infty$-categories of the corresponding $\infty$-analytic ringed toposes from Bambozzi-Ben-Bassat-Kremnizer \cite{12BBBK}:\\

\begin{align}
&\infty-\mathrm{Toposes}^{\mathrm{ringed},\mathrm{commutativealgebra}_{\mathrm{simplicial}}(\mathrm{Ind}\mathrm{Seminormed}_?)}_{\mathrm{Commutativealgebra}_{\mathrm{simplicial}}(\mathrm{Ind}\mathrm{Seminormed}_?)^\mathrm{opposite},\mathrm{Grothendiecktopology,homotopyepimorphism}}.\\
&\infty-\mathrm{Toposes}^{\mathrm{ringed},\mathrm{Commutativealgebra}_{\mathrm{simplicial}}(\mathrm{Ind}^m\mathrm{Seminormed}_?)}_{\mathrm{Commutativealgebra}_{\mathrm{simplicial}}(\mathrm{Ind}^m\mathrm{Seminormed}_?)^\mathrm{opposite},\mathrm{Grothendiecktopology,homotopyepimorphism}}.\\
&\infty-\mathrm{Toposes}^{\mathrm{ringed},\mathrm{Commutativealgebra}_{\mathrm{simplicial}}(\mathrm{Ind}\mathrm{Normed}_?)}_{\mathrm{Commutativealgebra}_{\mathrm{simplicial}}(\mathrm{Ind}\mathrm{Normed}_?)^\mathrm{opposite},\mathrm{Grothendiecktopology,homotopyepimorphism}}.\\
&\infty-\mathrm{Toposes}^{\mathrm{ringed},\mathrm{Commutativealgebra}_{\mathrm{simplicial}}(\mathrm{Ind}^m\mathrm{Normed}_?)}_{\mathrm{Commutativealgebra}_{\mathrm{simplicial}}(\mathrm{Ind}^m\mathrm{Normed}_?)^\mathrm{opposite},\mathrm{Grothendiecktopology,homotopyepimorphism}}.\\
&\infty-\mathrm{Toposes}^{\mathrm{ringed},\mathrm{Commutativealgebra}_{\mathrm{simplicial}}(\mathrm{Ind}\mathrm{Banach}_?)}_{\mathrm{Commutativealgebra}_{\mathrm{simplicial}}(\mathrm{Ind}\mathrm{Banach}_?)^\mathrm{opposite},\mathrm{Grothendiecktopology,homotopyepimorphism}}.\\
&\infty-\mathrm{Toposes}^{\mathrm{ringed},\mathrm{Commutativealgebra}_{\mathrm{simplicial}}(\mathrm{Ind}^m\mathrm{Banach}_?)}_{\mathrm{Commutativealgebra}_{\mathrm{simplicial}}(\mathrm{Ind}^m\mathrm{Banach}_?)^\mathrm{opposite},\mathrm{Grothendiecktopology,homotopyepimorphism}}.\\ 
&\mathrm{Ind}^\text{smoothformalseriesclosure}\infty-\mathrm{Toposes}^{\mathrm{ringed},\mathrm{commutativealgebra}_{\mathrm{simplicial}}(\mathrm{Ind}\mathrm{Seminormed}_?)}_{\mathrm{Commutativealgebra}_{\mathrm{simplicial}}(\mathrm{Ind}\mathrm{Seminormed}_?)^\mathrm{opposite},\mathrm{Grothendiecktopology,homotopyepimorphism}}. \\
&\mathrm{Ind}^\text{smoothformalseriesclosure}\infty-\mathrm{Toposes}^{\mathrm{ringed},\mathrm{Commutativealgebra}_{\mathrm{simplicial}}(\mathrm{Ind}^m\mathrm{Seminormed}_?)}_{\mathrm{Commutativealgebra}_{\mathrm{simplicial}}(\mathrm{Ind}^m\mathrm{Seminormed}_?)^\mathrm{opposite},\mathrm{Grothendiecktopology,homotopyepimorphism}}.\\
&\mathrm{Ind}^\text{smoothformalseriesclosure}\infty-\mathrm{Toposes}^{\mathrm{ringed},\mathrm{Commutativealgebra}_{\mathrm{simplicial}}(\mathrm{Ind}\mathrm{Normed}_?)}_{\mathrm{Commutativealgebra}_{\mathrm{simplicial}}(\mathrm{Ind}\mathrm{Normed}_?)^\mathrm{opposite},\mathrm{Grothendiecktopology,homotopyepimorphism}}.\\
&\mathrm{Ind}^\text{smoothformalseriesclosure}\infty-\mathrm{Toposes}^{\mathrm{ringed},\mathrm{Commutativealgebra}_{\mathrm{simplicial}}(\mathrm{Ind}^m\mathrm{Normed}_?)}_{\mathrm{Commutativealgebra}_{\mathrm{simplicial}}(\mathrm{Ind}^m\mathrm{Normed}_?)^\mathrm{opposite},\mathrm{Grothendiecktopology,homotopyepimorphism}}.\\
&\mathrm{Ind}^\text{smoothformalseriesclosure}\infty-\mathrm{Toposes}^{\mathrm{ringed},\mathrm{Commutativealgebra}_{\mathrm{simplicial}}(\mathrm{Ind}\mathrm{Banach}_?)}_{\mathrm{Commutativealgebra}_{\mathrm{simplicial}}(\mathrm{Ind}\mathrm{Banach}_?)^\mathrm{opposite},\mathrm{Grothendiecktopology,homotopyepimorphism}}.\\
&\mathrm{Ind}^\text{smoothformalseriesclosure}\infty-\mathrm{Toposes}^{\mathrm{ringed},\mathrm{Commutativealgebra}_{\mathrm{simplicial}}(\mathrm{Ind}^m\mathrm{Banach}_?)}_{\mathrm{Commutativealgebra}_{\mathrm{simplicial}}(\mathrm{Ind}^m\mathrm{Banach}_?)^\mathrm{opposite},\mathrm{Grothendiecktopology,homotopyepimorphism}}.\\ 
\end{align}

\begin{definition}
\indent One can actually define the derived prismatic cohomology presheaves through derived topological Hochschild cohomology presheaves, derived topological period cohomology presheaves and derived topological cyclic cohomology presheaves as in \cite[Section 2.2, Section 2.3]{12BMS}, \cite[Theorem 1.13]{12BS}:
\begin{align}
	\mathrm{Kan}_{\mathrm{Left}}\mathrm{THH},\mathrm{Kan}_{\mathrm{Left}}\mathrm{TP},\mathrm{Kan}_{\mathrm{Left}}\mathrm{TC},
\end{align}
on the following $(\infty,1)$-compactly generated closures of the corresponding polynomials\footnote{Definitely, we need to put certain norms over in some relatively canonical way, as in \cite[Section 4.2]{BBM} one can basically consider rigid ones and dagger ones, and so on. We restrict to the \textit{formal} one.} given over $A/I$ with a chosen prism $(A,I)$\footnote{In all the following, we assume this prism to be bounded and satisfy that $A/I$ is Banach.}:
\begin{align}
\mathrm{Ind}^\text{smoothformalseriesclosure}\infty-\mathrm{Toposes}^{\mathrm{ringed},\mathrm{commutativealgebra}_{\mathrm{simplicial}}(\mathrm{Ind}\mathrm{Seminormed}_{A/I})}_{\mathrm{Commutativealgebra}_{\mathrm{simplicial}}(\mathrm{Ind}\mathrm{Seminormed}_{A/I})^\mathrm{opposite},\mathrm{Grothendiecktopology,homotopyepimorphism}}. \\
\mathrm{Ind}^\text{smoothformalseriesclosure}\infty-\mathrm{Toposes}^{\mathrm{ringed},\mathrm{Commutativealgebra}_{\mathrm{simplicial}}(\mathrm{Ind}^m\mathrm{Seminormed}_{A/I})}_{\mathrm{Commutativealgebra}_{\mathrm{simplicial}}(\mathrm{Ind}^m\mathrm{Seminormed}_{A/I})^\mathrm{opposite},\mathrm{Grothendiecktopology,homotopyepimorphism}}.\\
\mathrm{Ind}^\text{smoothformalseriesclosure}\infty-\mathrm{Toposes}^{\mathrm{ringed},\mathrm{Commutativealgebra}_{\mathrm{simplicial}}(\mathrm{Ind}\mathrm{Normed}_{A/I})}_{\mathrm{Commutativealgebra}_{\mathrm{simplicial}}(\mathrm{Ind}\mathrm{Normed}_{A/I})^\mathrm{opposite},\mathrm{Grothendiecktopology,homotopyepimorphism}}.\\
\mathrm{Ind}^\text{smoothformalseriesclosure}\infty-\mathrm{Toposes}^{\mathrm{ringed},\mathrm{Commutativealgebra}_{\mathrm{simplicial}}(\mathrm{Ind}^m\mathrm{Normed}_{A/I})}_{\mathrm{Commutativealgebra}_{\mathrm{simplicial}}(\mathrm{Ind}^m\mathrm{Normed}_{A/I})^\mathrm{opposite},\mathrm{Grothendiecktopology,homotopyepimorphism}}.\\
\mathrm{Ind}^\text{smoothformalseriesclosure}\infty-\mathrm{Toposes}^{\mathrm{ringed},\mathrm{Commutativealgebra}_{\mathrm{simplicial}}(\mathrm{Ind}\mathrm{Banach}_{A/I})}_{\mathrm{Commutativealgebra}_{\mathrm{simplicial}}(\mathrm{Ind}\mathrm{Banach}_{A/I})^\mathrm{opposite},\mathrm{Grothendiecktopology,homotopyepimorphism}}.\\
\mathrm{Ind}^\text{smoothformalseriesclosure}\infty-\mathrm{Toposes}^{\mathrm{ringed},\mathrm{Commutativealgebra}_{\mathrm{simplicial}}(\mathrm{Ind}^m\mathrm{Banach}_{A/I})}_{\mathrm{Commutativealgebra}_{\mathrm{simplicial}}(\mathrm{Ind}^m\mathrm{Banach}_{A/I})^\mathrm{opposite},\mathrm{Grothendiecktopology,homotopyepimorphism}}. 
\end{align}
We call the corresponding functors are derived functional analytic Hochschild cohomology presheaves, derived functional analytic period cohomology presheaves and derived functional analytic cyclic cohomology presheaves, which we are going to denote these presheaves as in the following for any $\infty$-ringed topos $(\mathbb{X},\mathcal{O})=\underset{i}{\text{homotopycolimit}}(\mathbb{X}_i,\mathcal{O}_i)$:
\begin{align}
	&\mathrm{Kan}_{\mathrm{Left}}\mathrm{THH}_{\text{functionalanalytic,KKM},\text{BBM,formalanalytification}}(\mathcal{O}):=\\
	&(\underset{i}{\text{homotopylimit}}_{\text{sifted},\text{derivedcategory}_{\infty}(A/I-\text{Module})}\mathrm{Kan}_{\mathrm{Left}}\mathrm{THH}_{\text{functionalanalytic,KKM}}(\mathcal{O}_i)\\
	&)_\text{BBM,formalanalytification},\\
	&\mathrm{Kan}_{\mathrm{Left}}\mathrm{TP}_{\text{functionalanalytic,KKM},\text{BBM,formalanalytification}}(\mathcal{O}):=\\
	&(\underset{i}{\text{homotopylimit}}_{\text{sifted},\text{derivedcategory}_{\infty}(A/I-\text{Module})}\mathrm{Kan}_{\mathrm{Left}}\mathrm{TP}_{\text{functionalanalytic,KKM}}(\mathcal{O}_i)\\
	&)_\text{BBM,formalanalytification},\\
	&\mathrm{Kan}_{\mathrm{Left}}\mathrm{TC}_{\text{functionalanalytic,KKM},\text{BBM,formalanalytification}}(\mathcal{O}):=\\
	&(\underset{i}{\text{homotopylimit}}_{\text{sifted},\text{derivedcategory}_{\infty}(A/I-\text{Module})}\mathrm{Kan}_{\mathrm{Left}}\mathrm{TC}_{\text{functionalanalytic,KKM}}(\mathcal{O}_i)\\
	&)_\text{BBM,formalanalytification},
\end{align}
by writing any object $\mathcal{O}$ as the corresponding limit\footnote{Note that we are working with ind-$\infty$-ringed $\infty$-stacks.} 
\begin{center}
$\underset{i}{\text{homotopylimit}}_\text{sifted}\mathcal{O}_i:=\underset{i}{\text{homotopylimit}}_\text{sifted}\mathcal{O}_{\mathrm{Spectrumadic}_\mathrm{BK}(\mathrm{Formalseriesring}_i)}$.
\end{center}
These are quite large $(\infty,1)$-commutative ring objects in the corresponding $(\infty,1)$-categories for $?=A/I$ over $(\mathbb{X},\mathcal{O})$:

 \begin{align}
\mathrm{Ind}\mathrm{\sharp Quasicoherent}^{\text{presheaf}}_{\mathrm{Ind}^\text{smoothformalseriesclosure}\infty-\mathrm{Toposes}^{\mathrm{ringed},\mathrm{commutativealgebra}_{\mathrm{simplicial}}(\mathrm{Ind}\mathrm{Seminormed}_?)}_{\mathrm{Commutativealgebra}_{\mathrm{simplicial}}(\mathrm{Ind}\mathrm{Seminormed}_?)^\mathrm{opposite},\mathrm{Grotopology,homotopyepimorphism}}}. \\
\mathrm{Ind}\mathrm{\sharp Quasicoherent}^{\text{presheaf}}_{\mathrm{Ind}^\text{smoothformalseriesclosure}\infty-\mathrm{Toposes}^{\mathrm{ringed},\mathrm{Commutativealgebra}_{\mathrm{simplicial}}(\mathrm{Ind}^m\mathrm{Seminormed}_?)}_{\mathrm{Commutativealgebra}_{\mathrm{simplicial}}(\mathrm{Ind}^m\mathrm{Seminormed}_?)^\mathrm{opposite},\mathrm{Grotopology,homotopyepimorphism}}}.\\
\mathrm{Ind}\mathrm{\sharp Quasicoherent}^{\text{presheaf}}_{\mathrm{Ind}^\text{smoothformalseriesclosure}\infty-\mathrm{Toposes}^{\mathrm{ringed},\mathrm{Commutativealgebra}_{\mathrm{simplicial}}(\mathrm{Ind}\mathrm{Normed}_?)}_{\mathrm{Commutativealgebra}_{\mathrm{simplicial}}(\mathrm{Ind}\mathrm{Normed}_?)^\mathrm{opposite},\mathrm{Grotopology,homotopyepimorphism}}}.\\
\mathrm{Ind}\mathrm{\sharp Quasicoherent}^{\text{presheaf}}_{\mathrm{Ind}^\text{smoothformalseriesclosure}\infty-\mathrm{Toposes}^{\mathrm{ringed},\mathrm{Commutativealgebra}_{\mathrm{simplicial}}(\mathrm{Ind}^m\mathrm{Normed}_?)}_{\mathrm{Commutativealgebra}_{\mathrm{simplicial}}(\mathrm{Ind}^m\mathrm{Normed}_?)^\mathrm{opposite},\mathrm{Grotopology,homotopyepimorphism}}}.\\
\mathrm{Ind}\mathrm{\sharp Quasicoherent}^{\text{presheaf}}_{\mathrm{Ind}^\text{smoothformalseriesclosure}\infty-\mathrm{Toposes}^{\mathrm{ringed},\mathrm{Commutativealgebra}_{\mathrm{simplicial}}(\mathrm{Ind}\mathrm{Banach}_?)}_{\mathrm{Commutativealgebra}_{\mathrm{simplicial}}(\mathrm{Ind}\mathrm{Banach}_?)^\mathrm{opposite},\mathrm{Grotopology,homotopyepimorphism}}}.\\
\mathrm{Ind}\mathrm{\sharp Quasicoherent}^{\text{presheaf}}_{\mathrm{Ind}^\text{smoothformalseriesclosure}\infty-\mathrm{Toposes}^{\mathrm{ringed},\mathrm{Commutativealgebra}_{\mathrm{simplicial}}(\mathrm{Ind}^m\mathrm{Banach}_?)}_{\mathrm{Commutativealgebra}_{\mathrm{simplicial}}(\mathrm{Ind}^m\mathrm{Banach}_?)^\mathrm{opposite},\mathrm{Grotopology,homotopyepimorphism}}},\\ 
\end{align}
after taking the formal series ring left Kan extension analytification from \cite[Section 4.2]{BBM}, which is defined by taking the left Kan extension to all the $(\infty,1)$-ring objects in the $\infty$-derived category of all $A$-modules from formal series rings over $A$.
\end{definition}

\

\begin{definition}
Then we can in the same fashion consider the corresponding derived prismatic complex presheaves \cite[Construction 7.6]{12BS}\footnote{One just applies \cite[Construction 7.6]{12BS} and then takes the left Kan extensions.} for the commutative algebras as in the above (for a given prism $(A,I)$):
\begin{align}
\mathrm{Kan}_{\mathrm{Left}}\Delta_{?/A},	
\end{align}
by the regular corresponding left Kan extension techniques on the following $(\infty,1)$-compactly generated closures of the corresponding polynomials given over $A/I$ with a chosen prism $(A,I)$:\\

\begin{align}
\mathrm{Ind}^\text{smoothformalseriesclosure}\infty-\mathrm{Toposes}^{\mathrm{ringed},\mathrm{commutativealgebra}_{\mathrm{simplicial}}(\mathrm{Ind}\mathrm{Seminormed}_{A/I})}_{\mathrm{Commutativealgebra}_{\mathrm{simplicial}}(\mathrm{Ind}\mathrm{Seminormed}_{A/I})^\mathrm{opposite},\mathrm{Grothendiecktopology,homotopyepimorphism}}. \\
\mathrm{Ind}^\text{smoothformalseriesclosure}\infty-\mathrm{Toposes}^{\mathrm{ringed},\mathrm{Commutativealgebra}_{\mathrm{simplicial}}(\mathrm{Ind}^m\mathrm{Seminormed}_{A/I})}_{\mathrm{Commutativealgebra}_{\mathrm{simplicial}}(\mathrm{Ind}^m\mathrm{Seminormed}_{A/I})^\mathrm{opposite},\mathrm{Grothendiecktopology,homotopyepimorphism}}.\\
\mathrm{Ind}^\text{smoothformalseriesclosure}\infty-\mathrm{Toposes}^{\mathrm{ringed},\mathrm{Commutativealgebra}_{\mathrm{simplicial}}(\mathrm{Ind}\mathrm{Normed}_{A/I})}_{\mathrm{Commutativealgebra}_{\mathrm{simplicial}}(\mathrm{Ind}\mathrm{Normed}_{A/I})^\mathrm{opposite},\mathrm{Grothendiecktopology,homotopyepimorphism}}.\\
\mathrm{Ind}^\text{smoothformalseriesclosure}\infty-\mathrm{Toposes}^{\mathrm{ringed},\mathrm{Commutativealgebra}_{\mathrm{simplicial}}(\mathrm{Ind}^m\mathrm{Normed}_{A/I})}_{\mathrm{Commutativealgebra}_{\mathrm{simplicial}}(\mathrm{Ind}^m\mathrm{Normed}_{A/I})^\mathrm{opposite},\mathrm{Grothendiecktopology,homotopyepimorphism}}.\\
\mathrm{Ind}^\text{smoothformalseriesclosure}\infty-\mathrm{Toposes}^{\mathrm{ringed},\mathrm{Commutativealgebra}_{\mathrm{simplicial}}(\mathrm{Ind}\mathrm{Banach}_{A/I})}_{\mathrm{Commutativealgebra}_{\mathrm{simplicial}}(\mathrm{Ind}\mathrm{Banach}_{A/I})^\mathrm{opposite},\mathrm{Grothendiecktopology,homotopyepimorphism}}.\\
\mathrm{Ind}^\text{smoothformalseriesclosure}\infty-\mathrm{Toposes}^{\mathrm{ringed},\mathrm{Commutativealgebra}_{\mathrm{simplicial}}(\mathrm{Ind}^m\mathrm{Banach}_{A/I})}_{\mathrm{Commutativealgebra}_{\mathrm{simplicial}}(\mathrm{Ind}^m\mathrm{Banach}_{A/I})^\mathrm{opposite},\mathrm{Grothendiecktopology,homotopyepimorphism}}. 
\end{align}
We call the corresponding functors functional analytic derived prismatic complex presheaves which we are going to denote that as in the following:
\begin{align}
\mathrm{Kan}_{\mathrm{Left}}\Delta_{?/A,\text{functionalanalytic,KKM},\text{BBM,formalanalytification}}.	
\end{align}
This would mean the following definition{\footnote{Before the Ben-Bassat-Mukherjee $p$-adic formal analytification we take the corresponding derived $(p,I)$-completion.}}:
\begin{align}
&\mathrm{Kan}_{\mathrm{Left}}\Delta_{?/A,\text{functionalanalytic,KKM},\text{BBM,formalanalytification}}(\mathcal{O})\\
&:=	((\underset{i}{\text{homotopylimit}}_{\text{sifted},\text{derivedcategory}_{\infty}(A/I-\text{Module})}\mathrm{Kan}_{\mathrm{Left}}\Delta_{?/A,\text{functionalanalytic,KKM}}(\mathcal{O}_i))^\wedge\\
&)_\text{BBM,formalanalytification}
\end{align}
by writing any object $\mathcal{O}$ as the corresponding limit\footnote{Here we remind the readers of the corresponding foundation here, namely the presheaf $\mathcal{O}_i$ in fact takes the value in Koszul complex taking the following form:
\begin{align}
&\mathrm{Koszulcomplex}_{A/I\left<X_1,...,X_l\right>\left<T_1,...,T_m\right>}(a_1-T_1b_1,...,a_m-T_mb_m)\\
&= A/I\left<X_1,...,X_l\right>\left<T_1,...,T_m\right>/^\mathbb{L}(a_1-T_1b_1,...,a_m-T_mb_m).	
\end{align}
This is actually derived $p$-complete since each homotopy group is derived $p$-complete (from the corresponding Banach structure from $A/I\left<X_1,...,X_l\right>$ induced from the $p$-adic topology). Namely the definition of the presheaf:
\begin{align}
\mathrm{Kan}_{\mathrm{Left}}\Delta_{?/A,\text{functionalanalytic,KKM}}(\mathcal{O}_i)	
\end{align}
is directly the application of the derived prismatic functor from \cite[Construction 7.6]{12BS}.} 
\begin{center}
$\underset{i}{\text{homotopylimit}}_\text{sifted}\mathcal{O}_i$.
\end{center}
These are quite large $(\infty,1)$-commutative ring objects in the corresponding $(\infty,1)$-categories for $R=A/I$:
\begin{align}
\mathrm{Ind}\mathrm{\sharp Quasicoherent}^{\text{presheaf}}_{\mathrm{Ind}^\text{smoothformalseriesclosure}\infty-\mathrm{Toposes}^{\mathrm{ringed},\mathrm{commutativealgebra}_{\mathrm{simplicial}}(\mathrm{Ind}\mathrm{Seminormed}_R)}_{\mathrm{Commutativealgebra}_{\mathrm{simplicial}}(\mathrm{Ind}\mathrm{Seminormed}_R)^\mathrm{opposite},\mathrm{Grotopology,homotopyepimorphism}}}. \\
\mathrm{Ind}\mathrm{\sharp Quasicoherent}^{\text{presheaf}}_{\mathrm{Ind}^\text{smoothformalseriesclosure}\infty-\mathrm{Toposes}^{\mathrm{ringed},\mathrm{Commutativealgebra}_{\mathrm{simplicial}}(\mathrm{Ind}^m\mathrm{Seminormed}_R)}_{\mathrm{Commutativealgebra}_{\mathrm{simplicial}}(\mathrm{Ind}^m\mathrm{Seminormed}_R)^\mathrm{opposite},\mathrm{Grotopology,homotopyepimorphism}}}.\\
\mathrm{Ind}\mathrm{\sharp Quasicoherent}^{\text{presheaf}}_{\mathrm{Ind}^\text{smoothformalseriesclosure}\infty-\mathrm{Toposes}^{\mathrm{ringed},\mathrm{Commutativealgebra}_{\mathrm{simplicial}}(\mathrm{Ind}\mathrm{Normed}_R)}_{\mathrm{Commutativealgebra}_{\mathrm{simplicial}}(\mathrm{Ind}\mathrm{Normed}_R)^\mathrm{opposite},\mathrm{Grotopology,homotopyepimorphism}}}.\\
\mathrm{Ind}\mathrm{\sharp Quasicoherent}^{\text{presheaf}}_{\mathrm{Ind}^\text{smoothformalseriesclosure}\infty-\mathrm{Toposes}^{\mathrm{ringed},\mathrm{Commutativealgebra}_{\mathrm{simplicial}}(\mathrm{Ind}^m\mathrm{Normed}_R)}_{\mathrm{Commutativealgebra}_{\mathrm{simplicial}}(\mathrm{Ind}^m\mathrm{Normed}_R)^\mathrm{opposite},\mathrm{Grotopology,homotopyepimorphism}}}.\\
\mathrm{Ind}\mathrm{\sharp Quasicoherent}^{\text{presheaf}}_{\mathrm{Ind}^\text{smoothformalseriesclosure}\infty-\mathrm{Toposes}^{\mathrm{ringed},\mathrm{Commutativealgebra}_{\mathrm{simplicial}}(\mathrm{Ind}\mathrm{Banach}_R)}_{\mathrm{Commutativealgebra}_{\mathrm{simplicial}}(\mathrm{Ind}\mathrm{Banach}_R)^\mathrm{opposite},\mathrm{Grotopology,homotopyepimorphism}}}.\\
\mathrm{Ind}\mathrm{\sharp Quasicoherent}^{\text{presheaf}}_{\mathrm{Ind}^\text{smoothformalseriesclosure}\infty-\mathrm{Toposes}^{\mathrm{ringed},\mathrm{Commutativealgebra}_{\mathrm{simplicial}}(\mathrm{Ind}^m\mathrm{Banach}_R)}_{\mathrm{Commutativealgebra}_{\mathrm{simplicial}}(\mathrm{Ind}^m\mathrm{Banach}_R)^\mathrm{opposite},\mathrm{Grotopology,homotopyepimorphism}}},\\ 
\end{align}
after taking the formal series ring left Kan extension analytification from \cite[Section 4.2]{BBM}, which is defined by taking the left Kan extension to all the $(\infty,1)$-ring objects in the $\infty$-derived category of all $A$-modules from formal series rings over $A$.\\
\end{definition}

\newpage

\subsection{Functional Analytic Derived Preperfectoidizations for Inductive $(\infty,1)$-Analytic Stacks}

\

\noindent Then as in \cite[Definition 8.2]{12BS} we consider the corresponding perfectoidization in this analytic setting.

\begin{definition}
Let $(A,I)$ be a perfectoid prism, and we consider any $\mathrm{E}_\infty$-ring $\mathcal{O}$ in the following
\begin{align}
\mathrm{Ind}^\text{smoothformalseriesclosure}\infty-\mathrm{Toposes}^{\mathrm{ringed},\mathrm{commutativealgebra}_{\mathrm{simplicial}}(\mathrm{Ind}\mathrm{Seminormed}_{A/I})}_{\mathrm{Commutativealgebra}_{\mathrm{simplicial}}(\mathrm{Ind}\mathrm{Seminormed}_{A/I})^\mathrm{opposite},\mathrm{Grothendiecktopology,homotopyepimorphism}}. \\
\mathrm{Ind}^\text{smoothformalseriesclosure}\infty-\mathrm{Toposes}^{\mathrm{ringed},\mathrm{Commutativealgebra}_{\mathrm{simplicial}}(\mathrm{Ind}^m\mathrm{Seminormed}_{A/I})}_{\mathrm{Commutativealgebra}_{\mathrm{simplicial}}(\mathrm{Ind}^m\mathrm{Seminormed}_{A/I})^\mathrm{opposite},\mathrm{Grothendiecktopology,homotopyepimorphism}}.\\
\mathrm{Ind}^\text{smoothformalseriesclosure}\infty-\mathrm{Toposes}^{\mathrm{ringed},\mathrm{Commutativealgebra}_{\mathrm{simplicial}}(\mathrm{Ind}\mathrm{Normed}_{A/I})}_{\mathrm{Commutativealgebra}_{\mathrm{simplicial}}(\mathrm{Ind}\mathrm{Normed}_{A/I})^\mathrm{opposite},\mathrm{Grothendiecktopology,homotopyepimorphism}}.\\
\mathrm{Ind}^\text{smoothformalseriesclosure}\infty-\mathrm{Toposes}^{\mathrm{ringed},\mathrm{Commutativealgebra}_{\mathrm{simplicial}}(\mathrm{Ind}^m\mathrm{Normed}_{A/I})}_{\mathrm{Commutativealgebra}_{\mathrm{simplicial}}(\mathrm{Ind}^m\mathrm{Normed}_{A/I})^\mathrm{opposite},\mathrm{Grothendiecktopology,homotopyepimorphism}}.\\
\mathrm{Ind}^\text{smoothformalseriesclosure}\infty-\mathrm{Toposes}^{\mathrm{ringed},\mathrm{Commutativealgebra}_{\mathrm{simplicial}}(\mathrm{Ind}\mathrm{Banach}_{A/I})}_{\mathrm{Commutativealgebra}_{\mathrm{simplicial}}(\mathrm{Ind}\mathrm{Banach}_{A/I})^\mathrm{opposite},\mathrm{Grothendiecktopology,homotopyepimorphism}}.\\
\mathrm{Ind}^\text{smoothformalseriesclosure}\infty-\mathrm{Toposes}^{\mathrm{ringed},\mathrm{Commutativealgebra}_{\mathrm{simplicial}}(\mathrm{Ind}^m\mathrm{Banach}_{A/I})}_{\mathrm{Commutativealgebra}_{\mathrm{simplicial}}(\mathrm{Ind}^m\mathrm{Banach}_{A/I})^\mathrm{opposite},\mathrm{Grothendiecktopology,homotopyepimorphism}}. 
\end{align}
Then consider the derived prismatic object:
\begin{align}
\mathrm{Kan}_{\mathrm{Left}}\Delta_{?/A,\text{functionalanalytic,KKM},\text{BBM,formalanalytification}}(\mathcal{O}).
\end{align}	
Then as in \cite[Definition 8.2]{12BS} we have the following preperfectoidization:
\begin{align}
&(\mathcal{O})^{\text{preperfectoidization}}\\
&:=\mathrm{Colimit}(\mathrm{Kan}_{\mathrm{Left}}\Delta_{?/A,\text{functionalanalytic,KKM},\text{BBM,formalanalytification}}(\mathcal{O})\rightarrow \\
&\mathrm{Fro}_*\mathrm{Kan}_{\mathrm{Left}}\Delta_{?/A,\text{functionalanalytic,KKM},\text{BBM,formalanalytification}}(\mathcal{O})\\
&\rightarrow \mathrm{Fro}_* \mathrm{Fro}_*\mathrm{Kan}_{\mathrm{Left}}\Delta_{?/A,\text{functionalanalytic,KKM},\text{BBM,formalanalytification}}(\mathcal{O})\rightarrow...)^{\text{BBM,formalanalytification}},	
\end{align}
after taking the formal series ring left Kan extension analytification from \cite[Section 4.2]{BBM}, which is defined by taking the left Kan extension to all the $(\infty,1)$-ring object in the $\infty$-derived category of all $A$-modules from formal series rings over $A$. Then we define the corresponding perfectoidization:
\begin{align}
&(\mathcal{O})^{\text{perfectoidization}}\\
&:=\mathrm{Colimit}(\mathrm{Kan}_{\mathrm{Left}}\Delta_{?/A,\text{functionalanalytic,KKM},\text{BBM,formalanalytification}}(\mathcal{O})\longrightarrow \\
&\mathrm{Fro}_*\mathrm{Kan}_{\mathrm{Left}}\Delta_{?/A,\text{functionalanalytic,KKM},\text{BBM,formalanalytification}}(\mathcal{O})\\
&\longrightarrow \mathrm{Fro}_* \mathrm{Fro}_*\mathrm{Kan}_{\mathrm{Left}}\Delta_{?/A,\text{functionalanalytic,KKM},\text{BBM,formalanalytification}}(\mathcal{O})\longrightarrow...)^{\text{BBM,formalanalytification}}\times A/I.	
\end{align}
Furthermore one can take derived $(p,I)$-completion to achieve the derived $(p,I)$-completed versions:
\begin{align}
\mathcal{O}^\text{preperfectoidization,derivedcomplete}:=(\mathcal{O}^\text{preperfectoidization})^{\wedge},\\
\mathcal{O}^\text{perfectoidization,derivedcomplete}:=\mathcal{O}^\text{preperfectoidization,derivedcomplete}\times A/I.\\
\end{align}
These are large $(\infty,1)$-commutative algebra objects in the corresponding categories as in the above, attached to also large $(\infty,1)$-commutative algebra objects. When we apply this to the corresponding sub-$(\infty,1)$-categories of Banach perfectoid objects in \cite{BMS2}, \cite{GR}, \cite{12KL1}, \cite{12KL2}, \cite{12Ked1}, \cite{12Sch3},  we will recover the corresponding distinguished elemental deformation processes defined in \cite{BMS2}, \cite{GR}, \cite{12KL1}, \cite{12KL2}, \cite{12Ked1}, \cite{12Sch3}. 
\end{definition}

\begin{remark}
One can then define such ring $\mathcal{O}$ to be \textit{preperfectoid} if we have the equivalence:
\begin{align}
\mathcal{O}^{\text{preperfectoidization}} \overset{\sim}{\longrightarrow}	\mathcal{O}.
\end{align}
One can then define such ring $\mathcal{O}$ to be \textit{perfectoid} if we have the equivalence:
\begin{align}
\mathcal{O}^{\text{preperfectoidization}}\times A/I \overset{\sim}{\longrightarrow}	\mathcal{O}.
\end{align}
	
\end{remark}

\newpage

\subsection{Functional Analytic Derived de Rham Complexes for Inductive $(\infty,1)$-Analytic Stacks and de Rham Preperfectoidizations}

\indent As in \cite{12LL} we have the comparison between the derived prismatic cohomology and the corresponding derived de Rham cohomology in some very well-defined way which respects the corresponding filtrations, we can then in our situation take the corresponding definition of some derived de Rham complex as the one side of the comparison from \cite{12LL}\footnote{Our goal here is actually study the corresponding derived de Rham period rings and the corresponding applications in $p$-adic Hodge theory extending work of \cite{12DLLZ1}, \cite{12DLLZ2}, \cite{12Sch2} in the motivation from \cite{12GL}, when applyting the construction to derived $p$-adic formal stacks and derived logarithmic $p$-adic formal stacks.}. To be more precise after \cite[Chapitre 3]{12An1}, \cite{12An2}, \cite[Chapter 2, Chapter 8]{12B1}, \cite[Chapter 1]{12Bei}, \cite[Chapter 5]{12G1}, \cite[Chapter 3, Chapter 4]{12GL}, \cite[Chapitre II, Chapitre III]{12Ill1}, \cite[Chapitre VIII]{12Ill2}, \cite[Section 4]{12Qui}, and \cite[Example 5.11, Example 5.12]{BMS2} we define the corresponding:

\begin{definition}
Then we can in the same fashion consider the corresponding derived de Rham complex presheaves for the commutative algebras as in the above (for a given prism $(A,I)$):
\begin{align}
\mathrm{Kan}_{\mathrm{Left}}\mathrm{deRham}_{?/A},	
\end{align}
\footnote{After taking derived $p$-completion.}by the regular corresponding left Kan extension techniques on the following $(\infty,1)$-compactly generated closures of the corresponding polynomials given over $A/I$ with a chosen prism $(A,I)$:\\

\begin{align}
\mathrm{Ind}^\text{smoothformalseriesclosure}\infty-\mathrm{Toposes}^{\mathrm{ringed},\mathrm{commutativealgebra}_{\mathrm{simplicial}}(\mathrm{Ind}\mathrm{Seminormed}_{A/I})}_{\mathrm{Commutativealgebra}_{\mathrm{simplicial}}(\mathrm{Ind}\mathrm{Seminormed}_{A/I})^\mathrm{opposite},\mathrm{Grothendiecktopology,homotopyepimorphism}}. \\
\mathrm{Ind}^\text{smoothformalseriesclosure}\infty-\mathrm{Toposes}^{\mathrm{ringed},\mathrm{Commutativealgebra}_{\mathrm{simplicial}}(\mathrm{Ind}^m\mathrm{Seminormed}_{A/I})}_{\mathrm{Commutativealgebra}_{\mathrm{simplicial}}(\mathrm{Ind}^m\mathrm{Seminormed}_{A/I})^\mathrm{opposite},\mathrm{Grothendiecktopology,homotopyepimorphism}}.\\
\mathrm{Ind}^\text{smoothformalseriesclosure}\infty-\mathrm{Toposes}^{\mathrm{ringed},\mathrm{Commutativealgebra}_{\mathrm{simplicial}}(\mathrm{Ind}\mathrm{Normed}_{A/I})}_{\mathrm{Commutativealgebra}_{\mathrm{simplicial}}(\mathrm{Ind}\mathrm{Normed}_{A/I})^\mathrm{opposite},\mathrm{Grothendiecktopology,homotopyepimorphism}}.\\
\mathrm{Ind}^\text{smoothformalseriesclosure}\infty-\mathrm{Toposes}^{\mathrm{ringed},\mathrm{Commutativealgebra}_{\mathrm{simplicial}}(\mathrm{Ind}^m\mathrm{Normed}_{A/I})}_{\mathrm{Commutativealgebra}_{\mathrm{simplicial}}(\mathrm{Ind}^m\mathrm{Normed}_{A/I})^\mathrm{opposite},\mathrm{Grothendiecktopology,homotopyepimorphism}}.\\
\mathrm{Ind}^\text{smoothformalseriesclosure}\infty-\mathrm{Toposes}^{\mathrm{ringed},\mathrm{Commutativealgebra}_{\mathrm{simplicial}}(\mathrm{Ind}\mathrm{Banach}_{A/I})}_{\mathrm{Commutativealgebra}_{\mathrm{simplicial}}(\mathrm{Ind}\mathrm{Banach}_{A/I})^\mathrm{opposite},\mathrm{Grothendiecktopology,homotopyepimorphism}}.\\
\mathrm{Ind}^\text{smoothformalseriesclosure}\infty-\mathrm{Toposes}^{\mathrm{ringed},\mathrm{Commutativealgebra}_{\mathrm{simplicial}}(\mathrm{Ind}^m\mathrm{Banach}_{A/I})}_{\mathrm{Commutativealgebra}_{\mathrm{simplicial}}(\mathrm{Ind}^m\mathrm{Banach}_{A/I})^\mathrm{opposite},\mathrm{Grothendiecktopology,homotopyepimorphism}}. 
\end{align}
We call the corresponding functors functional analytic derived de Rham complex presheaves which we are going to denote that as in the following:
\begin{align}
\mathrm{Kan}_{\mathrm{Left}}\mathrm{deRham}_{?/A,\text{functionalanalytic,KKM},\text{BBM,formalanalytification}}.	
\end{align}
This would mean the following definition{\footnote{Before the Ben-Bassat-Mukherjee $p$-adic formal analytification we take the corresponding derived $(p,I)$-completion.}}:
\begin{align}
\mathrm{Kan}_{\mathrm{Left}}&\mathrm{deRham}_{?/A,\text{functionalanalytic,KKM},\text{BBM,formalanalytification}}(\mathcal{O})\\
&:=	((\underset{i}{\text{homotopylimit}}_{\text{sifted},\text{derivedcategory}_{\infty}(A/I-\text{Module})}\mathrm{Kan}_{\mathrm{Left}}\mathrm{deRham}_{?/A,\text{functionalanalytic,KKM}}\\
&(\mathcal{O}_i))^\wedge\\
&)_\text{BBM,formalanalytification}
\end{align}
by writing any object $\mathcal{O}$ as the corresponding colimit 
\begin{center}
$\underset{i}{\text{homotopylimit}}_\text{sifted}\mathcal{O}_i$.
\end{center}
These are quite large $(\infty,1)$-commutative ring objects in the corresponding $(\infty,1)$-categories for $R=A/I$:
\begin{align}
\mathrm{Ind}\mathrm{\sharp Quasicoherent}^{\text{presheaf}}_{\mathrm{Ind}^\text{smoothformalseriesclosure}\infty-\mathrm{Toposes}^{\mathrm{ringed},\mathrm{commutativealgebra}_{\mathrm{simplicial}}(\mathrm{Ind}\mathrm{Seminormed}_R)}_{\mathrm{Commutativealgebra}_{\mathrm{simplicial}}(\mathrm{Ind}\mathrm{Seminormed}_R)^\mathrm{opposite},\mathrm{Grotopology,homotopyepimorphism}}}. \\
\mathrm{Ind}\mathrm{\sharp Quasicoherent}^{\text{presheaf}}_{\mathrm{Ind}^\text{smoothformalseriesclosure}\infty-\mathrm{Toposes}^{\mathrm{ringed},\mathrm{Commutativealgebra}_{\mathrm{simplicial}}(\mathrm{Ind}^m\mathrm{Seminormed}_R)}_{\mathrm{Commutativealgebra}_{\mathrm{simplicial}}(\mathrm{Ind}^m\mathrm{Seminormed}_R)^\mathrm{opposite},\mathrm{Grotopology,homotopyepimorphism}}}.\\
\mathrm{Ind}\mathrm{\sharp Quasicoherent}^{\text{presheaf}}_{\mathrm{Ind}^\text{smoothformalseriesclosure}\infty-\mathrm{Toposes}^{\mathrm{ringed},\mathrm{Commutativealgebra}_{\mathrm{simplicial}}(\mathrm{Ind}\mathrm{Normed}_R)}_{\mathrm{Commutativealgebra}_{\mathrm{simplicial}}(\mathrm{Ind}\mathrm{Normed}_R)^\mathrm{opposite},\mathrm{Grotopology,homotopyepimorphism}}}.\\
\mathrm{Ind}\mathrm{\sharp Quasicoherent}^{\text{presheaf}}_{\mathrm{Ind}^\text{smoothformalseriesclosure}\infty-\mathrm{Toposes}^{\mathrm{ringed},\mathrm{Commutativealgebra}_{\mathrm{simplicial}}(\mathrm{Ind}^m\mathrm{Normed}_R)}_{\mathrm{Commutativealgebra}_{\mathrm{simplicial}}(\mathrm{Ind}^m\mathrm{Normed}_R)^\mathrm{opposite},\mathrm{Grotopology,homotopyepimorphism}}}.\\
\mathrm{Ind}\mathrm{\sharp Quasicoherent}^{\text{presheaf}}_{\mathrm{Ind}^\text{smoothformalseriesclosure}\infty-\mathrm{Toposes}^{\mathrm{ringed},\mathrm{Commutativealgebra}_{\mathrm{simplicial}}(\mathrm{Ind}\mathrm{Banach}_R)}_{\mathrm{Commutativealgebra}_{\mathrm{simplicial}}(\mathrm{Ind}\mathrm{Banach}_R)^\mathrm{opposite},\mathrm{Grotopology,homotopyepimorphism}}}.\\
\mathrm{Ind}\mathrm{\sharp Quasicoherent}^{\text{presheaf}}_{\mathrm{Ind}^\text{smoothformalseriesclosure}\infty-\mathrm{Toposes}^{\mathrm{ringed},\mathrm{Commutativealgebra}_{\mathrm{simplicial}}(\mathrm{Ind}^m\mathrm{Banach}_R)}_{\mathrm{Commutativealgebra}_{\mathrm{simplicial}}(\mathrm{Ind}^m\mathrm{Banach}_R)^\mathrm{opposite},\mathrm{Grotopology,homotopyepimorphism}}},\\ 
\end{align}
after taking the formal series ring left Kan extension analytification from \cite[Section 4.2]{BBM}, which is defined by taking the left Kan extension to all the $(\infty,1)$-ring objects in the $\infty$-derived category of all $A$-modules from formal series rings over $A$.\\
\end{definition}

\noindent Now we apply this construction to any $A/I$ formal scheme $(X,\mathcal{O}_X)$. But in order to use prismatic technology to reach the corresponding period derived de Rham sheaves, we need to consider more, this would be some definition for 'perfectoidizations':

\begin{definition}
Let $(A,I)$ be a perfectoid prism, and we consider any $\mathrm{E}_\infty$-ring $\mathcal{O}$ in the following
\begin{align}
\mathrm{Ind}^\text{smoothformalseriesclosure}\infty-\mathrm{Toposes}^{\mathrm{ringed},\mathrm{commutativealgebra}_{\mathrm{simplicial}}(\mathrm{Ind}\mathrm{Seminormed}_{A/I})}_{\mathrm{Commutativealgebra}_{\mathrm{simplicial}}(\mathrm{Ind}\mathrm{Seminormed}_{A/I})^\mathrm{opposite},\mathrm{Grothendiecktopology,homotopyepimorphism}}. \\
\mathrm{Ind}^\text{smoothformalseriesclosure}\infty-\mathrm{Toposes}^{\mathrm{ringed},\mathrm{Commutativealgebra}_{\mathrm{simplicial}}(\mathrm{Ind}^m\mathrm{Seminormed}_{A/I})}_{\mathrm{Commutativealgebra}_{\mathrm{simplicial}}(\mathrm{Ind}^m\mathrm{Seminormed}_{A/I})^\mathrm{opposite},\mathrm{Grothendiecktopology,homotopyepimorphism}}.\\
\mathrm{Ind}^\text{smoothformalseriesclosure}\infty-\mathrm{Toposes}^{\mathrm{ringed},\mathrm{Commutativealgebra}_{\mathrm{simplicial}}(\mathrm{Ind}\mathrm{Normed}_{A/I})}_{\mathrm{Commutativealgebra}_{\mathrm{simplicial}}(\mathrm{Ind}\mathrm{Normed}_{A/I})^\mathrm{opposite},\mathrm{Grothendiecktopology,homotopyepimorphism}}.\\
\mathrm{Ind}^\text{smoothformalseriesclosure}\infty-\mathrm{Toposes}^{\mathrm{ringed},\mathrm{Commutativealgebra}_{\mathrm{simplicial}}(\mathrm{Ind}^m\mathrm{Normed}_{A/I})}_{\mathrm{Commutativealgebra}_{\mathrm{simplicial}}(\mathrm{Ind}^m\mathrm{Normed}_{A/I})^\mathrm{opposite},\mathrm{Grothendiecktopology,homotopyepimorphism}}.\\
\mathrm{Ind}^\text{smoothformalseriesclosure}\infty-\mathrm{Toposes}^{\mathrm{ringed},\mathrm{Commutativealgebra}_{\mathrm{simplicial}}(\mathrm{Ind}\mathrm{Banach}_{A/I})}_{\mathrm{Commutativealgebra}_{\mathrm{simplicial}}(\mathrm{Ind}\mathrm{Banach}_{A/I})^\mathrm{opposite},\mathrm{Grothendiecktopology,homotopyepimorphism}}.\\
\mathrm{Ind}^\text{smoothformalseriesclosure}\infty-\mathrm{Toposes}^{\mathrm{ringed},\mathrm{Commutativealgebra}_{\mathrm{simplicial}}(\mathrm{Ind}^m\mathrm{Banach}_{A/I})}_{\mathrm{Commutativealgebra}_{\mathrm{simplicial}}(\mathrm{Ind}^m\mathrm{Banach}_{A/I})^\mathrm{opposite},\mathrm{Grothendiecktopology,homotopyepimorphism}}. 
\end{align}
Then consider the derived prismatic object:
\begin{align}
\mathrm{Kan}_{\mathrm{Left}}\mathrm{deRham}_{?/A,\text{functionalanalytic,KKM},\text{BBM,formalanalytification}}(\mathcal{O}).
\end{align}	
Then as in \cite[Definition 8.2]{12BS} we have the following de Rham  preperfectoidization:
\begin{align}
&(\mathcal{O})^{\text{deRham,preperfectoidization}}\\
&:=\mathrm{Colimit}(\mathrm{Kan}_{\mathrm{Left}}\mathrm{deRham}_{?/A,\text{functionalanalytic,KKM},\text{BBM,formalanalytification}}(\mathcal{O})\rightarrow \\
&\mathrm{Fro}_*\mathrm{Kan}_{\mathrm{Left}}\mathrm{deRham}_{?/A,\text{functionalanalytic,KKM},\text{BBM,formalanalytification}}(\mathcal{O})\\
&\rightarrow \mathrm{Fro}_* \mathrm{Fro}_*\mathrm{Kan}_{\mathrm{Left}}\mathrm{deRham}_{?/A,\text{functionalanalytic,KKM},\text{BBM,formalanalytification}}(\mathcal{O})\rightarrow...\\
&)^{\text{BBM,formalanalytification}},	
\end{align}
after taking the formal series ring left Kan extension analytification from \cite[Section 4.2]{BBM}, which is defined by taking the left Kan extension to all the $(\infty,1)$-ring object in the $\infty$-derived category of all $A/I$-modules from formal series rings over $A/I$. 
Furthermore one can take derived $(p,I)$-completion to achieve the derived $(p,I)$-completed versions.

\end{definition}

\newpage

\section{Functional Analytic Noncommutative Motives, Noncommutative Primatic Cohomologies and the Preperfectoidizations}

\indent We now work in the noncommutative setting after \cite{Kon1}, \cite{Ta}, \cite{KR1} and \cite{KR2}, with some philosophy rooted in some noncommtative motives and the corresponding nonabelian applications in noncommutative analytic geometry in the derived sense, and the noncommutative analogs of the corresponding Riemann hypothesis and the corresponding Tamagawa number conjectures. The issue is certainly that the usual Frobenius map looks strange, which tells us of the fact that actually we need to consider really large objects such as the corresponding Topological Hochschild Homologies and the corresponding nearby objects. Here we choose to consider \cite{12NS} in order to apply the constructions to certain $\infty$-rings, which we will call them Fukaya-Kato analytifications from \cite{12FK}. \\

\indent Now we consider the following $\infty$-categories of the corresponding $\infty$-analytic ringed toposes from Bambozzi-Ben-Bassat-Kremnizer \cite{12BBBK} in some paralle way:\\

\begin{align}
&\infty-\mathrm{Toposes}^{\mathrm{ringed},\mathrm{Noncommutativealgebra}_{\mathrm{simplicial}}(\mathrm{Ind}\mathrm{Seminormed}_?)}_{\mathrm{Noncommutativealgebra}_{\mathrm{simplicial}}(\mathrm{Ind}\mathrm{Seminormed}_?)^\mathrm{opposite},\mathrm{Grothendiecktopology,homotopyepimorphism}}.\\
&\infty-\mathrm{Toposes}^{\mathrm{ringed},\mathrm{Noncommutativealgebra}_{\mathrm{simplicial}}(\mathrm{Ind}^m\mathrm{Seminormed}_?)}_{\mathrm{Noncommutativealgebra}_{\mathrm{simplicial}}(\mathrm{Ind}^m\mathrm{Seminormed}_?)^\mathrm{opposite},\mathrm{Grothendiecktopology,homotopyepimorphism}}.\\
&\infty-\mathrm{Toposes}^{\mathrm{ringed},\mathrm{Noncommutativealgebra}_{\mathrm{simplicial}}(\mathrm{Ind}\mathrm{Normed}_?)}_{\mathrm{Noncommutativealgebra}_{\mathrm{simplicial}}(\mathrm{Ind}\mathrm{Normed}_?)^\mathrm{opposite},\mathrm{Grothendiecktopology,homotopyepimorphism}}.\\
&\infty-\mathrm{Toposes}^{\mathrm{ringed},\mathrm{Noncommutativealgebra}_{\mathrm{simplicial}}(\mathrm{Ind}^m\mathrm{Normed}_?)}_{\mathrm{Noncommutativealgebra}_{\mathrm{simplicial}}(\mathrm{Ind}^m\mathrm{Normed}_?)^\mathrm{opposite},\mathrm{Grothendiecktopology,homotopyepimorphism}}.\\
&\infty-\mathrm{Toposes}^{\mathrm{ringed},\mathrm{Noncommutativealgebra}_{\mathrm{simplicial}}(\mathrm{Ind}\mathrm{Banach}_?)}_{\mathrm{Noncommutativealgebra}_{\mathrm{simplicial}}(\mathrm{Ind}\mathrm{Banach}_?)^\mathrm{opposite},\mathrm{Grothendiecktopology,homotopyepimorphism}}.\\
&\infty-\mathrm{Toposes}^{\mathrm{ringed},\mathrm{Noncommutativealgebra}_{\mathrm{simplicial}}(\mathrm{Ind}^m\mathrm{Banach}_?)}_{\mathrm{Noncommutativealgebra}_{\mathrm{simplicial}}(\mathrm{Ind}^m\mathrm{Banach}_?)^\mathrm{opposite},\mathrm{Grothendiecktopology,homotopyepimorphism}}.\\ 
&\mathrm{Proj}^\text{smoothformalseriesclosure}\infty-\mathrm{Toposes}^{\mathrm{ringed},\mathrm{Noncommutativealgebra}_{\mathrm{simplicial}}(\mathrm{Ind}\mathrm{Seminormed}_?)}_{\mathrm{Noncommutativealgebra}_{\mathrm{simplicial}}(\mathrm{Ind}\mathrm{Seminormed}_?)^\mathrm{opposite},\mathrm{Grothendiecktopology,homotopyepimorphism}}. \\
&\mathrm{Proj}^\text{smoothformalseriesclosure}\infty-\mathrm{Toposes}^{\mathrm{ringed},\mathrm{Noncommutativealgebra}_{\mathrm{simplicial}}(\mathrm{Ind}^m\mathrm{Seminormed}_?)}_{\mathrm{Noncommutativealgebra}_{\mathrm{simplicial}}(\mathrm{Ind}^m\mathrm{Seminormed}_?)^\mathrm{opposite},\mathrm{Grothendiecktopology,homotopyepimorphism}}.\\
&\mathrm{Proj}^\text{smoothformalseriesclosure}\infty-\mathrm{Toposes}^{\mathrm{ringed},\mathrm{Noncommutativealgebra}_{\mathrm{simplicial}}(\mathrm{Ind}\mathrm{Normed}_?)}_{\mathrm{Noncommutativealgebra}_{\mathrm{simplicial}}(\mathrm{Ind}\mathrm{Normed}_?)^\mathrm{opposite},\mathrm{Grothendiecktopology,homotopyepimorphism}}.\\
&\mathrm{Proj}^\text{smoothformalseriesclosure}\infty-\mathrm{Toposes}^{\mathrm{ringed},\mathrm{Noncommutativealgebra}_{\mathrm{simplicial}}(\mathrm{Ind}^m\mathrm{Normed}_?)}_{\mathrm{Noncommutativealgebra}_{\mathrm{simplicial}}(\mathrm{Ind}^m\mathrm{Normed}_?)^\mathrm{opposite},\mathrm{Grothendiecktopology,homotopyepimorphism}}.\\
&\mathrm{Proj}^\text{smoothformalseriesclosure}\infty-\mathrm{Toposes}^{\mathrm{ringed},\mathrm{Noncommutativealgebra}_{\mathrm{simplicial}}(\mathrm{Ind}\mathrm{Banach}_?)}_{\mathrm{Noncommutativealgebra}_{\mathrm{simplicial}}(\mathrm{Ind}\mathrm{Banach}_?)^\mathrm{opposite},\mathrm{Grothendiecktopology,homotopyepimorphism}}.\\
&\mathrm{Proj}^\text{smoothformalseriesclosure}\infty-\mathrm{Toposes}^{\mathrm{ringed},\mathrm{Noncommutativealgebra}_{\mathrm{simplicial}}(\mathrm{Ind}^m\mathrm{Banach}_?)}_{\mathrm{Noncommutativealgebra}_{\mathrm{simplicial}}(\mathrm{Ind}^m\mathrm{Banach}_?)^\mathrm{opposite},\mathrm{Grothendiecktopology,homotopyepimorphism}}.\\ 
\end{align}

\indent Starting from the $\infty$-rings, we have:\\

\begin{definition}
\indent One can actually define the derived prismatic cohomologies through derived topological Hochschild cohomologies, derived topological period cohomologies and derived topological cyclic cohomologies as in \cite[Section 2.2, Section 2.3]{12BMS}, \cite[Theorem 1.13]{12BS}:
\begin{align}
	\mathrm{Kan}_{\mathrm{Left}}\mathrm{THH}^\mathrm{noncommutative},\mathrm{Kan}_{\mathrm{Left}}\mathrm{TP}^\mathrm{noncommutative},\mathrm{Kan}_{\mathrm{Left}}\mathrm{TC}^\mathrm{noncommutative},
\end{align}
\footnote{One has the corresponding $p$-completed versions as well.}on the following $(\infty,1)$-compactly generated closures of the corresponding polynomials\footnote{Definitely, we need to put certain norms over in some relatively canonical way, as in \cite[Section 4.2]{BBM} one can basically consider rigid ones and dagger ones, and so on. In this noncommutative setting we do not actually need to fix the type of the analytification, as in commutative setting since we are going to apply directly construction from Bhatt-Scholze and Koshikawa \cite{12BS} and \cite{12Ko1}, though the derived characterization of the prismatic cohomology might not need to be restricted to the corresponding $p$-adic formal scheme situations.} given over $A/I$ with a chosen prism $(A,I)$\footnote{In all the following, we assume this prism to be bounded and satisfy that $A/I$ is Banach.}:

\begin{align}
\mathrm{Object}_{\mathrm{E}_\infty\mathrm{Noncommutativealgebra},\mathrm{Simplicial}}(\mathrm{IndSNorm}_{A/I})^{\mathrm{smoothformalseriesclosure}},\\
\mathrm{Object}_{\mathrm{E}_\infty\mathrm{Noncommutativealgebra},\mathrm{Simplicial}}(\mathrm{Ind}^m\mathrm{SNorm}_{A/I})^{\mathrm{smoothformalseriesclosure}},\\
\mathrm{Object}_{\mathrm{E}_\infty\mathrm{Noncommutativealgebra},\mathrm{Simplicial}}(\mathrm{IndNorm}_{A/I})^{\mathrm{smoothformalseriesclosure}},\\
\mathrm{Object}_{\mathrm{E}_\infty\mathrm{Noncommutativealgebra},\mathrm{Simplicial}}(\mathrm{Ind}^m\mathrm{Norm}_{A/I})^{\mathrm{smoothformalseriesclosure}},\\
\mathrm{Object}_{\mathrm{E}_\infty\mathrm{Noncommutativealgebra},\mathrm{Simplicial}}(\mathrm{IndBan}_{A/I})^{\mathrm{smoothformalseriesclosure}},\\
\mathrm{Object}_{\mathrm{E}_\infty\mathrm{Noncommutativealgebra},\mathrm{Simplicial}}(\mathrm{Ind}^m\mathrm{Ban}_{A/I})^{\mathrm{smoothformalseriesclosure}}.
\end{align}
We call the corresponding functors are derived functional analytic Hochschild cohomologies, derived functional analytic period cohomologies and derived functional analytic cyclic cohomologies, which we are going to denote them as in the following:
\begin{align}
	&\mathrm{Kan}_{\mathrm{Left}}\mathrm{THH}^\mathrm{noncommutative}_{\text{functionalanalytic,KKM},\text{BBM,formalanalytification,nc}}(\mathcal{O}):=\\
	&(\underset{i}{\text{homotopycolimit}}_{\text{sifted},\text{derivedcategory}_{\infty}(A/I-\text{Module})}\mathrm{Kan}_{\mathrm{Left}}\mathrm{THH}^\mathrm{noncommutative}_{\text{functionalanalytic,KKM}}(\mathcal{O}_i)\\
	&)_\text{BBM,formalanalytification,nc},\\
	&\mathrm{Kan}_{\mathrm{Left}}\mathrm{TP}^\mathrm{noncommutative}_{\text{functionalanalytic,KKM},\text{BBM,formalanalytification,nc}}(\mathcal{O}):=\\
	&(\underset{i}{\text{homotopycolimit}}_{\text{sifted},\text{derivedcategory}_{\infty}(A/I-\text{Module})}\mathrm{Kan}_{\mathrm{Left}}\mathrm{TP}^\mathrm{noncommutative}_{\text{functionalanalytic,KKM}}(\mathcal{O}_i)\\
	&)_\text{BBM,formalanalytification,nc},\\
	&\mathrm{Kan}_{\mathrm{Left}}\mathrm{TC}^\mathrm{noncommutative}_{\text{functionalanalytic,KKM},\text{BBM,formalanalytification,nc}}(\mathcal{O}):=\\
	&(\underset{i}{\text{homotopycolimit}}_{\text{sifted},\text{derivedcategory}_{\infty}(A/I-\text{Module})}\mathrm{Kan}_{\mathrm{Left}}\mathrm{TC}^\mathrm{noncommutative}_{\text{functionalanalytic,KKM}}(\mathcal{O}_i)\\
	&)_\text{BBM,formalanalytification,nc},
\end{align}
by writing any object $\mathcal{O}$ as the corresponding colimit 
\begin{center}
$\underset{i}{\text{homotopycolimit}}_\text{sifted}\mathcal{O}_i$.
\end{center}
These are quite large $(\infty,1)$-commutative ring objects in the corresponding $(\infty,1)$-categories for $R=A/I$:

\begin{align}
\mathrm{Object}_{\mathrm{E}_\infty\mathrm{Noncommutativealgebra},\mathrm{Simplicial}}(\mathrm{IndSNorm}_R),\\
\mathrm{Object}_{\mathrm{E}_\infty\mathrm{Noncommutativealgebra},\mathrm{Simplicial}}(\mathrm{Ind}^m\mathrm{SNorm}_R),\\
\mathrm{Object}_{\mathrm{E}_\infty\mathrm{Noncommutativealgebra},\mathrm{Simplicial}}(\mathrm{IndNorm}_R),\\
\mathrm{Object}_{\mathrm{E}_\infty\mathrm{Noncommutativealgebra},\mathrm{Simplicial}}(\mathrm{Ind}^m\mathrm{Norm}_R),\\
\mathrm{Object}_{\mathrm{E}_\infty\mathrm{Noncommutativealgebra},\mathrm{Simplicial}}(\mathrm{IndBan}_R),\\
\mathrm{Object}_{\mathrm{E}_\infty\mathrm{Noncommutativealgebra},\mathrm{Simplicial}}(\mathrm{Ind}^m\mathrm{Ban}_R),
\end{align}
after taking the formal series ring left Kan extension analytification from \cite[Section 4.2]{BBM}, which is defined by taking the left Kan extension to all the $(\infty,1)$-ring objects in the $\infty$-derived category of all $A$-modules from formal series rings over $A$, into:
\begin{align}
\mathrm{Object}_{\mathrm{E}_\infty\mathrm{Noncommutativealgebra},\mathrm{Simplicial}}(\mathrm{IndSNorm}_{\mathbb{F}_1})_A,\\
\mathrm{Object}_{\mathrm{E}_\infty\mathrm{Noncommutativealgebra},\mathrm{Simplicial}}(\mathrm{Ind}^m\mathrm{SNorm}_{\mathbb{F}_1})_A,\\
\mathrm{Object}_{\mathrm{E}_\infty\mathrm{Noncommutativealgebra},\mathrm{Simplicial}}(\mathrm{IndNorm}_{\mathbb{F}_1})_A,\\
\mathrm{Object}_{\mathrm{E}_\infty\mathrm{Noncommutativealgebra},\mathrm{Simplicial}}(\mathrm{Ind}^m\mathrm{Norm}_{\mathbb{F}_1})_A,\\
\mathrm{Object}_{\mathrm{E}_\infty\mathrm{Noncommutativealgebra},\mathrm{Simplicial}}(\mathrm{IndBan}_{\mathbb{F}_1})_A,\\
\mathrm{Object}_{\mathrm{E}_\infty\mathrm{Noncommutativealgebra},\mathrm{Simplicial}}(\mathrm{Ind}^m\mathrm{Ban}_{\mathbb{F}_1})_A.
\end{align}
\end{definition}

\

\begin{definition}
Then we can in the same fashion consider the corresponding derived prismatic complexes \cite[Construction 7.6]{12BS}\footnote{One just applies \cite[Construction 7.6]{12BS} and then takes the left Kan extensions.} for the commutative algebras as in the above (for a given prism $(A,I)$):
\begin{align}
\mathrm{Kan}_{\mathrm{Left}}\Delta_{?/A}:=\mathrm{Kan}_{\mathrm{Left}}\mathrm{TP}^\mathrm{noncommutative,pcomplete}(?/A),	
\end{align}
by the regular corresponding left Kan extension techniques on the following $(\infty,1)$-compactly generated closures of the corresponding polynomials given over $A/I$ with a chosen prism $(A,I)$:

\begin{align}
\mathrm{Object}_{\mathrm{E}_\infty\mathrm{Noncommutativealgebra},\mathrm{Simplicial}}(\mathrm{IndSNorm}_{A/I})^{\mathrm{smoothformalseriesclosure}},\\
\mathrm{Object}_{\mathrm{E}_\infty\mathrm{Noncommutativealgebra},\mathrm{Simplicial}}(\mathrm{Ind}^m\mathrm{SNorm}_{A/I})^{\mathrm{smoothformalseriesclosure}},\\
\mathrm{Object}_{\mathrm{E}_\infty\mathrm{Noncommutativealgebra},\mathrm{Simplicial}}(\mathrm{IndNorm}_{A/I})^{\mathrm{smoothformalseriesclosure}},\\
\mathrm{Object}_{\mathrm{E}_\infty\mathrm{Noncommutativealgebra},\mathrm{Simplicial}}(\mathrm{Ind}^m\mathrm{Norm}_{A/I})^{\mathrm{smoothformalseriesclosure}},\\
\mathrm{Object}_{\mathrm{E}_\infty\mathrm{Noncommutativealgebra},\mathrm{Simplicial}}(\mathrm{IndBan}_{A/I})^{\mathrm{smoothformalseriesclosure}},\\
\mathrm{Object}_{\mathrm{E}_\infty\mathrm{Noncommutativealgebra},\mathrm{Simplicial}}(\mathrm{Ind}^m\mathrm{Ban}_{A/I})^{\mathrm{smoothformalseriesclosure}}.
\end{align}
We call the corresponding functors functional analytic derived prismatic complexes which we are going to denote that as in the following:
\begin{align}
\mathrm{Kan}_{\mathrm{Left}}&\Delta_{?/A,\text{functionalanalytic,KKM},\text{BBM,formalanalytification,nc}}\\
&:=\mathrm{Kan}_{\mathrm{Left}}\mathrm{TP}^\mathrm{noncommutative,pcomplete}(?/A)_{\text{functionalanalytic,KKM},\text{BBM,formalanalytification,nc}}.	
\end{align}
This would mean the following definition{\footnote{Before the Ben-Bassat-Mukherjee $p$-adic formal analytification we take the corresponding $p$-completion.}}:
\begin{align}
&\mathrm{Kan}_{\mathrm{Left}}\Delta_{?/A,\text{functionalanalytic,KKM},\text{BBM,formalanalytification,nc}}(\mathcal{O})\\
&:=	((\underset{i}{\text{homotopycolimit}}_{\text{sifted},\text{derivedcategory}_{\infty}(A/I-\text{Module})}\mathrm{Kan}_{\mathrm{Left}}\Delta_{?/A,\text{functionalanalytic,KKM}}(\mathcal{O}_i))^\wedge\\
&)_\text{BBM,formalanalytification,nc}\\
&=((\underset{i}{\text{homotopycolimit}}_{\text{sifted},\text{derivedcategory}_{\infty}(A/I-\text{Module})}\mathrm{Kan}_{\mathrm{Left}}\mathrm{TP}^\mathrm{noncommutative,pcomplete}\\
&(?/A)_{\text{functionalanalytic,KKM}}(\mathcal{O}_i))^\wedge)_\text{BBM,formalanalytification,nc}\\
\end{align}
by writing any object $\mathcal{O}$ as the corresponding colimit 
\begin{center}
$\underset{i}{\text{homotopycolimit}}_\text{sifted}\mathcal{O}_i$.
\end{center}
These are quite large $(\infty,1)$-commutative ring objects in the corresponding $(\infty,1)$-categories for $R=A/I$:
\begin{align}
\mathrm{Object}_{\mathrm{E}_\infty\mathrm{Noncommutativealgebra},\mathrm{Simplicial}}(\mathrm{IndSNorm}_R),\\
\mathrm{Object}_{\mathrm{E}_\infty\mathrm{Noncommutativealgebra},\mathrm{Simplicial}}(\mathrm{Ind}^m\mathrm{SNorm}_R),\\
\mathrm{Object}_{\mathrm{E}_\infty\mathrm{Noncommutativealgebra},\mathrm{Simplicial}}(\mathrm{IndNorm}_R),\\
\mathrm{Object}_{\mathrm{E}_\infty\mathrm{Noncommutativealgebra},\mathrm{Simplicial}}(\mathrm{Ind}^m\mathrm{Norm}_R),\\
\mathrm{Object}_{\mathrm{E}_\infty\mathrm{Noncommutativealgebra},\mathrm{Simplicial}}(\mathrm{IndBan}_R),\\
\mathrm{Object}_{\mathrm{E}_\infty\mathrm{Noncommutativealgebra},\mathrm{Simplicial}}(\mathrm{Ind}^m\mathrm{Ban}_R),\\
\end{align}
after taking the formal series ring left Kan extension analytification from \cite[Section 4.2]{BBM}, which is defined by taking the left Kan extension to all the $(\infty,1)$-ring objects in the $\infty$-derived category of all $A$-modules from formal series rings over $A$, into:
\begin{align}
\mathrm{Object}_{\mathrm{E}_\infty\mathrm{Noncommutativealgebra},\mathrm{Simplicial}}(\mathrm{IndSNorm}_{\mathbb{F}_1})_A,\\
\mathrm{Object}_{\mathrm{E}_\infty\mathrm{Noncommutativealgebra},\mathrm{Simplicial}}(\mathrm{Ind}^m\mathrm{SNorm}_{\mathbb{F}_1})_A,\\
\mathrm{Object}_{\mathrm{E}_\infty\mathrm{Noncommutativealgebra},\mathrm{Simplicial}}(\mathrm{IndNorm}_{\mathbb{F}_1})_A,\\
\mathrm{Object}_{\mathrm{E}_\infty\mathrm{Noncommutativealgebra},\mathrm{Simplicial}}(\mathrm{Ind}^m\mathrm{Norm}_{\mathbb{F}_1})_A,\\
\mathrm{Object}_{\mathrm{E}_\infty\mathrm{Noncommutativealgebra},\mathrm{Simplicial}}(\mathrm{IndBan}_{\mathbb{F}_1})_A,\\
\mathrm{Object}_{\mathrm{E}_\infty\mathrm{Noncommutativealgebra},\mathrm{Simplicial}}(\mathrm{Ind}^m\mathrm{Ban}_{\mathbb{F}_1})_A.
\end{align}
\end{definition}

\

\indent Then as in \cite[Definition 8.2]{12BS} we consider the corresponding 'perfectoidization' in this analytic setting.

\begin{definition}
Let $(A,I)$ be a perfectoid prism, and we consider any $\mathrm{E}_\infty$-ring $\mathcal{O}$ in the following
\begin{align}
\mathrm{Object}_{\mathrm{E}_\infty\mathrm{Noncommutativealgebra},\mathrm{Simplicial}}(\mathrm{IndSNorm}_{A/I})^{\mathrm{smoothformalseriesclosure}},\\
\mathrm{Object}_{\mathrm{E}_\infty\mathrm{Noncommutativealgebra},\mathrm{Simplicial}}(\mathrm{Ind}^m\mathrm{SNorm}_{A/I})^{\mathrm{smoothformalseriesclosure}},\\
\mathrm{Object}_{\mathrm{E}_\infty\mathrm{Noncommutativealgebra},\mathrm{Simplicial}}(\mathrm{IndNorm}_{A/I})^{\mathrm{smoothformalseriesclosure}},\\
\mathrm{Object}_{\mathrm{E}_\infty\mathrm{Noncommutativealgebra},\mathrm{Simplicial}}(\mathrm{Ind}^m\mathrm{Norm}_{A/I})^{\mathrm{smoothformalseriesclosure}},\\
\mathrm{Object}_{\mathrm{E}_\infty\mathrm{Noncommutativealgebra},\mathrm{Simplicial}}(\mathrm{IndBan}_{A/I})^{\mathrm{smoothformalseriesclosure}},\\
\mathrm{Object}_{\mathrm{E}_\infty\mathrm{Noncommutativealgebra},\mathrm{Simplicial}}(\mathrm{Ind}^m\mathrm{Ban}_{A/I})^{\mathrm{smoothformalseriesclosure}}.
\end{align}
Then consider the derived prismatic object:
\begin{align}
\mathrm{Kan}_{\mathrm{Left}}&\Delta_{?/A,\text{functionalanalytic,KKM},\text{BBM,formalanalytification,nc}}(\mathcal{O})\\
&:=\mathrm{Kan}_{\mathrm{Left}}\mathrm{TP}^\mathrm{noncommutative,pcomplete}(?/A)_{\text{functionalanalytic,KKM},\text{BBM,formalanalytification,nc}}(\mathcal{O}).
\end{align}	
Then as in \cite[Definition 8.2]{12BS} we have the following 'preperfectoidization':
\begin{align}
&(\mathcal{O})^{\text{preperfectoidization}}\\
&:=\mathrm{Colimit}(\mathrm{Kan}_{\mathrm{Left}}\Delta_{?/A,\text{functionalanalytic,KKM},\text{BBM,formalanalytification,nc}}(\mathcal{O})\rightarrow \\
&\phi_*\mathrm{Kan}_{\mathrm{Left}}\Delta_{?/A,\text{functionalanalytic,KKM},\text{BBM,formalanalytification,nc}}(\mathcal{O})\\
&\rightarrow \phi_* \phi_*\mathrm{Kan}_{\mathrm{Left}}\Delta_{?/A,\text{functionalanalytic,KKM},\text{BBM,formalanalytification,nc}}(\mathcal{O})\rightarrow...)^{\text{BBM,formalanalytification,nc}},	
\end{align}
after taking the formal series ring left Kan extension analytification from \cite[Section 4.2]{BBM}, which is defined by taking the left Kan extension to all the $(\infty,1)$-ring object in the $\infty$-derived category of all $A$-modules from formal series rings over $A$. Then we define the corresponding 'perfectoidization':
\begin{align}
&(\mathcal{O})^{\text{perfectoidization}}\\
&:=\mathrm{Colimit}(\mathrm{Kan}_{\mathrm{Left}}\Delta_{?/A,\text{functionalanalytic,KKM},\text{BBM,formalanalytification,nc}}(\mathcal{O})\longrightarrow \\
&\phi_*\mathrm{Kan}_{\mathrm{Left}}\Delta_{?/A,\text{functionalanalytic,KKM},\text{BBM,formalanalytification,nc}}(\mathcal{O})\\
&\longrightarrow \phi_* \phi_*\mathrm{Kan}_{\mathrm{Left}}\Delta_{?/A,\text{functionalanalytic,KKM},\text{BBM,formalanalytification,nc}}(\mathcal{O})\longrightarrow...\\
&)^{\text{BBM,formalanalytification,nc}}\times A/I.	
\end{align}
Furthermore one can take derived $(p,I)$-completion to achieve the derived $(p,I)$-completed versions:
\begin{align}
\mathcal{O}^\text{preperfectoidization,derivedcomplete}:=(\mathcal{O}^\text{preperfectoidization})^{\wedge},\\
\mathcal{O}^\text{perfectoidization,derivedcomplete}:=\mathcal{O}^\text{preperfectoidization,derivedcomplete}\times A/I.\\
\end{align}
These are large $(\infty,1)$-noncommutative algebra objects in the corresponding categories as in the above, attached to also large $(\infty,1)$-noncommutative algebra objects. When we apply this to the corresponding sub-$(\infty,1)$-categories of Banach perfectoid objects as in \cite{BMS2}, \cite{GR}, \cite{12KL1}, \cite{12KL2}, \cite{12Ked1}, \cite{12Sch3},  we will recover the corresponding noncommutative analogues of the distinguished elemental deformation processes defined in \cite{BMS2}, \cite{GR}, \cite{12KL1}, \cite{12KL2}, \cite{12Ked1}, \cite{12Sch3}.
\end{definition}

\

\begin{remark}
One can then define such ring $\mathcal{O}$ to be \textit{preperfectoid} if we have the equivalence:
\begin{align}
\mathcal{O}^{\text{preperfectoidization}} \overset{\sim}{\longrightarrow}	\mathcal{O}.
\end{align}
One can then define such ring $\mathcal{O}$ to be \textit{perfectoid} if we have the equivalence:
\begin{align}
\mathcal{O}^{\text{preperfectoidization}}\times A/I \overset{\sim}{\longrightarrow}	\mathcal{O}.
\end{align}
	
\end{remark}

\

\begin{definition}
\indent One can actually define the derived prismatic cohomology presheaves through derived topological Hochschild cohomology presheaves, derived topological period cohomology presheaves and derived topological cyclic cohomology presheaves as in \cite[Section 2.2, Section 2.3]{12BMS}, \cite[Theorem 1.13]{12BS}:
\begin{align}
	\mathrm{Kan}_{\mathrm{Left}}\mathrm{THH}^\mathrm{noncommutative},\mathrm{Kan}_{\mathrm{Left}}\mathrm{TP}^\mathrm{noncommutative},\mathrm{Kan}_{\mathrm{Left}}\mathrm{TC}^\mathrm{noncommutative},
\end{align}
on the following $(\infty,1)$-compactly generated closures of the corresponding polynomials\footnote{Definitely, we need to put certain norms over in some relatively canonical way, as in \cite[Section 4.2]{BBM} one can basically consider rigid ones and dagger ones, and so on. Again in this noncommutative setting we do not actually need to fix the type of the analytification, as in commutative setting since we are going to apply directly construction from Bhatt-Scholze and Koshikawa \cite{12BS} and \cite{12Ko1}, though the derived characterization of the prismatic cohomology might not need to be restricted to the corresponding $p$-adic formal scheme situations.} given over $A/I$ with a chosen prism $(A,I)$\footnote{In all the following, we assume this prism to be bounded and satisfy that $A/I$ is Banach.}:
\begin{align}
\mathrm{Proj}^\text{smoothformalseriesclosure}\infty-\mathrm{Toposes}^{\mathrm{ringed},\mathrm{Noncommutativealgebra}_{\mathrm{simplicial}}(\mathrm{Ind}\mathrm{Seminormed}_{A/I})}_{\mathrm{Noncommutativealgebra}_{\mathrm{simplicial}}(\mathrm{Ind}\mathrm{Seminormed}_{A/I})^\mathrm{opposite},\mathrm{Grothendiecktopology,homotopyepimorphism}}. \\
\mathrm{Proj}^\text{smoothformalseriesclosure}\infty-\mathrm{Toposes}^{\mathrm{ringed},\mathrm{Noncommutativealgebra}_{\mathrm{simplicial}}(\mathrm{Ind}^m\mathrm{Seminormed}_{A/I})}_{\mathrm{Noncommutativealgebra}_{\mathrm{simplicial}}(\mathrm{Ind}^m\mathrm{Seminormed}_{A/I})^\mathrm{opposite},\mathrm{Grothendiecktopology,homotopyepimorphism}}.\\
\mathrm{Proj}^\text{smoothformalseriesclosure}\infty-\mathrm{Toposes}^{\mathrm{ringed},\mathrm{Noncommutativealgebra}_{\mathrm{simplicial}}(\mathrm{Ind}\mathrm{Normed}_{A/I})}_{\mathrm{Noncommutativealgebra}_{\mathrm{simplicial}}(\mathrm{Ind}\mathrm{Normed}_{A/I})^\mathrm{opposite},\mathrm{Grothendiecktopology,homotopyepimorphism}}.\\
\mathrm{Proj}^\text{smoothformalseriesclosure}\infty-\mathrm{Toposes}^{\mathrm{ringed},\mathrm{Noncommutativealgebra}_{\mathrm{simplicial}}(\mathrm{Ind}^m\mathrm{Normed}_{A/I})}_{\mathrm{Noncommutativealgebra}_{\mathrm{simplicial}}(\mathrm{Ind}^m\mathrm{Normed}_{A/I})^\mathrm{opposite},\mathrm{Grothendiecktopology,homotopyepimorphism}}.\\
\mathrm{Proj}^\text{smoothformalseriesclosure}\infty-\mathrm{Toposes}^{\mathrm{ringed},\mathrm{Noncommutativealgebra}_{\mathrm{simplicial}}(\mathrm{Ind}\mathrm{Banach}_{A/I})}_{\mathrm{Noncommutativealgebra}_{\mathrm{simplicial}}(\mathrm{Ind}\mathrm{Banach}_{A/I})^\mathrm{opposite},\mathrm{Grothendiecktopology,homotopyepimorphism}}.\\
\mathrm{Proj}^\text{smoothformalseriesclosure}\infty-\mathrm{Toposes}^{\mathrm{ringed},\mathrm{Noncommutativealgebra}_{\mathrm{simplicial}}(\mathrm{Ind}^m\mathrm{Banach}_{A/I})}_{\mathrm{Noncommutativealgebra}_{\mathrm{simplicial}}(\mathrm{Ind}^m\mathrm{Banach}_{A/I})^\mathrm{opposite},\mathrm{Grothendiecktopology,homotopyepimorphism}}. 
\end{align}
We call the corresponding functors are derived functional analytic Hochschild cohomology presheaves, derived functional analytic period cohomology presheaves and derived functional analytic cyclic cohomology presheaves, which we are going to denote these presheaves as in the following for any $\infty$-ringed topos $(\mathbb{X},\mathcal{O})=\underset{i}{\text{homotopycolimit}}(\mathbb{X}_i,\mathcal{O}_i)$:
\begin{align}
	&\mathrm{Kan}_{\mathrm{Left}}\mathrm{THH}^\mathrm{noncommutative}_{\text{functionalanalytic,KKM},\text{BBM,formalanalytification,nc}}(\mathcal{O}):=\\
	&(\underset{i}{\text{homotopycolimit}}_{\text{sifted},\text{derivedcategory}_{\infty}(A/I-\text{Module})}\mathrm{Kan}_{\mathrm{Left}}\mathrm{THH}^\mathrm{noncommutative}_{\text{functionalanalytic,KKM}}(\mathcal{O}_i)\\
	&)_\text{BBM,formalanalytification,nc},\\
	&\mathrm{Kan}_{\mathrm{Left}}\mathrm{TP}^\mathrm{noncommutative}_{\text{functionalanalytic,KKM},\text{BBM,formalanalytification,nc}}(\mathcal{O}):=\\
	&(\underset{i}{\text{homotopycolimit}}_{\text{sifted},\text{derivedcategory}_{\infty}(A/I-\text{Module})}\mathrm{Kan}_{\mathrm{Left}}\mathrm{TP}^\mathrm{noncommutative}_{\text{functionalanalytic,KKM}}(\mathcal{O}_i)\\
	&)_\text{BBM,formalanalytification,nc},\\
	&\mathrm{Kan}_{\mathrm{Left}}\mathrm{TC}^\mathrm{noncommutative}_{\text{functionalanalytic,KKM},\text{BBM,formalanalytification,nc}}(\mathcal{O}):=\\
	&(\underset{i}{\text{homotopycolimit}}_{\text{sifted},\text{derivedcategory}_{\infty}(A/I-\text{Module})}\mathrm{Kan}_{\mathrm{Left}}\mathrm{TC}^\mathrm{noncommutative}_{\text{functionalanalytic,KKM}}(\mathcal{O}_i)\\
	&)_\text{BBM,formalanalytification,nc},
\end{align}
by writing any object $\mathcal{O}$ as the corresponding colimit\footnote{Here we assume that in the following the presheaves $\mathcal{O}_i$ are taking values in the derived $p$-completed $\mathbb{E}_1$-algebras over $A/I$, which then are generated by those $p$-adic formal series ring $A/I\left<Z_1,...,Z_n\right>$, $n=1,2,...$ by the corresponding derived colimit completion with free variables $Z_1,...,Z_n$, $n=1,2,...$.} 
\begin{center}
$\underset{i}{\text{homotopycolimit}}_\text{sifted}\mathcal{O}_i$.
\end{center}
These are quite large $(\infty,1)$-commutative ring objects in the corresponding $(\infty,1)$-categories for $?=A/I$ over $(\mathbb{X},\mathcal{O})$:

 \begin{align}
\mathrm{Ind}\mathrm{\sharp Quasicoherent}^{\text{presheaf}}_{\mathrm{Ind}^\text{smoothformalseriesclosure}\infty-\mathrm{Toposes}^{\mathrm{ringed},\mathrm{Noncommutativealgebra}_{\mathrm{simplicial}}(\mathrm{Ind}\mathrm{Seminormed}_?)}_{\mathrm{Noncommutativealgebra}_{\mathrm{simplicial}}(\mathrm{Ind}\mathrm{Seminormed}_?)^\mathrm{opposite},\mathrm{Grotopology,homotopyepimorphism}}}. \\
\mathrm{Ind}\mathrm{\sharp Quasicoherent}^{\text{presheaf}}_{\mathrm{Ind}^\text{smoothformalseriesclosure}\infty-\mathrm{Toposes}^{\mathrm{ringed},\mathrm{Noncommutativealgebra}_{\mathrm{simplicial}}(\mathrm{Ind}^m\mathrm{Seminormed}_?)}_{\mathrm{Noncommutativealgebra}_{\mathrm{simplicial}}(\mathrm{Ind}^m\mathrm{Seminormed}_?)^\mathrm{opposite},\mathrm{Grotopology,homotopyepimorphism}}}.\\
\mathrm{Ind}\mathrm{\sharp Quasicoherent}^{\text{presheaf}}_{\mathrm{Ind}^\text{smoothformalseriesclosure}\infty-\mathrm{Toposes}^{\mathrm{ringed},\mathrm{Noncommutativealgebra}_{\mathrm{simplicial}}(\mathrm{Ind}\mathrm{Normed}_?)}_{\mathrm{Noncommutativealgebra}_{\mathrm{simplicial}}(\mathrm{Ind}\mathrm{Normed}_?)^\mathrm{opposite},\mathrm{Grotopology,homotopyepimorphism}}}.\\
\mathrm{Ind}\mathrm{\sharp Quasicoherent}^{\text{presheaf}}_{\mathrm{Ind}^\text{smoothformalseriesclosure}\infty-\mathrm{Toposes}^{\mathrm{ringed},\mathrm{Noncommutativealgebra}_{\mathrm{simplicial}}(\mathrm{Ind}^m\mathrm{Normed}_?)}_{\mathrm{Noncommutativealgebra}_{\mathrm{simplicial}}(\mathrm{Ind}^m\mathrm{Normed}_?)^\mathrm{opposite},\mathrm{Grotopology,homotopyepimorphism}}}.\\
\mathrm{Ind}\mathrm{\sharp Quasicoherent}^{\text{presheaf}}_{\mathrm{Ind}^\text{smoothformalseriesclosure}\infty-\mathrm{Toposes}^{\mathrm{ringed},\mathrm{Noncommutativealgebra}_{\mathrm{simplicial}}(\mathrm{Ind}\mathrm{Banach}_?)}_{\mathrm{Noncommutativealgebra}_{\mathrm{simplicial}}(\mathrm{Ind}\mathrm{Banach}_?)^\mathrm{opposite},\mathrm{Grotopology,homotopyepimorphism}}}.\\
\mathrm{Ind}\mathrm{\sharp Quasicoherent}^{\text{presheaf}}_{\mathrm{Ind}^\text{smoothformalseriesclosure}\infty-\mathrm{Toposes}^{\mathrm{ringed},\mathrm{Noncommutativealgebra}_{\mathrm{simplicial}}(\mathrm{Ind}^m\mathrm{Banach}_?)}_{\mathrm{Noncommutativealgebra}_{\mathrm{simplicial}}(\mathrm{Ind}^m\mathrm{Banach}_?)^\mathrm{opposite},\mathrm{Grotopology,homotopyepimorphism}}},\\ 
\end{align}
after taking the formal series ring left Kan extension analytification from \cite[Section 4.2]{BBM}, which is defined by taking the left Kan extension to all the $(\infty,1)$-ring objects in the $\infty$-derived category of all $A$-modules from formal series rings over $A$.
\end{definition}

\begin{assumption}\mbox{(Technical Assumption)} 
Here we assume that in the following the presheaves $\mathcal{O}_i$ are taking values in the derived $p$-completed $\mathbb{E}_1$-algebras over $A/I$, which then are generated by those $p$-adic formal series ring $A/I\left<Z_1,...,Z_n\right>$, $n=1,2,...$ by the corresponding derived colimit completion with free variables $Z_1,...,Z_n$, $n=1,2,...$. However this assumption does not really matter so significantly in this noncommutative situation since we are just taking the direct definition through $\text{TP}$.  	
\end{assumption}

\

\begin{definition}
Then we can in the same fashion consider the corresponding derived noncommutative prismatic complex presheaves \cite[Construction 7.6]{12BS}\footnote{One just applies \cite[Construction 7.6]{12BS} and then takes the left Kan extensions.} for the commutative algebras as in the above (for a given prism $(A,I)$):
\begin{align}
\mathrm{Kan}_{\mathrm{Left}}\Delta_{?/A}:=\mathrm{Kan}_{\mathrm{Left}}\mathrm{TP}^\mathrm{noncommutative,pcomplete}(?/A),	
\end{align}
\footnote{The motivation for this definition comes from the corresponding commutative picture, namely the corresponding topological characterization of the corresponding prismatic cohomology and the corresponding completed version by using the corresponding Nygaard filtrations. See \cite[Theorem 1.13]{12BS}.}by the regular corresponding left Kan extension techniques on the following $(\infty,1)$-compactly generated closures of the corresponding polynomials given over $A/I$ with a chosen prism $(A,I)$:\\

\begin{align}
\mathrm{Proj}^\text{smoothformalseriesclosure}\infty-\mathrm{Toposes}^{\mathrm{ringed},\mathrm{Noncommutativealgebra}_{\mathrm{simplicial}}(\mathrm{Ind}\mathrm{Seminormed}_{A/I})}_{\mathrm{Noncommutativealgebra}_{\mathrm{simplicial}}(\mathrm{Ind}\mathrm{Seminormed}_{A/I})^\mathrm{opposite},\mathrm{Grothendiecktopology,homotopyepimorphism}}. \\
\mathrm{Proj}^\text{smoothformalseriesclosure}\infty-\mathrm{Toposes}^{\mathrm{ringed},\mathrm{Noncommutativealgebra}_{\mathrm{simplicial}}(\mathrm{Ind}^m\mathrm{Seminormed}_{A/I})}_{\mathrm{Noncommutativealgebra}_{\mathrm{simplicial}}(\mathrm{Ind}^m\mathrm{Seminormed}_{A/I})^\mathrm{opposite},\mathrm{Grothendiecktopology,homotopyepimorphism}}.\\
\mathrm{Proj}^\text{smoothformalseriesclosure}\infty-\mathrm{Toposes}^{\mathrm{ringed},\mathrm{Noncommutativealgebra}_{\mathrm{simplicial}}(\mathrm{Ind}\mathrm{Normed}_{A/I})}_{\mathrm{Noncommutativealgebra}_{\mathrm{simplicial}}(\mathrm{Ind}\mathrm{Normed}_{A/I})^\mathrm{opposite},\mathrm{Grothendiecktopology,homotopyepimorphism}}.\\
\mathrm{Proj}^\text{smoothformalseriesclosure}\infty-\mathrm{Toposes}^{\mathrm{ringed},\mathrm{Noncommutativealgebra}_{\mathrm{simplicial}}(\mathrm{Ind}^m\mathrm{Normed}_{A/I})}_{\mathrm{Noncommutativealgebra}_{\mathrm{simplicial}}(\mathrm{Ind}^m\mathrm{Normed}_{A/I})^\mathrm{opposite},\mathrm{Grothendiecktopology,homotopyepimorphism}}.\\
\mathrm{Proj}^\text{smoothformalseriesclosure}\infty-\mathrm{Toposes}^{\mathrm{ringed},\mathrm{Noncommutativealgebra}_{\mathrm{simplicial}}(\mathrm{Ind}\mathrm{Banach}_{A/I})}_{\mathrm{Noncommutativealgebra}_{\mathrm{simplicial}}(\mathrm{Ind}\mathrm{Banach}_{A/I})^\mathrm{opposite},\mathrm{Grothendiecktopology,homotopyepimorphism}}.\\
\mathrm{Proj}^\text{smoothformalseriesclosure}\infty-\mathrm{Toposes}^{\mathrm{ringed},\mathrm{Noncommutativealgebra}_{\mathrm{simplicial}}(\mathrm{Ind}^m\mathrm{Banach}_{A/I})}_{\mathrm{Noncommutativealgebra}_{\mathrm{simplicial}}(\mathrm{Ind}^m\mathrm{Banach}_{A/I})^\mathrm{opposite},\mathrm{Grothendiecktopology,homotopyepimorphism}}. 
\end{align}
We call the corresponding functors functional analytic derived prismatic complex presheaves which we are going to denote that as in the following:
\begin{align}
\mathrm{Kan}_{\mathrm{Left}}&\Delta_{?/A,\text{functionalanalytic,KKM},\text{BBM,formalanalytification,nc}}\\
&:=\mathrm{Kan}_{\mathrm{Left}}\mathrm{TP}^\mathrm{noncommutative,pcomplete}(?/A)_{\text{functionalanalytic,KKM},\text{BBM,formalanalytification,nc}}.	
\end{align}
This would mean the following definition{\footnote{Before the Ben-Bassat-Mukherjee $p$-adic formal analytification we take the corresponding $p$-completion.}}:
\begin{align}
\mathrm{Kan}_{\mathrm{Left}}&\Delta_{?/A,\text{functionalanalytic,KKM},\text{BBM,formalanalytification,nc}}(\mathcal{O})\\
&:=	((\underset{i}{\text{homotopycolimit}}_{\text{sifted},\text{derivedcategory}_{\infty}(A/I-\text{Module})}\mathrm{Kan}_{\mathrm{Left}}\Delta_{?/A,\text{functionalanalytic,KKM}}(\mathcal{O}_i))^\wedge\\
&)_\text{BBM,formalanalytification,nc}
\end{align}
by writing any object $\mathcal{O}$ as the corresponding colimit 
\begin{center}
$\underset{i}{\text{homotopycolimit}}_\text{sifted}\mathcal{O}_i$.
\end{center}
These are quite large $(\infty,1)$-commutative ring objects in the corresponding $(\infty,1)$-categories for $R=A/I$:
\begin{align}
\mathrm{Ind}\mathrm{\sharp Quasicoherent}^{\text{presheaf}}_{\mathrm{Ind}^\text{smoothformalseriesclosure}\infty-\mathrm{Toposes}^{\mathrm{ringed},\mathrm{Noncommutativealgebra}_{\mathrm{simplicial}}(\mathrm{Ind}\mathrm{Seminormed}_R)}_{\mathrm{Noncommutativealgebra}_{\mathrm{simplicial}}(\mathrm{Ind}\mathrm{Seminormed}_R)^\mathrm{opposite},\mathrm{Grotopology,homotopyepimorphism}}}. \\
\mathrm{Ind}\mathrm{\sharp Quasicoherent}^{\text{presheaf}}_{\mathrm{Ind}^\text{smoothformalseriesclosure}\infty-\mathrm{Toposes}^{\mathrm{ringed},\mathrm{Noncommutativealgebra}_{\mathrm{simplicial}}(\mathrm{Ind}^m\mathrm{Seminormed}_R)}_{\mathrm{Noncommutativealgebra}_{\mathrm{simplicial}}(\mathrm{Ind}^m\mathrm{Seminormed}_R)^\mathrm{opposite},\mathrm{Grotopology,homotopyepimorphism}}}.\\
\mathrm{Ind}\mathrm{\sharp Quasicoherent}^{\text{presheaf}}_{\mathrm{Ind}^\text{smoothformalseriesclosure}\infty-\mathrm{Toposes}^{\mathrm{ringed},\mathrm{Noncommutativealgebra}_{\mathrm{simplicial}}(\mathrm{Ind}\mathrm{Normed}_R)}_{\mathrm{Noncommutativealgebra}_{\mathrm{simplicial}}(\mathrm{Ind}\mathrm{Normed}_R)^\mathrm{opposite},\mathrm{Grotopology,homotopyepimorphism}}}.\\
\mathrm{Ind}\mathrm{\sharp Quasicoherent}^{\text{presheaf}}_{\mathrm{Ind}^\text{smoothformalseriesclosure}\infty-\mathrm{Toposes}^{\mathrm{ringed},\mathrm{Noncommutativealgebra}_{\mathrm{simplicial}}(\mathrm{Ind}^m\mathrm{Normed}_R)}_{\mathrm{Noncommutativealgebra}_{\mathrm{simplicial}}(\mathrm{Ind}^m\mathrm{Normed}_R)^\mathrm{opposite},\mathrm{Grotopology,homotopyepimorphism}}}.\\
\mathrm{Ind}\mathrm{\sharp Quasicoherent}^{\text{presheaf}}_{\mathrm{Ind}^\text{smoothformalseriesclosure}\infty-\mathrm{Toposes}^{\mathrm{ringed},\mathrm{Noncommutativealgebra}_{\mathrm{simplicial}}(\mathrm{Ind}\mathrm{Banach}_R)}_{\mathrm{Noncommutativealgebra}_{\mathrm{simplicial}}(\mathrm{Ind}\mathrm{Banach}_R)^\mathrm{opposite},\mathrm{Grotopology,homotopyepimorphism}}}.\\
\mathrm{Ind}\mathrm{\sharp Quasicoherent}^{\text{presheaf}}_{\mathrm{Ind}^\text{smoothformalseriesclosure}\infty-\mathrm{Toposes}^{\mathrm{ringed},\mathrm{Noncommutativealgebra}_{\mathrm{simplicial}}(\mathrm{Ind}^m\mathrm{Banach}_R)}_{\mathrm{Noncommutativealgebra}_{\mathrm{simplicial}}(\mathrm{Ind}^m\mathrm{Banach}_R)^\mathrm{opposite},\mathrm{Grotopology,homotopyepimorphism}}},\\ 
\end{align}
after taking the formal series ring left Kan extension analytification from \cite[Section 4.2]{BBM}, which is defined by taking the left Kan extension to all the $(\infty,1)$-ring objects in the $\infty$-derived category of all $A$-modules from formal series rings over $A$.\\
\end{definition}

\

\indent Then as in \cite[Definition 8.2]{12BS} we consider the corresponding noncommutative perfectoidization in this analytic setting.

\begin{definition}
Let $(A,I)$ be a perfectoid prism, and we consider any $\mathrm{E}_\infty$-ring $\mathcal{O}$ in the following
\begin{align}
\mathrm{Proj}^\text{smoothformalseriesclosure}\infty-\mathrm{Toposes}^{\mathrm{ringed},\mathrm{Noncommutativealgebra}_{\mathrm{simplicial}}(\mathrm{Ind}\mathrm{Seminormed}_{A/I})}_{\mathrm{Noncommutativealgebra}_{\mathrm{simplicial}}(\mathrm{Ind}\mathrm{Seminormed}_{A/I})^\mathrm{opposite},\mathrm{Grothendiecktopology,homotopyepimorphism}}. \\
\mathrm{Proj}^\text{smoothformalseriesclosure}\infty-\mathrm{Toposes}^{\mathrm{ringed},\mathrm{Noncommutativealgebra}_{\mathrm{simplicial}}(\mathrm{Ind}^m\mathrm{Seminormed}_{A/I})}_{\mathrm{Noncommutativealgebra}_{\mathrm{simplicial}}(\mathrm{Ind}^m\mathrm{Seminormed}_{A/I})^\mathrm{opposite},\mathrm{Grothendiecktopology,homotopyepimorphism}}.\\
\mathrm{Proj}^\text{smoothformalseriesclosure}\infty-\mathrm{Toposes}^{\mathrm{ringed},\mathrm{Noncommutativealgebra}_{\mathrm{simplicial}}(\mathrm{Ind}\mathrm{Normed}_{A/I})}_{\mathrm{Noncommutativealgebra}_{\mathrm{simplicial}}(\mathrm{Ind}\mathrm{Normed}_{A/I})^\mathrm{opposite},\mathrm{Grothendiecktopology,homotopyepimorphism}}.\\
\mathrm{Proj}^\text{smoothformalseriesclosure}\infty-\mathrm{Toposes}^{\mathrm{ringed},\mathrm{Noncommutativealgebra}_{\mathrm{simplicial}}(\mathrm{Ind}^m\mathrm{Normed}_{A/I})}_{\mathrm{Noncommutativealgebra}_{\mathrm{simplicial}}(\mathrm{Ind}^m\mathrm{Normed}_{A/I})^\mathrm{opposite},\mathrm{Grothendiecktopology,homotopyepimorphism}}.\\
\mathrm{Proj}^\text{smoothformalseriesclosure}\infty-\mathrm{Toposes}^{\mathrm{ringed},\mathrm{Noncommutativealgebra}_{\mathrm{simplicial}}(\mathrm{Ind}\mathrm{Banach}_{A/I})}_{\mathrm{Noncommutativealgebra}_{\mathrm{simplicial}}(\mathrm{Ind}\mathrm{Banach}_{A/I})^\mathrm{opposite},\mathrm{Grothendiecktopology,homotopyepimorphism}}.\\
\mathrm{Proj}^\text{smoothformalseriesclosure}\infty-\mathrm{Toposes}^{\mathrm{ringed},\mathrm{Noncommutativealgebra}_{\mathrm{simplicial}}(\mathrm{Ind}^m\mathrm{Banach}_{A/I})}_{\mathrm{Noncommutativealgebra}_{\mathrm{simplicial}}(\mathrm{Ind}^m\mathrm{Banach}_{A/I})^\mathrm{opposite},\mathrm{Grothendiecktopology,homotopyepimorphism}}. 
\end{align}
Then consider the derived prismatic object:
\begin{align}
\mathrm{Kan}_{\mathrm{Left}}\Delta_{?/A,\text{functionalanalytic,KKM},\text{BBM,formalanalytification,nc}}(\mathcal{O}).
\end{align}	
Then as in \cite[Definition 8.2]{12BS} we have the following preperfectoidization:
\begin{align}
&(\mathcal{O})^{\text{preperfectoidization}}\\
&:=\mathrm{Colimit}(\mathrm{Kan}_{\mathrm{Left}}\Delta_{?/A,\text{functionalanalytic,KKM},\text{BBM,formalanalytification,nc}}(\mathcal{O})\rightarrow \\
&\phi_*\mathrm{Kan}_{\mathrm{Left}}\Delta_{?/A,\text{functionalanalytic,KKM},\text{BBM,formalanalytification,nc}}(\mathcal{O})\\
&\rightarrow \phi_* \phi_*\mathrm{Kan}_{\mathrm{Left}}\Delta_{?/A,\text{functionalanalytic,KKM},\text{BBM,formalanalytification,nc}}(\mathcal{O})\rightarrow...)^{\text{BBM,formalanalytification,nc}},	
\end{align}
after taking the formal series ring left Kan extension analytification from \cite[Section 4.2]{BBM}, which is defined by taking the left Kan extension to all the $(\infty,1)$-ring object in the $\infty$-derived category of all $A$-modules from formal series rings over $A$. Then we define the corresponding perfectoidization:
\begin{align}
&(\mathcal{O})^{\text{perfectoidization}}\\
&:=\mathrm{Colimit}(\mathrm{Kan}_{\mathrm{Left}}\Delta_{?/A,\text{functionalanalytic,KKM},\text{BBM,formalanalytification,nc}}(\mathcal{O})\longrightarrow \\
&\phi_*\mathrm{Kan}_{\mathrm{Left}}\Delta_{?/A,\text{functionalanalytic,KKM},\text{BBM,formalanalytification,nc}}(\mathcal{O})\\
&\longrightarrow \phi_* \phi_*\mathrm{Kan}_{\mathrm{Left}}\Delta_{?/A,\text{functionalanalytic,KKM},\text{BBM,formalanalytification,nc}}(\mathcal{O})\longrightarrow...)^{\text{BBM,formalanalytification,nc}}\times A/I.	
\end{align}
Furthermore one can take derived $(p,I)$-completion to achieve the derived $(p,I)$-completed versions:
\begin{align}
\mathcal{O}^\text{preperfectoidization,derivedcomplete}:=(\mathcal{O}^\text{preperfectoidization})^{\wedge},\\
\mathcal{O}^\text{perfectoidization,derivedcomplete}:=\mathcal{O}^\text{preperfectoidization,derivedcomplete}\times A/I.\\
\end{align}
These are large $(\infty,1)$-commutative algebra objects in the corresponding categories as in the above, attached to also large $(\infty,1)$-commutative algebra objects. When we apply this to the corresponding sub-$(\infty,1)$-categories of Banach perfectoid objects in \cite{BMS2}, \cite{GR}, \cite{12KL1}, \cite{12KL2}, \cite{12Ked1}, \cite{12Sch3},  we will have the corresponding noncommutative analogues of the distinguished elemental deformation processes defined in \cite{BMS2}, \cite{GR}, \cite{12KL1}, \cite{12KL2}, \cite{12Ked1}, \cite{12Sch3}. 
\end{definition}

\begin{remark}
One can then define such ring $\mathcal{O}$ to be \textit{preperfectoid} if we have the equivalence:
\begin{align}
\mathcal{O}^{\text{preperfectoidization}} \overset{\sim}{\longrightarrow}	\mathcal{O}.
\end{align}
One can then define such ring $\mathcal{O}$ to be \textit{perfectoid} if we have the equivalence:
\begin{align}
\mathcal{O}^{\text{preperfectoidization}}\times A/I \overset{\sim}{\longrightarrow}	\mathcal{O}.
\end{align}
	
\end{remark}

\newpage

\subsection*{Acknowledgements}

This paper is based on the Chapter 12 of the author's 2021 UCSD dissertation \cite{XT2} under the advice and direction from Professor Kedlaya. We would like to thank Professor Kedlaya for some motivating discussion in the preparation for this paper especially on the corresponding Robba bundles and Frobenius bundles, as well as the corresponding $\infty$-categorical and homotopical motivic constructions being considered.


\newpage


\begin{thebibliography}{10}

\bibitem[Huber1]{12Huber1} Huber, Roland. "A generalization of formal schemes and rigid analytic varieties." Mathematische Zeitschrift 217, no. 1 (1994): 513-551.

\bibitem[Huber2]{12Huber2} Huber, Roland. \'Etale cohomology of rigid analytic varieties and adic spaces. Vol. 30. Springer, 2013.

\bibitem[Ked1]{12Ked1} Kedlaya, K. "Sheaves, stacks, and shtukas, lecture notes from the 2017 Arizona Winter School: Perfectoid Spaces." Math. Surveys and Monographs 242.

\bibitem[HK]{12HK} Hansen, David, and Kiran S. Kedlaya. "Sheafiness criteria for Huber rings." (2020). Https://kskedlaya.org/papers/.

\bibitem[XT1]{12XT1} Tong, Xin. "Topologization and Functional Analytification I: Intrinsic Morphisms of Commutative Algebras." Https://arxiv.org/abs/2102.10766.

\bibitem[BBBK]{12BBBK} Bambozzi, Federico, Oren Ben-Bassat, and Kobi Kremnizer. "Analytic geometry over $\mathbb{F}_1$ and the Fargues-Fontaine curve." Advances in Mathematics 356 (2019): 106815.


\bibitem[BK]{12BK} Bambozzi, Federico, and Kobi Kremnizer. "On the Sheafyness Property of Spectra of Banach Rings." arXiv preprint arXiv:2009.13926 (2020).


\bibitem[BBK]{12BBK} Ben-Bassat, Oren, and Kobi Kremnizer. "Fr\'echet Modules and Descent." arXiv preprint arXiv:2002.11608 (2020).


\bibitem[B1]{12B1} Bhatt, Bhargav. "$p$-adic derived de Rham cohomology." arXiv preprint arXiv:1204.6560 (2012).


\bibitem[B2]{12B2} Bhatt, Bhargav. "Lectures on prismatic cohomology." Http://www-personal.umich.edu/~bhattb/teaching/prismatic-columbia/.


\bibitem[G1]{12G1} Guo, Haoyang. "Crystalline Cohomology of Rigid Analytic Spaces." Https://sites.google.com/umich.edu/hyg.


\bibitem[GL]{12GL} Guo, Haoyang, and Shizhang Li. "Period sheaves via derived de Rham cohomology." arXiv preprint arXiv:2008.06143 (2020).


\bibitem[Ill1]{12Ill1} Illusie, Luc. Complexe cotangent et d\'eformations I. Vol. 239. Springer, 2006.

\bibitem[Ill2]{12Ill2} Illusie, L. "Complexe cotangent et d\'eformations II, SLN, 283 (1972)." Zbl0224 13014.


\bibitem[KL1]{12KL1} Kedlaya, Kiran Sridhara, and Ruochuan Liu. 2015. Relative $p$-adic Hodge theory: foundations. Aste\'erisque. Paris, Soci\'et\'e Math\'ematique de France.

\bibitem[KL2]{12KL2} Kedlaya, Kiran S., and Ruochuan Liu. "Relative $p$-adic Hodge theory, II: Imperfect period rings." arXiv preprint arXiv:1602.06899 (2016).


\bibitem[R]{12R} Gregoric, Rok. "The crystalline space and Divided Power Completion." Https://web.ma.utexas.edu/users/gregoric/CrystalMath.pdf.


\bibitem[Bel1]{12Bel1} Bellovin, Rebecca. "Galois representations over pseudorigid spaces." arXiv preprint arXiv:2002.06687 (2020).

\bibitem[Bel2]{12Bel2} Bellovin, Rebecca. "Cohomology of $(\varphi,\Gamma) $-modules over pseudorigid spaces." arXiv preprint arXiv:2102.04820 (2021).



\bibitem[Lu1]{12Lu1} Lurie, Jacob. "Higher algebra. 2014." Preprint, available at http://www.math.harvard.edu/$\sim$lurie (2016).

\bibitem[Lu2]{12Lu2} Lurie, Jacob. "Spectral algebraic geometry." unpublished paper (2018).

\bibitem[Lu3]{Lu3} Lurie, Jacob. Higher Topos Theory (AM-170), Princeton: Princeton University Press, 2009. 

\bibitem[Dr1]{12Dr1} Drinfeld, Vladimir. "A stacky approach to crystals." arXiv preprint arXiv:1810.11853 (2018).

\bibitem[Dr2]{12Dr2} Drinfeld, Vladimir. "Prismatization." arXiv preprint arXiv:2005.04746 (2020). 


\bibitem[L]{12L} Louren\c{c}o, Joao NP. "The Riemannian Hebbarkeitss\" atze for pseudorigid spaces." arXiv preprint arXiv:1711.06903 (2017).


\bibitem[BS1]{12BS} Bhatt, Bhargav, and Peter Scholze. "Prisms and prismatic cohomology." arXiv preprint arXiv:1905.08229 (2019).


\bibitem[BMS]{12BMS} Bhatt, Bhargav, Matthew Morrow, and Peter Scholze. "Topological Hochschild homology and integral $ p $-adic Hodge theory." Publications math\'ematiques de l'IH\'ES 129, no. 1 (2019): 199-310.

\bibitem[BMS2]{BMS2} Bhatt, Bhargav, Matthew Morrow, and Peter Scholze. "Integral $p$-adic Hodge theory." Publications Math\'ematiques de l'IH\'ES 128, 219-397 (2018). 


\bibitem[O]{12O} Olsson, Martin C. "The logarithmic cotangent complex." Mathematische Annalen 333, no. 4 (2005): 859-931.

\bibitem[Pau1]{12Pau1} Paugam, Fr\'ed\'eric. "Overconvergent global analytic geometry." arXiv preprint arXiv:1410.7971 (2014).


\bibitem[DLLZ1]{12DLLZ1} Diao, Hansheng, Kai-Wen Lan, Ruochuan Liu, and Xinwen Zhu. "Logarithmic adic spaces: some foundational results." arXiv preprint arXiv:1912.09836 (2019).

\bibitem[DLLZ2]{12DLLZ2} Diao, Hansheng, Kai-Wen Lan, Ruochuan Liu, and Xinwen Zhu. "Logarithmic Riemann-Hilbert correspondences for rigid varieties." arXiv preprint arXiv:1803.05786 (2018).


\bibitem[CS2]{12CS2} Clausen, Dustin, and Peter Scholze. Lectures on Analytic Geometry. Https://www.math.uni-bonn.de/people/scholze/Notes.html.

\bibitem[CS1]{12CS1} Clausen, Dustin, and Peter Scholze. Lectures on Condensed Mathematics. Https://www.math.uni-bonn.de/people/scholze/Notes.html.






\bibitem[LL]{12LL} Li, Shizhang, and Tong Liu. "Comparison of prismatic cohomology and derived de Rham cohomology." arXiv preprint arXiv:2012.14064 (2020).


\bibitem[FS]{12FS} Fargues, Laurent, and Peter Scholze. "Geometrization of the local Langlands correspondence." arXiv preprint arXiv:2102.13459 (2021).


\bibitem[SP]{12SP} Stacks Project Authors. "Stacks Project." (2021).


\bibitem[Ko1]{12Ko1} Koshikawa, Teruhisa. "Logarithmic Prismatic Cohomology I." arXiv preprint arXiv:2007.14037 (2020).


\bibitem[Sch1]{12Sch1} Scholze, Peter. "\'Etale cohomology of diamonds." arXiv preprint arXiv:1709.07343 (2021).

\bibitem[Sch2]{12Sch2} Scholze, Peter. "$ p $-adic Hodge theory for rigid-analytic varieties." In Forum of Mathematics, Pi, vol. 1. Cambridge University Press, 2013.

\bibitem[Sch3]{12Sch3} Scholze, Peter. "Perfectoid Spaces." Publications math\'ematiques de l'IH\'ES (2012).

\bibitem[GR]{GR} Gabber, O., and  L. Ramero. "Foundations for Almost Ring Theory--Release 7.5." Mathematics (2004). arxiv preprint arXiv: 0409583v13.




\bibitem[L]{12L} Louren\c{c}o, Joao NP. "The Riemannian Hebbarkeitss\"atze for pseudorigid spaces." arXiv preprint arXiv:1711.06903 (2017).


\bibitem[NS]{12NS} Nikolaus, Thomas, and Peter Scholze. "On topological cyclic homology." Acta Mathematica 221, no. 2 (2018): 203-409.



\bibitem[FK]{12FK} Fukaya, Takako and Kazuya Kato "A formulation of conjectures on $p$-adic zeta functions in noncommutative Iwasawa theory." In Proceedings of the St. Petersburg Mathematical Society, pp. 1-85. 2006.


\bibitem[ELS]{12ELS} Eriksen, Eivind, Olav Arnfinn Laudal, and Arvid Siqveland. Noncommutative Deformation Theory. CRC Press, 2017.

\bibitem[Ta]{12Ta} Tate, John. "Rigid analytic spaces." Inventiones mathematicae 12, no. 4 (1971): 257-289.

\bibitem[MP]{12MP} May, J. Peter, and Kate Ponto. More concise algebraic topology: localization, completion, and model categories. University of Chicago Press, 2011.

\bibitem[N]{12N} Neisendorfer, Joseph. Algebraic methods in unstable homotopy theory. Vol. 12. Cambridge University Press, 2010.

\bibitem[BS2]{12BS2} Bhatt, Bhargav, and Peter Scholze. "The pro-\'etale topology for schemes." Ast\'erisque 369 (2015): 99-201.


\bibitem[Bei]{12Bei} Beilinson, Alexander. "$p$-adic periods and derived de Rham cohomology." Journal of the American Mathematical Society 25, no. 3 (2012): 715-738. 

\bibitem[An1]{12An1} Andr\'e, Michel. Homologie des algebres commutatives. Vol. 206. Springer, 2013. 

\bibitem[An2]{12An2} Andr\'e, Michel. M\'ethode simpliciale en algebre homologique et algebre commutative. Vol. 32. Berlin-Heidelberg-New York: Springer, 1967.

\bibitem[Qui]{12Qui} Quillen, Daniel. "On the (co-) homology of commutative rings." In Proc. Symp. Pure Math, vol. 17, no. 2, pp. 65-87. 1970.

\bibitem[M]{M} Mao, Zhouhang. "Revisiting Derived Crystalline Cohomology." ArXiv Preprint ArXiv:2107.02921, 2021.

\bibitem[Ra]{Ra} Raksit, Arpon. "Hochschild Homology and the Derived de Rham Complex Revisited." Arxiv Preprint ArXiv:2007.02576. 

 

\bibitem[KKM]{KKM} Kelly, Jack, Kobi Kremnizer and Devarshi Mukherjee. "Analytic Hochschild-Kostant-Rosenberg Theorem." Arxiv preprint arXiv: 2111.03502.

\bibitem[BBM]{BBM} Ben-Bassat, Oren and Devarshi Mukherjee.  "Analytification, Localization and Homotopy Epimorphisms." ArXiv preprint arXiv: 2111.04184.


\bibitem[XT2]{XT2} Xin Tong. "Geometric and Representation Theoretic Aspects of $p$-adic Motives." University of California San Diego Dissertation, 2021.

\bibitem[Ked2]{Ked2} Kedlaya, Kiran. (2015). "Reified valuations and adic spectra." Research in Number Theory. 1. 10.1007/s40993-015-0021-7. 

\bibitem[Kon1]{Kon1} Kontsevich, Maxim. "Noncommutative motives." Talk at the Institute for Advanced Study on the occasion of the 61st birthday of Pierre Deligne (2005).

\bibitem[Ta]{Ta} Tabuada, Gon\c{c}alo. Noncommutative motives. Vol. 63. American Mathematical Soc., 2015.

\bibitem[KR1]{KR1} Kontsevich, Maxim, and Alexander Rosenberg. "Noncommutative spaces." Max-Planck-Inst, 2004.

\bibitem[KR2]{KR2} Kontsevich, Maxim, and Alexander Rosenberg. "Noncommutative stacks." 2004.

\bibitem[G2]{G2} Guo, Haoyang. "Prismatic cohomology of rigid analytic spaces over de Rham period ring." Arxiv preprint arxiv: 0112.14746.



\



\end{thebibliography}
\end{document}